\documentclass{amsart}
\pdfoutput=1
\usepackage[utf8]{inputenc}

\usepackage[T1]{fontenc}    
\usepackage{ae,aecompl}     

\usepackage{url}
\usepackage{color}
\usepackage[hyperfootnotes]{hyperref}
\usepackage{mathrsfs} %
\usepackage{enumerate}
\usepackage{amsmath,upref,amssymb,amsthm}
\usepackage{tensor}
\usepackage[all]{xy}
\RequirePackage{calrsfs}
\CompileMatrices

\numberwithin{equation}{subsection} 
\swapnumbers
\theoremstyle{plain}
\newtheorem*{thm}{Theorem}
\newtheorem*{lem}{Lemma}
\newtheorem*{prop}{Proposition}
\newtheorem*{cor}{Corollary}

\theoremstyle{definition}
\newtheorem*{defn}{Definition} 

\newtheorem*{quest}{Question}
\newtheorem*{exmp}{Example}
\newtheorem*{ass}{Assumption}

\theoremstyle{remark}
\newtheorem*{rem}{Remarks}

\newcommand{\ol}{\overline}

\newcommand{\al}{\alpha}

\newcommand{\Acc}{\operatorname{Acc}}

\newcommand{\bt}{\beta}

\newcommand{\g}{\gamma}

\newcommand{\e}{\epsilon}

\newcommand{\wt}{\widetilde}

\newcommand{\ck}[1]{{#1}^\vee}

\newcommand{\eset}{\emptyset}

\newcommand{\sneq}{\subsetneq}

\newcommand{\Hom}{\text{\rm Hom}}

\newcommand{\Int}{\mathbb {Z}}

\DeclareMathOperator{\stab}{{\mathrm Stab}}

\newcommand{\op}{^{\text{\rm op}}}

\newcommand{\mpair}[1]{\pair{\,#1\,}}

\newcommand{\mset}[1]{\set{\,#1\,}}

\newcommand{\pair}[1]{\langle #1\rangle}

\newcommand{\set}[1]{\{#1\}}

\newcommand{\GL}{\text{\rm GL}}

\newcommand{\seq}{\subseteq}

\newcommand{\sm}{\setminus}

\newenvironment{conds}{
                       
                        \begin{enumerate} }
                     {\end{enumerate} }

\newenvironment{num}{
                      
                      \begin{enumerate} }
                    {\end{enumerate} }
                    
 \newenvironment{subconds}{
                       
                        \begin{enumerate} }
                     {\end{enumerate} }

\newcommand{\Nat}{\mathbb {N}}

\newcommand{\rat}{\mathbb{Q}}

\newcommand{\real}{\mathbb{R}}

\newcommand\mc{\mathcal}
\newcommand\mcKc{\mathcal{K}^{\text{\rm c}}}
\newcommand\mcZc{\mathcal{Z}^{\text{\rm c}}}

  \DeclareMathOperator{\lin}{\mathrm{lin}}
 \DeclareMathOperator{\ccl}{\mathrm{cone}}
  \DeclareMathOperator{\Exp}{\mathrm{Exp}}
  \DeclareMathOperator{\SExp}{\mathrm{SExp}}
   \DeclareMathOperator{\Ext}{\mathrm{Ext}} 
    \DeclareMathOperator{\ASExp}{\mathrm{ASExp}}
   \DeclareMathOperator{\AExp}{\mathrm{AExp}} 
   \DeclareMathOperator{\aff}{\mathrm{aff}}
   \DeclareMathOperator{\conv}{\mathrm{conv}}
   \DeclareMathOperator{\Stab}{{\mathrm{Stab}}}
   \DeclareMathOperator{\ri}{{\mathrm{ri}}}
\DeclareMathOperator{\rb}{{\mathrm{rb}}}
\DeclareMathOperator{\cl}{{\mathrm{cl}}}
\DeclareMathOperator{\wcl}{{\mathrm{wcl}}}
\DeclareMathOperator{\inter}{{\mathrm{int}}}
\DeclareMathOperator{\ray}{{\mathrm{ray}}}
\DeclareMathOperator{\rext}{{R_{\mathrm{ext}}}}
\DeclareMathOperator{\rexp}{{R_{\mathrm{exp}}}}
\newcommand\CC{{\mathcal{C}}}
\newcommand\G{{\Gamma}}
\newcommand\D{{\Delta}}
\newcommand{\sreq}{{\supseteq}}
\newcommand{\flc}{flc }
\newcommand{\wh}{\widehat}
\newcommand{\Zg}{\ol{\mc{Z}}_{\text{\rm gen}}}
\begin{document}
\title[Imaginary cone and  reflection subgroups]{Imaginary cone and reflection subgroups   \\  of Coxeter groups
} 

\author {Matthew J. Dyer}
 \address{Department of Mathematics\\
255 Hurley Building\\ University of Notre
Dame\\ Notre Dame, Indiana, 46556-4618\\U.S.A.}

\email{dyer.1@nd.edu}



%
\begin{abstract}  The imaginary cone of a Kac-Moody Lie algebra is the 
convex hull of   zero and  the  positive  imaginary 
roots. This paper studies the   imaginary cone for a class of root 
systems of  general Coxeter groups $W$.  It is shown that the imaginary cone of a reflection
 subgroup of $W$ is contained in that of $W$, and that for irreducible  infinite $W$ of finite rank, the closed imaginary cone is the only non-zero, closed, pointed $W$-stable cone contained 
 in the pointed cone spanned by the simple roots.  For $W$  of finite rank, various natural 
 notions of  faces of  the imaginary cone are shown to coincide, the face lattice  is explicitly 
 described       in terms of 
 the lattice of facial reflection subgroups and it is shown that the  Tits 
 cone and imaginary cone are related by a duality closely analogous to 
 the standard duality for polyhedral cones, even though neither of them 
 is a closed cone in general. 
 Some of these results have   application, to be given in  sequels to this paper, to 
 dominance order of Coxeter groups, associated automata, and construction of modules for   
  generic Iwahori-Hecke algebras.    \end{abstract}
\maketitle

\section*{}\label{x0}
 The imaginary cone (\cite[Ch 5]{Kac}) of a Kac-Moody Lie algebra is the convex hull of zero 
 and the positive  imaginary roots. The combinatorial characterization of imaginary roots from  \cite{Kac} can be used  to give a definition (\cite{HeeTh}, \cite{HeeIm}) of the imaginary cone which makes sense for possibly non-crystallographic root systems of  Coxeter groups (which do not have imaginary roots in general).

 This paper  systematically studies the imaginary cone for  a class of root systems of general Coxeter 
 groups $W$. As well as in \cite{Kac} for the crystallographic case, some of the basic  facts may   be found   in \cite{HeeIm}, \cite{HeeTh},
  \cite{FuCone} and \cite{HLR}. One  of the  main new results (Theorem \ref{x6.3})
   is   that  the imaginary cone of a reflection subgroup
  $W'$ of $W$ is contained in the imaginary cone of $W$;  the corresponding result for 
  the closures of the imaginary cones is easier to prove  but  much less 
  useful. Another main result is that,
for irreducible infinite  $W$ of finite rank, the   closed imaginary cone is the only non-zero closed pointed $W$-stable cone contained 
 in the pointed cone spanned by the simple roots
  (Theorem
   \ref{x7.6}).
  A third  main result  gives  algebraic descriptions of the face lattices of the imaginary cone and Tits cone. Namely, it is  shown in Section \ref{x11}  that (under mild finiteness and 
  non-degeneracy conditions) several notions of faces of the imaginary cone (and of the Tits cone) coincide, and that  the face lattice of the imaginary cone
 is isomorphic to the lattice of special  facial  subgroups of $W$ (facial subgroups with  no finite 
 components), and is dual to the lattice of faces of the Tits cone. Here, 
 the facial subgroups are defined as  the (parabolic) reflection  
 subgroups arising as stabilizers of points of the Tits cone.   In the framework of standard crystallographic root 
   systems of Kac-Moody Lie algebras, with linearly independent simple roots and simple 
   coroots,  corresponding results on the  faces and the 
   face lattice of the Tits cone (but not the imaginary cone) have been obtained in work of 
   Looijenga,  Slodowy  and Mokler  (\cite{Loo}, \cite{Sl1}, \cite{Sl2}, \cite{Mok1}, \cite{Mok2}, \cite{Mok6}). Those results play a significant role   in the study of the face monoid associated to a  Kac-Moody  group  (see  \cite{Mok7}). 

  The relation between the Tits and imaginary cones  is  quite delicate since, in general, neither is a closed cone and many of the analogous  facts do not even  hold  for closed 
 cones. Another subtlety in the relationships between  
 imaginary cones of reflection subgroups  is that, even  if all parabolic 
 subgroups  of $W$ are facial,  not all  parabolic subgroups of (finite  rank) reflection 
   subgroups $W'$ of $W$ are facial; these relationships have  not been known  in the 
   Kac-Moody case where, due to the nature of standard realizations of  generalized Cartan matrices usually used, work has been mostly restricted to the special case where  
   parabolic and facial subgroups coincide.

  The results involving relationships between the imaginary cone of $W$ 
  and   its dihedral reflection subgroups     have applications,  to be given in subsequent papers,  to the study of dominance order on 
  root systems of general Coxeter groups, and   certain  associated finite 
  state  automata. Some of these have also been obtained  in \cite{Edgar}, \cite{FuCone} or \cite{FuDomHeir}. The results  have connections with generic    Iwahori-Hecke algebras; they lead to a  construction of modules for certain   Iwahori-Hecke algebras which provide a  proof  for finite rank Coxeter systems of    a  
  weakened version of a conjecture of Lusztig  on boundedness of the $a$-function
(recent work of Nanhua Xi \cite{Xi} provides a
  proof of Lusztig's  conjecture itself  in the special case of Coxeter groups with complete Coxeter graphs). 
 
In some detail,  the contents of the paper are as follows.  Section \ref{x1} gives basic properties of the class of reflection representations and 
 root systems of general Coxeter groups which are used in this paper.  
 It is obviously  necessary  to use a class of root systems  closed under passage to 
 subsystems of reflection subgroups, for which simple roots need not be linearly 
 independent.  We take for simplicity a minimal natural  class of root systems with this 
 property,  in real vector spaces $V$ equipped with a 
 $W$-invariant symmetric bilinear form. 
 
   Several  of the basic  results   of this paper extend  to more general classes of 
  possibly non-symmetrizable root systems considered   in \cite{DyRig}, \cite{RankTwo}, 
  and \cite{FuNonortho}, which include also the standard 
  crystallographic root systems of Kac-Moody Lie algebras.
   Although it would be desirable to give these extensions, as 
   suggested by the above comments on what is known in the 
   Kac-Moody setting, we do not go into  this.   The extension of results 
   involving isotropic vectors, totally isotropic  faces of the imaginary cone  etc presents particular difficulties    and should be especially interesting.  

  Section \ref{x2} records basic properties of facial subgroups (the stabilizers in $W$ of points of the Tits cone).
It would  be interesting to develop a suitable theory of root systems  and corresponding  facial subgroups, as in Section \ref{x2}  in (possibly infinite and even  infinite rank) oriented matroids and to determine to what extent results proved here  hold in such an abstract setting.

Section \ref{x3} gives several equivalent  definitions of the imaginary cone $\mc{Z}=\mc{Z}_{W}$ of $W$. The one corresponding most closely  to the original definition in the Kac-Moody setting is as follows: $\mc{Z}_{W}=W\mc{K}_{W}$ where $\mc{K}_{W}$ is the intersection of the negative of the fundamental chamber with the cone $\real_{\geq 0}\Pi$ of non-negative linear combinations of simple roots $\Pi$ of $W$. In particular, $\mc{Z}$ is contained in the negative of the Tits cone
(which is defined as the union of $W$-orbits of points of the 
fundamental chamber), and so
the set $\mc{K}_{W}$ is a fundamental domain for $W$ acting on 
$\mc{Z}$. Also, $\mc{Z}\subseteq \real_{\geq 0}\Pi$.
A fact proved  in Section \ref{x3} which is important for later 
developments is that the imaginary cone of a facial reflection 
subgroup $W'$ is the intersection of the imaginary cone of $W$ 
with the subspace spanned by the roots of $W'$.

In Sections \ref{x4}--\ref{x11},  additional 
finiteness 
and non-degeneracy assumptions  are  imposed on the root 
system; in particular, $W$ is of finite rank throughout these 
sections. These conditions and their  consequences   are 
discussed in Section \ref{x4}.

Section \ref{x5}  establishes some   facts, originally due to Kac (in the 
Kac-Moody setting),  about the closure of the imaginary cone.  
In particular, it is  shown that  the closed imaginary cone is dual to 
the closure of the  Tits cone, and it is the convex hull of the union 
of zero and all the limit rays of rays spanned by positive roots.
The structure of the set of these limit rays  has  also been 
investigated in \cite{HLR}. We  state without proof here an 
important characterization of these limit rays from \cite{HLR}:  
namely,  the set of   limit rays of positive roots is the closure of the
union of  the sets of limit rays of positive roots of dihedral 
reflection subgroups.  

Section \ref{x6} proves the first main new result of this paper,  Theorem \ref{x6.3}, which asserts that $\mc{Z}_{W'}\subseteq \mc{Z}$ if $W'$ is a  
reflection subgroup of $W$. An important corollary is  that if $W''$ 
is a facial reflection subgroup of $W$, then
 $\mc{Z}_{W'}\cap \mc{Z}_{W''}=\mc{Z}_{W'''}$ where 
 $W'''=W'\cap W''$ (which is a facial reflection subgroup of $W'$ by 
 Section \ref{x2}).

Section  \ref{x7} gives our  second  main new result, Theorem \ref{x7.6}, which asserts that for irreducible infinite 
 finite rank $W$, 
the closed imaginary cone is 
   the only non-zero pointed closed $W$-invariant cone contained in the 
   pointed cone spanned by the simple roots. It is proved by   showing
 that   the limit rays of  a $W$-orbit of rays in the imaginary cone include the limit rays of  all
    dihedral reflection subgroups and invoking  the above-mentioned 
    result of \cite{HLR}. The result has additional  significant applications to the 
    study of the $W$-action on the  closed imaginary cone
    which  are deferred to    \cite{DHR}. In  particular, Theorem \ref{x7.6} opens the way for the study 
 of fractal properties of the imaginary cone and the dynamics of the 
 $W$-action on it.

  For an element $v$ of $\real_{\geq 0} \Pi$,  a support of $v$ is defined 
 to be a subset $\Delta$ of
$\Pi$ such that $v=\sum_{\alpha\in \Delta}c_{\alpha }\alpha$ with all $c_{\alpha}>0$.
 This notion of support is more subtle than in the case of linearly independent 
$\Pi$, when each $v$ has a unique support,   and  plays an important technical  role throughout this paper. Section \ref{x8} 
collects some basic   properties of supports  additional to those already 
used in previous sections. 

 Section \ref{x9} discusses the closed imaginary cones of hyperbolic groups  $W$
  or hyperbolic reflection subgroups $W'$ of (possibly non-hyperbolic)
   Coxeter groups $W$, and, as an extended example,  the imaginary 
   cones of  ``generic'' universal  root systems (those for which any rank 
   two parabolic root subsystem  is infinite and not affine).  These 
   classes of root systems play an important role in relation to general 
   root systems. The main result obtained  in the generic  universal case is that  (in 
   rank at least three) the space of components  of the set of  points of the  closed imaginary cone outside 
   the imaginary cone is a Cantor space 
   (the components themselves are totally isotropic cones of unknown 
   dimensions and structure).
      This section also raises some natural  questions left open in  this 
      work; others  are deferred to   \cite{DHR}.

  In preparation for the main results on facial structure of the imaginary 
  and Tits cones,  the reader should consult  the Appendix
   \ref{xA} which discusses  the general   algebraic aspects of facial 
   structure of cones and introduces the (partly non-standard) 
   terminology we use. It also collects for the reader's convenience  at the end  
  a few additional facts used extensively throughout   this paper.   A semidual pair of (possibly non-closed) cones
  is defined to be a pair of cones, each contained in the dual of the other.
  There are several 
    notions   of  face (of each cone in the pair) which are applicable in such a context; (non-empty) extreme subsets,  exposed (equivalently, semi-exposed) subsets, and stable sets for the natural Galois 
  connection, corresponding to the relation of orthogonality of their elements, between  
  subsets of  the  two cones. For trivial reasons, the second and third notions don't necessarily 
  coincide for non-closed cones    while, as is well known,  the first and second notions  don't 
  necessarily agree even for a pair of mutually dual closed cones.  We say the semidual pair 
  is a dual pair of cones  if  all three notions of face of each cone in the pair coincide and  the 
  closures of the two cones are mutually dual in the usual sense.    Polyhedral cones and their duals provide  examples of dual pairs.

  Section \ref{x10} is devoted to a detailed discussion of the facial 
  structure of the imaginary cone and Tits cone. 
   The main result (Theorem \ref{x10.3}) on the face lattices obtained here  
   is that the  imaginary cone and Tits cone form a dual pair as 
   discussed above,  so that there are natural notions of face lattice of each,   and  that the 
   face lattice of $\mc{Z}$ 
  (and therefore   the  dual of  the face  lattice of 
  the Tits cone) is canonically   isomorphic to the lattice of special  facial reflection subgroups of $W$. 
   
The facial closure of a subset of $W$ is defined to be  the  unique inclusion-minimal facial 
subgroup containing it.  Section \ref{x11} gives two algorithms for computing the  facial 
closure of finitely generated subgroups of $W$. The first applies only to reflection subgroups, 
and determines the facial closure by means of   geometry of the imaginary cone and root 
system, while the second  applies to any finitely-generated subgroup and makes   use of the 
solvability of the conjugacy problem for $W$.

  Section \ref{x12}  drops the finiteness and non-degeneracy assumptions  \ref{x4.1}(i)--(iii)    and  returns to the more general framework of Sections \ref{x1}--\ref{x3}.
  Under these relaxed assumptions, some  of the statements about  $\mc{Z}$  proved in Sections \ref{x4}--\ref{x11} would need to be modified if they are to hold, and we do not go into this. Three  complicating factors are the  choice  of topology on the ambient real vector space,  the appropriate generalization of the notion of  dual pair of cones   and   the fact that, for infinite rank Coxeter systems, arbitrary intersections of facial subgroups need not be facial (see \cite{NuidaLocPar}
  for the analogous fact for  parabolic subgroups). However,  the main result of Section \ref{x6} is extended by  showing that the imaginary cone of a reflection subgroup $W'$ of $W$ is contained in that of $W$ in general.

The main part of this work  was done in the academic year 
2008--2009 while on sabbatical from the University of Notre Dame 
at the University of Sydney, though the results of Section \ref{x7}, \ref{x2.13}--\ref{x2.17} and  \ref{x9.9}--\ref{e9.18} were obtained later. I gratefully acknowledge the support of
 both institutions. I also thank Bob Howlett and Xiang Fu for some useful 
 conversations and Claus Mokler  for helpful communications.  Finally, I thank the authors of \cite{HLR} for  useful 
 conversations and  for permission to use   results in  a 
 preliminary version of \cite{HLR} here.

 \section{Recollections on Coxeter groups and root systems}
 \label{x1} This section fixes some commonly used  notation and 
 terminology,  and describes basic properties of Coxeter groups and 
 their root systems. 
\subsection{} \label{x1.1}   It is assumed that  the reader is familiar with 
  basic results on convexity, especially  properties of polyhedral cones, 
  their faces, extreme rays,
 dual polyhedral cones etc  as described in \cite{Br} and \cite{Bar} for 
 instance. A 
 detailed discussion of some aspects of convexity which are  
 particularly relevant  to the latter parts of this paper, especially 
 notions of facial structure of general cones,    is given in Appendix
  \ref{xA}. This    does not  depend on  other sections of this paper and 
  can be consulted for    more background and terminology on cones 
  and convexity  than given in this subsection.
 
Vector spaces are regarded as affine spaces, and affine spaces 
regarded as vector spaces by a choice of an origin, in the standard 
way.
 Finite dimensional vector  spaces are always endowed  with their 
standard topology. Except  in  Appendix \ref{xA}, an 
 infinite dimensional real vector space is always given a fixed but 
 arbitrary   topology  coherent with the standard  topologies  on its 
 finite dimensional subspaces. 

 Let $V$ be a real vector space. 
  A subset of $V$ which is closed under multiplication by non-zero scalars is called a \emph{possibly non-convex cone} in $V$.   A \emph{cone}  in  $V$ is    a possibly non-convex  cone which is   convex as a subset of   $V$. Equivalently, 
 it is an additive subsemigroup of $(V,+)$ which   is closed under  multiplication by positive scalars. The zero cone $\set{0}$ is often written  just as $0$ and the  empty cone as $\emptyset$.    If $A\seq B\seq V$ with $A$ a cone, we say that  $A$ is a  \emph{conical subset} of $B$ (or a \emph{subcone} of $B$, if $B$ is a  cone).  A \emph{xA.11} of a possibly non-convex cone $A$ is a subset $B$ of $A\sm\set{0}$ such that the map $(\lambda,v)\mapsto \lambda v\colon \real_{> 0}\times B\to A\sm\set{0}$ is bijective. Every possibly non-convex cone has a base, but a  cone need  not   necessarily have  a convex base. A (possibly non-convex) cone $\mc{Y}$ is said to be  \emph{pointed} if $0\in \mc{Y}$ and  \emph{blunt} if 
$0\not\in {Y}$.
 
For subsets $A,B$ of $V$, define $A+B$, $A-{B}$ by 
$A\pm B=\mset{a\pm b\mid a\in A,b\in B}$
and set $\lambda A:=\mset{\lambda a\mid a\in A}$ for $\lambda\in \real$. The linear (resp.,
 affine, convex or conical) span of $A$  is defined to be  the 
 smallest vector subspace (resp., affine set, convex set, cone) in $V$ which contains $A$ and  
 is denoted $\lin(A)$
 (resp., $\aff(A)$, $\conv(A)$, $\ccl(A)$).  These sets have 
 well-known descriptions in terms of  linear (resp., affine, 
 convex, positive  linear)  combinations of elements of $A$. 
 One has $\ccl(A)=\mset{\lambda v\mid v\in \conv(A), \lambda\in \real_{>0}}$ and  
$\aff(A)=A+\ccl(A-A)$. If $A\neq \eset$,  then $\aff(A)=a+\ccl(A-A)$ for any $a\in A$ and $\aff{A}$ then has a natural structure of affine space. One has $\lin(A)=\aff(A\cup\set{0})=\ccl(A\cup\set{0}\cup-A)$ where $-A:=(-1)A$.  If $A$ is a non-empty cone, then
 $\lin(A)=\aff(A)$. We abbreviate $\real A:=\lin(A)$ and $\real_{\geq 0}A:=\ccl(A\cup\set{0})$. Define  $\real_{>0}A:=\mset{\sum_{a\in A}c_{a}a\mid c_{a}\in \real_{>0}\text{ \rm for all $a\in A$}}$ if $A$ is finite; in particular,  $\real_{\geq 0}\eset=\set{0}$ and  $\real_{>0}\eset=\eset$. We say that $\Gamma\seq V$ is \emph{positively independent} if $0\not \in \real_{>0}A$ for any finite $A\seq \Gamma$.
 Frequently, we do not distinguish notationally between a singleton
$\set{v}$ and $v$, writing for example $A-v$ for $A-\set{v}$ and $A\sm v$ for $A\sm\set{v}$ (where $A\setminus B$ denotes set difference).  

A cone $\mc{Y}$ is said to be  \emph{salient} if $\mc{Y}\cap -\mc{Y}=\set{0}$ 
and \emph{generating} if $\mc{Y}-\mc{Y}=V$.   A \emph{ray} of
 $V$ is a cone  of the form $\real_{\geq 0}\alpha$ for some 
 non-zero $\alpha\in V$; we say that $\alpha$ \emph{spans}, or 
 \emph{generates}, the ray $\real_{\geq 0}\alpha$.

Assume $V$ is  topologized as described above.
The closure of a subset $A$ of $V$  is denoted  as $\overline A$ or $\cl(A)$. The interior of $A$ is denoted $\inter(A)$.
The \emph{relative interior}, denoted  $\ri(A)$ or sometimes  $A^{\circ}$, of a subset $A$ of $V$ is defined to be the interior of $A$ relative to the subspace topology  of $ \aff(A)$.  The \emph{relative boundary} $\rb(A)$ of $A$ is defined by $\rb(A):=\cl(A)\setminus \ri(A)$. The set of limit points of a set or sequence $X$ in a topological space is denoted $\Acc(X)$.
Recall that any non-empty convex set (e.g. a non-empty cone) 
 is connected.

    \subsection{} \label{x1.2}   Consider  a real vector space $V$ equipped with a  symmetric $\real$-bilinear form $\mpair{-,-}\colon V\times V\to \real$.
    We call the pair $(V,\mpair{-,-})$ a \emph{quadratic space} (over $\real$).
       In general, for $A,B\subseteq V$,  write $\mpair{A,B}=\mset{\mpair{a,b} \mid a\in A,b\in B}$. Define $A^{\perp}:=\set{v\in V\mid \mpair{v,A}\seq \set{0}}$ and $A^{*}:=\set{v\in V\mid \mpair{v,A}\seq \real_{\geq 0}}$.  Write $A\perp B$ if $B\subseteq A^{\perp}$. 
It is well-known (see Appendix \ref{xA} more generally) that $A^{*}$ is a 
pointed cone  and that if $V$ is finite dimensional and $\mpair{-,-}$ is non-singular,  then $A^{*}$ is closed  and  $A^{**}=\cl(\real_{\geq 0}A)$.
The quadratic space (or form) is said to be non-singular if $V^{\perp}=0$.  If $V$ is finite-dimensional, we say that the form (or quadratic space) 
  $\mpair{-,-}$ is of  signature $(k,l,m)$ if  a matrix of the form has $k$ positive eigenvalues, $l$ negative 
  eigenvalues and $m$ zero eigenvalues.

   A vector $\al\in V$ is said to be \emph{isotropic}
 (resp., \emph{non-isotropic}) if $\mpair{\al,\al}=0$ (resp., $\mpair{\al,\al}\neq 0$).  If $\al\neq 0$, the     
 ray $\real_{\geq 0}\alpha$ spanned by $\al$     will be  said to be    \emph{positive},  \emph{negative} or \emph{isotropic} according as 
 $\mpair{\alpha,\alpha}$ is positive, negative or zero.
  Let $\mc{Q}:=\mset{\al\in V\mid \mpair{\al,\al}=0}$ denote the possibly non-convex cone of  isotropic vectors in $V$.
 
  For non-isotropic $\alpha\in V$,  let $s_{\alpha}: V\rightarrow V$ denote the \emph{reflection} in $\alpha$, given by \[s_{\alpha}(v):=v-\mpair{v,\ck \alpha}\alpha, \qquad \ck \alpha:=
\frac{2}{\mpair{\alpha,\alpha}}\alpha.\]    For any  $\Gamma\seq V$ consisting of non-isotropic vectors,  let $W_{\Gamma}:=\mpair{s_{\gamma}\mid \gamma\in \Gamma}$ denote the subgroup of $\GL(V)$ generated by reflections in elements of $\Gamma$.

\subsection{}   \label{x1.3} Assume given   a based root system $(\Phi,\Pi)$   for 
     $(V,\mpair{-,-})$) 
 with associated Coxeter system $(W,S)$ in the sense of \cite[\S 3]{DySd}. 
By definition, this  means that      
 \begin{conds}
\item $(V,\mpair{-,-})$ is a quadratic space (over $real$). 

   \item $\Pi\subseteq V$ is positively independent.\item $\mpair{\alpha,\alpha}=1$ for $\alpha\in \Pi$.  Set $m_{\al,\al}:=1$ for all $\al\in \Pi$.
\item For $\alpha\neq \beta\in \Pi$,  either  $\mpair{\alpha,\beta}=-\cos \frac{\pi}{m_{\al,\bt}}$ where $m_{\alpha,\beta}\in \Nat_{\geq 2}$ or   $\mpair{\alpha,\alpha}\leq -1$, in which case we set $m_{\al,\bt}:=\infty$.
\item $S=\mset{s_{\alpha}\mid \alpha\in \Pi}$ and $W=W_{\Pi}$ is the group of $\real$-linear automorphisms of $V$ generated by $S$.
\item   $\Phi:=\mset{w\alpha\mid w\in W, \alpha\in \Pi}$.
\end{conds}

 The pair $(W,S)$ is a Coxeter system with Coxeter matrix $(m_{\al,\bt})_{\al,\bt\in \Pi}$; in particular,  for $\alpha,\beta\in \Pi$, $s_{\alpha}s_{\beta}$
 has order $m_{\al, \bt}$.  Any Coxeter system  is isomorphic to one attached to  a ``standard'' based root system as above, in which $\Pi$ is a basis for $V$ and $\mpair{\al,\beta}\geq -1$ for all $\al,\beta\in \Pi$ 
 (see \cite{Hum}).  It  also isomorphic to one attached to a based root system in which  $\Pi$ is linearly  independent and $\mpair{-,-}$ is  non-singular.   Coxeter 
 systems in this paper  are possibly of infinite rank except where otherwise  stated
 (as by blanket assumption in Sections \ref{x4}--\ref{x11}). The notion of subgroup   always refers to a subgroup of $W$ unless the overgroup is more precisely 
  indicated.

 One  calls $\Pi$ the set of \emph{simple roots} and $\Phi$ the \emph{root system}.  The set $\Phi_+:=\Phi\cap \real_{\geq 0}\Pi$ is called the set of \emph{positive roots}. We have $\Phi=\Phi_{+}\dot \cup\,  \Phi_{-}$, where $\Phi_{-}:=-\Phi_{+}$ and $\dot\cup$ indicates that a union is of disjoint sets. 
 The cone $\real_{\geq 0}\Pi=\real_{\geq 0}\Phi_{+}$ will be called the \emph{positive root cone}.
 Denote the \emph{standard length function} of $(W,S)$  as $l\colon W\to \Nat$ and define  the set $T:=\mset{wsw^{-1}\mid w\in W,s\in S}$ of \emph{reflections}.  The map
 $\alpha\mapsto s_\alpha\colon \Phi_+\rightarrow T$ is a bijection. For $t\in T$, we sometimes  denote  the unique positive root $\alpha$ with $s_\alpha=t$ as  $\alpha_t$.   

 For standard  properties of
 Coxeter systems and root systems, see \cite{Bour},  \cite{Hum} and \cite{BjBr}.
 Frequently,  results in these and other  references  are stated under stronger conditions than imposed in this paper (e.g. with $\Pi$ finite, or for the standard root system), but the proofs of our stated results  under the weaker conditions here are essentially the same unless otherwise indicated.
 Useful  properties of the  based root systems considered in this paper are   listed in \cite{DySd}. The conditions (i)--(iv)
 imply that  for any finite subset  $\Delta$ of $\Pi$, the pointed conical hull $C:=\real_{\geq 0}\Delta$ of $\Delta$  is a polyhedral cone and the 
   extreme rays of $C$ are  spanned by the elements of $\Delta$;
    hence  there is a linear map $\rho\colon V\to \real$ with $\rho(\Delta)\subseteq \real_{>0}$. Therefore if $\Pi$ is finite, then $(V, \mpair{\cdot \mid \cdot },\Pi)$ is a root basis in the sense of \cite{K} and the results of \cite{K} and \cite{HRT} apply. 
    
 \begin{rem}  (1) Suppose given a based root system $(\Phi,\Pi)$ in $(V,\mpair{-,-})$. For any  subspace $V'\sreq \real \Pi$ of $V$, we may regard $(\Phi,\Pi)$ as a based root system in $(V',\mpair{-,-}')$ where $\mpair{-,-}'$ is obtained by restriction of $\mpair{-,-}$ to a symmetric bilinear form on $V'$; the latter based root system is said to be obtained by \emph{restriction (of quadratic space)} of the former, and the former by 
    \emph{extension (of quadratic space)}  of the latter.
    The associated  Coxeter systems of  based root systems related by extension or restriction are canonically isomorphic.

  (2)   A based root system $(\Phi,\Pi)$ in $(V,\mpair{-,-})$  will be  said to be \emph{ample}  for a subspace $U$ of $V$
    if any $\real$-linear map $U\to \real$ is of the form $u\mapsto
    \mpair{u,v}$, for $u\in V$, for some $v\in V$; this holds for instance if $U$ is finite-dimensional and $\mpair{-,-}$ is non-singular. 
     In case $U=\real\Pi$,     we say simply that 
the based root system is ample or that
       $(V, \mpair{-,-})$ is ample for $(\Phi,\Pi)$. 
It is not difficult to show that  for any based root system $(\Phi,\Pi)$ in $(V,\mpair{-,-})$ and any subspace $U$ of $V$, there is an extension which is ample for $U$ and such that the associated quadratic space is non-singular.

\subsection{}  \label{x1.4} For   proofs of some (but not all) facts concerning $(W,S)$ or $\Phi$,  it is possible to replace the root system by one with linearly independent roots by a standard construction used for instance in  the proof of \cite[Proposition 2.9]{HRT} and  \cite[(3.5)]{DyTh}. We recall this construction  below. 
\end{rem}

   Let $V'$ be a real vector space with linearly independent subset $\Pi'=\mset{\al'\mid \al\in \Pi}$ in bijection  with $\Pi$ by a map $\al'\mapsto \al\colon \Pi'\xrightarrow{\cong} \Pi$. Suppose given
a linear map $L\colon  V'\to V$ which restricts to this bijection $\Pi'\to \Pi$ (if $\Pi'$ is a basis for $V'$, there is a unique such map $L'$).  
   Define a 
  symmetric bilinear form $\mpair{-,-}'\colon V'\times V'\to \real$ by $\mpair{u_{1},u_{2}}'=\mpair{L(u_{1}),L(u_{2})}$. For non-isotropic 
  $u\in V'$, let $r_{u}\colon V'\to V'$ denote the reflection in $u$.
  Let $ W':=\pair{S'}\seq \GL(V')$ where $S':=\mset{r_{\al}\mid \al\in \Pi'}\seq \GL(V')$,
  $\Phi':=W'\Pi'$ and $\Phi'_{+}=\Phi'\cap \real_{\geq 0}\Pi'$.
  Then $(\Phi',\Pi')$ is a based root system for $(V',\mpair{-,-}')$ with 
  associated Coxeter system $(W',S')$ and the map 
  $r_{\al'}\mapsto s_{\al}\colon S'\to S$ for $\al\in \Pi$ extends  to an 
  isomorphism of Coxeter systems $\theta\colon (W',S')\to (W,S)$. Further, 
  $L$ restricts to bijections $\Phi'\xrightarrow{\cong} \Phi$ and 
  $\Phi'_{+}\xrightarrow{\cong} \Phi_{+}$   and satisfies 
  $L(wu)=\theta(w)L(u) $ for all $w\in W'$ and $u\in V'$ (i.e. identifying $W'$ with $W$ by $\theta$, $L$ is a $W$-equivariant map).  
  
   If $\Pi'$ is a basis of $V'$,  we shall call $(\Phi',\Pi')$ on 
   $(V',\mpair{-,-}')$ a \emph{canonical lift}  of $(\Phi,\Pi)$ on 
   $(V,\mpair{-,-})$ and  call $L$ the associated \emph{canonical map}. 
     This notion will be used several times in this paper to reduce proofs 
     to the case of linearly independent simple roots (i.e from the positive 
     root cone being a general polyhedral cone  to a simplicial cone). It 
     seems likely that these arguments could   be replaced by  arguments 
     using more  of  the theory of polyhedral cones.

      The diversity of  root bases which may be possible  for fixed $(W,S)$ is suggested by the following example.
    
    \begin{exmp}   Suppose  $(W,S)$ is universal  (i.e  there are no braid relations) and  of rank $n\geq 3$.
  We consider  certain based  root systems  $(\Phi, \Pi)$ on spaces $(V,\mpair{-,-})$   with $V=\real \Pi$ and 
  with associated Coxeter system isomorphic to $(W,S)$. 
  The requirement that $W$ is universal is equivalent to the condition  $\mpair{\al,\bt}\leq -1$ for all distinct $\al,\bt\in \Pi$.  
   Write  
  $\Pi=\set{\al_{1},\ldots,\al_{n}}$ and  assume first that  $\Pi$ is linearly independent. The standard   based 
  root system   from \cite{Bour} or \cite{Hum} has $\mpair{\al_{i},\al_{j}}=-1$ for $i\neq j$. There is also   a based root system $(\Phi,\Pi)$ on $(V,\mpair{-,-})$ in which $\Pi$ is linearly independent  and  $\mpair{\al_{i},\al_{j}}:=1-8(i-j)^{2}$ for all $i,j$. (The latter based root system arises naturally  as a canonical lift  of   a subsystem of a based root system considered in Example \ref{x1.7}).  It is obvious from its natural origin (and easy to check directly) that the latter form on $V$ has signature $(2,1,n-3)$.
    By perturbing $\mpair{-,-}$ on $V$ and dividing out by a subspace of the radical
  (taking  the image of $\Pi$ in the quotient space as new 
  simple roots,   see \cite[6.1]{K})  one can find based  root systems,  $(\Phi',\Pi')$ on  $(V',\mpair{-,-}')$,  of $(W,S)$ such that the form 
  $\mpair{-,-}'$ on $\real \Pi'$ is of any signature $(k,l,m)$ with $k+l+m\leq n$, $k\geq 2$ and  $l\geq 1$
  and with  the positive root cone 
  $\real_{\geq 0}\Pi'$ of any of a wide  variety of combinatorial types of $k+l+m$-dimensional 
  cones with $n$ extreme rays.  It is easily checked that this is the full range of possible  signatures (for the bilinear form on the span of the  roots)  for   irreducible, infinite non-affine root systems.   
       \end{exmp}
       
        \begin{rem} Despite the diversity of based root systems, their root systems $\Phi$  (for fixed $W$) are all canonically isomorphic  as $W\times\set{\pm 1}$-sets. For the  bijection $(t,\e)\mapsto \e\al_{t}\colon T\times\set{\pm 1}\to \Phi$ may be used  to transfer the $W\times\set{\pm 1}$-action on $\Phi$ to one on  $T\times\set{\pm 1}$; the resulting action depends only on   $(W,S)$ as Coxeter system
 (it is described in \cite[Ch IV, \S 1, no. 4]{Bour}).   
  \end{rem}

 \subsection{} \label{x1.5}
  It is well known (see \cite{Hum}, for example) that for $w\in W$,  \[N(w):=\mset{t\in T\mid l(tw)<l(w)}=\mset{s_{\alpha}\mid \alpha\in \Phi_{+}\cap w(-\Phi_{+})}.\]
 The equality above will be used frequently without special comment.
  Any \emph{reflection subgroup} $W'$ of $W$ (i.e. a subgroup $W'=\mpair{W'\cap T}$ generated by the reflections it contains) has a \emph{canonical set  of  Coxeter generators} $\chi(W')=\mset{t\in T\mid N(t)\cap W'=\set{t}}$ (\cite[3.3]{Ref}). We always consider reflection subgroups as Coxeter groups with these Coxeter generators, unless otherwise stated.
   It is easily shown (see \cite[Lemma 2.8]{HRT})  that if $\G,\G'\seq \Phi$ with  $W_{\G}=W_{\G'}$, then $\real\G=\real\G'$.

   A reflection subgroup $W'$  is said to be dihedral if  it is generated by two (distinct) reflections, or equivalently (by \cite[3.11]{Ref})  if $\vert \chi(W')\vert =2$.    Any dihedral reflection subgroup
 $\mpair{s_\alpha,s_\beta}$ is contained in a unique maximal (under inclusion) dihedral reflection subgroup, namely
 $\mpair{s_\gamma\mid \gamma\in \Phi\cap (\real \alpha+\real \beta)}$ (see \cite[3.2]{BrGr}).
 
\subsection{} \label{x1.6} For any reflection subgroup $W'$ and any $w\in W$, the coset $W'w$ has a unique element $x$ of minimal length, characterized by $x\in W'w$ and $N(x)\cap (W'\cap T)=\emptyset$
 or equivalently, $x\in W'w$ and $N(x)\cap \chi(W')=\emptyset$ (\cite[3.4]{Ref}, \cite[1.4]{BrGr}), or equivalently again, $x\in W'w$ and $x^{-1}(\Pi_{W'})\seq \Phi_{+}$. One  has $\chi(w^{-1}W'w)=x^{-1}\chi(W')x$ by \cite[Lemma 1]{DyParClos}.

\subsection{} \label{x1.7}  For any reflection subgroup  $W'$  of $(W,S)$, let
  \[\Phi_{W'}:=\mset{\alpha\in \Phi\mid s_{\alpha}\in W'}, \qquad\Phi_{{W',+}}:=\Phi_{{W'}}\cap \Phi_{+}\]  
denote the corresponding sets of roots and positive roots, and  \[\Pi_{W'}:=
\mset{\alpha\in \Phi_{+}\mid s_{\alpha}\in \chi(W')}.\]  Then
 $(\Phi_{W'},\Pi_{W'})$ is a based root system in $V$ with positive roots 
 $\Phi_{W',+}$ and associated Coxeter system $(W',\chi(W'))$ (\cite[Lemma 
 3.5]{DySd}). From \cite{Ref}, a  subset $\G$ of $\Phi_{+}$ is of the form $\G=\Pi_{W'}$ for some reflection subgroup $W'$ of $W$ (necessarily $W'=W_{\Gamma})$ if and only if $\mpair{\al,\bt}\in (-\infty,-1]\cup\set{-\cos \frac {\pi}{n}\mid n\in \Nat_{\geq 2}}$ for all distinct $\al,\bt\in \Pi$.

 If $w\in W$ with $N(w)\cap W'=\emptyset$  then using  \ref{x1.5}--\ref{x1.6} and the 
 definitions, one has  \[\Pi_{w^{-1}W'w}=w^{-1}\Pi_{W'}, \qquad \Phi_{w^{-1}W'w}=
 w^{-1}\Phi_{W'},\qquad
\Phi_{w^{-1}W'w,+}=w^{-1}\Phi_{W',+}.\]
\begin{exmp} (cf.  \cite[Example 3.19]{DyTh})
Let $(\Phi, \Pi)$ on $(V,\mpair{-,-})$ be the standard based root system of a rank $3$ 
universal Coxeter group, with 
$\Pi=\set{\al,\bt,\g}$ as $\real$-basis for  $V$ and 
$\mpair{\delta,\delta}=-\mpair{\delta,\epsilon}=1$ for all $\delta\neq \epsilon $ in $\Pi$. 
For $k\in \Int$, define $\gamma_{k}=2k(2k+1)\al+2k(2k-1)\bt+\g$. A simple computation shows 
that $s_{\al}s_{\bt}(\g_{k})=\g_{k+1}$. Since $\g_{0}=\g$, it follows that  
$\g_{k}=(s_{\al}s_{\bt})^{k}(\g)$ for all $k\in \Int$. We have 
$\mpair{\g_{i},\g_{j}}=\mpair{\g,\g_{j-i}}=1-8(i-j)^{2}$ for all $i,j\in \Int$.
 Setting $W':=\mpair{s_{\g_{k}}\mid k\in \Int}$, it follows that
  $\Pi_{W'}=\mset{\g_{k}\mid k\in \Int}$. In particular,  even though
$(\Phi,\Pi)$ on $(V,\mpair{-,-})$ is a standard based root system,  of finite rank, the based  root 
system  $ (\Phi_{W'},\Pi_{W'})$ on
$(V,\mpair{-,-})$ is of  infinite rank  and is  not  standard; $\Pi_{W'}$ is linearly dependent  and 
$\mpair{\delta,\epsilon}<-1$ for distinct $\delta,\epsilon\in \Pi_{W'}$. This phenomenon   
provides the principal  reason for using   the  class of root systems  in \ref{x1.3} rather than 
the standard root systems as in \cite{Hum}. Infinite rank reflection subgroups appear quite 
naturally; for example,  even in finite rank root systems, the stabilizer of a root $\al$ is a 
semidirect product of its (sometimes infinite rank) normal reflection subgroup generated by  reflections in  all
roots orthogonal to $\al$,  and a free group.\end{exmp}

\subsection{} \label{x1.8} In general, for any subset $I$ of $S$, let  $W_{I}$ denote the subgroup of $W$ generated by $I$.  Let $I^{\perp}:=\mset{r\in S\sm I\mid sr=rs \text{ \rm for all $s\in I$}}$. The subgroups $W_{I}$ for $I\subseteq S$ are called \emph{standard parabolic subgroups} of $W$, and their conjugates in $W$ are called \emph{parabolic subgroups} of $W$. One has $\chi(W_{I})=I$.  Following \cite{Loo}, we shall say that $I\seq S$ (resp., $W_{I}$, $wW_{I}w^{-1}$)  is a \emph{special} subset of $S$ (resp., \emph{special} standard parabolic subgroup,
\emph{special}  parabolic subgroup) if $I$ has no  component $J$ with $W_{J}$ finite.
 For any $I\subseteq S$,   abbreviate $\Pi_{I}:=\Pi_{W_{I}}=\mset{\al\in \Pi\mid s_{\al}\in I}$,
$\Phi_{I}:=\Phi_{W_{I}}=W_{I}\Pi_{I}$, $\Phi_{I,+}=\Phi_{W_{I},+}$
  etc. Denote the set of minimal length coset representatives of $W/W_{I}$ as $W^{I}$. Then \begin{equation*}
  W^{I}=\mset{x\in W\mid N(x^{-1})\cap I=\eset}=\mset{x\in W\mid x(\Pi_{I})\seq \Phi_{+}}.
  \end{equation*} If $I, J,K\seq S$,  let $W_{J}^{I}:=W_{J} \cap W^{I}$ and $^{K}W_{J}^{I}:=
  W_{J}^{I}\cap (W^{K})^{-1}$.
  \begin{rem}
  We shall use at times standard properties of shortest double coset representatives which can be found for finite $W$ in \cite[2.7]{Ca} (they can be proved similarly for general $W$). 
  In fact, many of these results can be  generalized to apply to $(W',W_{J})$ double cosets
  (where $W'$ is an arbitrary reflection  subgroup) and more general length functions than the standard ones, but we don't go into this here.
  \end{rem}
 
\subsection{} \label{x1.9} We shall now introduce  several cones which play an important role in this paper.
Let $W'$ be a reflection subgroup of $W$. Let $\mc{C}_{W'}$ denote the \emph{fundamental chamber} for $W'$ on $V$. By definition, this is the  cone
 \begin{equation*}\begin{split}\mc{C}_{W'}&=\mset{v\in V\mid \mpair{v,\alpha}\geq 0\text{ \rm for all $\alpha\in \Pi_{W'}$ }}\\
 &=\mset{v\in V\mid \mpair{v,\alpha}\geq 0\text{ \rm for all $\alpha\in \Phi_{W',+}$ }}\end{split}\end{equation*}
 Let  $\mc{X}_{W'}:=\cup_{w\in W'}w(\mc{C}_{W'})$ denote the \emph{Tits cone} (see \ref{x1.10} below) of $W'$.
 Abbreviate $\CC:=\CC_{W}$ and $\mc{X}:=\mc{X}_{W}$.
 \begin{rem}  The fundamental chamber and Tits cone may  change significantly  under extension or restriction of quadratic space.
 Note that according to our definitions,   $\mc{C}\cap -\mc{C}=\Pi^{\perp}$. This differs from other treatments (e.g. \cite{Bour}) in which the fundamental chamber and Tits cone are defined as  
 subsets of the dual space of the linear span  $\real\Pi$ of the simple roots (so  one has $\mc{C}\cap -\mc{C}=\set{0}$).  If $\mpair{-,-}$ is non-singular and $V$ is 
 finite dimensional, then one may identify $V$ with its dual space and the definition is consistent  with that in sources 
 (e.g. \cite{K}, \cite{HRT}, \cite{V}, \cite{Kac}) where $\mc{C}$  is  
 defined as a subset of the dual space of $V$.  In general, if the quadratic space $(V,\mpair{-,-})$ is ample for $(\Phi,\Pi)$,  then the dual space of $\real\Pi$ identifies with   quotient  subspace of $V$ by $\Pi^{\perp}$ and our notions have similar properties as in these other sources. In full generality,  it is not obvious that $\mc{C}$ or $\mc{X}$ is non-zero, but   we may always replace the based root system by an ample extension if that is  inconvenient (see  \ref{x12.1}).
 
  \end{rem}
\subsection{} \label{x1.10} The following  lemma records some basic properties of the Tits cone and fundamental chamber. 

\begin{lem} \begin{num}  \item $\mc{X}$is the set of all $v\in V$ which have negative inner product $\mpair{\alpha,v}$ with only a finite number of positive roots i.e.  \[\mc{X}_{W}=\mset{v\in V\mid \vert \mset{ \alpha\in \Phi_{+}\vert \mpair{v,\alpha}<0}\vert \text{ \rm is finite}}.\]
\item  $\mc{C}$ and $\mc{X}$ are  pointed cones.  
\item 
 For $x,y\in \mc{C}$ and $w\in W$, one has $wx=y$ if and only if  $x=y$ and $w$ is in the standard parabolic subgroup generated by   reflections in simple roots $\alpha$ with $\mpair{\alpha,x}=0$. 
 \item The stabilizer $\stab_{W}(x)$ in $W$  of any point  $x$ of $\mc{X}$ is generated by the  reflections in the roots $\alpha$ with $\mpair{x,\alpha}=0$, and is a parabolic subgroup.   \item  For a reflection subgroup 
   $W'$ of $W$,  $\mc{C}_{W'}\supseteq \mc{C}_{W}$ and $\mc{X}_{W'}\supseteq  \mc{X}_{W}$. 
 \item If $W_{1},\ldots, W_{n}$ are reflection subgroups of 
 $W$ such that $\Phi=\cup_{i}\Phi_{W_{i}}$
 then 
  $\mc{C}_{W}=\cap_{i} \mc{C}_{W_{i}} $ and $\mc{X}_{W}=\cap_{i}  \mc{X}_{W_{i}}$. 
  \item For a reflection subgroup 
   $W'$ of $W$ and $w\in W$, $\mc{X}_{wW'w^{-1}}=w(\mc{X}_{W'})$.
     \item If $W$ is finite, then $\mc{X}=V$.
        \end{num}\end{lem}
        \begin{rem}
        Part (f) applies if $W$ has finitely many components $W_{1},\ldots, W_{n}$.
        \end{rem}
 \begin{proof} Parts (a)-(d) and (h) are  standard (see \cite[Ch 3]{Kac}, \cite{Bour}, \cite{K}). Parts (e)--(f)
   are   consequence of the definition of $\mc{C}$ and (a). If $w\in W'$, (g) is trivial.  By \ref{x1.6}, it is therefore  enough to check (g) if $w$ is of minimal length in $wW'$.  In that case, (g) follows  using  (a) and the fact from \ref{x1.7} that  $\Phi_{wW'w^{-1},+}=w(\Phi_{W',+})$.  \end{proof}
   
   \subsection{} \label{x1.11} The next lemma refines the  relationship between fundamental chambers of  reflection  subgroups
   described in  Lemma \ref{x1.10}(e).
     \begin{lem} Let $W'$ be a reflection subgroup of $W$ and  let $W''$  be the set of minimal length coset representatives in $W'\backslash W$. Then  \begin{num}   \item If $x\in W''$, then   $\CC_{x^{-1}W'x}=x^{-1}(\CC_{W'})$.
\item    $\CC_{W'}\cap \mc{X}_{W}= \bigcup_{w\in W''
 }w(\CC_{W})$.\end{num}
 \end{lem}

\begin{proof} Part (a) follows from the definition of $\CC_{W'}$, using      $\Pi_{x^{-1}W'x}=x^{-1}(\Pi_{W'})$ from \ref{x1.7}. 
For (b), note first 
 that if $ w\in W''$, then (a) and Lemma
 \ref{x1.10}(e) imply that 
$w^{-1}(\CC_{W'})=\CC_{w^{-1}W'w}\sreq \CC_{W}$. Hence 
$\cup_{w\in W''}w(\CC_{W})\seq \mc{C}_{W'}\cap \mc{X}_{W}$.
To prove the reverse inclusion, let $v\in \CC_{W'}\cap \mc{X}_{W}$.  Write 
$v=w(v')$ where $v'\in \CC_{W}$ and  $w\in W$ is of minimal length  $l(w)$. It  will suffice to show that $l(s_{\alpha}w)\geq l(w)$ for all $\alpha\in \Pi_{W'}$, for then $w\in W''$ by \ref{x1.6}. Suppose first that $\mpair{\alpha,v}=0$.
Then $v=s_{\alpha}(v)=(s_{\alpha}w)(v')$ with $s_{\alpha}w\in W$. By choice of $w$,  $l(s_{\alpha}w)\geq l(w)$. On the other  hand,
suppose $\mpair{\alpha,v}\neq 0$. Since $v\in \CC_{W'}$, this forces
$0<\mpair{v,\alpha}=\mpair{w(v'),\alpha}=\mpair{v',w^{-1}(\alpha)}$. Since $v'\in \CC_{W}$ and $w^{-1}(\alpha)\in \Phi$, it follows that
$w^{-1}(\alpha)\in \Phi_{+}$ and so $l(s_{\alpha}w)>l(w)$ as required.
\end{proof}

   \subsection{}\label{x1.12}
 The following result is in marked contrast to the case of finite $W$ (see \ref{x1.10}(h)).
     \begin{prop}    Suppose that  $(W,S)$ is infinite, irreducible and not affine, and $S$ is finite. Let $V'=\real\Pi$ be the subspace of $V$ spanned by $\Pi$ and  $X:=\mc{X}\cap V' $.  Then $\ol{X}\cap -\ol{X}=\set{0}$.
  \end{prop}
    \begin{proof} See \cite[Proposition 3.2]{HRT},  \cite[Theorem 2.16]{K}  and \cite[Lemma 15]{V}. 
  \end{proof}
 
 \subsection{}\label{x1.13} The  next lemma records  well-known useful facts about $W$-orbits on $V$ \begin{lem}\begin{num} \item If $w=s_{\alpha_{1}}\cdots s_{\alpha_{n}}$,
  with all $\alpha_{i}\in \Phi$
  then  for $v\in V$,
  \[ v-w(v)=\sum_{i=1}^{n}\mpair{v,\ck\alpha_{i}}\beta_{i}\]
  where  $\beta_{i}:=s_{\alpha_{1}}\cdots s_{\alpha_{i-1}}(\alpha_{i})$.
  Further, $w=s_{\beta_{n}}\ldots s_{\beta_{1}}$.
  \item If  $w=s_{\alpha_{1}}\cdots s_{\alpha_{n}}$ is a reduced expression for $w$,
  with all $\alpha_{i}\in \Pi$, then $\beta_{1},\ldots, \beta_{n}$ are distinct positive roots, with
  $\Phi_{+}
\cap w(-\Phi_{+})=\set{\beta_{1},\ldots, \beta_{n}} $. \item  For any $w\in W$ and $v\in \mc{C}$, we have 
  $v-w(v)\in \real_{\geq 0}\Pi$.
  \item If $W$ is irreducible and $v\in V$ satisfies $v\not\perp \Pi$, then $\aff(Wv)=v+\real \Pi$.
    \end{num}\end{lem}

   \begin{proof}  Parts  (a) and  (b) are readily proved by induction  (see   \cite[Ch 3]{Kac}, \cite{Bour} or  \cite{Hum}) and  part (c) follows immediately from them.   Part (d) is a variant of \cite[Ch 5, \S4, Proposition 7]{Bour}.   Let $U$ be the translation space of $\aff(Wv)$,
 so  $
   U=\lin(\mset{xv-yv\mid x,y\in W})=\lin(\mset{v-wv\mid w\in W})
   $  and  $\aff(Wv)=v+U$. We have $U\seq \real\Pi$ by (c).
   Since  $v\not\perp \Pi$, there is    $\al\in \Pi$ with $\mpair{v,\al}\neq 0$. Then $\al\in U$ since $\mpair{v,\ck\al}\al=v-s_{\al}(v)$. Further, if $\bt\in \Pi\cap U$ and $\g\in \Pi$ with $\mpair{\bt,\g}\neq 0$, then $\g\in U$ since $\mpair{\bt,\ck\g}\g=\bt-s_{\g}(\bt)$. By irreducibility of $W$, $\Pi\seq U$ and (d) follows.
  \end{proof}

 \subsection{}\label{x1.14} Let $\D\seq \Phi$ with
  $\D=\Pi_{W'}$ where $W':=W_{\D}$ (for instance, 
    $\Delta\seq\Pi$).
 Then $\D$ is said to  have a certain property of based root 
 systems  (e.g. be irreducible, have $n$ components, be 
 irreducible of affine type etc) if the based root system
  $(\Phi_{W'},\D)$ has that property. If $W'$ is finite, its longest element is denoted $w_{\D}$ or by $w_{I}$ where $I:=\mset{s_{\alpha}\mid \alpha\in \D}$ is the set of its simple reflections.

Sometimes  a reflection subgroup $W'$ (resp.,  subset $I$ of $S$) will be said to have   a certain property of root systems 
(be  irreducible etc) if $\Pi_{W'}$ (resp., $\Pi_{I}$) has that 
property; we may abuse terminology by doing  this even when 
that property is a property of root systems and not just of the 
corresponding Coxeter group $W'$ (resp., $W_{I}$). For 
example, see Remark \ref{x4.5}(i).  We view the vertex set of the Coxeter graph of a reflection subgroup $W'$
as either $\Pi_{W'}$ or $\chi(W')$ as convenience dictates, 
 and often identify   full subgraphs of the Coxeter graph with 
 their vertex set.  Thus, we refer   to the connected (or 
 equivalently, irreducible or indecomposable) components of
  $\Pi_{W'}$ or $\chi(W')$ or $W'$
(they are, respectively, subsets of $\Pi_{W'}$, subsets of
 $\chi(W')$, and subgroups of $W'$).

Subsets  $\Delta$, $\Delta'$ of $\Phi$ are defined to be \emph{separated} if $\Delta\perp \Delta'$.
Similarly, subsets $I,J$ of $S$ are \emph{separated} if $\Pi_{I}$ and $\Pi_{J}$ are separated.
A subset $I$ of $S$ (resp., $\Delta$ of $\Pi$) is said to be \emph{of finite type} if $W_{I}$ (resp., $W_{\Delta}$) is finite.

\subsection{} \label{x1.15}  The basic  fact below,  due to Deodhar \cite{D} (see also
\cite{BH2}), will be required  several times in this paper. 
\begin{thm} Suppose that $\Delta, \Gamma\subseteq \Pi$  and $w\in W$ with $w(\Delta)=\Gamma$.
Then there exist $n\in \Nat$, $w_{1},\ldots, w_{n}\in W$, $ \Delta_{0},\ldots, \Delta_{n}\subseteq \Pi$
and  $ \alpha_{1},\ldots,  \alpha_{n}\in \Pi$, such that
\begin{conds}\item $\Delta_{0}=\Delta$ and $\Delta_{n}=\Gamma$.
\item For $i=1,\ldots, n$,  $\alpha_{i}\not\in\Delta_{{i-1}}$, the connected component $\Gamma_{i}$ of $\Delta_{i-1}\cup\set{\alpha_{i}}$ containing $\alpha_{i}$ is of finite type,  and $w_{i}=w_{\Gamma_{i}}w_{\Gamma_{i}\setminus \set{\alpha_{i}}}$. 
\item $\Delta_{i}=w_{i}\Delta_{i-1}$ for $i=1,\ldots, n$.
\item $w=w_{n}\cdots w_{1}$ and $l(w)=l(w_{1})+\cdots +l(w_{n})$.
 \end{conds}
\end{thm}
\begin{rem}
Suppose above that $\Delta$ has the property that if $\al\in\Pi\sm \bigl(\Delta\cup (\Pi\cap \Delta^{\perp})\bigr)$, then
the connected component of $\Delta\cup\set{\al}$ containing $\al$ is  infinite. Then  above, each $\al_{i}\in \Pi\cap \Delta^{\perp}$ and $w_{i}=s_{\al_{i}}$,  so $w\in W_{\Pi\cap \Delta^{\perp}}$ and $\Gamma=w(\Delta)=\Delta$ (in fact, $w$ fixes $\Delta$ pointwise).
This applies in particular if every component of $\Delta$ is infinite (cf. \cite{BH2}   and \cite[1.23]{DyRig}). 
\end{rem}

  \subsection{}  \label{x1.16}The following is well known, but we provide  a  proof for ease of reference.   \begin{lem} 
   Let $\alpha_{0},\alpha_{1},\ldots, \alpha_{n}\in \Pi$ be distinct. Assume  that for each $i>0$, there is some $j<i$ with $\mpair{\alpha_{i},\alpha_{j}}\neq 0$. Let $\alpha:=s_{\alpha_{n}}\cdots s_{{\alpha}_{1}}(\alpha_{0})$.
   Then
   \begin{num}
   \item
   $\alpha=\sum_{i=0}^{n}c_{i}\alpha_{i}$ for some $c_{i}\in \real_{>0}$.
   \item $l(s_{\alpha})=2n+1$.
   \item $s_{\alpha}=s_{\alpha_{n}}\cdots s_{\alpha_{1}}s_{\alpha_{0}}s_{\alpha_{1}}\cdots s_{\alpha_{n}}$ is a reduced expression for $s_{\alpha}$.
     \end{num}
   \end{lem}
   \begin{proof}  This is proved by induction on $n$ as follows.
   For $n=0$, the assertions are trivial.  Suppose inductively that  $n>0$.
   Write $\beta:=s_{\alpha_{n-1}}\cdots s_{{\alpha}_{1}}(\alpha_{0})=\sum_{i=0}^{n-1}c_{i}\alpha_{i}$ for some $c_{i}\in \real_{>0}$.
   Then $c_{n}:=-\mpair{\beta,\ck\alpha_{n}}=-\sum_{i=0}^{n-1}c_{i}
   \mpair{\alpha_{i},\ck\alpha_{n}}>0$ since $c_{0},\ldots, c_{n-1}>0$,
$\mpair{\alpha_{i},\ck\alpha_{n}}\leq 0$ for all $i<n$ since $\alpha_{i}\neq \alpha _{n}$
and $\mpair{\alpha_{i},\ck\alpha_{n}} <0$ for some $i<n$ by assumption.  It follows that $\alpha=s_{\alpha_{n}}(\beta)=\beta+c_{n}\alpha_{n}$ is as in (a). Noting that  $\beta\neq \alpha_{n}$, \cite[Lemma 3.4]{DySd} implies that $l(s_{\alpha})=l(s_{\beta})+2$. Since $s_{\alpha}=s_{\alpha_{n}}s_{\beta}s_{\alpha_{n}}$, (b) and (c) follow by induction.   \end{proof}

\subsection{}\label{x1.17} The next  lemma is for use in the proof of Lemma \ref{x2.10}.\begin{lem} Let $I,J\subseteq S$. Then the following conditions are equivalent:\begin{conds}\item $W=W_{I}W_{J}$.
\item $T=(T\cap W_{I})\cup(T\cap W_{J})$.
\item $\Phi=\Phi_{I}\cup \Phi_{J}$.
\item For every  irreducible component $\D$ of $\Pi$,  either $\D$  is contained in $\Pi_{I}$ or $\D$ 
is contained in $\Pi_{J}$.
\end{conds}
\end{lem}
\begin{proof}  We first show that (i) implies (ii). Assume that (i) holds. Let $t\in T$.
We may write $t=uv$ where $u\in W_{I}$ and $v\in W_{J}$. Choosing reduced expressions for $u$ and $v$ and substituting in $t=uv$ gives an expression for $t$ which
may be cancelled (by repeated deletion of appropriate pairs of simple reflections)
to a reduced expression $t=s_{1}\ldots s_{m} $ where $s_{1},\ldots, s_{n}\in I$ and
$s_{n+1},\ldots, s_{m}\in J$. Note that  $m=l(t)=2k+1$ for some $k$.
If $n\geq k+1$, we have $t=s_{1}\cdots s_{k}s_{k+1}s_{k}\cdots s_{1}\in W_{I}$ by \cite[(2.7)]{Ref}. 
On the other hand, if $n\leq k$, then similarly  $t=t^{-1}=s_{2k+1}\cdots s_{1}=s_{2k+1}\cdots s_{k+2}s_{k+1}s_{k+2}\cdots s_{2k+1}\in W_{J}$. Hence (ii) holds.
 
 The equivalence of (ii) and (iii) is clear. Next we show that (iii) implies (iv).
 Assume that (iii) holds but (iv) doesn't. Suppose that $\Delta$ is an irreducible 
 component of $\Pi$ which is contained in neither $\Pi_{I}$ nor $\Pi_{J}$.  
 Then there is some  finite subset $\Delta'$ of $\Delta$ which is connected 
 but contained in neither $\Pi_{I}$ nor $\Pi_{J}$. By Lemma \ref{x1.16}, 
   there is a root  $\alpha\in \Phi$
  such that $s_{\alpha}$ has a reduced expression containing $s_{\beta}$ for all $\beta\in \Delta'$.
 Then
  $\alpha\not\in \Phi_{I}\cup \Phi_{J}$, a contradiction.
 This completes the proof that (iii) implies (iv).

 Finally, we show that (iv) implies (i). Let $w\in W$.  Write $w=w_{1}\ldots w_{n}$ where
 $w_{i}\in W_{\Delta_{i}}$ for some irreducible component  $\Delta_{i}$ of $\Pi$. 
 Since  $w_{i}$ and $w_{j}$ commute if $\Delta_{i}\neq \Delta_{j}$, it may be  assumed that the $\Delta_{i}$ are pairwise distinct and, by reindexing if necessary, it  may also  be  assumed  that $\Delta_{1},\ldots,  \Delta_{m}$
 are contained in $\Pi_{I}$ while $\Delta_{m+1},\ldots,  \Delta_{n}$ are contained in $\Pi_{J}$.
 Then $w_{1},\ldots,w_{m}\in W_{I}$ and $w_{m+1},\ldots, w_{n}\in W_{J}$ so $w\in W_{I}W_{J}$ as required to prove that (iv) implies (i).
\end{proof}

\subsection{}\label{x1.18}  The  proposition below collects assorted useful facts  which  show   how certain subsets of roots (orbits, complements of parabolic subsystems and cofinite subsets of the roots)   may be regarded as being  ``large'' in various natural senses.
\begin{prop}   Suppose that $(W,S)$ is irreducible.
\begin{num}
 \item 
   Let  $\alpha\in \Phi$ and let  $Z:=\mset{w\in W\mid w(\alpha)=\alpha}$ denote the stabilizer of $\alpha$ in $W$. Then $\real(W\alpha)=\real\Phi$,  $W$ acts faithfully on the $W$-orbit $W\alpha$ and $\cap_{w\in W}\,wZw^{-1}
   =\set{1_{W}}$. In particular, if $W\alpha$ is finite, then $W$ is finite.
 \item    Let $I\subsetneq S$ and $\G:=\Phi\setminus \Phi_{I}$ Then  $W_{\G}=W$, $W\G=\Phi$,   and $\real\G=\real\Phi$.  \item Suppose that $W$ is infinite.  If $\Delta\subseteq \Phi_{+}$  with $\vert \Phi_{+}\setminus \Delta\vert$ finite, then $W_{\Delta}=W$,
 $W\Delta=\Phi$ and $\real\Delta=\real\Pi$.
\end{num}\end{prop}
\begin{proof}
     To prove (a), assume without loss of generality that   $\alpha\in \Pi$.
     Since $\al\not\perp \Pi$ by irreducibility of $W$,  one has 
     $\aff(W\al)=\al+\real \Pi=\real \Pi$ by Lemma \ref{x1.13}(d).
    Hence 
    $\real(W\alpha)=\real \Pi=\real\Phi$, and the rest of (a) follows since
    $W$ acts faithfully on $\Phi$.

To prove (b), it will suffice to show that $\Pi\seq W_{\G}\G$.    Fix  $\al\in \Pi\sm\Pi_{I}$ and let  $\bt\in \Pi$.
         By irreducibility of $W$, there exist $n\in \Nat_{>0}$,   and a sequence $\al=\al_{0},\ldots, \al_{n}=\bt$ of distinct simple roots  such that
    $\mpair{\al_{i-1},\al_{i}}\neq 0$ for $i=1,\ldots, n$.
    Define $\g_{i}:=s_{\al _{0}}\cdots s_{\al _{i-1}}(\al _{i})$ for $i=0,\ldots, n$. 
    By Lemma \ref{x1.16}, $s_{\g_{i}}$ has a reduced expression containing $s_{\al _{0}}$ and hence $\g_{i}\in \Phi\sm \Phi_{I}=\G$.   It follows by induction 
    on $i$ that   $\al _{i}\in W_{\G}\G$ and $s_{\al_{i}}\in W_{\G}$  for $i=0,\ldots, n$. In particular, $\bt=\al_{n}\in W_{\G}\G$. Since $\bt\in \Pi$ is arbitrary, this gives $\Pi\seq W_{\G}\G$ as required to complete the proof of the first statement in (b).

   If there is  a proper standard  parabolic subgroup $W_{I}$ with $\Phi_{+}\sm \D\seq \Phi_{I}$, then (c) follows from (b) since $\D\supseteq \Phi_{+}\sm \Phi_{I}$. Such a subgroup $W_{I}$ certainly exists if $W$ is of infinite rank, so we  may and do assume for the rest of this proof that $\Pi$ is finite.
    It is easy to check that (c) holds if $(W,S)$ is infinite 
    dihedral, and we now reduce to that case. In fact, let
     $\beta\in \Phi$. Since  $(W,S)$ is irreducible of finite rank and infinite, it follows 
     from   \cite[Proposition 3.13]{HRT} that
    there is some $\alpha\in \Phi$ such that
     $W':=\mpair{s_{\alpha},s_{\beta}}$ is infinite. Let 
     $\Psi:=\Phi_{W'}$, $\G:=\D\cap \Psi\seq \Psi_{+}$.
     Note that $\Psi_{+}\sm \G\seq \Phi_{+}\sm \D$ is a finite set. From the dihedral case, it follows that $s_{\beta}\in W_{\G}\seq W_{\D}$,
     $\beta \in W' \G\seq W\D$  and $\beta\in \real \G\seq \real \D$.
     Since $\beta\in \Phi $ is arbitrary, the conclusions of (c)  now follow trivially. 
        \end{proof}
        
     \subsection{}\label{x1.19}   Despite  its formulation below, the 
     next fact is purely graph-theoretic in nature.  Though a  more general 
     fact is established in \cite{Wat}, we provide  for  the reader's 
     convenience a   self-contained proof of the special case used 
     here.
\begin{lem}  Suppose $W$ is  irreducible, $S$ is infinite  and the Coxeter graph is locally finite (i.e. for each $\al\in \Pi$ there are only finitely many $\bt$ in $\Pi$ such that $\mpair{\al,\bt}\neq 0$).
Fix $\al\in \Pi$. Then there is an infinite sequence $\al=\al_{0},\al_{1},\ldots$  of distinct  roots in $\Pi$ such that for $i<j\in \Nat$,  the distance from $\al_{i}$ to $\al_{j}$ in the Coxeter graph is $j-i$. In particular, if $i<j$, then  $\mpair{\al_{i},\al_{j}}\neq 0$ if and only if $j=i+1$. \end{lem}
\begin{proof}
For $\g,\bt\in \Pi$,   a path from $\g$ to $\bt$ is a sequence $\g=\g_{0},\ldots, \g_{n}=\bt$ in $\Pi$ with $\mpair{\g_{i-1},\g_{i}}\neq 0$  for $i=1,\ldots, n$.  
Such paths exist by irreducibility of $W$. Define $d(\g,\bt)$ to be the minimum length $n$ of all such paths from $\g$ to $\bt$. Note that if the above path
satisfies $n=d(\g,\bt)$, then $d(\g_{i},\g_{j})=j-i$ for all $0\leq i\leq j\leq n$.

For each $n\in \Nat$, let $\Pi_{n}:=\mset{\bt\in \Pi\mid d(\al,\bt)=n}$. 
If $\bt\in \Pi_{n}$, we may choose a path 
$\al=\al_{0},\ldots, \al_{n}=\bt$ from $\al$ to $\bt$, and one 
necessarily has $\al_{i}\in \Pi_{i}$ (and so $\Pi_{i}\neq \eset$) for 
$i=0,\ldots, n$.
Using local finiteness of the Coxeter graph, it follows by induction on 
$n$ that $\Pi_{n}$ is finite for all $n$, since for $n>0$, any root in
 $\Pi_{n}$ is the endpoint of one of the (finitely many, by induction) 
edges with one vertex in $\Pi_{n-1}$. Since   $\Pi$ is infinite,  there 
is for any $m\in \Nat$ some $n>m$ with $\Pi_{n}\neq \eset$, and as above, this implies $\Pi_{m}\neq \eset$. Hence 
each  $\Pi_{n}$ is a non-empty finite set. Note $\Pi_{0}=\set{\al}$. Give $\Pi_{n}$ the discrete topology and  endow the product set
$X=\prod_{n\in \Nat}  \Pi_{n}$ with  the product topology.  Thus, $X$ is a compact Hausdorff space. For each $n\in \Nat$, let \begin{equation*}
X_{n}:=\mset{(\al_{0},\al_{1},\ldots)\in X\mid \al_{0},\al_{1},\ldots, 
\al_{n} \text{ \rm is a path from $\al=\al_{0}$ to $\al_{n}$}}.\end{equation*} One has for $n\in \Nat$ that $X_{n}\neq \eset$; for one may 
choose in turn $\al_{n}\in \Pi_{n}$, a path $\al_{0},\ldots, \al_{n}$ from $\al$ to $\al_{n}$ and then arbitrary $\al_{m}\in\Pi_{m}$ for 
$m>n$ to obtain an element $(\al_{0}, \al_{1},\ldots)\in X_{n}$. Also,
one easily sees from the definition of product topology that $X\sm X_{n}$
is open in $X$ and hence   that $X_{n}$ is closed in $X$. Obviously, $X_{0}\sreq X_{1}\sreq\ldots$. By compactness of $X$, $Y:=\cap_{n\in \Nat}X_{n}\neq \eset$.
Let $(\al_{0},\al_{1},\ldots)\in Y$. From above, $d(\al_{i},\al_{j})=j-i$ for all $i< j$ in $\Nat$. In particular,  if  $\mpair{\al_{i},\al_{j}}\neq 0$, then
$1=d(\al_{i},\al_{j})=j-i$, and 
so  the sequence $\al_{0},\al_{1},\ldots$ has the required properties.
\end{proof}    

 \subsection{} \label{x1.20}  In the special case that $S$ is finite, the next result  has been proved by Deodhar (and (a) 
independently  by Howlett; see \cite[Proposition 4.2(x) and Remark 4.3]{D}).  
\begin{lem}
Suppose that  $(W,S)$ is an  irreducible Coxeter system and  $I\sneq S$\begin{num}
\item If  $W/W_{I}$ is finite, then $W$ is finite. 
\item If $\Phi\sm \Phi_{I}$ is finite, then $W$ is finite.
\end{num}  
\end{lem}
\begin{proof} First we prove (a). Using the references preceding the statement of the lemma,  suppose  without loss of generality  that $S$ is infinite.  We shall  obtain a contradiction after first showing   that $(W,S)$ is of \emph{locally finite type}  i.e. that $W_{K}$ is finite
for  every  finite subset $K$ of $S$. (This notion should not be confused with that of $(W,S)$ having locally finite Coxeter graph in \ref{x1.19}. We remark that it follows easily from the 
classification of finite Coxeter groups that the possible types of  infinite rank, irreducible, Coxeter systems of locally finite type are   $A_{\infty}$, $A_{\infty,\infty}$, $B_{\infty}$ or $D_{\infty}$; see  \cite{Kac}  or  \cite{DyRig} for the notation, 
or \cite[Figure 1]{NuidaLocPar} where $A_{\infty}$, $A_{\infty,\infty}$ are denoted as $A_{\infty}^{(1)}$, $A_{\infty}^{(2)}$. This classification can be used instead of the reference to Lemma \ref{x1.19}  in the proof below.)

Since $W^{I}$ is a set of coset representatives for $W/W_{I}$, it is finite. Hence  there is 
 some standard parabolic subgroup $W_{J}$, with $J\seq S$ finite, such that 
$W^{I}\seq W_{J}$. By irreducibility of $W$,   one may assume
without loss of generality  that $W_{J}$ is irreducible and $J$ is sufficiently large  that $K\seq J$ and   
$J\cap I\sneq J$.  Then $W_{J}/W_{J\cap I}\cong W_{J}^{J\cap I} \seq W^{I}
$ is finite, which implies $W_{J}$ is finite by the special case. Since
 $K\seq J$, $W_{K}$ is finite too.
Since the vertex degrees in Coxeter graphs of finite Coxeter systems are 
bounded above, 
the Coxeter graph of $(W,S)$ is locally finite. Lemma \ref{x1.19} implies  that  there is 
for any $\al\in \Pi$ a sequence $\al=\al_{0},\al_{1},\ldots, \al_{n},\ldots$ 
of distinct simple roots such that $\mpair{\al_{n-1},\al_{n}}\neq 0$ for all  $n\geq 1$. Choose $\al\in \Pi\sm \Pi_{I}$. For all  $n\in \Nat$,
the element   $w_{n}:=s_{\al_{n}}\cdots s_{\al_{0}}$  of $W$ is of length $n+1$
(using   Lemma \ref{x1.16}, for instance). Clearly,   $w_{n}$ has a unique reduced expression.  Since $s_{\al_{0}}\not\in I$, the elements $w_{n}$, for  $n\in \Nat$, are distinct elements of $W^{I}$. Hence $W^{I}$ is infinite,  which is a contradiction. 

Now (b) follows from (a) in the same way as its special case in which 
$S$ is finite, by the argument of  \cite[Remark 4.4]{D}. Alternatively, note that
$\Phi\sm \Phi_{I}$ finite implies $W=W_{\Phi\sm \Phi_{I}}$ is finitely generated by Proposition \ref{x1.18}, so $S$ is finite
    (since it is a minimal set of generators of $W$) and $W$ is finite by the result of \cite[Remark 4.4]{D}.
 \end{proof}

 \subsection{} The next result extends   \cite[Lemma 1.22]{DyQuo}.
 
 \label{x1.21}
 \begin{lem}
 Assume that $\Pi$ is finite. Let  $\rho\colon V\to \real$ be  a linear function such that
$\rho(\Pi)\seq \real_{>0}$.
 \begin{num}
 \item There is a positive scalar $\e_{0}$ such that
 for any two non-orthogonal roots $\alpha,\beta\in \Phi$ one has $\vert \mpair{\alpha,\ck\beta}\vert \geq \e_{0}$. 
 \item Let $\e:=\min(\mset{\rho(\al)\e_{0}/2\mid \al\in \Pi}$.
Then for all $\bt\in \Phi_{+}$,  $\rho(\bt)\geq\e l(s_{\bt})$.
 \end{num} 
 \end{lem}
\begin{proof}
For (a), see  \cite[Lemma 1.22(i)]{DyQuo}.
Part (b) is proved by induction on $l(s_{\bt})$ in the same way as   \cite[Lemma 1.22(ii)]{DyQuo}, making use of \cite[Proposition 3.4]{DySd}.
\end{proof}

\subsection{}\label{x1.22} Roots  in the fundamental chamber 
 may be regarded as general analogues of highest (long or short) 
 roots in root systems of  irreducible finite Weyl groups. For 
 irreducible $W$, their existence is equivalent to finiteness of $W$,
  as the next result shows. There are  generalizations      concerning orbits of standard   parabolic subgroups of $W$ on   $\Phi$ which we do not go into here.
 \begin{cor}\begin{num}
 \item If $W$ is finite, then the $W$-orbit of any root contains a unique element of $\Phi\cap \CC_{W}$.
\item  Suppose that  $\al\in \Phi\cap \CC_{W}$. Then $\al\in \Phi_{+}$. If, further,  $(W,S)$ is irreducible, then $W$ is finite. 
 \end{num}  \end{cor}

\begin{proof}
Part (a) follows from \ref{x1.10}(h),(c). Now let $\al\in\Phi\cap \CC$. Write $\al=\sum_{\g\in \G}c_{\g}\g$ where $\G\seq \Pi$, $\G$ is finite and  either $c_{\g}< 0$ for all $\g\in \G$ or 
$c_{\g}> 0$ for all $\g\in \G$. 
Since  $1=\mpair{\al,\al}=\sum_{\g\in \G}c_{\g}\mpair{\g,\al}$
where each $\mpair{\g,\al}\geq 0$ (because $\al\in \CC$), it follows that we must have $c_{\g}>0$ for all $\g\in \G$, and so $\al\in \Phi_{+}$.

Now assume that $(W,S)$ is irreducible. We claim that $\G=\Pi$. If not, by irreducibility of $\Pi$, there is some $\bt\in \Pi\sm\G$ and $\g_{0}\in \G$ such that $\mpair{\g_{0},\bt}<0$.
One also has $\mpair{\g,\bt}\leq 0$ for all $\g\in \G$.
Since $\al\in\CC$, it follows that  \begin{equation*}
0\leq\mpair{\bt,\al}=\sum_{\g\in \G}c_{\g}\mpair{\bt,\g}\leq c_{\g_{0}}\mpair{\bt,\g_{0}} <0, \end{equation*} a contradiction.

Since $\Pi=\G$ is finite, we may choose $\rho,\e$ as in Lemma
\ref{x1.21}.  By Lemma 1.6, for any $w\in W$,  one has
$\al-w(\al)\in \real_{\geq 0}\Pi$  and so $\rho(w\al)\leq \rho( \al)$.
Assume that $w\al\in \Phi_{+}$.
By Lemma \ref{x1.21} $l(s_{w(\al)})\leq \e^{-1}\rho(w\al)\leq \e^{-1} \rho(\al)
=:N$.  But there are at most $\vert \Pi\vert^{m}<\infty$ elements of $W$ of 
length $m$ for any  $m\in \Nat_{\leq N} $. This implies that $ W\al\cap \Phi_
{+}$ is finite, and hence
$W\al=(W\al\cap \Phi_{+})\cup -(W\al\cap \Phi_{+})$ is finite.
By Proposition \ref{x1.18}(a), it follows that 
$W$ is finite.
\end{proof}

\section{Facial subgroups}
\label{x2}
This section contains basic facts about facial subgroups of $W$, which are defined as
the stabilizers of  points of the Tits cone. Some of these facts are proved  under stronger assumptions   in  \cite{Bour}, \cite{V}, and \cite{K}.
\subsection{} \label{x2.1} A subset  $\Pi'$ of $ \Pi$ is said to be a 
\emph{facial subset} of $\Pi$, or said to be  \emph{facial}, 
  if $\Pi'=\Pi\cap v^{\perp}$ for some $v\in \mc{C}$.  
  Obviously, $\Pi=\Pi\cap 0^{\perp}$ itself is facial. It is not clear
without additional assumptions on the root system  if there are other facial subsets.

 If $\Pi'$ is facial, then  $I=\mpair{s_{\alpha}\mid \alpha\in \Pi'}$ is 
 said to be a \emph{facial subset} of $S$ and   $W_{I}$ is called  a 
 \emph{standard  facial subgroup}.   Using that $\CC$ is a cone, 
 one sees that a finite  intersection of facial subsets 
of $\Pi$ (resp., facial subsets of $S$, resp., standard facial  subgroups) is 
again a  facial subset of $\Pi$ (resp., facial subset of $S$, resp., standard 
facial  subgroup).
A \emph{facial subgroup} of $W$ is defined to be  to be a conjugate in $W$ of a  standard  facial subgroup.
\begin{rem}  The   
  facial subsets correspond naturally to the exposed faces (with respect to the semidual pair of cones $(\real_{\geq 0}\Pi,\mc{C})$) of $\real\Pi$  in the sense of Appendix \ref{xA}. Certain subtleties in the notion of facial subsets reflect subtleties in the notion of exposed faces. For instance, compare Remark \ref{x2.7} and Lemma \ref{x4.2}(f) with the facts that, for  general semidual pairs of cones, exposed faces of exposed faces  need not be exposed faces, whereas this  does hold for dual pairs of polyhedral cones. The notion of facial subset, subgroup etc   may change under  restriction or extension of quadratic space.    \end{rem}

\subsection{}\label{x2.2}  For $p\in V$, let $W_{p}$ denote the stabilizer of $v$ in $W$.
Then for all $w\in W$, one has  $W_{wp}=wW_{p}w^{-1}$.\begin{lem} Let $q\in \mc{X}$. Let    $w\in W$ be such that  $p:=w^{-1}q\in \mc{C}$. Let 
 $J$ be the facial subset of  $S$  with  $\Pi_{J}=\Pi\cap p^{\perp}$.
 \begin{num} \item  $W_{p }=W_{J}$,  $\chi(W_{p})=J$,
$ \Phi_{W_{p}}= \Phi\cap p^{\perp}$ and  $\Pi_{W_{p}}=\Pi\cap p^{\perp}$. 
\item $W_{q}=wW_{J}w^{-1}$  with $\Phi_{W_{q}}=\Phi\cap q^{\perp}$. \end{num}
\end{lem}
\begin{proof} By  \ref{x1.10}(c), $W_{p}=W_{J}$. The definition of $\chi(W_{J})$ easily implies that  $J\subseteq \chi(W_{J})$.
Hence $\chi(W_{J})=J$, since (see \cite[Ch IV]{Bour}) any set of Coxeter generators of $W_{J}$, such as $\chi(W_{J})$, is a minimal set of generators of $W_{J}$.   From $\chi(W_{J})=J$, it follows by definition that $\Pi_{W_{p}}=\Pi_{J}=\Pi\cap p^{\perp}$. The  assertion in (a) that $ \Phi_{W_{p}}= \Phi\cap p^{\perp}$, and (b), follow from \ref{x1.10}(d).\end{proof}

\subsection{}\label{x2.3}  For a  subset $I$ of $S$, let $\mc{C}(I):=\mset{v\in \CC\mid v^{\perp}\cap \Pi=\Pi_{I}}$, which is a cone in $V$. A non-empty set  of the  form
 $w\mc{C}(I)$ with $w\in W$ and $I\subseteq S$ is called a \emph{facet} of $\mc{X}$. (Note that this conflicts in some situations with a common usage of the word facet for a  codimension one face of a  cone; in general, $w\mc{C}(I)$ is not a face of $\mc{X}$.)   
 
 For any $\phi\in V$,  and a relation $\prec$ on $\real$ which is one of  the standard equality, order or inequality relations 
 $=$, $\leq$, $<$, $\geq$, $>$, or  $\neq$ let $H_{\phi}^{\prec}:=\mset{v\in V\mid \mpair{v,\phi}\prec 0}$.
 
 Basic properties of facial subgroups and facets are given in the following Lemma.
\begin{lem}\begin{num}\item  For $I\subseteq S$, 
 $\mc{C}(I)\neq \emptyset$ if and only if  $I$ is  facial.  \item Suppose that $I,J$ are  facial subsets of $S$
and $x,y\in W$. Then the intersection   $x\mc{C}(I)\cap y\mc{C}(J)$
is non-empty  if and only if  $I=J$ and $x^{-1}y\in W_{J}$, in which case $yx^{-1}\in xW_{J}x^{-1}=yW_{J}y^{-1}$ fixes $x\mc{C}(I)=y\mc{C}(J)$ pointwise.
\item The stabilizers of the points of $\mc{C}$ (resp.,  $\mc{X}$) are exactly the
standard facial subgroups (resp.,  facial subgroups).
\item Any (standard) facial subgroup of $W$ is (standard) parabolic.
\item The cone $\mc{X}$ is the disjoint  union of its facets.\item The facets of $\mc{X}$ are the non-empty intersections of the form  $\cap_{\alpha\in \Phi_{+}}H_{\alpha}^{\prec_{\alpha}}$ such that for $\alpha\in \Phi_{+}$, $\prec_{\alpha}$ is one of the relations $=$, $<$ or $>$,  with only finitely many  $\alpha$ for which  $\prec_{\alpha}$ is equal to $<$.
\item The conical hull of any finite union of facets of $\mc{X}$  meets only finitely many facets. In particular, the conical hull of any finite subset of $\mc{X}$ meets only finitely many  facets.\end{num}\end{lem}
\begin{proof}   Part (a) follows from the definition of facial subsets of $\Pi$, (b) follows from Lemma \ref{x1.10}, (c)--(d) follow from Lemma \ref{x2.2} and (e) follows from (b) (cf. \cite{K}). 

Now we prove (f). If $v\in F:=\cap_{\alpha\in \Phi_{+}}H_{\alpha}^{\prec_{\alpha}}$, then
for all $\alpha\in \Phi_{+}$, one has $\mpair{v,\alpha}\prec_{\alpha}0$ and hence  the relations $\prec_{\alpha}$ (amongst $=$, $<$, $>$) are uniquely determined. 
This implies that the intersections as in (f) are pairwise disjoint, and they are  contained in
$\mc{X}$ by Lemma \ref{x1.10}.
Let $w\in W$. Choose relations $\prec'_{\alpha}$ (amongst $=$, $<$, $>$)
satisfying $\mpair{wv,\alpha}\prec'_{\alpha}0$. Note that 
$\prec'_{\al}$ and $\prec_{w^{-1}\al}$ coincide  if $w^{-1}(\al)\in \Phi_{+}$.
If $w^{-1}(\al)\in \Phi_{-}$,  then
$\prec'_{\al}$ and $\prec_{-w^{-1}\al}$  are either both  $=$, or one of 
them is $<$ and the other is $>$. 
 It follows that $w(\cap_{\alpha\in \Phi_{+}}H_{\alpha}^{\prec_{\alpha}})=\cap_{\alpha\in \Phi_{+}}H_{\alpha}^{\prec'_{\alpha}}$.
 By Lemma \ref{x1.10}, only finitely many of the $\prec'_{\alpha}$ are equal to $<$.  This implies that the set of non-empty intersections as in (f) is $W$-stable. Hence it will suffice to show that the facets
of $\mc{X}$ which contain a point of $\mc{C}$ are precisely  the non-empty intersections $F$
with all $\prec_{\alpha}$ amongst $=$, $>$. But if $v\in\mc{C}$, then the unique facet
containing $v$ is  $\mc{C}(I)$ where  $\Pi_{I}=\mset{\alpha\in \Pi\mid \mpair{v,\alpha}=0}$ and the unique intersection as in (f) which contains $v$ is $F:=\cap_{\alpha\in \Phi_{+}}H_{\alpha}^{\prec_{\alpha}}$ where  $\mpair{v,\alpha}\prec_{\alpha}0$ with $\prec_{\alpha}$ either $=$ or $>$.
Consider a root $\alpha=\sum_{\beta\in \Pi}c_{\beta }\beta\in \Phi$ with all
$c_{\beta}\geq 0$. Then  $\prec_{\alpha}$ is  $=$ if and only if
$c_{\beta}=0$ for all $\beta\in\Pi\sm \Pi_{I}$. It is clear then that  \[F=\mset{v\in \mc{C}\mid \mpair{v,\alpha}=0\text{ \rm for all $\alpha\in \Pi_{I}$}}= \mc{C}(I)\] completing the proof of (f).

For the first claim of (g), let $F_{1},\ldots, F_{n}$ be facets. Write $F_{i}=\cap_{\alpha\in \Phi_{+}}H_{\alpha}^{\prec_{\alpha,i}}$ as in (f).  It will suffice to show that there are only finitely many facets $F=\cap_{\alpha\in \Phi_{+}}H_{\alpha}^{\prec_{\alpha}}$ containing some point
$x=\sum_{i=1}^{n}\lambda_{i}x_{i}$ with $x_{i}\in F_{i}$ and all  $\lambda_{i}>0$. Assume  $x\in F$, and let $\alpha\in \Phi_{+}$. Clearly,  $\prec _{\alpha}$ is $=$ if all $\prec _{\alpha,i}$ are $=$,
while $\prec _{\alpha}$ is $>$ if no $\prec _{\alpha,i}$ is $<$ and 
at least one  $\prec _{\alpha,i}$ is $>$. There are only finitely many $\alpha\in \Phi_{+}$
which do not fall under one of those two cases, namely the $\alpha$ for which some
$\prec_{\alpha,i}$ is $<$, and for  each of these, there are at most three possibilities ($=$, $<$, 
$>$) for $\prec_{\alpha}$.   It follows that there are only finitely many possible families $(\prec_{\alpha})_{\alpha\in \Phi_{+}}$ and so only finitely many possible $F$. This proves  the first assertion of (g), and the second follows since each point of $\mc{X}$ lies in some facet
(in the case of two points, compare \cite[Ch V, \S 4, Prop 6]{Bour}).
\end{proof}
 \begin{rem}  If 
$\Pi$ is linearly independent and  $(\Phi,\Pi)$ on $(V,\mpair{-,-})$ is ample (see Remark \ref{x1.3}(2)), then the classes of  (standard) facial subgroups  and (standard)  parabolic subgroups coincide. 
\end{rem}

\subsection{} \label{x2.4} Parts  (a)--(c) at least of the following are trivial or well-known in case $\Pi$ is linearly independent
(see \cite[Proposition 3.3]{BH2}).

\begin{lem} Let $I$ be a facial subset of $S$.  Fix $a\in \mc{C}$ such that
$\Pi_{I}=\Pi\cap a^{\perp}$.Then \begin{num}
\item  $\real_{\geq 0}\Pi\cap a^{\perp}=\real_{\geq 0}\Pi\cap \real \Pi_{I}=\real_{\geq 0}\Pi_{I}$.
\item $\Phi\cap a^{\perp}=\Phi\cap \real \Pi_{I}=W_{I}\Pi_{I}=\Phi_{W_{I}}$.
\item   $\Phi_{+}\cap a^{\perp}=\Phi_{+}\cap \real \Pi_{I}=\Phi_{+}\cap \real_{\geq 0}\Pi_{I}=W_{I}\Pi_{I}\cap \Phi_{+}=\Phi_{W_{I},+}$. 
\item If $k\in \mc{C}$, $w\in W$ and $k-wk\in \real\Pi_{I}$,
then  $wk=w'k$ for some $w'\in W_{I}$.
\end{num}
\end{lem}
\begin{proof}  The statement obtained by replacing ``$=$'' by ``$\supseteq$'' in (a) is clear.
For the reverse inclusions, suppose that $v=\sum_{\alpha\in \Pi}c_{\alpha}\alpha\in
\real_{\geq 0}\Pi\cap a^{\perp}$, where all $c_{\alpha}\geq 0$. We have  $0=\mpair{v,a}=\sum_{\alpha\in \Pi}c_{\alpha}\mpair{\alpha,a}$.
Since $\mpair{\alpha,a}\geq 0$ for all $\alpha\in \Pi$, we get $c_{\alpha}=0$ if $\mpair{\alpha,a}\neq 0$ i.e. if $\alpha\not\in \Pi_{I}$, which proves (a).

In (c), the first two equalities follow  by intersecting each part of (a) with $\Phi_{+}$. It is also clear that  $\Phi_{W_{I},+}=W_{I}\Pi_{I}\cap \Phi_{+}\subseteq \Phi_{+}\cap \real_{\geq 0}\Pi_{I}$. 
For the reverse inclusion, let  $\beta\in \Phi_{+}\cap \real_{\geq 0}\Pi_{I}$.
Write $\beta=\sum_{\alpha\in \Pi_{I}}c_{\alpha}\alpha$ with all $c_{\alpha}\geq 0$.
If $\beta\in \Pi$, then $\beta\in \Pi_{I}$, else  $\mpair{\beta,\Pi_{I}}\subseteq \real_{\leq 0}$ and 
\[1=\mpair{\beta,\beta}=\sum_{\alpha\in \Pi_{I}} c_{\alpha}\mpair{\beta,\alpha}\leq 0.\] 
We show that $\beta\in W_{I}\Pi_{I}$ in general by induction on $l(s_{\beta})$. Assume $l(s_{\beta})>1$.
Using \cite[Proposition 3.4]{DySd}, choose $\gamma\in \Pi$ with $l(s_{\gamma}s_{\beta}s_{\gamma})=l(s_{\beta})-2$.
and  $\mpair{\gamma,\beta}>0$.
If  $\gamma\in \Pi_{S\setminus I}$, then $\mpair{\beta,\gamma}=\sum_{\alpha\in \Pi_{I} }c_{\alpha}\mpair{\alpha,\gamma}\leq 0$, a contradiction.  Hence $\gamma\in \Pi_{I}$. Since $\beta\in \Phi_{+}$, $\gamma\in \Pi$ and $\gamma\neq \beta$, we have $s_{\gamma}\beta\in \Phi_{+}$. Hence  
  \[s_{\gamma}(\beta)=\beta-\mpair{\beta,\ck\gamma}\gamma\in \Phi_{+}\cap \real
\Pi_{I}= \Phi_{+}\cap \real_{\geq 0}\Pi_{I}.\]By induction, $s_{\gamma}(\beta)\in W_{I}\Pi_{I}$ and so $\beta\in W_{I}\Pi_{I}\cap \Phi_{+}$, completing the inductive proof of (c). Part (b) follows easily from (c).

For (d), choose a reduced expression $w=s_{\alpha_{1}}\cdots s_{\alpha_{n}}$ for $w$ in $W$, with $\alpha_{i}\in \Pi$.
We have, using Lemma \ref{x1.13},  
\[w(k)-k=-\sum_{i=1}^{n }\mpair{k,\ck\alpha_{i}}\beta_{i}\] where
$\beta_{i}:=s_{\alpha_{1}}\ldots s_{\alpha_{i-1}}(\alpha_{i})$. Therefore
\[0=\mpair{wk-k,a}=-\sum_{i}\mpair{k,\ck\alpha_{i}}\mpair{\beta_{i},a}\]
where  $\mpair{k,\alpha_{i}}\geq 0$,  and $\mpair{\beta_{i},a}\geq 0$ since $\beta_{i}\in \Phi_{+}$.
It follows that  $\mpair{\beta_{i},a}=0$ for any $i$ with $\mpair{k,\alpha_{i}}\neq 0$. If $\mpair{k,\alpha_{i}}\neq 0$, then $\beta_{i}\in \real_{\geq 0}
\Pi\cap a^{\perp}=\real_{\geq 0}
\Pi_{I}$. Then by (b), $\beta_{i}\in W_{I}\Pi_{I}\cap \Phi_{+}$ if  $\mpair{k,\alpha_{i}}\neq 0$
and so $s_{\beta_{i}}\in W_{I}$ for such $i$. Suppose that $\mpair{k,\alpha_{i}}\neq 0$
for $i=i_{1},\ldots, i_{m}$ where $1\leq i_{1}<\ldots <i_{m}\leq n$ and that $\mpair{k,\alpha_{i}}=0$ for $i\not\in \mset{i_{1},\ldots, i_{m}}$. Then
\[s_{\beta_{i_{1}}}\cdots s_{\beta_{i_{m}}}w(k)=s_{\alpha_{1}}\cdots  \widehat s_{\alpha_{i_{i}}}
\cdots \widehat s_{i_{m}}\cdots s_{\alpha_{n}} (k)=k\] where circumflexes denote factors  omitted from the product, 
and therefore\[w(k)=s_{\beta_{i_{m}}}\cdots  s_{\beta_{i_{1}}}(k)\in W_{I}k.\qedhere\] \end{proof}
\subsection{} Some of the  results above   hold for  facial subgroups in general. Notably: \label{x2.5}
\begin{cor} Let $W'$ be a facial subgroup of $W$, say $W'=\stab_{W}(q)$ where $q\in \mc{X}$ . Then $\Phi_{W',+}=W'\Pi_{W'}\cap \Phi_{+}=\real\Pi_{W'}\cap \Phi_{+}=\real_{\geq 0}\Pi_{W'}\cap \Phi=\Phi_{+}\cap q^{\perp}$.
\end{cor}
\begin{proof} Write $q=wp$ where $w\in W$ and $p\in \CC$, so  $W'=wW_{I}w^{-1}$ where $W_{I}=\stab_{W}(p)$ is standard facial. Assume without loss of generality that  $w\in W^{I}$. 

By  \ref{x1.7}, $\Pi_{W'}=w\Pi_{I}$ and by
 Lemma  \ref{x2.4}(b), $\Phi_{I}=W_{I}\Pi_{I}=\real \Pi_{I}\cap \Phi=\Phi\cap p^{\perp}$. Hence 
 \begin{equation*}
W'\Pi_{W'}= \Phi_{W'}=w(\Phi_{I})=\real w(\Pi_{I})\cap w(\Phi)=\real\Pi_{W'}\cap \Phi. \end{equation*} Intersecting with $\Phi_{+}$ gives the first two equalities in the statement of the corollary.
The third holds since  $\real_{\geq 0}\Pi_{W'}\cap \Phi\seq \real\Pi_{W'}\cap  \Phi_{+}=\Phi_{W',+}\seq \real_{\geq 0}\Pi_{W'}\cap \Phi$.
 Lemma \ref{x1.10}(d) gives
$\Phi_{W'}=\Phi\cap q^{\perp}$ and intersecting with $\Phi_{+}$ gives the remaining assertion.

\end{proof}

 \subsection{} \label{x2.6} Next we describe the intersections of reflection subgroups with (standard) facial subgroups.

 \begin{prop}  Let $W'$ be a reflection subgroup of $W$. \begin{num}\item  Let $J\subseteq S$ be facial. Then $W'':=W'\cap W_{J}$ is a standard facial subgroup of $W'$, with  $\real_{\geq 0} \Pi_{W''}=\real_{\geq 0}\Pi_{W'}\cap \real \Pi_{J}$,  $\Pi_{W''}=\Pi_{W'}\cap \Phi_{J}=\Pi_{W'}\cap \real \Pi_{J}$,
 $\Phi_{W'',+}=\Phi_{W',+}\cap \Phi_{J}=\Phi_{W',+}\cap \real \Pi_{J}$ and  $\chi(W'')=
\chi(W')\cap W_{J}$.  \item 
If  $W''$ is any facial subgroup of $W$, then $W''':=W'\cap  W''$ is a  facial subgroup of $W'$
with $\Phi_{W''',+}=\Phi_{W',+}\cap \Phi_{W''}=\Phi_{W',+}\cap \real\Pi_{W''}$.
\end{num}\end{prop}
\begin{proof}  To prove (a), suppose that $W_{J}=W_{p}$ where $p\in \mc{C}$.
Then $J=\Pi\cap p^{\perp}$. We have $p\in \mc{C}_{W'}$ by \ref{x1.10}(e),
so by \ref{x2.2}(a) applied to $W'$, $W''=W'\cap W_{p}=(W')_{p}$ is a standard facial subgroup of $W'$.
By two applications of Lemma \ref{x2.4}(a), 
\begin{equation*}\begin{split}
\real_{\geq 0} \Pi_{W''}&=\real_{\geq 0}\Pi_{W'}\cap p^{\perp}= 
\real_{\geq 0}\Pi_{W'}\cap( \real_{\geq 0}\Pi\cap p^{\perp})\\
&=\real_{\geq 0}\Pi_{W'}\cap ( \real_{\geq 0}\Pi\cap \real \Pi_{J})=
\real_{\geq 0}\Pi_{W'}\cap  \real \Pi_{J}.\end{split}
\end{equation*}  Clearly, 
$\Pi_{W''}\seq\Pi_{W'}\cap \Phi_{J}\seq\Pi_{W'}\cap \real \Pi_{J}$ and 
 $\Phi_{W'',+}\seq\Phi_{W',+}\cap \Phi_{J}\seq\Phi_{W',+}\cap \real \Pi_{J}$.
 Intersecting $\real_{\geq 0} \Pi_{W''}=\real_{\geq 0}\Pi_{W'}\cap \real \Pi_{J}$
 with $\Pi_{W''}$ and $\Phi_{W'',+}$ in turn shows that the above inclusions are actually equalities. Finally for (a), note that 
 \begin{equation*}
 \chi(W')=\mset{s_{\alpha}\mid \alpha\in \Pi_{W''}}=\mset{s_{\alpha}\mid \alpha\in \Pi_{W'}\cap \Phi_{J}}=\chi(W')\cap \Phi_{J}.
 \end{equation*}

For  (b), write $W''=W_{q}$ for some $q\in \mc{X}$. By \ref{x1.10}(e) and \ref{x2.2}(b),
 $W'''=W'\cap W_{q}=(W')_{q}$ is a facial subgroup of $W'$
with $\Phi_{W'''}=\Phi_{W'}\cap q^{\perp}$. The rest of (b) follows using Corollary \ref{x2.5}. \end{proof}

 \subsection{} \label{x2.7}  The preceding lemma implies that an intersection of two facial subgroups is a facial subgroup of both of them. The next lemma shows in particular that the  intersections of two facial subgroups of $W$ is a facial subgroup of $W$.

  \begin{lem}\begin{num}
 \item The stabilizer of a point  $x$ of $\mc{X}$ coincides with the pointwise stabilizer of the facet of $\mc{X}$ containing  $x$.
 \item The    pointwise  stabilizer of a finite subset $Y$ of $\mc{X}$   is the stabilizer of some point  in the convex hull of $Y$. 
 \item The intersection of finitely many (standard) facial subgroups of $W$ is a (standard) facial subgroup of $W$.
 \end{num}\end{lem}
 \begin{rem}  This would follow from the previous lemma if  it was known that a (standard) facial subgroup of a (standard) facial subgroup of $W$ is a (standard) facial subgroup of $W$,  but it seems possible that this fails in this generality (see \ref{x4.2}(f) and Remark \ref{x2.1}).
 \end{rem}
 \begin{proof} Part (a) follows from \ref{x2.3}(b).
 For the proof of (b), denote the stabilizer in $W$ of $x\in \mc{X}$ as $W_{x}$. It suffices to show that if $x,y\in \mc{X}$, then $W_{x}\cap W_{y}=W_{z}$ for some $z$ in the closed segment $[x,y]=\set{tx+(1-t)y\mid
 t\in \real,  0\leq t\leq 1}$ (cf. the proof of \cite[Lemma 2.1.2]{K}). We may assume $x\neq y$. The set 
 $\Gamma:=\mset{\alpha\in \Phi_{+}\mid \mpair{x,\alpha}\mpair{y,\alpha}<0}$ is finite by Lemma \ref{x1.10}(a). We may and do choose $z=tx+(1-t)y$ with $0<t<1$ in $ [x,y]$ such that  for all $\alpha\in \Gamma$, $\mpair{z,\alpha}\neq 0$. Then for $\al\in \Phi_{+}$,  $\mpair{z,\alpha}= 0$  implies  that
 $\mpair{\alpha,x}=\mpair{\alpha,y}=0$. By \ref{x1.10}(d), we get $W_{z}\subseteq W_{x}\cap W_{y}$.
 The reverse inclusion  $W_{z}\supseteq W_{x}\cap W_{y}$ is clear, proving (b). Then (c) follows from (b) and \ref{x2.3}(c) on recalling that $\mc{X}$ (and $\mc{C}$) is convex. \end{proof}
\subsection{}\label{x2.8} Slightly extending the terminology of \cite{K}, define  a subgroup $W''$ of $W$ to be a \emph{parabolic closure}
(resp., \emph{facial closure}) of a subset $X$ of $W$ if $W''$ is a parabolic (resp., facial) subgroup $W$, and any parabolic (resp., facial) subgroup of $W$ containing $W'$ contains $W''$.
Such a parabolic (resp facial) closure exists if and only if
 the intersection of all parabolic (resp., facial) subgroups containing $X$ is
 parabolic (resp., facial), in which case it coincides with that intersection   and so is unique.
From the above, it follows that $X$ has a parabolic (resp., facial) closure 
if it is contained in some finite rank parabolic (resp., facial) subgroup  of $W$; 
then the parabolic (resp., facial)  closure  is the parabolic (resp., facial) subgroup of minimal rank amongst those containing $X$. Similarly, we define the \emph{standard parabolic closure} and \emph{standard facial closure} of any subset of $W$. The standard parabolic closure of any subset exists, and the standard facial closure of a subset $X$ exists if $X$ is contained in a finite rank standard facial subgroup.   
 \subsection{}\label{x2.9} Well-known properties of  finite subgroups of Coxeter groups and their parabolic closures are collected in the lemma below. For proofs, see for example   \cite{Hum}, \cite{Bour}  and   \cite[1.2.6, 3.2]{K}. 
\begin{lem} Let $W'$ be a reflection subgroup of $W$ and  $W''$ be the reflection subgroup of $W$ with  $\Phi_{W''}=\real\Pi_{W'}\cap \Phi$. \begin{num}
\item   A subgroup of $W$ is finite if and only if  it is contained in   some  finite parabolic subgroup of $W$.
\item The following conditions $\text{\rm (i)--(iii)}$ are equivalent:
\begin{subconds}\item  $W'$ is finite 
\item  $\Pi_{W'}$ is finite and the ``Cartan matrix'' 
$(\mpair{\alpha,\beta})_{\alpha,\beta\in \Pi_{W'}}$ is positive definite. 
\item   $\real\Pi_{W'}$ is finite dimensional and
the restriction of $\mpair{-,-}$ to $\real \Pi_{W'}$  is positive definite.\end{subconds}
\item If $W'$ is finite, then $\Pi_{W'}$ is linearly independent, $W''$ is finite, $W''$ has the same rank as $W'$ and $W''$ is the parabolic closure of $W'$. Further, every subset of $\Pi_{W'}$ is facial in $\Pi_{W'}$ (for  the Coxeter group $W'$). \end{num}\end{lem}
\begin{rem} In particular,
$\Pi_{W'}$ is non-degenerate if $W'$ is finite. 
 In general, all parts of (c) fail without finiteness of $W'$. \end{rem}

\subsection{} \label{x2.10}
The next result gives  information on the facial subsets of $S$.
The main facts we require  are (b) (which extends    \cite[Theorem 4]{V}  and  \cite[2.2.4]{K})
 and (d) below.

\begin{lem}\begin{num}
\item Let $\tau\in \mc{C}$ and $I\subseteq S$ such that $W_{I}$ is finite.
Let $\Delta$ denote the union of all components of $\Pi_{I}$ which are not contained in
$\tau^{\perp}$. Then  $\tau':=\sum_{w\in W_{I}}w\tau
\in \mc{C}$ and $\tau^{\prime\perp}\cap \Pi=\Pi_{I}\cup(\Pi\cap \tau^{\perp}\cap \Delta^{\perp})$
\item  Let $I$, $J$   be  subsets of $S$ such that  $J$ is facial, $W_{I}$ is finite, and $I$ and $J$ are separated. Then $I\cup J$ is facial. In particular,  if $\emptyset$ is facial, then $I$ is facial.
\item Let $I$ be a facial subset of $S$ and $K$ be a subset of $S$ such that $W_{K}$ is finite,
$K$ is separated from $I\setminus K$, and each component of $K$ contains an  element of $S\setminus I$. Then $I\setminus K$ is facial.
\item If $I$ is a facial subset of $S$ and $J\subseteq S$ is $W$-conjugate to $I$, then $J$ is  a facial subset of $S$.
\end{num}     \end{lem}\begin{proof}

First we prove (a). Let $J\subseteq S$ with $\Pi_{J}=\Lambda:=\Pi_{I}\cap \tau^{\perp}$.
 For $\alpha\in \Pi_{I}$, we have $s_{\alpha }\tau'=\tau'$ so $\mpair{\alpha,\tau'}=0$. Next, consider $\alpha\in \Pi\setminus \Pi_{I}$. We have $\mpair{\tau',\alpha}=\sum_{w\in W_{I}}\mpair{\tau,w\alpha}$.
Since $\alpha\in -\mc{C}_{I}$, we have $w(\alpha)\in \alpha+\real_{\geq 0 }\Pi_{I}$ by \ref{x1.13}.
In particular, $w(\alpha)\in \Phi_{+}$.
Hence each term $\mpair{\tau, w\alpha}\geq 0$, so it follows that $\mpair{\tau',\alpha}\geq 0$.
This shows that $\tau'\in \mc{C}$. Further, $\mpair{\tau',\alpha}=0$ if and only if   $\mpair{\tau,\alpha}=0$ and for each $w\in W_{I}$, $w(\alpha)-\alpha\in \real_{\geq 0}\Pi_{I}\cap \tau^{\perp}=\real_{\geq 0}\Lambda$.
By Lemma \ref{x2.9},  $\Lambda$ is facial in $\Pi_{I}$.
 Lemma \ref{x2.4}(c) and \ref{x1.13} therefore show  that $w(\alpha)-\alpha\in \real_{\geq 0}\Lambda$ if and only if  $w(\alpha)=w'(\alpha)$ for some $w'\in W_{J}$. But for $w'\in W_{J}$,  Lemma  \ref{x1.10} implies  that $w^{-1}w'\in W_{I}$ stabilizes
$\alpha\in -\mc{C}_{I}$ if and only if  $w^{-1}w'\in W_{I_{\alpha}}$ where $\Pi_{I_{\alpha}}=\Pi_{I}\cap \alpha^{\perp}$.  It follows that 
(for $\alpha\in \Pi\setminus \Pi_{I}$), $\mpair{\tau',\alpha}=0$ if and only if  
$\mpair{\tau,\alpha}=0$  and $W_{I}=W_{J}W_{I_{\alpha}}$. By Lemma \ref{x1.17},  however,  $W_{I}=W_{J}W_{I_\alpha}$ if and only if  every component of $\Pi_{I}$ which is not contained in $\Lambda$ (i.e. is not contained in $\tau^{\perp}$) is contained in $\Pi_{I_{\alpha}}$.  That is, for $\al\in \Pi\sm \Pi_{I}$, one has $\al\in (\tau')^{\perp}$ if and only if $\al\in \tau^{\perp}\cap \D^{\perp}$.
Since  it has already been noted that $\Pi_{I}\seq  (\tau')^{\perp}$,  this completes the proof of (a).

For (b), choose $\tau\in \mc{C}$ with $\tau^{\perp}\cap \Pi=\Pi_{J}$ and let $\tau'\in \mc{C}$ and $\Delta$ be as in (a). We have $\Delta=\Pi_{I}$ since $\Pi_{I}\cap \tau^{\perp}=\Pi_{I}\cap \Pi_{J}=\eset$. Hence  $\Pi\cap( \tau^{\prime})^{\perp}=\Pi_{I}\cup(\Pi_{J}\cap \Pi_{I}^{\perp})=\Pi_{I\cup J}$ by (a), since $\Pi_{I}$ and $\Pi_{J}$ are separated. 

Next we prove (c). Let   $\tau\in \mc{C}$ with $\Pi\cap \tau^{\perp}=\Pi_{I}$. Set $L:=I\setminus K$ and $J:=I\cap K$.
Note that $\Phi_{+}\setminus \Phi_{K\cup I}$  and $\Phi_{L,+}$  are $W_{K}$-invariant,
and $\Phi_{I\cup K}=\Phi_{K}\dot \cup \Phi_{L}$ since $I\cup K=K\dot\cup L$  where $K$ and $L$ are separated.  We have 
$ \mpair{\tau,\alpha}>0$ for all $\alpha\in \Phi_{+}\setminus \Phi_{I\cup K}$ and for all
$\alpha\in \Phi_{K,+}\setminus \Phi_{J}$, while 
$\mpair{\tau,\alpha}=0$ for all $\alpha\in \Phi_{L}$  and all  $\alpha\in \Phi_{J}$.  By \ref{x2.7}(b), the set of points $\mu$ in the convex hull $Y$ of $\mset{w\tau\mid w\in W_{K}}$ 
 with stabilizer $W_{\mu}=\cap_{w\in W_{K}}W_{w\tau}$  is non-empty, and it is clearly $W_{K}$-stable. Since $V=\mc{X}_{K}=W_{K}\mc{C}_{K}$ by finiteness of $W_{K}$ and Lemma \ref{x1.10}(h), we may choose a point  $\mu\in  Y\cap \mc{C}_{K}$ such that $W_{\mu}=\cap_{w\in W_{K}}W_{w\tau}$. In particular,
 $\mpair{\mu,\alpha}\geq 0$ for all $\alpha\in \Phi_{K,+}$. We have also 
 $ \mpair{\mu,\alpha}>0$ for all $\alpha\in \Phi_{+}\setminus \Phi_{I\cup K}$  and 
$\mpair{\mu,\alpha}=0$ for all $\alpha\in \Phi_{L}$ since $\mu$ is in the convex closure of $Y$.  We now show that
$\mpair{\mu,\alpha}>0$ for $\alpha\in \Phi_{K,+}$, which will imply that 
$\mu\in \mc{C}$ and $\Pi\cap \mu^{\perp}=\Pi_{L}=\Pi_{I\setminus J}$ as required to prove (c).
Note  that by the assumptions and   Lemma \ref{x1.16} applied to each component of $K$, each $W_{K}$-orbit on $\Phi_{K}$ contains a point of $ \Phi_{K,+}\setminus \Phi_{J}$. Hence for  $\alpha\in \Phi_{K,+}$ there is some $w\in W_{K}$ and $\beta\in \Phi_{K,+}\setminus \Phi_{J} $
with $w^{-1}\alpha=\beta$. We then have $\mpair{w\tau,\alpha}=\mpair{\tau,\beta}>0$,
so $s_{\alpha}\not \in W_{w\tau}$, $s_{\alpha}\not\in W_{\mu}$ by definition of $\mu$, and
 hence $\mpair{\mu,\alpha}\neq 0$. But we've already seen $\mpair{\mu,\alpha}\geq 0$, so
$\mpair{\mu,\alpha}>0$ as required. 

Finally we prove (d). Suppose that $wW_{I}w^{-1}=W_{J}$ where $I,J\subseteq S$ and $I$ is facial. Without loss of generality, we may assume that $w$ is a minimal length  $(W_{J},W_{I})$-double coset representative. Then $w\Pi_{I}=\Pi_{J}$. Using induction,   we may assume by \ref{x1.15}, that there
exists $\alpha\in \Pi\setminus \Pi_{I}$ such that, setting $s=s_{\alpha}\in S$,  the connected component $K$ of  $I\cup \set{s}$ containing  $s$ is of finite type and
$w=w_{K}w_{K\setminus\set{s}}$. Let $t=w_{K}sw_{K}$.
We have $\Pi_{I}=\Pi_{I\setminus K}\,\dot\cup\, \Pi_{K\setminus \set{s}}$ so \[\Pi_{J}=w\Pi_{I}=w(\Pi_{I\setminus K})\,\dot\cup\, w(\Pi_{K\setminus \set{s}})=
\Pi_{I\setminus K}\,\dot\cup\, \Pi_{K\setminus \set{t}}.\]
  By (c), we conclude that 
 $I\setminus K$ is facial. Since $K\setminus \set{t}$ is separated from $I\setminus K$ and is of finite type, we conclude from (b) that $J=(I\setminus K)\cup (K\setminus \set{t})$ is facial.
  \end{proof}
  \subsection{} \label{x2.11} The following   is an immediate corollary  of part (d) of the preceding lemma.
   \begin{cor} The    standard facial subgroups of $W$ are exactly the   standard parabolic subgroups of $W$ which are facial.  \end{cor}
   \subsection{} \label{x2.12}  The result below    refines  Lemma \ref{x2.7}(c) and the well-known formula of  Kilmoyer (see e.g. \cite[1.3.5(d)]{K}) for the intersection of two parabolic subgroups of $W$.  \begin{prop}
For any facial subsets  $I,J\subseteq S$ and shortest $(W_{I},W_{J})$ double coset representative $d$,  then
 \[W_{I}\cap dW_{J}d^{-1}=W_{I\cap dJd^{-1}}\]
 is a standard facial subgroup of $W$.
 \end{prop}
 \begin{proof}
 Let $W'=W_{I}\cap dW_{J}d^{-1}$, which by \ref{x2.7}(c) is a facial subgroup of $W$.
  Clearly, $d$ is of minimal length in its cosets $W_{I}d$, $dW_{J}$
 and $W'd$.
By \ref{x1.6} and  Proposition \ref{x2.6}(a) we  have $\chi(dW_{J}d^{-1})=dJd^{-1}$, $\chi(d^{-1}W_{I}d)=d^{-1}Id$, 
 $\chi(W')=dJd^{-1}\cap W_{I} $, and  \[d^{-1}\chi(W')d=\chi(d^{-1}W'd)=\chi(d^{-1}W_{I}d\cap W_{J})=d^{-1}Id\cap W_{J}.\] 
 Hence $\chi(W')=\chi(W')\cap d\chi(d^{-1}W'd)d^{-1}=I\cap dJd^{-1}$. Since $I\cap dJd^{-1}\subseteq S$, it follows that $W'=W_{I\cap dJd^{-1}}$ is a facial standard parabolic subgroup, and hence a standard facial subgroup 
 by Corollary \ref{x2.12}.
 \end{proof}
 \begin{rem} Suppose that the standard facial subgroups and standard parabolic subgroups of $W$ coincide (see Remark \ref{x2.3}). Then the Proposition specializes to Kilmoyer's formula
and  the argument  above   simplifies  since $W_{I \cap dJd^{-1}}$ is obviously  not just parabolic, but standard parabolic. This affords a simple proof of Kilmoyer's formula (not needing \ref{x2.11}). \end{rem}

  \subsection{}\label{x2.13} The proof of the following observation  is left to the reader.  See also Lemma  \ref{x8.5} for similar results under stronger hypotheses. 
  \begin{lem} Suppose that $V=U\oplus U'$ is an orthogonal direct sum decomposition of $V$ such that $\Pi=\Delta\dot\cup \Delta'$  where $\Delta:=\Pi\cap U$ and $\Delta':=\Pi\cap U'$.
  Then the  facial subsets of $\Pi$ are precisely the unions of facial subsets of $\Delta$ (for the group $W_{\Delta}$ acting naturally either as reflection group on $V$ or on
  $U$) and the facial subsets of $\Delta'$ (for $W_{\Delta'}$ as reflection group on $V$ or on $U'$).\end{lem}
  \begin{rem}
  Let $\Delta$ be a non-empty finite subset of $\Pi$ which is separated from
  $\Delta':=\Pi\setminus \Delta$ and such that that the restriction of $\mpair{-,-}$ to  $\real\Delta$
  is non-singular.  For example, $\Delta$ could be a finite union of finite type components of $\Pi$. Then the above applies with $U=\real \Delta$, 
  $U'=\Delta^{\perp}$ and $\Delta'=\Pi\sm \Delta$\end{rem}
\subsection{}\label{x2.14} 
Part (c) of the following was also obtained in \cite{NuidaLocPar}.

\begin{lem}\begin{num}\item  The intersection of any directed (by  reverse inclusion) family of reflection subgroups of $W$  is a reflection subgroup.
\item The pointwise stabilizer of a subset $X$ of $\mc{X}$  is the reflection subgroup with  root system   $\Phi\cap X^{\perp}$.
\item Any intersection of parabolic subgroups is a reflection subgroup.
\end{num}
\end{lem}
\begin{proof}  Following \cite{DyParClos}, we observe that for any reflection subgroup $W'$ of $W$ and $w\in W$, we have
$w\in W'$ if and only if  there exist $n\in \Nat$ and reflections  $t_{1},\ldots, t_{n}$ of $T$ such that,
setting $w_{i}=t_{i}\ldots t_{1}$ for $i=0,1,\ldots, n$ (so $w_{0}=1_{W}$), we have 
\begin{conds}\item $w_{n}=w$ 
\item  $l(w_{0})<l(w_{1})<\ldots <l(w_{n})$
\item $t_{i}\in W'\cap T$ for all $i$.
\end{conds}
In fact, $w\in W'$ if and only if  such $t_{i}$ exist with $t_{i}\in \chi(W')$ for all $i$; then $n=l_{(W',\chi(W'))}(w)$.

Suppose that $D$ is a family of reflection subgroups of $D$ which is directed by reverse inclusion  i.e. for $W',W''$ in $D$, there exists $W'''\in D$ with $W'''\subseteq W'\cap W''$. Let $W':=\cap_{U\in D}U$. Now for fixed $w\in W$, there are only finitely many tuples $(n,t_{1},\ldots, t_{n})$ with $t_{i}\in T$ satisfying (i) and (ii) (since $n\leq l(w)$ and each $w_{i}$ lies in the (finite)  Bruhat interval $[1,w]$). It follows that $w\in W'$ if and only if  there exists such a tuple  $(n,t_{1},\ldots, t_{n})$ with $t_{i}\in U$  for all $i=1,\ldots, n$ and all $U\in D$.
But then each $t_{i}\in W'\cap T$ and $w=w_{n}=t_{n}\ldots t_{1}\in \mpair{W'\cap T}$.
Hence $W'$ is a reflection subgroup as required to prove (a). We observe that that the argument also shows that $w\in \chi(W')$ if and only if  $l_{(W',\chi(W')}(w)=1$ if and only if   every tuple
$(n,t_{1},\ldots, t_{n})$ as above has $n=1$ (in which case, such a tuple is unique).

Part (c) follows since the intersection of a family of parabolic subgroups is the intersection of the
(directed by reverse inclusion) family of their finite intersections, and these finite intersections are (parabolic) reflection subgroups by Kilmoyer's formula.
In (b), note that  the pointwise stabilizer of $X$ is the intersection of the stabilizers of the points of $X$, which are parabolic (in fact, facial) subgroups. 
Hence the pointwise stabilizer is a reflection subgroup by (c). The remaining 
assertions are clear since a reflection $s_{\alpha}$ fixes $X$ pointwise if and
 only if  it fixes each point of $X$ i.e. if and only if   $\alpha$  is in $X^{\perp}$.   \end{proof}
\subsection{} \label{x2.15}  In \cite{NuidaLocPar}, an example is given to show that an intersection of an arbitrary family of parabolic 
subgroups of an infinite rank Coxeter system need not be a 
parabolic subgroup.   
The following result  gives a simple class of examples which illustrates the ubiquity of this phenomenon.
\begin{prop} Let  $(W,S)$ be an infinite rank irreducible Coxeter system with locally 
finite Coxeter graph. Then there is a family $(W_{n})_{n\in \Nat}$, of  
parabolic subgroups of $W$ such that $\cap _{n\in \Nat}W_{n}$ is a non-parabolic reflection subgroup with $\vert \Nat\vert $ components, which are   all of type $A_{1}$. \end{prop}
\begin{rem}  
In the example in \cite{NuidaLocPar}, the parabolic subgroups involved, and their intersection, are all
 irreducible of type $A_{\infty}$, which,  together with $A_{\infty,\infty}$, is  exceptional amongst all irreducible Coxeter groups in some respects (see \cite{DyRig}).    \end{rem}
\begin{proof} By Lemma \ref{x1.19}, there is a subset $\Pi_{0}=\set{\al_{0},\al_{1},\ldots}$  of $\Pi$ such that for $i<j\in \Nat$, $\mpair{\al_{i},\al_{j}}\neq 0$ if 
 and only if $j=i+1$. Let  $s_{i}:=s_{\al_{i}}$ for $i\in \Nat$.  It is clear that there is a subset 
 $J$ of $S$ with irreducible components $J_{i}$ for $i\in \Nat$ such that $\vert  J_{i}\vert =i+1$. For example, one could take $J:=\cup_{i\in \Nat }J_{i}$
 where $J_{i}:=\set{s_{2i^{2}},s_{2i^{2}+1},\ldots, s_{2i^{2}+i}}$.  By Lemma \ref{x1.16},  there is a reflection $t_{i}\in W_{J_{i}}$ of length $l(t_{i})=2i+1$.
 Let $K_{i}=\set{t_{0},t_{1},\ldots,t_{i}}\cup\bigcup _{j>i}J_{j}$.
 Since $\mpair{t_{k}}$ is a  parabolic subgroup of $W_{J_{k}}$ for $k\in \Nat$, it follows 
   that $W_{K_{i}}$ is a parabolic subgroup of $W$ (with $\chi(W_{K_{i}})=K_{i}$ as its set  of canonical generators). Let $W'=\mpair{t_{0},t_{1},\ldots}$, which is a reflection subgroup of $W$ with $\chi(W')=R:=\set{t_{0},t_{1},\ldots}$  as canonical generators (since $R$ is both the set of all reflections of $W'$ and  a minimal generating  set of $W'$).   Obviously, $W'$ has $\vert \Nat\vert$ components, all of type $A_{1}$.     
   We claim that $\cap_{i\in \Nat}W_{K_{i}}=W'$.
   Clearly, $W'\seq \cap W_{K_{i}}$ since $t_{j}\in W_{K_{i}}$ for all
    $j,i\in \Nat$. For the  reverse inclusion, let $w\in \cap W_{K_{i}}$.  Set 
   $L_{n}:=J_{1}\cup\ldots \cup J_{n}$.  Note that  
   $w\in W_{J}=\cup_{n\in \Nat}W_{L_{n}}$ so $w\in W_{L_{n}}$ for some  
   $n\in \Nat$. Hence 
   $w\in W_{L_{n}}\cap W_{K_{n}}=\mpair{\set{t_{0},\ldots, t_{n}}}\seq W'$ as claimed.
   
   To complete the proof, it suffices to show that $R$ is not $W$-conjugate to any subset of $S$. But since $\set{l(t)\mid t\in R}$ is not bounded above, it follows that for any $u\in W$,  $\set{l(utu^{-1})\mid t\in R}$
   is also not bounded above (since $l(utu^{-1})\geq l(t)-2l(u)$)  and so $uRu^{-1}\not\seq S$.
  \end{proof}

\subsection{Weakly facial subgroups and locally parabolic subgroups}\label{x2.16} Call a (reflection) subgroup of $W$  \emph{weakly facial} if it is an
 intersection of  (possibly infinitely many) facial subgroups, or equivalently, if it is the pointwise stabilizer of a subset of the Tits cone. There is a natural Galois connection between the subsets of $\mc{X}$ and subsets of $W$ (cf. \cite[2.1]{K}) with the weakly facial subgroups as the stable subsets of $W$. Also,   
  the set of weakly facial subgroups forms a complete lattice when ordered by inclusion.
 Any subset $X$ of $W$ has a weakly facial closure, defined to be  the weakly facial subgroup obtained by intersecting all (weakly) facial subgroups containing $X$.
 If $W$ is finitely generated,  the  set of weakly facial subgroups is equal to the set of facial 
 subgroups. 
Some of the results on facial subgroups have analogues for weakly facial subgroups; for 
example, intersections of weakly facial subgroups are weakly facial, and (using \ref{x1.10}(e)) 
the intersection of a weakly facial subgroup of $W$ with a reflection subgroup $W'$ of $W$ is 
a weakly facial subgroup of $W'$.

In \cite{NuidaLocPar}, a reflection subgroup $W'$  of $W$ is defined to be a 
\emph{locally parabolic subgroup} if every finite subset of its canonical generators
 $\chi(W')$ is $W$-conjugate to a subset of $S$. It is shown in op. cit. that parabolic 
subgroups are locally parabolic, and  that  arbitrary intersections of locally parabolic  (e.g. 
parabolic)  subgroups are locally parabolic.  Also, if $S$ is finite,  the notions of  parabolic 
and locally parabolic subgroups of $W$ coincide.

In particular, weakly facial subgroups are locally parabolic, and if $S$ is finite
and all parabolic subgroups are facial,  
the notions of parabolic, weakly facial and locally parabolic subgroups coincide.
It is not known for infinite $S$ whether every locally parabolic subgroup is equal
to some intersection of parabolic subgroups (or equivalently, if the weakly facial and 
locally parabolic subgroups coincide  when every parabolic subgroup is facial). We also leave open 
the question as to whether there is   a natural  analogue  of the notion of locally parabolic 
subgroups,  
reducing to the  facial subgroups for $W$ of finite rank, in the general situation  in which  not 
all parabolic subgroups are facial. 

\subsection{Parabolic closure of union of two special parabolic subgroups.} \label{x2.17}  
 The following result  extends to (possibly infinite rank) Coxeter groups a result established   for Weyl groups of   Kac-Moody Lie algebras in \cite[Theorem 4(b)]{Mok6}.
 
 \begin{prop} Let  $(W,S)$ be  any Coxeter system.
 \begin{num}
 \item Suppose that $J,K\seq S$ are special and $d$ is  of  minimal length  in its double coset
  $W_{J\cup J^{\perp}}dW_{K\cup K^{\perp}}$. Then the parabolic closure of  $W_{J}\cup dW_{K}w^{-1}$ exists and is  equal to the standard parabolic closure $W_{L}$ of $J\cup \set{d}\cup K$. 
 \item In \text{\rm (a)}, $W_{L}$ is special.
 \item The union of a  finite family of special   parabolic subgroups  has a parabolic closure, which is itself a  special parabolic subgroup. 
 \end{num}
  \end{prop}
  \begin{rem} We do not know if,  in arbitrary $W$,
the  union of finitely many parabolic subgroups necessarily has a parabolic closure.   Proposition \ref{x2.15} implies that in any infinite rank irreducible Coxeter group with locally finite 
Coxeter graph, there is a denumerable family of parabolic  subgroups (all of type $A_{1}$) for which their union does not have a parabolic closure.  
  \end{rem}
  \begin{proof}
     By definition of standard parabolic closure, in part (a) one has 
     $L:=J\cup K\cup\set{s_{1},\ldots, s_{n}}$ for some (and every) reduced expression $d=s_{1}\ldots s_{n}$. 
   Clearly, $W_{J}\cup dW_{K}d^{-1}\seq W_{L}$. To show in (a) that $W_{L}$ is the desired 
   parabolic closure of $W_{J}\cup dW_{K}w^{-1}$,  it will suffice to show that if 
   $u\in W$ and $M\seq S$ with $W_{J}\cup dW_{K}w^{-1}\seq uW_{M}u^{-1}$, then
     $W_{L}\seq uW_{M}u^{-1}$. For this,  replace $u$ by the minimal length element of its coset $uW_{M}$ to  assume without loss of generality that
   $u\in W^{M}$. Write $u=u_{0}u_{1}$ where $u_{0}\in W_{J^{\perp}\cap dK^{\perp}d^{-1}}$
   and  $u_{1}^{-1}\in  W^{J^{\perp}\cap dK^{\perp}d^{-1}}$.  Since conjugations by elements 
   of $W_{J^{\perp}\cap dK^{\perp}d^{-1}}$ fix $W_{J}\cup dW_{K}d^{-1}$ pointwise,
   $W_{J}\cup dW_{K}w^{-1}\seq u_{1}W_{M}u_{1}^{-1}$. Now $l(u)=l(u_{0})+l(u_{1})$, so  
    $u_{1}\in W^{M}$. This implies that $\Pi_{J}\seq u_{1}\Pi_{M}$, by \ref{x1.7}.
    Hence $u_{1}^{-1}\Pi_{J}\seq \Pi_{M}$. Since all irreducible components of $\Pi_{J}$ are 
    infinite, it follows by Remark  \ref{x1.15}  that $u_{1}\in W_{J^{\perp}}$ and 
    $\Pi_{J}\seq \Pi_{M}$ i.e. $J\seq M$.
    
     Now we also have $W_{K}\seq d^{-1}u_{1}W_{M}u_{1}^{-1}d$. Write $d^{-1}u_{1}=xw$ 
     where $x\in W^{M}$ and $w\in W_{M}$, so $W_{K}\seq xW_{M}x^{-1}$. A similar argument 
     to that just  above shows that $\Pi_{K}\seq x\Pi_{M}$, $x\in W_{K^{\perp}}$ and
$K\seq M$. Now we have $w=x^{-1}d^{-1}u_{1}\in W_{M}$ where $x^{-1}\in  W_{K^{\perp}}$,
$u_{1}^{-1}\in W_{J^{\perp}} ^{J^{\perp}\cap dK^{\perp}d^{-1}}$ and
 $d^{-1}\in {}^{K^{\perp}}W^{J^{\perp}}$ (actually, (a)--(b) only require this condition on $d$, not the stronger assumed condition 
  $d^{-1}\in {}^{K\cup K^{\perp}}W^{J\cup J^{\perp}}$, but the two formulations are easily equivalent).  By \cite[Proposition 2.7]{Ca} (for general Coxeter 
 groups), one has $l(w)=l(x^{-1})+l(d^{-1})+l(u_{1})$.  Since $w\in W_{M}$, one has 
 $x,d,u_{1}\in W_{M}$.
Hence $J\cup\set{d}\cup K\seq W_{M}$ and therefore $L\seq M$. In case
$W$ is of finite rank, we can shortcut the rest of the proof of (a) at this point by noting that the 
above already shows that $W_{L}$ is a parabolic subgroup of minimal rank containing  
$W_{J}\cup dW_{K}w^{-1}$ and hence it is the parabolic closure as desired.

We continue the proof of (a) in general. The above shows that $u_{1}\in W^{M}\cap W_{M}=
\set{1}$ so $u=u_{0}\in  W_{J^{\perp}\cap dK^{\perp}d^{-1}}$.
Let $u=r_{k}\cdots r_{1}$, $r_{i}\in J^{\perp}\cap dK^{\perp}d^{-1}$, be a reduced expression 
for $u$. Note  $u\in W^{M}\seq W^{L}$. We claim  that $r_{1},\ldots, r_{k}\in L^{\perp}$ so $u\in W_{L^{\perp}}$. This would imply $uW_{M}u^{-1}\sreq uW_{L}u^{-1}=W_{L}$, completing the proof of (a). To prove the claim, it suffices to show that
 if $r_{1},\ldots, r_{i-1}\in L^{\perp}$  where $1\leq i\leq k$, 
then $r_{i}\in L^{\perp}$. The hypothesis and $u\in W^{L}$  imply that $r_{k}\cdots r_{i}\in W^{L}$ and so $r_{i}\not\in L$. We have by definition $r_{i}\in J^{\perp}$ and  $r_{i}d=dr_{i}'$ where
$r_{i}'\in K^{\perp}$.  Note   $d\in W_{L}$ and $r_{i}\not \in W_{L}$,  so $\set{\al_{r_{i}'}}=d^{-1} \set{\al_{r_{i}}}$ and 
applying Theorem \ref{x1.15} shows that
$d$ has a reduced expression involving only simple reflections in $r_{i}^{\perp}\seq S\sm \set{r_{i}}$. In particular,  $r_{i}'=r_{i}$, which gives $K\seq r_{i}^{\perp}$. This shows that  $W_{r_{i}^{\perp}}\sreq J  \cup K\cup\set{d}$ so $L\seq r_{i}^{\perp}$ and $r_{i}\in L^{\perp}$.
This proves the claim and hence (a).

For the proofs of  (b)-(c), consider first two parabolic subgroups $aW_{J}a^{-1}$ and $bW_{K}b^{-1}$ where $a,b\in W$ and 
   $J,K\seq S$ be special. Write 
 $a^{-1}b=pdq$ where $p\in W_{J\cup J^{\perp}}$, $q\in W_{K\cup K^{\perp}}$ and $d$ is the shortest length double coset representative in $W_{J\cup J^{\perp}}a^{-1}bW_{K\cup K^{\perp}}$. Then  \begin{equation*}
 aW_{J}a^{-1}\cup bW_{J}b^{-1}=a(W_{J}\cup pdqW_{K}(pdq)^{-1})a^{-1}=ap(W_{J}\cup dW_{K}d^{-1})(ap)^{-1}.
 \end{equation*}  
 Letting $W_{L}$ be the parabolic closure of  $W_{J}\cup dW_{K}d^{-1}$, with $L$ as in (a),
 we see that   $aW_{J}a^{-1}\cup bW_{K}b^{-1}$ has a parabolic closure $ apW_{L}(ap)^{-1}$.    Also, if (b) holds then $ apW_{L}(ap)^{-1}$ is special.
 Hence we need only prove (b) (for then (c) follows by induction on the number of parabolic subgroups).
 
 Finally we prove (b) (i.e. that $W_{L}$ is special in (a)) by induction on $l(d)$.
 If $l(d)=0$, then $L=J\cup K$ is clearly special since any component of $J\cup K$ contains
 either a component of $J$ or one of $K$, and  hence is of infinite type because $J$ and $K$ are
special. Suppose $l(d)>0$ and write $d=sf$ where $s\in L$ and $l(f)<l(d)$.
Since $d^{-{1}}\in {W}^{J\cup J^{\perp}}$, we have  $s\not \in J\cup J^{\perp}$. Therefore, $H:=J\cup\set{s}$ is connected, necessarily of  infinite
type since it contains $J$. Since $s\in L$, the parabolic closure of  $W_{H}\cup dW_{K}d^{-1}$ clearly exists and   is clearly to $W_{L}$ also. Let $e$ be the minimal length double coset representative in $W_{H\cup H^{\perp}}dW_{K\cup K^{\perp}}$. From above, $W_{L}$ is $W$-conjugate to the parabolic closure $W_{L'}$  of
$W_{H}\cup eW_{K}e^{-1}$. Since $f=sd$ is in the above double coset, $l(e)\leq l(f)<l(d)$.
 By induction, $W_{L'}$ is special and hence  its conjugate $W_{L}$ is special too (in fact, $L'=L$ by \ref{x1.15}, though we don't need this fact). 
 \end{proof}

\section{The imaginary cone}
\label{x3}
\subsection{} For  \label{x3.1} any reflection subgroup $W'$ of $W$,
 define the $W'$-invariant cone 
  \[\mc{Y}_{W'}:=\bigcap_{{w\in W'}}w(\real_{\geq 0}\Pi_{W'})\] i.e. 
  $\mc{Y}_{W'}$ is the set of all
 $v\in V$ such that for each $w\in W'$, $w(v)$ is expressible as a non-negative linear combination of positive roots $\alpha\in \Pi_{W'}$.
 
 For any subset $X$  of $V$,   let $X^{*}:=\mset{v\in V\mid\mpair{v,X}\subseteq \real_{\geq 0}}$
 (see Appendix \ref{xA}).
 
  \begin{lem} Let $W'$ be any reflection subgroup of $W$. Then 
  \begin{num}\item  $\mc{Y}_{W'}\subseteq \mc{Y}_{W}$. 
 \item For any $w\in W$, $\mc{Y}_{wW'w^{-1}}=w\mc{Y}_{W'}$.
     \item If $W'$ is facial, then $\mc{Y}_{W'}= \mc{Y}_{W}\cap \real \Pi_{W'}$.
\item If $\Pi_{W'}$ is finite and $\mpair{-,-}$ is non-singular, then  $\real_{{\geq 0}}\Pi_{W'}=\mc{C}_{{W'}}^{*}$  and $\mc{Y}_{W'}=\mc{X}_{W'}^{*}$.
   \end{num}\end{lem}
   \begin{rem}  Even if $\mpair{-,-}$ is non-singular, $V$ is finite dimensional  and $\Pi$ is finite, the cone $\real_{\geq 0}\Pi_{W'}$ need not be closed if $\chi(W')$ is infinite.\end{rem}
\begin{proof} 
 Suppose that $v\in \mc{Y}_{W'}$.
 For (a), we  have to show that for $w\in W$, $w(v)\in \real_{\geq 0}\Phi_{+}$.
 Write $w=w'w''$ where $w''\in W'$ and $N(w^{\prime-1})\cap W'=\emptyset.$
 Since $v\in \mc{Y}_{W'}$, we have $w''(v)\in\real_{\geq 0}  \Pi_{W'}$.
 But for $\alpha\in \Pi_{W'}$, we have $w'({\alpha})\in \Phi_{+}\subseteq \real_{\geq 0}\Pi$ 
  so $w(v)=w'(w''(v))\in \real_{\geq 0}\Pi$ as required for (a).
    Part (b) follows readily from the definitions   on noting that $\Pi_{wW'w^{-1}}=x\Pi_{W'}$ where $x$ is the element of minimal length in the coset $wW'$.

  To prove (c), write  $W'=wW_{I}w^{-1}$ where $W_{I}$ is standard facial and $w$ is of minimal length in $wW_{I}$.
  By (a) and the definitions, \begin{equation*}
  \mc{Y}_{W_{I}}\seq \mc{Y}_{W}\cap \real_{\geq 0}\Pi_{I}
  \seq \mc{Y}_{W}\cap \real\Pi_{I}\seq \real_{\geq 0}\Pi\cap \real\Pi_{I}=\real_{\geq 0}\Pi_{I}
  \end{equation*} where the last equality uses  Lemma \ref{x2.4}(a).  Since $\mc{Y}_{W}$ and $ \real\Pi_{I}$ are $W_{I}$-invariant, so is   $\mc{Y}_{W}\cap \real\Pi_{I}$ and therefore  
  $\mc{Y}_{W}\cap \real\Pi_{I}\seq \cap_{w\in W_{I}}w(\real_{\geq 0}\Pi_{I})=\mc{Y}_{W_{I}}$. Hence $\mc{Y}_{W_{I}}=\mc{Y}\cap \real \Pi_{I}$. Finally by (b) and \ref{x1.6}--\ref{x1.7}, 
  \begin{equation*}
  \mc{Y}_{W'}=w\mc{Y}_{W_{I}}=w(\mc{Y}_{W}\cap \real\Pi_{I})=w(\mc{Y}_{W})\cap\real w(\Pi_{I})=\mc{Y}_{W}\cap\real\Pi_{W'},
  \end{equation*} completing the proof of (c).

  For (d), note that if  $\Pi_{W'}$ is finite 
 and  $\mpair{-,-}$ is non-singular on $V$, then one has  $(\real_{\geq 0}\Pi_{W'})^{*}=\mc{C}_{W'}$ by definition and so   $\real_{{\geq 0}}\Pi_{W'}=\mc{C}_{{W'}}^{*}$  by the duality theorem for polyhedral cones (see Lemma \ref{xA.10}). This implies that 
 \[ \mc{Y}_{W'}=\bigcap_{w\in W'} w(\real_{\geq 0}\Pi_{W'})=\bigcap_{{w\in W'}}w(\mc{C}_{W'}^{*})=
\Bigl( \ \bigcup_{{w\in W'}}w(\mc{C}_{W'})\Bigr)^{*}=\mc{X}_{W'}^{*}\] is the dual cone of $\mc{X}_{W'}$.
 \end{proof}

  \subsection{} \label{x3.2} We now define a $W$-stable cone  $\mc{Z}$ 
  called  the \emph{imaginary cone} of $W$ on $V$ (see \cite[\S5.3--5.4, 5.8]{Kac} and \ref{x10.1}) which is the basic object of study in this paper.
It  is  a $W$-stable subset of $\mc{Y}$
 defined in terms of  a natural fundamental domain
  $\mc{K}$.  It will eventually  be shown that, under mild finiteness conditions, $\mc{Z}$ contains the interior of $\mc{Y}$ and satisfies analogues of the properties of $\mc{Y}$ stated  in Lemma
  \ref{x3.1}.

   Define  $\mc{K}=\mc{K}_{W}:=\real_{\geq 0}\Pi_{W} \cap -\mc{C}_{W}$ and  $\mc{Z}=\mc{Z}_{W}:=\cup_{w\in W} w(\mc{K}_{W})$. 
   Note that any $W$-orbit on $\mc{Z}$ contains a unique point of $\mc{K}$, since any $W$-orbit on $\mc{X}$ contains a unique point of $\mc{C}$.   
   
    \begin{prop}\begin{num}\item  $\mc{Z}= \mc{Y}\cap -\mc{X}$.
\item $\mc{K}$ and  $\mc{Z}$ are  cones.
 \item If $z,z'\in \mc{Z}$, then $\mpair{z,z'}\leq 0$. 
 \item Let $W_{i}$, for $i\in I$, denote the irreducible components of $(W,S)$.
 Then $\mc{K}_{W}=\sum_{i}\mc{K}_{W_{i}}$ and $\mc{Z}_{W}=\sum_{i}\mc{Z}_{W_{i}}$.
 \item For any reflection subgroup $W'$ of $W$ and any $w\in W$, $\mc{Z}_{wW'w^{-1}}=w\mc{Z}_{W'}$. If also $w$ is the element of minimal length in $wW'$, then $\mc{K}_{wW'w^{-1}}=w\mc{K}_{W'}$. 
  \end{num}\end{prop}
 
 \begin{proof} 
    By \ref{x1.13},  if $k\in \mc{K}$ and $w\in W$, then \[wk=(wk-k)+k\in \real_{\geq 0}\Pi+\real_{\geq 0}\Pi\subseteq \real_{\geq 0}\Pi
.\] Thus, $\mc{Z}\subseteq \real_{\geq 0}\Pi$ and since $\mc{Z}$ is $W$-stable,
$\mc{Z}\subseteq \mc{Y}$.   Also, since $\mc{K}\subseteq -\mc{C}$, we have
$\mc{Z}=W\mc{K}\subseteq W(-\mc{C})=-\mc{X}$. This proves that
 $\mc{Z}\subseteq \mc{Y}\cap -\mc{X}$. On the other hand, let $z\in 
 \mc{Y}\cap -\mc{X}$. Since $z\in -\mc{X}$, there is some $w\in W$ with $w^{-1}z\in- \mc{C}$.
 Since $z\in \mc{Y}$, we have $w^{-1}z\in \mc{Y}\subseteq \real_{\geq 0}\Pi$ and so $w^{-1}z\in \real_{\geq 0}\Pi\cap -\mc{C}=\mc{K}$. Hence $z=w(w^{-1}z)\in W\mc{K}=\mc{Z}$.
 This shows  $\mc{Z}\supseteq \mc{Y}\cap -\mc{X}$ and proves  (a).
 
 Part (b) follows from the definitions and (a), using that $\real_{\geq 0}\Pi$, $\mc{C}$,
 $\mc{X}$ and $\mc{Y}$ are cones and that an intersection of cones is a cone.

To prove (c), choose $w\in W$ with $wz\in \mc{K}$.
Then $wz\in -\mc{C}$ and $wz'\in \mc{Z}\subseteq \real_{\geq 0}\Pi$, so
\[\mpair{z,z'}=\mpair{wz,wz'}\in \mpair{\real_{\geq 0}\Pi,-\mc{C}}\subseteq \real_{\leq 0}.\]

Now we prove (d). If $\alpha\in \Pi_{i}:=\Pi_{W_{i}}$, then $\mpair{\alpha,\beta}=0$ for all
$\beta\in \Pi\setminus \Pi_{i}$.
This implies that $\real_{\geq 0}\Pi_{i}\cap -\mc{C}_{i}=\real_{\geq 0}\Pi_{i}\cap -\mc{C}$
where $\mc{C}_{i}:=\mc{C}_{W_{i}}$. Hence $\mc{K}_{i}:=\mc{K}_{W_{i}}\subseteq \mc{K}$ and
$\sum_{i}\mc{K}_{i}\subseteq \mc{K}$ by (b). We also get $\mc{Z}_{i}:=W_{i}\mc{K}_{i}\subseteq
W\mc{K}=\mc{Z}$ and $\sum_{i}\mc{Z}_{i}\subseteq \mc{Z}$.

For the other direction, let $k\in \mc{K}\subseteq \real_{\geq 0}\Pi$. We may write $k=\sum_{i}k_{i}$
where $k_{i}\in \real_{\geq 0}\Pi_{i}$ and almost all $k_{i}=0$. For $\alpha\in \Pi_{i}$ we have
$\mpair{\alpha,k_{i}}= \mpair{\alpha,k}\leq 0$ and so $k_{i}\in \mc{K}_{i}$.
So $k\in \sum_{i}\mc{K}_{i}$. If $z\in \mc{Z}$, choose $w\in W$ with $w^{{-1}}z\in \mc{K}$, say
$w^{-1}(z)=\sum_{i}k_{i}$ with all $k_{i}\in \mc{K}_{i}$ and almost all $k_{i}=0$. 
Also, write $w=\prod_{i\in I}w_{i}$ with $w_{i}\in W_{i}$, almost all $w_{i}=1$.
We have $wk_{i}=w_{i}k_{i}\in \mc{Z}_{i}$ and \[z=w(w^{-1}z)=\sum_{i}wk_{i}=\sum_{i}w_{i}k_{i}\in \sum_{i}\mc{Z}_{i}\] as required for (d). Finally, the first part of (e) follows from (a), \ref{x1.10}(g) and \ref{x3.1}(b) and the second follows from Lemma \ref{x1.11}(a) and \ref{x1.6}--\ref{x1.7}.
\end{proof}

\subsection{}\label{x3.3}  For  $I\seq S$,   abbreviate 
$\mc{Y}_{I}:=\mc{Y}_{W_{I}}$
$\mc{K}_{I}:=\mc{K}_{W_{I}}$, $\mc{Z}_{I}=\mc{Z}_{W_{I}}$ etc.
\begin{lem} \begin{num}\item If $I\subseteq S$, then $\mc{K}_{I}\subseteq \mc{K}$ and
$\mc{Z}_{I}\subseteq \mc{Z}$.
\item $\mc{K}=\cup_{I}\mc{K}_{I}$ and $\mc{Z}=\cup_{I}\mc{Z}_{I}$ where in both unions, $I$ ranges over the finite subsets of $S$.\end{num} \end{lem}
\begin{proof}  First we prove (a). Let $k\in \mc{K}_{I}$. Then $k\in \real_{\geq 0}\Pi_{I}$ and
$\mpair{k,\Pi_{I}}\subseteq \real_{\leq 0}$. Since $\Pi_{I}\subseteq \Pi$ and
$\mpair{\Pi_{I},\Pi_{S\setminus I}}\subseteq \real_{\leq 0}$, we get 
$k\in \real_{\geq 0}\Pi$ and
$\mpair{k,\Pi}\subseteq \real_{\leq 0}$ i.e. $k\in \mc{K}$.
Hence $\mc{K}_{I}\subseteq \mc{K}$, and therefore $\mc{Z}_{I}=W_{I}
\mc{K}_{I}\subseteq W\mc{K}=\mc{Z}$ proving (a).

To prove (b), let $z\in \mc{Z}$. Choose $w\in W$ such that $wz=k\in \mc{K}$.
We may choose a finite subset $I$ of $S$ such that both $w\in W_{I}$ and $k\in \real_{\geq 0}\Pi_{I}$. Clearly, $\mpair{k,\Pi_{I}}\subseteq \mpair{k,\Pi}\subseteq \real_{\leq 0}$
so $k\in \mc{K}_{I}$ and $z=w^{-1}k\in W_{I}\mc{K}_{I}=\mc{Z}_{I}$. This shows that
$\mc{Z}\subseteq \cup_{I}\mc{Z}_{I}$ with union over finite $I\subseteq S$, and the reverse inclusion holds by (a). Now if $z\in \mc{K}$, we may take $w=1$ in the above argument to see
that $z=k\in \mc{K}_{I}$, hence $\mc{K}\subseteq \cup_{I}\mc{K}_{I}$ with union over finite $I\subseteq S$, and the reverse inclusion holds by (a) again.

\end{proof}

\subsection{} \label{x3.4}In the case of a facial subset $I$ of $S$, there is the following stronger connection
between imaginary cones of $W$ and $W_{I}$.
\begin{lem} Let $I$ be a facial subset of $S$. Then \begin{num}
 \item $\mc{K}_{I}=\mc{K}\cap \real\Pi_{I}=\mc{K}\cap \real_{\geq 0}\Pi_{I}$.
\item $\mc{Z}_{I}=\mc{Z}\cap \real\Pi_{I}=\mc{Z}\cap \real_{\geq 0}\Pi_{I}$.\end{num}
\end{lem}
\begin{proof} It is enough to prove the first equality in each of (a)--(b), since
$\mc{K}_{I},\mc{Z}_{I}\subseteq \real_{\geq 0}\Pi_{I}$. For (a), note that by  Lemma \ref{x2.4}, \begin{equation*}\begin{split} \mc{K}\cap \real\Pi_{I}&=\mset{v\in \real_{\geq 0}\Pi\mid \mpair{v,\alpha}\leq 0 \text{ \rm for all }\alpha\in \Pi}\cap \real \Pi_{I}\\
&=\mset{v\in \real_{\geq 0}\Pi_{I}\mid \mpair{v,\alpha}\leq 0 \text{ \rm for all }\alpha\in \Pi}\\
&=\mset{v\in \real_{\geq 0}\Pi_{I}\mid \mpair{v,\alpha}\leq 0 \text{ \rm for all }\alpha\in \Pi_{I}}=\mc{K}_{I}
\end{split}
\end{equation*}
since for $v\in \real_{\geq 0}\Pi_{I}$, $\alpha\in \Pi_{S\setminus I}$ we have $\mpair{v,\alpha}\leq 0$.

For (b), note first that using (a), we have 
\[\mc{Z}_{I}=\bigcup_{w\in W_{I}}w(\mc{K}_{I})\subseteq \real \Pi_{I}\cap  \bigcup_{w\in W}w(\mc{K})=\mc{Z}\cap \real \Pi_{I}.\]
Conversely, suppose that $z\in \mc{Z}\cap \real\Pi_{I}$. Write $z=wk$ where $w\in W$ and $k\in \mc{K}$. We have $z=(wk-k)+k$ where $k$ and $wk-k$ are in $\real_{\geq 0}\Pi$, using \ref{x1.13}. Since $z\in\real_{\geq 0}\Pi\cap \real
\Pi_{I}=\real_{\geq 0}\Pi_{I}$ and $\Pi_{I}$ is facial, it follows that $k, wk-k\in \real_{\geq 0}\Pi_{I}$. By \ref{x2.4}, $wk=w'k$ for some $w'\in W_{I}$. Hence $k\in \mc{K}\cap \real_{\geq 0}\Pi_{I}=\mc{K}_{I}$
 and $z=wk=w'k\in W_{I}\mc{K}_{I}=\mc{Z}_{I}$ as required for (b).
\end{proof}

\subsection{}\label{x3.5}  The next fact is an immediate corollary of \ref{x3.4} and \ref{x3.2}(a), using the $W$ action and \ref{x1.6}--\ref{x1.7}.
\begin{lem} Let $W'$ be a facial subgroup of $W$. Then \begin{num}
 \item $\mc{K}_{W'}=\mc{K}\cap \real\Pi_{W'}=\mc{K}\cap \real_{\geq 0}\Pi_{W'}$.
\item $\mc{Z}_{W'}=\mc{Z}\cap \real\Pi_{W'}=\mc{Z}\cap \real_{\geq 0}\Pi_{W'}$.\end{num}
\end{lem}
\begin{rem}  (1)  It is easy to see that, as $W$-subsets of $\real\Pi$, the cones  $\mc{Y}$, $\mc{K}$ and $\mc{Z}$ are unchanged by extension or restriction of quadratic space.
 
(2) If $\Pi$ is linearly independent, the lemma remains true with ``facial'' replaced by ``parabolic,'' by taking an ample extension  and using (1) and  Remark \ref{x2.3}. Similar remarks apply to many
other facts about the imaginary cone. \end{rem}
 \subsection{}  \label{x3.6} We conclude this section with the following observations about  the  relation of the general case to  the  case of linearly independent simple roots. Consider $V',\Pi',\Phi', W',S',\mc{Q}'$ and the linear map $L\colon V'\to V$ associated to
 $V,\Pi,\Phi,W,S,\mc{Q}$ as in \ref{x1.4}. Attach to $W'$ acting on 
 $V'$  the fundamental chamber $\mc{C}'$, Tits cone $\mc{X}'$, 
 imaginary cone $\mc{Z}'$ and its fundamental domain $\mc{K}'$  just as for $W$ acting on $V$.
 
 \begin{prop}
 \begin{num}\item    $\real_{\geq 0}\Pi'\cap L^{-1}(\set {0})=\set{0}$.
 \item  $\mc{Q}'=L^{-1}(\mc{Q})$, $\mc{C}'=L^{-1}(\mc{C})$, $\mc{X}'=L^{-1}(\mc{X})$, $\mc{K}'=\real_{\geq 0}\Pi'\cap L^{-1}(\mc{K})$  and $\mc{Z}'=\real_{\geq 0}\Pi'\cap L^{-1}(\mc{Z})$.
 \item $L$ restricts to surjective maps   $\real_{\geq 0}\Pi'\to \real_{\geq 0}\Pi$, $\mc{K}'\to \mc{K}$ and  $\mc{Z}'\to \mc{Z}$.
 \item $L$ restricts to maps $\mc{Q}'\to \mc{Q}$, $\mc{C}'\to \mc{C}$ and 
 $\mc{X}'\to \mc{X}$; these  maps are surjective if $L$ is surjective. 
  \end{num}
 \end{prop}
 \begin{rem}   None of the maps in (c)--(d)  are bijective in general, since their domains and codomains are cones of possibly unequal dimension. There is also a map $\mc{Y}'\to \mc{Y}$ induced by $L$, but it is not obvious whether it satisfies analogues of the above properties.
  \end{rem}
 \begin{proof}  Part (a) is equivalent to the  assumed positive independence of   $\Pi$. 
 Part 
 (b) follows from the definitions using the properties of $L$ stated in \ref{x1.4} and  (c)--(d) then follow using (b) and the fact $L$ induces a bijection $\Pi'\to \Pi$.
 \end{proof}

\section{Finiteness and non-degeneracy conditions}
\label{x4}
\subsection{}  Throughout \label{x4.1} Sections \ref{x4}--\ref{x11}, the    following conditions (i)--(iii) on 
$(V,\mpair{-,-})$ and $(\Phi,\Pi)$ are assumed  except where   explicitly  indicated:\begin{ass} \begin{conds}\item $V$ is finite dimensional 
\item $\mpair{-,-}$ is non-singular
\item $\Pi$ is finite.\end{conds}\end{ass}

We give $V$ its standard topology as finite dimensional real vector space, and we always consider subsets of $V$ in the induced topology unless otherwise indicated.  
We shall say that  a  subset $\Delta$  of $\Phi$,
(usually of the form $\Delta=\Pi_{W'}$ for some reflection subgroup $W'$) is non-degenerate if the restriction of $\mpair{-,-}$ to $\real \Delta$ is a non-singular form.

\begin{rem} The remainder of this paper makes more extensive use  of standard properties of polyhedral cones (the most  frequently used properties  are listed in \ref{xA.10}).
Some   of the results and arguments  below can be adapted   to hold under more general hypotheses. 
 \end{rem}

\subsection{}\label{x4.2}   Simple   properties of   $\real_{\geq 0 }\Pi$ and  $\mc{C}$   are recorded below.
\begin{lem} \begin{num}
 \item $\real_{\geq 0}\Pi$ and $  \mc{C}=(\real_{\geq 0}\Pi)^{*}$ are dual polyhedral cones.
 \item The  extreme rays of $\real_{\geq 0}\Pi$ are the sets
$\real_{\geq 0} \alpha$ for $\alpha\in \Pi$.
\item $\real_{\geq 0}\Pi$ is salient and   $\mc{C}$ has an interior point $\rho$. 
\item  We have $\mpair{\rho,\alpha}>0$ for all $\alpha\in \Pi$. The intersection of $\real_{\geq 0}\Pi$ with the affine hyperplane $\mset{v\in V\mid \mpair{v,\rho}=1}$ is a convex polytope $P$ with vertices
$\mpair{\rho,\alpha}^{-1}\alpha$ for $\alpha\in \Pi$. Further, $P$  is a compact, convex base of $\real_{\geq 0}\Pi$.
\item  A subset  $\Pi'$ of $\Pi$ is facial if and only if  it is a set of representatives of the extreme rays of some face of  $\real_{\geq 0}\Pi$. 
\item If $I$ is a facial subset of $S$ and $J\seq I$, then $J$  is a facial subset of $I$ (for the group $W_{I}$ with root system $\Phi_{I}$ and simple roots $\Pi_{I})$) if and only if  $J$ is a facial subset of $S$.
\item If $\Pi$ is linearly independent, then any subset of $\Pi$ is facial.
\item If $\Pi=\Pi_{1}\dot\cup \Pi_{2}$ where $\real \Pi=\real\Pi_{1}\oplus \real \Pi_{2}$ (direct sum) then a subset $\Delta$ of $\Pi$ is facial if and only if $\Delta\cap \Pi_{i}$ is facial for $i=1,2$. This applies in particular if  $\Pi_{1}$ is the union of  some of the non-degenerate components of $\Pi$.
\end{num}\end{lem}
\begin{proof}  Part (a) follows using  the definitions from \ref{x1.3}(ii) and finiteness of $\Pi$. Part (b) holds since for $\alpha\in \Pi$, $\mpair{\alpha,\alpha}>0$ while $\mpair{\alpha,\beta}\leq 0$ for $\beta\in \Pi\setminus \set{\alpha}$ 
(i.e. the hyperplane $\alpha^{\perp}$\ weakly separates $\alpha$ from $\Pi\setminus \set{\alpha}$).  The first part of (c) follows from 
\ref{x1.3}(ii) and the second is well known to follow from the first 
(see \ref{xA.11}).   Part (d) follows from (b)--(c) (see \ref{xA.11} more generally). Part  (e) follows from (a), the definition of facial subset of $\Pi$ and the fact (\ref{xA.10}(d)), often taken as their definition, that faces of polyhedral cones are exposed faces.  Part (f) also  follows from  Lemma \ref{xA.10}(d). Finally under the assumptions of (g), $\real_{\geq 0}\Pi$ is a simplicial cone
and so, using (e), any subset of $\Pi$ is facial as required. The straightforward  proof of (h) is omitted.
 \end{proof}

\subsection{} \label{x4.3} 

The following lemma  collects from the literature various  consequences of the above assumptions \ref{x4.1}(i)--(iii).

\begin{lem} 
\begin{num} 
\item $W$ is discrete and closed in  $\text{\rm GL}(V)$.
\item $\Phi$ is  discrete and closed in  $V$.
\item The interior $\inter(\mc{X})$ of $\mc{X}$ consists of the set of all points $x$ of $\mc{X}$ which have finite stabilizer in $W$ i.e. such that $x^{\perp}\cap \Phi$ is finite. 
\item For any $x,y\in \inter(\mc{X})$ there are only finitely many
$\alpha\in \Phi_{+}$ such that $\alpha^{\perp}$ contains a point of   the closed interval with endpoints $x$ and $y$.\item For any $N\in \real_{>0}$ and $\rho\in \inter(\mc{X})$,  there are only finitely many $\alpha\in \Phi_{+}$
with $\mpair{\alpha,\rho}\leq N$.
\item There are only finitely many $W$-orbits of pairs $(\alpha,\beta)$ of roots such that 
$\vert \mpair{\alpha,\beta}\vert < 1$.  In fact, each such pair is in the $W$-orbit of such a pair lying  in the root system of  finite standard parabolic subgroup of rank two.
\item For any $N\in \real_{>1}$, there  are  only finitely many $W$-orbits of pairs $(\alpha,\beta)$ of roots with
$1<\vert \mpair{\alpha,\beta}\vert \leq N$.\end{num}
\end{lem}
\begin{rem} If $\Phi$  contains a root system of  affine type   as a parabolic subsystem,  then there are infinitely many $W$-orbits of pairs $(\alpha,\beta)$ of roots with
$\mpair{\alpha,\beta}=1$.\end{rem}
 \begin{proof} We use  \cite{K} as a convenient  source for proofs of  most of these facts, though (excepting (e) and (g))  
 most parts  can be found in or easily deduced from   \cite{V} or \cite{Bour}.
 For (a), note that a discrete subgroup of a Hausdorff topological group is closed and  see   \cite[Ch V, \S 4, Cor 3]{Bour}.  For (b), argue  using (a) as in the proof of \cite[Lemma 1.2.5]{K}. For (c),
see \cite[Corollary 2.2.5]{K}. Part (d) follows from (c) and Proposition \ref{x2.3}(g). For (e),  see \cite[Lemma 5.7.1]{K}.  Part (f) follows using \cite[Ch V, \S4, Ex 2(d)]{Bour} (or \cite[Proposition 3.1(a)]{K}) and Lemma \ref{x2.9}
since $\mpair{s_{\alpha},s_{\beta}}$ is finite when $-1<\mpair{\alpha,\beta}<1$. For  (g),  see \cite[Proposition 6.6.2 ]{K}.
\end{proof}

\subsection{} \label{x4.4}  It will be convenient to fix an element $\rho$ in the interior of $\mc{C}$ throughout the remainder of Sections \ref{x4}--\ref{x11}.
 Observe that the conditions (i)--(iii) on $(V,\mpair{-,-})$ and $(\Phi,\Pi)$ also are satisfied by
$(V,\mpair{-,-})$ and $(\Phi_{W'},\Pi_{W'})$ for any finitely generated reflection subgroup
$W'$ of $W$. Hence all consequences of these conditions for $W$, including those in \ref{x4.3} and those given in the next lemma, apply  to such  reflection subgroups $W'$ as well as to $W$.
\begin{lem}  \begin{num}
\item If $v\in \mc{C}_{W}$,  then  $\mset{v-wv\mid w\in W}$ is  closed and discrete in $\real_{\geq 0}\Pi$, or equivalently,    the orbit $Wv$ of $v$ is closed and discrete in $V$.
\item $\overline{\mc{X}_{W}}=\mset{v\in V\mid v+t\rho\in \mc{X}_{W}\text{ \rm for all $t\in \real_{>0}$}}$. 
\item  $W$  has at most $\text{\rm dim} (\real \Pi)$ irreducible components.
\item If $W_{1}$, \ldots, $W_{n}$ are reflection subgroups of $W$ (e.g. the components of $W$) which satisfy  $\Phi=\cup_{i}\Phi_{W_{i}}$,  then $\overline{\mc{X}_{W}}=\cap_{i=1}^{n}\overline{\mc{X}_{W_{i}}}$.
\end{num} 
\end{lem}
\begin{proof}  For (a), note that the closed polyhedral cone $\real_{\geq 0}\Pi$ is, by \ref{x4.2}(d), the union of compact sets
(in fact, pyramidal polytopes) 
 $\mset{v\in \real_{\geq 0}\Pi\mid \mpair{v,\rho}\leq N}$  for
$N\in \real_{> 0}$. Clearly, every point in $\real_{\geq 0} \Pi$ has a neighborhood (in $\real_{\geq 0}\Pi$) contained in one of these sets.
 To prove (a), it will therefore  suffice, by \ref{x1.13}, to show that  for fixed $v\in \mc{C}$ and for all $N\in \real_{> 0}$,   
 \begin{equation}\label{x4.4.1}
 \mset{v-wv\mid w\in W, \mpair{v-wv,\rho}\leq N} \text{ \rm  is finite.}
 \end{equation}
By \ref{x1.13},   for fixed $w\in W$, $ v-wv=\sum_{i=1}^{n} \mpair{v,\ck\alpha_{i}}\beta_{i}$ where  $\beta_{1},\ldots, \beta_{n}\in \Phi_{+}$ are distinct,
 $\alpha_{i}\in \Pi$ and   $\mpair{v,\alpha_{i}}\geq 0$. Let  $A:=\mset{\mpair{v,\ck\alpha}\mid \alpha\in \Pi}\cup\set{0}\subseteq \real_{\geq 0}$. If $A=\set{0}$, then $v=wv$ for all $w\in W$ and we are done. Otherwise,   let 
  $\epsilon :=\min(A\setminus\set{0})>0$. Assume that $\mpair{\rho,v-wv}\leq N$. Let $\Psi$ be the finite (by \ref{x4.3}(e)) set of positive roots $\beta$ with $\mpair{\rho,\beta}\leq N/\epsilon$.
  Then for any $i$ with $\mpair{v,\alpha_{i}}\neq 0$, we have $\mpair{\rho,\beta_{i}}\leq N/\epsilon$ i.e $\beta_{i}\in \Psi$. It follows that $v-wv=\sum_{\beta\in \Psi}c_{\beta}\beta$
  for some $c_{\beta}\in A$. Since $\Psi$ and $A$ are finite, independent of $w$, there are only finitely many sums of this form, and (a) is proved.

Since $\mc{C}\subseteq \mc{X}$, we have that  $\rho$ is an interior point of $\mc{X}$ 
and (b) follows (see Lemma \ref{xA.3} more generally).

Let  $W_{1}, \ldots, W_{n}$ be distinct irreducible components of  $W$.
Choose $\alpha_{i}\in \Pi_{W_{i}}$. Then $\mpair{\alpha_{i},\alpha_{j}}=\delta_{i,j}$
so $\alpha_{1},\ldots, \alpha_{n}$ are distinct and linearly independent in $\real \Pi$. Hence $n$ is at most $\text{\rm dim} (\real \Pi)$, proving (c).

Part (d) follows from (b) and Lemma \ref{x1.10}(e)--(f), noting
that $\rho$ is also an interior point of $\mc{C}_{W_{i}}\supseteq \mc{C}$ for all $i$.
\end{proof}

 \subsection{} \label{x4.5} Proofs of some results in the following sections proceed by reducing to the case of irreducible $W$ and considering  cases depending on the type (finite, affine or indefinite) of  $(W,S)$.
 The necessary background  recalled below is from \cite{V}, though we use the terminology of \cite{Kac}
and  express the results in terms of the cones $\real_{\geq 0}\Pi$, $\mc{K}$ and $\mc{X}$.
 
  Assume that $W$ is finitely generated and irreducible. Let $A$ denote the $\Pi\times \Pi$ real  matrix  with entries $A_{\alpha,\beta}:=\mpair{\alpha,\beta}$ for $\alpha,\beta\in \Pi$.
Then the matrix $A$ is finite, indecomposable, symmetric and has non-positive off-diagonal entries.  In particular, it satisfies the condition \cite[(m1)--(m3)]{Kac}.
According to the classification there, $A$ is of \emph{finite}, \emph{affine} or  \emph{indefinite type}, and these types are mutually exclusive. In fact, $A$ is of finite type if and only if  it is positive definite, and is of affine type 
if and only if it is positive semi-definite of corank $1$. Otherwise, it is of  indefinite type.
Define the \emph{type} of $\Phi$ (or $\Pi$) to be the same as  that of $A$ (i.e. finite, affine or irreducible).

The matrix $A$ is of finite type if there is $v\in \real_{\geq 0}\Pi$ of the form $v=\sum_{\alpha\in \Pi}c_{\alpha}\alpha$ with all $c_{\alpha}>0$ such that $\mpair{v,\Pi}\subseteq \real_{\geq 0}$.
Then  $(W,S)$ is an irreducible finite Coxeter system, 
 $\Pi$ is linearly independent and if $v=\sum_{\alpha\in \Pi}c_{\alpha}\alpha$ with all $c_{\alpha}\in \real$ satisfies $\mpair{v,\Pi}\subseteq \real_{\geq 0}$, then either  $v=0$ 
or all $c_{\alpha}>0$.
In particular, there is  no non-zero $v\in \real_{\geq 0}\Pi$  such that $\mpair{v,\Pi}\subseteq \real_{\leq 0}$.
 It follows from Lemma  \ref{x1.10}(h) that $\mc{X}=V$, and from  the properties  above that $\mc{K}=0$, hence $\mc{Z}=0$ also.

The matrix $A$ is of affine type if there is $\delta\in \real_{\geq 0}\Pi$ of the form $\delta=\sum_{\alpha\in \Pi}c_{\alpha}\alpha$ with all $c_{\alpha}> 0$ such that $\mpair{\delta,\Pi}=\set{0}$. The group  $W$ is an irreducible affine Weyl group. Moreover, up to rescaling
the roots (multiplying them each by a positive scalar depending only on the $W$-orbit of the root), the root system $\Phi$ is the usual  affine root system (of real roots) attached to the (crystallographic) root system of the corresponding finite Weyl group as in \cite{Kac}. 
The set $\Pi$ is linearly independent and the form $\mpair{-,-}$ restricted to $\real \Pi$ is positive semi-definite with radical spanned by $\delta$.  In particular,
$\delta$ is uniquely determined  up to multiplication by a positive scalar.
 If $v\in \real\Pi$ with $\mpair{v,\Pi}\subseteq \real_{\geq 0}$ then $v\in \real\delta$ and so $\mpair{v,\Pi}=0$.
 From the description of the untwisted affine root systems in \cite{Kac}, one has  
\begin{equation*}\mc{X}=\mset{v\in V\mid \mpair{v,\delta}>0}\cup\mset{v\in V\mid \mpair{v,\Pi}=0}.\end{equation*}  Note here that $\mpair{v,\Pi}=0$ implies $\mpair{v,\delta}=0$ so $\mc{X}\subseteq \mset{v\in V\mid \mpair{v,\delta}\geq 0}$. Moreover, $\mset{v\in \real \Pi\mid \mpair{v,\Pi}=0}=\real \delta$. These properties of affine type matrices imply that  $\mc{K}=-\mc{C}\cap \real_{\geq 0}\Pi=\real_{\geq 0}\delta$ and hence that  $\mc{Z}=W\mc{K}=\real_{\geq 0}\delta$ also.
Call $\real_{\geq 0}\delta$ the  \emph{isotropic ray} of $\Pi$ (or  of $\Phi$ or of $W$).

Now we describe the situation in case $A$ is of indefinite  type. There is  then some $\beta\in V$ such that 
$\beta=\sum_{\alpha\in \Pi}c_{\alpha}\alpha$ with all $c_{\alpha}>0$, and all $\mpair{\beta,\alpha}<0$. Any such element $\beta$ is in the relative interior $\mc{K}^{0}$ of $\mc{K}$, and so
$\real\mc{K}^{0}=\real\Pi$. 
If $v=\sum_{\alpha}d_{\alpha }\alpha \in \real_{\geq 0}\Pi$ where all $d_{\alpha}\geq 0$ and
$\mpair{v,\Pi}\subseteq \real_{\geq 0}$ then all $d_{\alpha}=0$. 
We do not have any more explicit description of $\mc{X}$  or $\mc{Z}$ 
than that given by the general results and definitions already given, 
though we shall give several other descriptions of $\overline{\mc{Z}}$
 and one of  $\mc{K}$ in the next section.

 Define the \emph{type} of 
$\Pi$, $\Phi$, or $W$   to be that of the matrix $A$ above
(so the type is either finite, affine or indefinite).
We shall say that $A$ (or $\Pi$, or $\Phi$, or $W$) is of \emph{infinite type} if it is of affine type or indefinite type.

\begin{rem}(1) If   $(W,S)$ is infinite dihedral, then $A$ above may be of either affine or indefinite type, depending on the root system $\Phi$. In any other (finite rank irreducible) case, the type of $A$  coincides with  the type (finite, affine or indefinite) of $(W,S)$ in the usual sense.

(2) Assume  that  $(W,S)$ is infinite irreducible  of  finite rank at least three. Then
$(W,S)$ is of affine type if and only if it has  a free abelian subgroup 
of finite  index in $W$, or equivalently, if  ``$W$ is of polynomial 
growth'' (that is, there exist $C\in \real_{>0}$, $k\in \Nat$
such that $\vert\mset{w\in W\mid l(w)\leq n}\vert \leq C(n^{k}+1)$ for all
 $n\in \Nat$). Further, $(W,S)$ is of  indefinite type  if and 
 only if it has  a non-abelian free group as  subgroup, or equivalently, if 
  ``$W$ is of exponential growth'' (that is, there exist 
$\lambda \in \real_{>1}$ such that $\vert\mset{w\in W\mid l(w)\leq n}\vert \geq  \lambda^{n}$ for 
all $n\in \Nat$). See \cite{dlH} and  \cite{MV}.

 (3) If $A$ is of infinite  type, then $\mc{X}\cap -\mc{X}=\mc{C}\cap -\mc{C}=\mset{v\in V\mid \mpair{v,\Pi}=0}$
by the argument of the proof of \cite[Proposition 3.2]{HRT} (see also \cite{K}, \cite{V}).

  \end{rem}

\section{The closed imaginary cone}\label{x5}
In this section, we give several characterizations of the closed imaginary cone. 
The  analogous results for root systems and Weyl groups of Kac-Moody Lie algebras
were  proved  by Kac (see \cite[Ch 5, especially Section 5.14]{Kac}).

\subsection{}  \label{x5.1} We first show that (under the standing  assumptions
\ref{x4.1}(i)--(iii)), the closure of $\mc{Z}$ is the dual of the Tits cone
(cf. \cite[\S 5.8]{Kac}). \begin{thm} 
\begin{num}
\item The closures  $\overline{\mc{X}}$ and $\overline{\mc{Z}}$ of $\mc{X}$ and $\mc{Z}$ are dual cones. 
\item  $\overline{\mc{Z}}=\mc{Y}$.
\item If $W'$ is a finitely generated facial subgroup of $W$, then  $\ol{\mc{Z}_{W'}}=\ol{\mc{Z}}\cap \real \Pi_{W'}$. \item $\ol{\mc{Z}}=\sum_{i=1}^{n}
\ol{\mc{Z}_{W_{i}}}$ if  $W_{1},\ldots, W_{n}$ are finite rank reflection subgroups (e.g. the components of $W$) such that $\Phi=\cup_{i=1}^{n}\Phi_{W_{i}}$.
\end{num} 
\end{thm}
\begin{proof}  First we prove (a). We have $\mc{Z}\subseteq \real_{\geq 0}\Pi$, so taking duals using
Lemma \ref{x3.1}(d) shows that $\mc{C}\subseteq \mc{Z}^{*}$. Since $\mc{Z}$ is 
$W$-invariant, so is
 $\mc{Z}^{*}$ and hence $\mc{X}=\cup_{w\in W}w\mc{C}\subseteq \mc{Z}^{*}$.
 Since $\mc{Z}^{*}=\overline{\mc{Z}}^{*}$ is closed, we get that $\overline{\mc{X}}\subseteq 
 \overline{\mc{Z}}^{*}$.

We shall  prove the reverse inclusion first  under the additional assumption that $(W,S)$ is irreducible.
If $\Phi$ is of finite type, then
 $\overline{\mc{X}}=V$ and $\overline{\mc{Z}}=0$ so the result holds.
If $\Phi$ is of affine type, then   $\overline{\mc{X}}=\mset{v\in V\mid \mpair{v,\delta}\geq 0}$ and
 $\overline{\mc{Z}}=\real_{\geq 0}\delta$, so the result holds in this case also.
 
 Now assume that $\Phi$ is of indefinite type.   Choose $\beta\in \mc{K}$ and $\epsilon >0$  such that $\beta=\sum_{\alpha\in \Pi}d_{\alpha}\alpha$ with  $d_{\alpha}>0$ and  $\mpair{\beta,\ck\alpha}<-\epsilon $
 for all $\alpha\in \Pi$. 
 Consider $z\in \mc{Z}^{*}$ and $\gamma\in \Phi_{+}$. Write (by \cite[Lemma 3.3]{DySd}) $\ck\gamma =\sum_{\alpha\in \Pi}c_{\alpha}\ck\alpha$ where all $c_{\alpha}\geq 0$ and some $c_{\alpha}\geq 1$.
 We have $s_{\gamma}(\beta)=\beta+s\gamma$ where \[s=-\mpair{\beta,\ck\gamma}=
 \sum_{\alpha}c_{\alpha}(-\mpair{\beta,\ck\alpha})>\epsilon.\] 
 But $\beta$ and $s_{\gamma}(\beta)$ are in $\mc{Z}$, so by definition of $\mc{Z}^{*}$, we have 
 $\mpair{z,\beta}\geq 0$ and $\mpair{z,s_{\gamma}(\beta)}\geq 0$.
 Here,  $\mpair{z,s_{\gamma}\beta}=\mpair{z,\beta}+s\mpair{z,\gamma}$ so  $\mpair{z,\gamma}\geq -\frac{1}{s}\mpair{z,\beta}\geq- \epsilon ^{-1}\mpair{z,\beta}$. Hence  for $t\in \real_{>0}$,
 \[\mpair{ z+t \rho, \gamma}\geq -\frac{1}{\epsilon } \mpair{z,\beta}+t\mpair{\rho,\gamma}.\]
This implies that if  $\mpair{ z+t \rho, \gamma}<0$, then 
$\mpair{\rho,\gamma}\leq \frac{1}{\epsilon t}\mpair{z,\gamma}$. For fixed $t>0$, there are only finitely many $\gamma\in \Phi_{+}$ satisfying this latter condition by Lemma \ref{x4.3}(e), and it follows by Lemma \ref{x1.10}(a)
that $z+t\rho\in {\mc{X}}$. Since $t\in \real_{> 0}$ was arbitrary, we get $z\in \overline{\mc{X}}$ by Lemma \ref{x4.4}(b). This completes the proof that $\overline{\mc{Z}}^{*}=\overline{\mc{X}}$ if $(W,S)$ is irreducible.

To prove $\overline{\mc{Z}}^{*}=\overline{\mc{X}}$ in general, let $ W_{1},\ldots, W_{n}$ be the  irreducible components of $W$.
Then from above and  the Lemmas \ref{x3.2}(d) and \ref{x4.4}(d), we have \[\overline{\mc{Z}}^{*}=\mc{Z}^{*}=\Bigl(\sum_{i}\mc{Z}_{W_{i}}\Bigr)^{*}=
\bigcap_{i}\mc{Z}_{W_{i}}^{*}=\bigcap_{i}\overline{\mc{Z}_{W_{i}}}^{*}=\bigcap_{i}\overline{\mc{X}_{W_{i}}}=\overline{\mc{X}}.\]

Part (b), (c)   are restatements using (a)  of   Lemma \ref{x3.1}(d),(c)  while (d) follows by taking dual cones in Lemma \ref{x4.3}(a) using (a).   Several other results from \S\ref{x3}--\ref{x4} admit similar restatements using (a)  which we shall not list explicitly.
 \end{proof}
 \subsection{} \label{x5.2} Subsections  \ref{x5.3}--\ref{x5.6} give another description of the closure of  $\mc{Z}$ (cf. \cite[Lemma 5.8 and Exercise 5.12]{Kac}), 
 using 
 a topology on the set of rays in $V$ which is defined in this subsection.   
 
 Let $\wt{ {R}}:=\mset{\real_{\geq 0}\alpha\mid \alpha\in V\sm\set{0}}$ denote the set of rays of $V$ (see \ref{x1.1}). For any $E\seq \wt{R}$, let $\cup E:=\cup_{e\in E}\,e\seq V$ denote the union of the rays $e$ in the set $E$. 
   If $V=\set{0}$, then $\wt{ {R}}=\eset$, which we give its only possible topology. 
 To avoid trivialities,  assume henceforward  that $V\neq \set{0}$.
   Choose
 a compact convex body $B$ (e.g. a closed ball) with $0$ in its interior, and let $B'$ denote the boundary of $B$.
 The map $\wt{ {R}}\rightarrow B'$ taking each ray to the unique point  of $B'$ which the ray contains
 is a bijection, and we topologize $\wt{ {R}}$ by declaring this map to be a homeomorphism.
 This gives $\wt{ {R}}$ a topology, independent of the choice of $B$, in which it is homeomorphic to the standard $(\text{\rm dim}(V)-1)$-sphere. Note that $W$ acts naturally on $\wt{R}$ as a group of homeomorphisms; this action is faithful since if an element of $w$ fixes all rays, it cannot make any positive root negative and so must be the identity.  One may find in \ref{xA.11}--\ref{xA.12} some useful well known facts concerning the subset of $ {R}$ which has as its  elements the   rays contained in any fixed closed salient cone. 
 
 For any pointed, possibly non-convex 
 cone  $X\seq V$ let $\ray(X):=\mset{\real_{\geq 0}\al
 \mid \al\in X\sm\set{0}}\seq \wt{ {R}}$ denote the set of rays 
 of $V$ through non-zero points of  $X$, topologized as a subspace of 
 $\wt{R}$. We call it the space of rays of $X$.  In particular,  $\wt{R}=\ray(V)$. For $R\seq \wt{ {R}}$, let
  $\cup R\seq V$ be the possibly non-convex pointed cone arising 
  as the union of all rays in $R$. The maps $X\mapsto  R:=\ray(X)$ 
  and $R\mapsto X:=\cup R$ define inverse bijections between the 
  set of pointed, possibly non-convex cones in $V$ and the power 
  set of $\ray(V)$.

\subsection{}   \label{x5.3}
  Let $Q:=\ray(\mc{Q})\seq \ray(V) $ denote the set  all isotropic rays of $V$,
$ {R}_{{+}}:=\mset{\real_{\geq 0}\alpha
\mid \alpha\in \Phi_{+}}\subseteq \ray(V) $ and $ {R}_{0}:=\overline{ {R}_{+}}\setminus { {R}_{+}}$. Results closely related to the following appear in \cite{HLR}.
 \begin{prop}  \begin{num}\item  $ {R}_{+}$ consists of positive rays and is discrete in $\ray(V)$.
\item  $ {R}_{0}\seq Q$ and $R_{0}$ is closed in $\ray(V)$.
 \item $ {R}_{0}$ is the set of limit rays  (i.e. limit points) of $ {R}_{+}$ i.e. $R_{0}=\Acc(R_{+})$.
 \end{num}   \end{prop}

\begin{proof}  Consider   the  set $ R_{\Pi}:=\ray(\real_{\geq 0}\Pi)\subseteq \ray(V) $ of  rays  of $V$ contained in  $\real_{\geq 0}\Pi$. 
Since   $\mpair{\rho,\alpha}>0$ for all $\alpha\in \Pi$,
  we may define a map \[\tau\colon \real_{\geq 0}\Pi\setminus\set{0}\rightarrow H:=\mset{v\in V\mid \mpair{v,\rho}=1}\] by $v\mapsto \mpair{v,\rho}^{-1}v$.
The image of $\tau$  is  the set 
 $P:=\mset{v\in \real_{\geq 0}\Pi\mid \mpair{\rho,v}=1}$, which  by Lemma \ref{x4.2}(d), is a convex polytope 
 in the affine  hyperplane $H$,  with the points $\mpair{\rho,\alpha}^{{-1}}\alpha$ for $\alpha\in \Pi$ as vertices. 
 We may and do choose $B$ above so that $P\subseteq B'$. It follows that 
 the map taking a ray in $ {R}_{\Pi}$ to its intersection with $P$ is a homeomorphism
 $\theta\colon  {R}_{\Pi}\xrightarrow{\cong} P$,
 explicitly  given by $\real_{\geq 0}\alpha\mapsto \tau(\alpha)$ for  non-zero $\alpha\in \real_{\geq 0}\Pi$. (This is essentially the realization  of rays in $\real_{\geq 0}\Pi$ used in \cite{HLR},  the possible  hyperplanes $V_{1}$ ``transverse'' to $\Pi$ used there are exactly the hyperplanes $H$ as above for varying $\rho$). Clearly, $ {R}_{\Pi}$ is closed (in fact, compact) and  $ {R}_{+}\subseteq 
  {R}_{\Pi}$, so $ {R}_{0}\subseteq\overline{ {R}}_{+}\subseteq  {R}_{\Pi}$ also.
 
 Now $\theta(\overline{ {R}_{+}})=\overline{\tau(\Phi_{+}})$. The right hand side consists of all limit points  of sequences $(\tau(\alpha_{n}))$ (in $P$) for sequences $(\alpha_{n}) _{n\in \Nat}$ of positive roots. Consider a limit point $\alpha\in P$ of such a sequence. We may assume without loss of generality that the sequence $(\tau(\alpha_{n}))$ actually converges  to $\alpha$.
 We consider  two possibilities. The first case is that  the sequence $(p_{n}):=(\mpair{\rho,\alpha_{n}})$ is bounded.
  Then by Lemma \ref{x4.3}(e),  there are only finitely many possibilities for each $\alpha_{n}$. In this case,   the sequences
  $(\alpha_{n})$ and  $(\tau(\alpha_{n}))$
  must be eventually constant, and  $\alpha=\tau(\alpha_{n})$ for all sufficiently large $n$.
 This corresponds to an (obviously positive) limit ray $\theta^{{-1}}(\alpha)$ in  $ {R}_{+}$.  In the contrary case, 
 the sequence  $p_{n}$ is unbounded, and passing to a subsequence we may assume it has limit $+\infty$. We have 
 \[\mpair{\alpha,\alpha}=\lim_{n\rightarrow \infty }\mpair{\tau(\alpha_{n}),\tau(\alpha_{n})}=
 \lim_{n\rightarrow \infty }p_{n}^{-2}\mpair{\alpha_{n},\alpha_{n}}= \lim_{n\rightarrow \infty }p_{n}^{-2}=0.\]  This case yields   an isotropic limit ray
 $\theta^{-1}(\alpha)\in  {R}_{0}$.
  Observe that any limit of a sequence of rays in $ {R}_{0}$ is obviously an  isotropic ray and is contained in $\overline{ {R}_{+}}$, so it must be in $ {R}_{0}$. Therefore  $ {R}_{0}$ is  closed.
 From the above,   $ {R}_{+}$ is discrete, as  any sequence in $ {R}_{+}$ which converges to an element  of $ {R}_{+}$ is eventually constant (since otherwise it converges to an isotropic ray).
  This completes the proof of (a)--(c).
 \end{proof}
 \subsection{} \label{x5.4} Now we show  the closed imaginary cone is the conical closure of the union of the limit rays of the set of  rays spanned by positive roots.
 \begin{thm}  The  imaginary cone  ${\mc{Z}}$ 
 satisfies
  $\overline{\mc{Z}}=\real_{\geq 0}\bigl(\bigcup {R}_{0})$.   \end{thm}
\begin{proof}  Maintain  notation  and assumptions from the proof of the preceding proposition.
 Since $ {R}_{0}$ is topologically closed, $\theta( {R}_{0})$ is a topologically closed, hence  compact subset of the polytope  $P$. The conical closure $F$  of  $\bigcup_{r\in  {R}_{0}}r\cup\set{0}$ is equal to that of  $\theta( {R}_{0})\cup\set{0}$. Let $K$ be the convex closure of $\theta( {R}_{0})$ in $P$;
 it is compact since $\theta( {R}_{0})$ is topologically  closed. Clearly,  $F=\mset{\lambda k\mid \lambda\in \real_{\geq 0},k\in K}$, which is easily seen to be closed in $V$.

 To prove the theorem,  we first reduce to the case that $W$ is irreducible. Let $W_{1},\ldots, W_{n}$ be the irreducible components of $W$. Denote the analogues of  $ {R}$, $ {R}_{+}$ and 
 $ {R}_{0}$ for $W_{i}$ as $ {R}_{W_{i}}$, $ {R}_{W_{i},+}$ and 
 $ {R}_{W_{i},0}$ respectively.  Since $ {R}_{+}=\cup_{i} {R}_{W_{i},+}$ (disjoint union) we have 
  $\overline{ {R}_{+}}=\cup_{i}\overline{ {R}_{W_{i},+}}$. Note $ {R}_{W_{i},+}\cap  {R}_{W_{j},0}=\eset$  (since a ray cannot be both positive and isotropic). Therefore 
  ${ {R}}_{0}=\cup_{i}{ {R}_{W_{i},0}}$. If the theorem is known for irreducible Coxeter groups, then the conical closure  of $ \cup_{r\in {R}_{W_{i},0}}r\cup\set{0}$ is $\overline{ \mc{Z}_{W_{i}}}$
  and hence the conical closure $F$ of $ \cup_{r\in {R}_{0}}r\cup\set{0}$  is $\sum_{i }\overline{\mc{Z}_{W_{i}}}$.
 Now  $F$ is topologically    closed from above,  so $F$ is the topological  closure of  $\sum_{i }{\mc{Z}_{W_{i}}}$,  which in turn is the topological closure of $\mc{Z}$ as required by Lemma  \ref{x3.2}(d).
  
  We now may and do assume that $W$ is irreducible.  If  $\Phi$ is of finite type,
  then $ {R}_{0}=\emptyset$ since $ {R}_{+}$ is finite, and $\overline{\mc{Z}}=0$ so the result holds in this case. If $\Phi$ is of affine type, then using  the standard description of the root system of affine Weyl groups, one easily sees  $ {R}_{0}=\set{\real_{\geq 0} \delta}$ where $\delta$ is as in \ref{x4.5}, and $F=\real_{\geq 0}\delta=\overline{\mc{Z}}$ as required.
  Henceforward we assume that $\Phi$ is of indefinite type. Since $F$ is a closed cone,  it will suffice by Theorem \ref{x5.3} to show that
  $F\subseteq \overline{\mc{Z}}$ and $F^{*}\subseteq \overline{\mc{X}}$.
  
We shall first show that $\theta( {R}_{0})\subseteq \overline{\mc{Z}}$, which implies  that
 $F\subseteq \overline{\mc{Z}}$.
 Fix $\beta\in \mc{K}$ and $\epsilon >0$  with $\mpair{\beta,\ck\alpha}<-\epsilon$ for all $\alpha\in \Pi$. 
  For any root $\gamma\in \Phi_{+}$, we may write $\ck\gamma=\sum_{\alpha\in \Pi}c_{\alpha }\ck\alpha$ where all $c_{\alpha}\geq 0$ and some $c_{\alpha}\geq 1$.
  Then $s_{\gamma}(\beta)=\beta+s\gamma\in \mc{Z}$ where  $s=-\pair{\beta,\ck\gamma}\geq \epsilon $ and $\mpair{\rho, \beta}> 0$. Hence $\mpair{\rho,s_{\gamma}(\beta)}=\mpair{\rho,\beta}+s \mpair{\rho,\gamma}\geq \epsilon \mpair{\rho,\gamma}$. 
    One computes that  \[e_{\gamma}:=\tau(\gamma)-\tau(s_{\gamma}(\beta))=a_{\gamma}\gamma-b_{\gamma}\beta\]
  where 
  \[a_{\gamma}=\frac{\mpair{\rho,\beta}}{\mpair{\rho,\gamma}\mpair{\rho,s_{\gamma}(\beta)}},\qquad b_{\gamma}=\frac{1}{\mpair{\rho,s_{\gamma}(\beta)}}.\] 
  Here, \[0\leq \mpair{\rho,a_{\gamma}\gamma}\leq \frac{\mpair{\rho,\beta}}{\epsilon \mpair{\rho,\gamma}},\quad 
  0\leq \mpair{\rho,b_{\gamma}\beta}\leq \frac{\mpair{\rho,\beta}}{\epsilon\mpair{\rho,\gamma}} .\] 
  
  Let $f\in \theta( {R}_{0})$. Then there exists a sequence $(\gamma_{n})_{n\in \Nat}$ of positive roots with $\lim_{n\rightarrow \infty}\mpair{\rho,\gamma_{n}}=\infty$ and 
  $\lim_{n\rightarrow \infty}\tau(\gamma_{n})=f$.
  Note that $a_{\gamma_{n}}\gamma_{n}$ and  $b_{\gamma_{n}}\beta$ are in $\real_{\geq 0}\Pi$ while   \[\lim_{n\rightarrow \infty}\mpair{\rho,a_{\gamma_{n}}\gamma_{n}}=0= \lim_{n\rightarrow \infty}\mpair{\rho,b_{\gamma_{n}}\beta}.\] It follows that \[\lim_{n\rightarrow \infty}a_{\gamma_{n}}\gamma_{n}=0= \lim_{n\rightarrow \infty}b_{\gamma_{n}}\beta\]
  and hence $\lim_{n\rightarrow \infty}e_{\gamma_{n}}=0$. Therefore 
  \[\lim_{n\rightarrow \infty }\tau(s_{\gamma_{n}}(\beta))=
  \lim_{n\rightarrow \infty }\tau(\gamma_{n})=f.\]
  Since $\tau(s_{\gamma_{n}}(\beta))\in \mc{Z}$, we deduce that $f\in \overline{\mc{Z}}$. This completes the proof that $\theta( {R}_{0})\subseteq \overline{\mc{Z}}$.
  
  We have proved that
  $F\subseteq \overline{\mc{Z}}$.  It remains to   prove  that $F^{*}\subseteq \overline{\mc{X}}$.
  For this, it will suffice by Lemma \ref{x4.4}(b)  to show that if $x\in V$ with $\mpair{x, \theta( {R}_{0})}\subseteq \real_{\geq 0}$, then  $x+t\rho\in \mc{X}$ for all $t\in \real_{>0}$.
  If $x+t\rho\not \in \mc{X}$ for some $t>0$, there is by Lemma  \ref{x1.10}(a) a sequence of distinct  roots  $\gamma_{n}\in \Phi_{+}$ with $\mpair{x+t\rho,\gamma_{n}}<0$ for all $n$. Passing to a subsequence, we may assume by Lemma \ref{x4.3}(e) that $\lim_{{n\rightarrow \infty}}\mpair{\rho,\gamma_{n}}=\infty$ and 
  $\lim_{n\rightarrow \infty}\tau(\gamma_{n})=f\in \theta( {R}_{0})$.
  We have $\mpair{\rho, f}=1$, and 
  by assumption on $x$, we have $\mpair{x,f}\geq 0$, and so  $\mpair{x+t\rho,f}\geq t>0$. On the other hand, 
  from $\mpair{x+t\rho,\gamma_{n}}<0$ for all $n$, it follows that 
  \[\mpair{x+t\rho,f}=\lim_{n\rightarrow \infty}\mpair{x+t\rho, \tau(\gamma_{n})}\leq 0,\] a contradiction which  completes the proof.\end{proof}
  
  \subsection{}  \label{x5.5} For each reflection subgroup  $W'$ of $\Phi$, let $ {R}_{W',+}:=\mset{\real_{\geq 0}\al\mid \al\in \Phi_{W',+}}$ and 
   $ {R}_{W',0}:=\ol{ {R}_{W',+}}
   \sm {R}_{W',+}$.  These are subsets of $\ray(\real_{\geq 0}\Pi')\seq \ray(\real_{\geq 0}\Pi)\seq \ray(V) $.
  
    \begin{cor} Let $W'$ be a reflection subgroup of $W$. Then \begin{num}
    \item  $ {R}_{W',+}\seq  {R}_{+}$.  
 \item   $ {R}_{W',0}\seq  {R}_{0}$. 
   \item  $ {R}_{W',0}=\eset$ if $W'$ is finite.
   \item If $W'$ is infinite dihedral, then $ {R}_{W',0}$ is a single $W'$-orbit of rays and  consists of a  set of  one (resp., two) isotropic rays of $V$ according as to whether $W'$ is of  affine or indefinite type.
\item     $\ol{\mc{Z}}_{W'}\seq \ol{\mc{Z}}$.
    \end{num}
     \end{cor} \begin{proof} Part (a) is obvious, and (b) holds since   since  $ {R}_{W',0}\seq \ol{ {R}_{+}}$ contains only isotropic rays.
     Part (c) holds since $ {R}_{W',+}$ is finite (and hence closed) if $W$ is finite, and (d) is well known (see for instance \cite{HLR}).
   For the proof of (e),  assume without loss of generality by  Theorem \ref{x3.3}(b) that $W'$ is finitely generated. Then  (e)  follows  by taking closed convex hulls of the union of each side of (b) and using  Theorem \ref{x5.3}(a) (or from  Lemma \ref{x3.1}(a)  and Theorem \ref{x5.1}).
     \end{proof} 
     \subsection{} \label{x5.6} The set of limit rays of rays spanned by positive 
  roots has been independently  studied in \cite{HLR},
 which contains, as well as some basic results not proved here, several instructive examples (including  diagrams in low rank).     For use later 
  in this paper and in \cite{DHR}, we reformulate below
   in the framework of this paper  a fundamental fact proved in   \cite{HLR}, which implies  in particular that the set of limit rays of positive roots is the closure of the union of the  sets of  limit rays of positive roots of dihedral reflection subgroups.   
  Let $ {R}'_{0}$ (resp.,
  $ {R}''_{0}$) denote $\bigcup_{W'} {R}_{W',0}$
  where the union is over  the infinite dihedral reflection subgroups $W'$ of $W$ (resp., the infinite dihedral reflection subgroups $W'$ containing a simple reflection of $W$ i.e. with $\Pi_{W'}\cap \Pi\neq \eset$).  
 
  \begin{thm}   \begin{num}\item
   $ {R}''_{0}\seq  {R}'_{0}\seq  {R}_{0}$.  \item $\text{\rm (\cite[Theorems 4.2 and 4.5]{HLR})}$  $\ol{ {R}''_{0}}=\ol{  {R}'_{0}}=  {R}_{0}$.
  \item  $\ol{\mc{Z}}=\real_{\geq 0}(\bigcup{R}_{0})=\real_{\geq 0}(\cl(\bigcup{R}_{0}'))=\real_{\geq 0}(\cl(\bigcup  {R}_{0}''))$.
  \end{num}
\begin{rem} The validity of the weaker result in (b), that $ \ol{  {R}'_{0}}=  {R}_{0}$, was raised as a question in an earlier version of this paper.
\end{rem}  
  
  \end{thm}\begin{proof}
  Part (a) is obvious from Corollary \ref{x5.5}.  
  For (b), choose 
 the convex body  $B$ in \ref{x5.3} so that $P:=\mset{v\in \real_{\geq 0}\Pi\mid \mpair{v,\rho}=1}$ is a subset of the 
 boundary $B'$ of $B$ and thereby identify $P$ as a compact 
 subset of $\ray(V) $. For a reflection subgroup $W'$,
 $\Psi_{W',+}:=\mset{\rho(\al)^{-1}\al\mid \al\in \Phi_{W',+}}\seq P$ was called the set of normalized roots of $W'$ in \cite{HLR}. Define $E(W'):=\ol{\Psi_{W',+}}\sm \Psi_{W',+}$, and $E=E(W)$. Also,  let   $E_{2}$ (resp., $E_{2}^{\circ}$) be the union of the sets $E(W')$ for $W'$ ranging over  the infinite dihedral reflection subgroups $W'$ of $W$ (resp., the infinite dihedral reflection subgroups $W'$ with $\Pi_{W'}\cap \Pi\neq \eset$).  The cited results from \cite{HLR} then are the statements that
 $\ol{E_{2}^{\circ}}=\ol{E_{2}}=E$, which are obviously equivalent to (b) via the above identification.
 From Lemma \ref{xA.12}, it follows from (b) that
$ \cl(\bigcup_{r\in  {R}_{0}'}r)=\cl(\bigcup_{r\in  {R}_{0}''}r)= {R}_{0}$, and then (c) follows by  Theorem \ref{x5.4}.
  \end{proof}

\subsection{}\label{x5.7} We introduce next some  notation for several families of facial subsets of $S$ which we shall consider 
subsequently. Let $F_{\mathrm{all}}$ be the set of all  facial subsets of $S$ and   $F_{\mathrm{m}} $ be the set of all maximal proper facial subsets of $S$. Also, let 
 $F_{\mathrm{m.ind}}  $ (resp., $F_{\mathrm{inf}} $) be the set of all elements $I$ of $F_{\mathrm{m}} $ (resp., of $F_{\mathrm{all}} $) such that
  $\Pi_{I}$ has all its irreducible components of indefinite type (resp., infinite type).
  Thus,  $F_{\mathrm{inf}} $ consists of the special facial subsets of $S$.
  
    \subsection{} \label{x5.8}
We now give a description of $\mc{K}$ and $\overline{\mc{Z}}$ which is special to the case of  irreducible  $W$ of indefinite type (see \cite[Exercise 5.11]{Kac}).

\begin{prop} Suppose that $\Pi$ is irreducible  of indefinite type. For each $I\in F_{\mathrm{all}} $, choose $\phi_{I}\in \mc{C}$ such
that  $\Pi_{I}=\Pi\cap \phi_{I}^{\perp}$. Then
\begin{num}
\item ${\mc{K}}=\mset{v\in -{\mc{C}}\cap\real \Pi\mid \mpair{v,\phi_{I}}\geq 0 \text{ \rm for all }I\in F_{\mathrm{m.ind}}  }$.
\item $\overline{\mc{Z}}=\mset{v\in- \overline{\mc{X}}\cap \real\Pi\mid \mpair{v,w\phi_{I}}\geq 0 \text{ \rm for all $I\in F_{\mathrm{m.ind}}  $ and $w\in W$}}$.\end{num}\end{prop}
 
\begin{proof}  Recall the notation $H_{\phi}^{\prec}$ of \ref{x2.3}. 
We use below the fact that any  polyhedral cone of  maximum possible dimensional in a finite dimension in real vector space is the intersection of a unique minimal family of closed half-spaces  (corresponding naturally to the facets (maximal proper faces) of the cone). 
 For $\real_{\geq 0}\Pi$ in $\real\Pi$, this gives  $\real_{\geq 0}\Pi=\bigcap_{I\in F_{\mathrm{m}}}(\real\Pi\cap H_{\phi_{I}}^{\geq})$.
  Since   $-\mc{C} =\cap_{\phi\in -\Pi}\,H_{\phi}^{\geq}$, the 
  polyhedral  cone $\mc{K}$ is the intersection of  the
    half-spaces $\real\Pi\cap H_{\phi}^{\geq } $ of $\real\Pi$ for $\phi\in \Gamma:=-\Pi\cup\mset{\phi_{I}\mid I\in F_{\mathrm{m}} }$.
Since $\mc{K}$ is a full-dimensional   polyhedral  cone in $\real\Pi$, in order 
to prove (a),
 it will suffice to show that for  each $I\in F_{\mathrm{m}} \setminus F_{\mathrm{m.ind}}  $,  $\mc{K}$ is the intersection of  the  closed half-spaces $H_{\phi}^{\geq }\cap \real\Pi$ of $\real\Pi$ for $\phi\in \Gamma_{I}:=\Gamma\setminus \set{\phi_{I}}$.

 Suppose to the contrary  that $z\in\cap_{\phi\in \Gamma_{I}}(\real\Pi\cap {H}_{\phi}^{\geq})
\setminus \mc{K}$. Then $\mpair{z,\phi_{I}}<0$.
Since $\Pi$ is of  indefinite type, we may choose $\beta=\sum_{\alpha\in \Pi}c_{\alpha}\alpha$  with all $c_{\alpha}\in \real_{> 0}$ and  $\mpair{\beta,\alpha}<0$ for all $\alpha\in \Pi$.
In particular,  $\mpair{z,\phi_{I}}< 0$ and $\mpair{\beta,\phi_{I}}>0$, so
there is a unique real number  $t$  with $0<t<1$ such that $z':=tz+(1-t)\beta$ satisfies  $\mpair{z',\phi_{I}}=0$. For all $\alpha\in \Pi$,  $\mpair{z,\alpha}\leq 0$  and $\mpair{\beta,\alpha}<0$ so $\mpair{z',\alpha}<0$.  Since $\mpair{z,\phi_{J}}\geq 0$  and $\mpair{\beta,\phi_{J}}>0$ for all $J\in \Gamma_{I}$,  we have
 $\mpair{z',\phi_{J}}>0$ for all such $J$, and   $\mpair{z',\phi_{I}}=0$ by choice of $t$. So $z'$ lies  in the facet $\real_{\geq 0}\Pi_{I}=\real_{\geq 0}\Pi\cap \phi_{I}^{\perp}$ of $\real_{\geq 0}\Pi$ but is in none of the other facets.  Hence we may write $z'=\sum_{\alpha\in \Pi_{I}}d_{\alpha}\alpha$ with all $d_{\alpha}>0$. 
 The above also showed that 
$\mpair{z',\alpha}<0$ for all $\alpha\in \Pi_{I}$.  But by \ref{x4.5}, this implies that each component
of $\Pi_{I}$ is of indefinite type and so $I\in F_{\mathrm{m.ind}}  $, contrary to assumption. This proves (a).

For (b), first note that  $\mc{Z}=\mc{Y}\cap -\mc{X}$ by Lemma \ref{x3.2}(a), so $\overline {\mc{Z}}\subseteq -\overline{\mc{X}}$. Also, $\mc{Z}\subseteq \real\Pi$ and 
 \[ \mc{Z}\subseteq \mc{Y}=\bigcap_{w\in W }w(\real_{\geq 0} \Pi)=\bigcap_{\substack{w\in W\\ I\in \mc{F}_{\mathrm{m}}}}w(H^{\geq}_{\phi})=\bigcap_{\substack{w\in W\\ I\in \mc{F}_{\mathrm{m}}}}H^{\geq}_{w\phi}\subseteq \bigcap_{\substack{w\in W\\ I\in \mc{F}_{\mathrm{m.ind}}}}H^{\geq}_{w\phi}\] where the right hand side is closed and hence contains $\overline{\mc{Z}}$. 
Hence the inclusion ``$\subseteq$'' in (b) holds. For the reverse inclusion, it suffices 
by Lemma \ref{x4.4}(b) to show that
if $z$ is in the right hand side, then $z+t\rho$ is in $\mc{Z}$ for all $t\in \real_{>0}$.
But since $z\in -\overline{\mc{X}}$ and  $-\rho$ is in the interior of $-\mc{C}\subseteq -\mc{X}$, we have $z+t\rho\in -\mc{X}$ and so $w(z+t\rho)\in -\mc{C}$ for some $w\in W$. 
Now for $I\in F_{\mathrm{m.ind}}  $, we have $\mpair{wz,\phi_{I}}\geq 0$ by assumption, and we have $\mpair{tw\rho,\phi_{I}}\geq 0$ since $tw\rho\in \mc{Z}\subseteq \real_{\geq 0}\Pi$. Hence 
$\mpair{w(z+t\rho),\phi_{I}}\geq 0$, and  by (a) we conclude that $w(z+t\rho)\in \mc{K}$. Hence
$z+t\rho \in \mc{Z}$ for all $t>0$ as required to prove (b).
\end{proof}

 \subsection{}\label{x5.9}  Let us say that a subset  $U$ of $V$ is   \emph{totally isotropic} if $\mpair{U,U}\seq \set{0}$. 
 Observe that a cone  in $V$ is totally isotropic  if and only if it is contained in $\mc{Q}$ i.e. contains only isotropic vectors. Therefore, the notion of a totally isotropic subspace of $V$ in this sense  coincides with the usual definition. 
 \begin{lem}
 \begin{num}
 \item If $x,y\in \ol{\mc{Z}}$, then $\mpair{x,y}\leq 0$.
 \item Let $n,m\in \Nat$, $x_{1},\ldots, x_{n},y_{1},\ldots, y_{m}\in\ol{\mc{Z}}$, 
 $x=x_{1}+\ldots+x_{n}$ and  $y:=y_{1}+\ldots +y_{m}$.
 Then $\mpair{x,y}=0$ if and only if $\mpair{x_{i},y_{j}}=0$ for all $i=1,\ldots,n$ and $j=1,\ldots, m$.
\item Let $n\in \Nat$, $x_{1},\ldots, x_{n}\in\ol{\mc{Z}}$ and
 $x=x_{1}+\ldots+x_{n}$. Then  $x$ is isotropic if and only if   $\real\set{x_{1},\ldots, x_{n}}$ is a totally  isotropic subspace of $\real\mc{Z}$.  \end{num}\end{lem}
 \begin{proof}
 Part (a) follows directly  from Proposition \ref{x3.2}(c). Part (b) holds since
  $\mpair{x,y}=\sum_{i,j}\mpair{x_{i},y_{j}}$ where  $\mpair{x_{i},y_{j}}\leq 0$ for all $i,j$ by (a). 
  Finally, (c) follows from (b) taking $m=n$ and $y_{i}=x_{i}$ for all $i$. 
 \end{proof}
 \subsection{} \label{x5.10} Make assumptions as in  \ref{x1.4}
 and \ref{x3.6} but assume that $(\Phi,\Pi)$ on $(V,\mpair{-,-})$
 satisfies \ref{x4.1}(i)--(iii) and  $V'$ is finite-dimensional.   Note that
the radical of $\mpair{-,-}' $  is $\ker(L)$, so $(\Phi',\Pi')$ on $(V',\mpair{-,-}')$  does not necessarily satisfy \ref{x4.1}(i)--(iii).

  Let $R_{W,+}=\mset{\real_{\geq 0}\al\mid \al\in \Phi_{+}}\seq \ray(\real_{\geq 0} \Pi)$
 and $R_{W',+}=\mset{\real_{\geq 0}\al\mid \al\in \Phi'_{+}}\seq \ray(\real_{\geq 0} \Pi')$ denote the set of rays spanned by  roots in $\Phi_{+}$ and $\Phi'_{+}$ respectively,  and  let
 $R_{W,0}=\ol{R_{W,+}}\sm R_{W,+}$ and 
 $R_{W',0}=\ol{R_{W',+}}\sm R_{W',+}$ 
denote the corresponding  sets of limit rays (i.e limit points).
By  Proposition \ref{x3.6}(a), there is a  map 
 $L'\colon \ray(\real_{\geq 0}\Pi')\to \ray(\real_{\geq 0}\Pi)$, determined 
 by $L'(\real_{\geq 0}\alpha)=\real_{\geq 0}L(\alpha)$.

 \begin{prop}
 \begin{num}
 \item $L'$ restricts to a surjective map $R_{W',0}\to R_{W,0}$.
 \item $L$ restricts to a surjective map $\ol{\mc{Z}}_{W'}\to \ol{\mc{Z}}_{W}$.
 \end{num}
 \end{prop}
 \begin{proof} Clearly, $L'$ is continuous and it surjective by 
 Proposition \ref{x3.6}(b). So $L'$ is a continuous 
 surjective map 
 between compact Hausdorff spaces and is therefore a closed map.   By \ref{x1.9},  $L'$ restricts to a 
 bijection $R_{W',+}\to R_{W,+}$. This implies that 
 $L'(\ol{R_{W',+}})=\ol{L'(R_{W',+})}=\ol{R_{W,+}}$.  Note that 
 $L'$ preserves isotropic rays (and non-isotropic rays), so it induces a 
 surjection from the  set of isotropic rays  in $\ol{R_{W',+}}$ to the set of isotropic rays in
 $\ol{R_{W,+}}$. That is, by  Proposition \ref{x5.3},
  $L'(R_{W',0})=R_{W,0}$, proving (a).  
  By Proposition \ref{x3.6}(c)), we have  $
  \ol{\mc{Z}}=\ol{L'(\mc{Z}_{W'})}=L'(\ol{\mc{Z}_{W'}})$ since $L'$ is closed, proving (b).
   (An alternative proof of (b) could be given using  (a) and Theorem \ref{x5.4}.)
     \end{proof}
 
\section{Imaginary cone of a reflection subgroup}
\label{x6}

  \subsection{} In  \label{x6.1}  preparation for  the proof of  Theorem \ref{x6.3}, we state  two 
  lemmas. The first is of interest even apart from its role in the proof of the theorem. Recall that $\rho$ denotes an arbitrary but fixed element of the interior of $\mc{C}$.
	  \begin{lem} Suppose that $v\in\overline{ \mc{Z}}$. Then
	  \begin{num}
	  \item $\overline{Wv}\seq \overline{\mc{Z}}\seq \real_{\geq 0}\Pi$.
	  \item $0\leq \lambda:=\inf(\mset{\mpair{x,\rho}\mid x\in \overline{Wv}})$.
	  \item $\eset\neq M_{v}:=\mset{x\in \overline{Wv}\mid \mpair{x,\rho}=\lambda}$.
	  \item $M_{v}\seq \mc{K}$.
	  \item $\ol{Wv}\cap \mc{K}\neq \eset$.
	  \item  If $v$ is non-isotropic, or $v\in \mc{Z}\sm\set{0}$, then $\lambda>0$.
	  \end{num} \end{lem}
  \begin{proof} Part (a) holds since $\ol{\mc{Z}}$ is $W$-invariant (since $\mc{Z}$ is),
   $\mc{Z}\seq \real\Pi$ and $\real_{\geq 0}\Pi$ is closed (being a polyhedral cone).
  Part (b) follows from (a) since $\mpair{\rho,\real_{\geq 0}\Pi}\seq \real_{\geq 0}$. 
   One may choose a sequence $(w_{n})_{n\in \Nat}$ in $W$ with 
   $\lim_{n\rightarrow \infty }\mpair{\rho, w_{n}v}=\lambda$.  Since 
   $\mset{z\in \real_{\geq 0}\Pi \mid \mpair{\rho, z}\leq \lambda+1}$ is compact,  by passing to 
   a subsequence we may assume that the sequence $(w_{n}v)$ is convergent, say to  
   $x\in \overline{Wv}$.
   Then $x\in M_{v}$.
   This proves (c). 
     Now let $x\in M_{v}$ be arbitrary. One may choose a sequence $(w_{n})$ in $W$ such 
     that $(w_{n}v)$ converges to $x$.   To prove  (d), it will suffice to show that 
     $x\in -\mc{C}$  (for then  $x\in -\mc{C}\cap  \real_{\geq 0}\Pi\seq  \mc{K}$).
Suppose to the contrary that $\mpair{x,\alpha}=c>0$ for some $\alpha\in \Pi$.
We have $\mpair{\rho,x}=\lambda$.  Recall that $\mpair{\rho,\alpha}>0$. Since 
$\lim_{n\rightarrow \infty}w_{n}v=x$, we may   choose $n\in \Nat$ sufficiently large   that  both
$\mpair{\rho, w_{n}v}<\lambda+c\mpair{\rho,\alpha}$ and $d:= \mpair{w_{n}v,\ck\alpha}>\frac{1}{2}\mpair{x,\ck\alpha}=c$.
Then since $s_{\alpha}(w_{n}v)\in \ol{Wv}$, one has \[ \lambda\leq \mpair{\rho,s_{\alpha}(w_{n}v)}=\mpair{\rho, w_{n}v-d\alpha}<
\lambda+c\mpair{\rho,\alpha}- d\mpair{\rho,\alpha}<\lambda,\] which is a contradiction proving (d). Part (e) is immediate from (c) and (d). 
In (f),  the case in which $v\in \mc{Z}$  follows since there is
$v''\in Wv\cap \mc{K}$ and then $\rho(v'')\geq \rho(v')$ for all $v''\in Wv$, hence for all 
$v''\in \ol{Wv}$. Assume now that $\mpair{v,v}\neq 0$ and  let $x\in M_{v}$. One has  $\mpair{y,y}=\mpair{v,v}$ for all $y\in Wv$ and hence for all $y\in \ol{Wv}$. In particular, $\mpair{x,x}=\mpair{v,v}$.  If $\mpair{v,v}\neq 0$, then $\mpair{x,x}\neq 0$. This gives $x\in \real_{\geq 0}\Pi\sm\set{0}$ and so
$\lambda=\mpair{\rho, x}>0$.
 \end{proof}

\subsection{} \label{x6.2} The second lemma required in our proof of Theorem \ref{x6.3} is purely technical;    it follows from  the theorem and Lemma \ref{x4.4}(a). \begin{lem}  Let $W'$ be a finitely-generated reflection subgroup of $W$ and $z\in \mc{Z}_{W'}$.
Then the $W$-orbit $Wz$ of $z$ is discrete and closed in $V$. \end{lem}
\begin{proof} Without loss of generality, we may assume that $z=k\in\mc{K}_{W'}$.
It will suffice to prove the following claim: if
 $(w_{n})_{n\in \Nat}$ is a sequence in $W$ such that
the sequence $(w_{n}z)_{n}$ converges in $V$, say to $x$, then there is a subsequence $(w_{n_{m}})_{m\in \Nat} $ such that $w_{n_{m}}z=x$ for all sufficiently large $m$. We will show below equivalently that  after passing to a suitable subsequence of $(w_{n})$, we have $w_{n}z=x$ for all $n$.
 
   Now  $k$ lies in some face of the polyhedral cone  $\real_{\geq 0}\Pi_{W'}$, say that corresponding to the facial subset $I'\subseteq S'$ of $S':=\chi(W')$.
   By Lemma \ref{x3.4}(a), we have  $k\in \real_{\geq 0}\Pi_{W'_{I'}}\cap \mc{K}_{{W'}}=\mc{K}_{W'_{I'}}$. 
   If the Lemma holds with $W'$ replaced by $W'_{I'}$, it holds for $W'$.
   Hence we may and do assume without loss of generality that 
   $I'=S'$ i.e. $k$ is a point of the relative interior of the cone $\real\Pi_{W'}$.
   It follows that there is an expression $k=\sum_{\alpha\in \Pi_{W'} }c_{\alpha}\alpha$ with all
   $c_{\alpha}>0$, since $\Pi_{W'}$ is a set of representatives of the extreme rays of this cone. Let $\epsilon=\min(\mset{c_{\alpha}\mid {\alpha\in \Pi_{W'}}})>0$.
   For any $w\in W$, we write $w=w''w'$ where $w'\in W'$ and $w''\in W$ satisfies $N(w^{''-1})\cap W'=\emptyset$ i.e. $w''$ is the unique element of minimum length in the coset $wW'$.
   Using Lemma \ref{x1.13}(c),  write $w'k=\sum_{\alpha\in \Pi'}(c_{\alpha}+d_{w',\alpha})\alpha$ where all 
   $d_{w',\alpha}\geq 0$.
   
   Note that the sequence $(\mpair{\rho,w_{n}k})_{n}$ is bounded above, say  $\mpair{\rho,w_{n}k}\leq M$ where $M\in \real_{\geq 0}$.   
  We have \begin{equation*}\begin{split}M&\geq  \mpair{\rho,w_{n}k}=\mpair{\rho,w_{n}''w_{n}'k}=\mpair{\rho, \sum_{\alpha\in \Pi'}(c_{\alpha}+d_{w_{n}',\alpha})w_{n}''\alpha}\\ &\geq \epsilon \max(\set{\mpair{\rho,w_{n}''\alpha}\mid \alpha\in \Pi_{W'}})\end{split}\end{equation*}
   where all $w_{n}''(\alpha)\in\Phi_{+}$. Hence $\mpair{\rho,w_{n}''\alpha}\leq M/\epsilon$ for all
   $n\in \Nat$ and all $\alpha\in \Pi_{W'}$. By Lemma \ref{x4.3}(e), there are only finitely many possibilities for the $\Pi_{W'}$-indexed families $(w_{n}''\alpha)_{\alpha\in \Pi_{W'}}$ for $n\in \Nat$. Replacing $(w_{n})$ by a subsequence, we may assume that the sequence of families  $(w_{n}''\alpha)_{\alpha\in \Pi_{W'}}$ for $n\in \Nat$ is constant, so $(w_{n}''\alpha)_{\alpha\in \Pi_{W'}}=(w_{0}''\alpha)_{\alpha\in \Pi_{W'}}$ for all $n\in \Nat$.
   Then \begin{equation*} \begin{split}  w_{n}k&= \sum_{\alpha\in \Pi'}(c_{\alpha}+d_{w_{n}',\alpha})w_{n}''\alpha=\sum_{\alpha\in \Pi'}(c_{\alpha}+d_{w_{n}',\alpha})w_{0}''\alpha\\&=w_{0}''\sum_{\alpha\in \Pi'}(c_{\alpha}+d_{w_{n}',\alpha})\alpha=w_{0}''w'_{n}k \end{split} \]
   
   It follows that the sequence $w_{n}'k$ converges to $w_{0}^{\prime\prime -1}x$ and in particular, the sequence $\mpair{\rho, w_{n}'k-k}$ for $n\in \Nat$ is bounded above.
Recall that $\rho$ is an interior point of  $\mc{C}_{W'}$, by \ref{x1.10}(e).   By  \eqref{x4.4.1} applied to $W'$, we see that  $w_{n}'(k)-k$ has a constant subsequence. Passing to an appropriate  subsequence of $(w_{n})$ yet  again, we may  therefore
   assume that $w_{n}'k$ is constant. Hence $w_{n}'k=w_{0}^{\prime\prime -1}x$ for all $n$,
   and $w_{n}k=x$ for all $n$ as required to complete the proof.
   
    \end{proof}

   \subsection{}\label{x6.3}
  The following is a main result of this work.
  \begin{thm}  Let $W'$ be a  reflection subgroup  of $W$.  
 Then   $\mc{Z}_{W'}\subseteq \mc{Z}_{W}$.\end{thm}
  \begin{proof}  Using \ref{x3.3}(b),  assume without loss of generality that $W'$ is finitely generated.   It is clear from Theorem \ref{x5.1} and   Lemma \ref{x3.1}(a), that $\overline{\mc{Z}_{W'}}\subseteq\overline{ \mc{Z}_{W}}$.    Let $z\in \mc{Z}_{W'}$.
   Then  $z\in \overline{\mc{Z}}$, so by Lemma \ref{x6.1} there is a sequence $(w_{n})_{n\in \Nat}$ in $W$ such that 
   the sequence $(w_{n}z)$ converges to an element $x\in \mc{K}$.
   By Lemma \ref{x6.2},  $w_{n}z=x\in \mc{K}$ for all but finitely many $n$. Hence $z=w^{-1}_{n}x\in \mc{Z}$ for some $n$, as required.\end{proof}

\subsection{} \label{x6.4} The final result  in this section refines  Proposition \ref{x2.6}.
\begin{cor} Assume that $W''$ is a facial subgroup of $W$ and $W'$ is a finitely generated reflection subgroup of $W$. Then $W''':=W'\cap W''$ is a finitely-generated reflection subgroup
of $W$ and  $\mc{Z}_{W'}\cap \mc{Z}_{W''}=\mc{Z}_{W'''}$.
\end{cor}\begin{proof}
Note that $W'''$ is a finite rank reflection subgroup by \ref{x2.6}.
Assume  first that  $W''$ is standard facial,
say $W''=W_{J}$ for facial $J\subseteq S$.
Then $\chi(W'')=J':=
\chi(W')\cap W_{J}$ by \ref{x2.6}.  
By Theorem \ref{x6.3}, we have $\mc{Z}_{W'''}\subseteq \mc{Z}_{W'}\cap\mc{Z}_{W_{J}}$.
For the reverse inclusion, note that 
\begin{equation*}\begin{split} \mc{Z}_{W'}\cap \mc{Z}_{W_{J}}&\subseteq \bigl(\mc{Z}_{W'}\cap\real_{\geq 0}\Pi_{W'}\bigr)\cap \real\Pi_{J}\\ 
&= \mc{Z}_{W'}\cap
 \real_{\geq 0}\Pi_{W'_{J'}}=\mc{Z}_{W'_{J'}}.\end{split}\end{equation*}
 Here, we use \ref{x2.6}(a)  to get the first equality, 
and  Lemma \ref{x3.4}(a) applied to $W'$ to get the second (recalling $J'$ is facial in $\chi(W')$ by \ref{x2.6}). 
This proves the Corollary in the special case that $W''$ is a standard facial subgroup of $W$.
The  case of a general facial subgroup $W''$ reduces easily to the special  case just treated
by writing $W''=wW_{J}w^{-1}$ for some facial $J\subseteq S$ and using Lemma
  \ref{x3.2}(e). 
\end{proof}

  \section{Action on the closed imaginary cone} \label{x7}
  \subsection{} The following  simple fact will prove useful.
  \begin{lem}\label{x7.1}
  Suppose that $\Phi$ is irreducible infinite but not of affine type.   If $\al\in \real_{\geq 0}\Pi$ is non-zero, then $\al\not\perp\Pi$ and
  $\aff(W\al)=\real(W\al)=\real \Pi$.
  \end{lem}
  \begin{proof} Write $\al=\sum_{\g\in \G}c_{\g}\g$ for some (finite)
 $ \G\seq \Pi$  where all $c_{\g}>0$.  If $\G\sneq \Pi$, then, by 
 irreducibility of $\Pi$, there is some $\bt\in \Pi\sm \G$ with 
 $\mpair{\G, \bt}\neq 0$ and then $\mpair{\bt,\al}<0$ also. Otherwise, 
 $\G=\Pi$ and so $\mpair{\Pi,\al}\neq \set{0}$ since $\Pi$ is not of affine 
 type. Hence $\al\not\perp \Pi$. By Lemma \ref{x1.13}(d),
 $\real(W\al)\sreq \aff(W\al)= \al+\real\Pi=\real\Pi\sreq \real(W\al)$ and 
 equality must hold throughout. 
   \end{proof}
  \subsection{}  Recall from Section \ref{x5}  the definitions of  the  set $\ray(V) $ of rays of $V$ (as 
 topologized in \ref{x5.3}) and its subsets $R_{0}\sreq R'_{0}\sreq R''_{0}$ of isotropic rays.   The analogously defined subsets for a reflection subgroup $W'$ are denoted $R_{W',0}\sreq R'_{W',0}\sreq R''_{W',0}$.
  
 \label{x7.2}
  \begin{lem} Suppose that $\Phi$ is irreducible and not of affine type and that $\al\in \mc{Z}\sm\set{0}$. Let $W'$ be a non-trivial reflection subgroup of $W$.
  \begin{num}
  \item
  There exists $\bt\in W\al\cap -\mc{C}_{W'}$ such that 
  $\bt\not\perp \Pi_{W'}$.
  \item Assume further that $W'$ is infinite irreducible and let  $\bt$ be  as in $\text{\rm (a)}$. Then there exists $\epsilon>0$ such that $\Psi:=\mset{\g\in \Phi_{W',+}\mid \mpair{\bt,\ck\g}<-\epsilon}$ is infinite.
  \item Let assumptions be as in $\text{\rm (b)}$. Then there is 
  a sequence $(\g_{n})_{n\in \Nat}$  of distinct roots in $\Psi$ such that 
  $(\real_{\geq 0}\g_{n})$ converges  in $\ray(V)$ to a ray 
  $\real_{\geq 0}\g$ in $R_{W',0}$. Further,  for any such sequence $(\g_{n})$, $(\real_{\geq 0}s_{\g_{n}}(\beta))_{n\in \Nat}$ then necessarily  converges in $\ray(V)$ to
  $\real_{\geq 0}\g$.

  \end{num}
  \end{lem}
  
    \begin{proof} First we prove (a).
    Without loss of generality, assume that $\al\in \mc{K}\seq -\mc{C}_{W'}$. 
 If $\al\not \perp \Pi_{W'}$, then $\bt:=\al$ satisfies the requirements of (a). Henceforward,  suppose that 
 $\al\in \mc{K}\cap \Pi_{W'}^{\perp}$. Set $\D:=\Pi\cap \al^{\perp}$. By Lemma \ref{x7.1}, we have $\Delta\neq \Pi$.

 For each $\g\in \Pi_{W'}$, choose a finite set $\G_{\g}\seq \Pi$
 and scalars  $c_{\g,\delta}>0$ for $\delta\in\G_{\g}$ such that
 $\g=\sum_{\delta\in \G_{\g}}c_{\g,\delta}\delta$. 
 We have $0=\mpair{\g,\al}= \sum_{\delta\in \G_{\g}}c_{\g,\delta}\mpair{\delta,\al}$
 where each $\mpair{\delta,\al}\leq 0$ since $\al\in -\CC$ and 
 $\mpair{\delta,\al}<0$ if $\delta\in \Pi\sm \D$. This implies that 
 $\G_{\g}\seq \D$ and hence $\G:=\cup_{\g\in \Pi_{W'}}\Gamma_{\gamma}\seq \D$. 
 Note that $\Pi_{W'}\seq \real\G$. 
  Also,  $\G\neq \eset$ since $\Pi_{W'}\neq \eset$ by the assumed 
  non-triviality of $W'$.  
  
  Using connectedness of the Coxeter graph of $W$, we may choose 
  $p\in \Nat$ and   a sequence $\delta_{0},\delta_{1},\ldots,\delta_{p}$ of simple roots 
  with $\delta_{0}\in \Pi\sm \Delta$, $\delta_{p}\in \G$ and
  $\mpair{\delta_{i-1},\delta_{i}}\neq 0$ for $i=1,\ldots, p$.  Suppose 
  $p$ and the sequence is   chosen so $p$ is minimal amongst all such 
  sequences. Then $p\geq 1$, $\delta_{1},\ldots, \delta_{p}\in \Delta$ and $\delta_{0},\ldots, \delta_{p-1}\not\in \Gamma$. Moreover, if  $0\leq i\leq p-2$, then $\delta_{i}$ is not joined in the Coxeter  graph of  $W$ to any element of $\Gamma$ and so $\delta_{i}\perp \Pi_{W'}$.
  
  Set $\bt:=s_{\delta_{p-1}}\ldots s_{\delta_{0}}(\al)\in W\al$. By Lemma \ref{x1.13}(a), $\bt=\al+c \tau$ where $c=-\mpair{\al,\ck\delta_{0}}>0$ 
  (since $\delta_{0}\not \in\Delta$) and  $\tau:= s_{\delta_{p-1}}\ldots s_{\delta_{1}}(\delta_{0})$.
  Using Lemma \ref{x1.16}(a), one has
  $\bt=\al+b_{0}\delta_{0}+\ldots +b_{p-1}\delta_{p-1}$ for some 
  $b_{1},\ldots, b_{p-1}\in \real_{>0}$.
  Now for $\g\in \Pi_{W'}$,  \begin{equation*}
  \mpair{\bt,\g}=\mpair{\al+\sum_{i=0}^{p-1}b_{i}\delta_{i},\g}=b_{p-1}\mpair{\delta_{p-1},\g}=\sum_{\delta\in \G_{\g}}b_{p-1}c_{\g,\delta}\mpair{\delta_{p-1},\delta}.
  \end{equation*}
  Since $\delta_{p-1}\not\in \G$ and $\G\sreq \G_{\g}$, it follows that $\mpair{\bt,\g}\leq 0$ and $\bt\in -\mc{C}_{W'}$. Moreover, one has $\delta_{p}\in \G$, so $\delta_{p}\in \G_{\g}$ for some $\g\in \Pi_{W'}$.
  Then $\mpair{\bt,\g}\leq b_{p-1}c_{\g,\delta_{p}}\mpair{\delta_{p-1},\delta_{p}}<0$ and so $\bt\not\in \Pi_{W'}^{\perp}$. This proves (a).
  
  For the proof of (b), choose $\g_{0}\in \Pi_{W'}$ with
   $\e:=-\mpair{\bt,\ck \g_{0}}>0$. Let $W'':=W_{\Pi_{W'}\sm\set{\g_{0}}}$.
   Let $\Psi':=\Phi_{W',+}\sm \Phi_{W'',+}$ which is infinite by 
   Lemma \ref{x1.20}(b). We claim that $\Psi'\seq \Psi$ i.e.
   $\mpair{\bt,\ck\tau}<-\epsilon$ for all $\tau\in \Psi'$.  This may be regarded as a statement purely in terms of inner products in the root system $\Phi_{W'}$, for the proof of which we may assume $\Pi_{W'}$ is linearly independent 
   By Lemma \ref{x2.4}, each element $\tau$ of $\Psi'$ can be 
   written in the form $\tau=\sum_{\g\in \Pi_{W'}}c_{\g}\g$ where  all 
   $c_{\g}\geq 0$ and $c_{\g_{0}}>0$.  By  \ref{x1.4} and  a result of Brink (see
    \cite[Lemma (3.3)]{DySd}), one may  choose the $c_{\gamma}$ so  $c_{\g_{0}}\geq 1$. Hence
      $\mpair{\bt,\ck\tau}= \mpair{\bt, \sum_{\g\in \Pi_{W'}}c_{\g}\ck\g}\leq c_{\g_{0}}\mpair{\bt, \ck\g_{0}}<
   -\epsilon$. This proves (b).
   
   Finally, we prove (c).  Since $\Psi$ is infinite, there exists an infinite sequence
   $(\g_{n})_{n\in \Nat}$ of distinct elements of $\Psi$.
  Since $W$ is of finite rank,   $\rho(\g_{n})\to \infty$ as $n\to \infty$ by \eqref{x4.4.1}.  Since $\ray(V) $ is 
   compact, by passing to a subsequence if necessary,    we may assume  this sequence converges in $\ray(V) $ to a ray   $\real_{\geq 0}\g\in R_{W',0}$,
   where $\g\in \real_{\geq 0}\Pi\sm\set{0}$. Now let $(\g_{n})$ be any sequence of distinct roots in $\Psi$ with $\real_{\geq 0}\g_{n}\to \real_{\geq 0}\g$.
   Then $\mpair{\rho,\g_{n}}^{-1}\g_{n}\to \mpair{\rho,\g}^{-1}\g $ and $\mpair{\rho,\g_{n}}\to \infty$ as $n\to \infty$.
     Now $s_{\g_{n}}(\bt)=\bt-\mpair{\bt,\ck\g_{n}}\g_{n}$.
     Let $c_{n}:=-\mpair{\bt,\ck\g_{n}}^{-1}\mpair{\rho,\g_{n}}^{-1}$. 
     Since $\epsilon < -\mpair{\bt,\ck \g_{n}}$ for all $n$, $c_{n}>0$ and $c_{n}\to 0$ as $n\to \infty$.
    One has $\real_{\geq 0}s_{\g_{n}}(\bt)=\real_{\geq 0}\delta_{n}$ 
    where  
    $\delta_{n}:=c_{n}s_{\g_{n}}(\bt)=c_{n}\bt_{n}+
    \mpair{\rho,\g_{n}}^{-1}\g_{n}$.  Clearly, as $n\to \infty$, 
    $\delta_{n}\to \mpair{\rho,\g}^{-1}\g$ and so $\real_{\geq 0}s_{\g_{n}}
    (\bt)\to \real_{\geq 0}\g$ as required.
    \end{proof}
       \begin{rem} (1)  The  proofs above of Lemma \ref{x7.1}  and (a)--(b) hold in the framework in Sections \ref{x1}--\ref{x3} (in particular, $W$ need not be of finite rank). However, if $W$ is finite or of locally finite type, then $\mc{Z}=\set{0}$, so no $\al$ as in the statement of the above lemma exists.
   
(2)  In the case $W'$ is infinite dihedral,   a simpler proof of (b) is as follows. Write 
   $\Pi_{W'}=\set{\delta,\delta'}$. Then $\Phi_{W',+}\sm\Pi_{W'}\seq \Psi$ 
   since, as is  well known and easily checked, for $\tau\in \Phi_{W',+}\sm \Pi_{W'}$, one has $\tau=c\delta+d\delta'$ where $c,d\geq 1$. 
 \end{rem}  
 \subsection{} \label{x7.3}
 If $w\in W$ and $\mc{Z}\neq 0$, then $w$ has an eigenvector $\al$ in 
 $\ol{\mc{Z}}$ with strictly positive eigenvalue  equal to the 
 spectral radius of $w$ on $\real\mc{Z}$, by Perron-Frobenius 
  theory (see \cite{Va}, \cite{RVa}). 
  
\begin{lem}
Let $\al\in \ol{\mc{Z}}$ be an eigenvector of $w\in W$
  with (real)  eigenvalue $\lambda$.   So  $\al\neq 0$, $w\al=\lambda\al$ and   $\lambda>0$. Let $V_{w,\lambda}:=\mset{\g\in V\mid w(\g)=\lambda \g}$ be the $\lambda$-eigenspace of $w$ on $V$
\begin{num}
\item  If $\lambda\neq 1$, then $V_{w,\lambda}\cap \ol{\mc{Z}}$ is a  totally isotropic subset of $V$.  
\item   If $\lambda>1$, then $\mpair{\rho,w^{n}\al}\to\infty$ and 
   $w^{-n}\al\to 0$ as $n\to \infty$.
 Similarly, if  $\lambda<1$,  then $w^{n}\al\to 0$ and $\mpair{\rho,w^{-n}\al}\to\infty$ as $n\to \infty$. 
 \item If $\al'$ and $\al''$ are linearly independent eigenvectors of $w$ in $\ri(\mc{Z})$ with corresponding eigenvalues $\lambda'$ and $\lambda''$, then $\lambda'=\lambda''$ and there is an eigenvector
  $\al''\in \rb(\mc{Z})$  of $w$ with eigenvalue $\lambda'$.
\end{num}    
\end{lem}
\begin{proof} One has $\lambda\neq 0$ since $w$ acts invertibly on $\real{\mc{Z}}$ and $\lambda\not<0$ since $\ol{\mc{Z}}$ is salient. Hence $\lambda>0$. Assume $\lambda\neq 1$.
If $\bt,\g\in V_{w,\lambda}$, then $\mpair{\bt,\g}=\mpair{w\bt,w\g}=\lambda^{2}\mpair{\bt,\g}$ and so $\mpair{\bt,\g}=0$ since $\lambda^{2}\neq1$. Part (a) follows readily.
Suppose now that $\lambda>1$. Then  $\mpair{\rho,\al}\neq 0$ since $\al\neq 0$, $\mpair{\rho,w^{n}\al}=\lambda^{n}\mpair{\rho,\alpha}\to\infty$ and $\mpair{\rho,w^{-n}\al}=\lambda^{-n}\mpair{\rho,\alpha}\to 0$ as $n\to \infty$. Hence $w^{-n}\alpha\to 0$ by Remark \ref{xA.11}. This proves the first part of (b), and the second follows from it.
  Part (c) is a special case of   \cite[Lemma 4.1]{Va}.
\end{proof}

\subsection{}  \label{x7.4} 
 For $\al\in \ol{\mc{Z}}\sm\set{0}$, set 
  ${E}_{\al}:=\mset{\real_{\geq 0}\bt\mid \bt\in W\al} \seq \ray(V) $ and let 
  $E'_{\al}\seq \ol{E_{\al}}$ be the set of limit points  of $E_{\al}$ in
    $ \ray(V)$.  One has $\ol{E_{\al}'}=E_{\al}'$. \begin{lem} If $\al\in (\ol{\mc{Z}}\sm\set{0})\cap (\mc{Z}\cup\mc{Q})$, then $E'_{\al}\seq Q$. 
  \end{lem}  
 \begin{proof}
 Note that if two rays $\real_{\geq 0} x(\al)$, $\real_{\geq 0} y(\al)$, where $x,y\in W$ and    $\al\in\ol{ \mc{Z}}\sm\set{0}$,  are equal, then $\al$ is an eigenvector in $\ol{\mc{Z}}$ of $x^{-1}y$. 
 Suppose that $\real_{\geq 0}\bt$, $0\neq\bt\in \real_{\geq 0}\Pi$, is a limit ray of $E_{\al}$.
 Then there is a sequence $(w_{n})_{n\in \Nat}$ of elements of $W$ such that the rays $\real_{\geq 0}w_{n}\al$ in $\ol{\mc{Z}}$ are pairwise  distinct and converge in $\ray(V)$ to $\real_{\geq 0}\bt$. That is,
 $\g_{n}:=\mpair{\rho,w_{n}\al}^{-1}w_{n}\al\to \mpair{\rho,\bt}^{-1}\bt$ as $n\to \infty$ in $V$.
 If $\al$ is isotropic, then so is each $w_{n}\al$ and so $\bt$ is isotropic 
 as required. If $\al\in \mc{Z}$,  then, since
  $W\al\cap -\mc{C}\neq \eset$,  \eqref{x4.4.1} implies that  for any $\eta>0$, there are only finitely 
  many $\g\in W\al$ with $\mpair{\g,\rho}<\eta$. Hence as $n\to \infty$,
   $\mpair{\rho, w_{n}\al} \to\infty$ . Therefore
 $  \mpair{\g_{n},\g_{n}}=\mpair{\rho,w_{n}\al}^{-2}\mpair{w_{n}\al,w_{n}\al}\to 0$ since $\mpair{w_{n}\al,w_{n}\al}=\mpair{\al,\al}$.
 But also $  \mpair{\g_{n},\g_{n}}\to  \mpair{\rho,\bt}^{-2}\mpair{\bt,\bt}$, so $\bt$  is isotropic as claimed.

  \end{proof} 
  \subsection{} \label{x7.5} The  main consequence  of the next result is stated as Theorem \ref{x7.6}.
    \begin{thm}
  Assume  $W$ is irreducible.  
\begin{num} 
\item   If  $\al\in{ \mc{Z}}\sm\set{0}$, then
$R_{0}\seq \ol{E_{\al}}$. 
\item If $\al\in\ol{ \mc{Z}}\sm\set{0}$, then  $\ol{\mc{Z}}=
\cl\bigl(\conv(\bigcup E_{\al})\bigr)=\conv\bigl(\cl(\bigcup E_{\al})\bigr)$. 
 \end{num}
\end{thm} 
\begin{proof} If $W$ is finite, then (a)--(b) hold vacuously and if $W$ is affine, they hold trivially since $\al$ spans the isotropic ray $\real_{\geq 0}\al=\mc{K}=\mc{Z}=\ol{\mc{Z}}$ and is fixed by $W$. Henceforward assume $W$ is of indefinite type. 
For the proof of  (a),  fix $\al\in{ \mc{Z}}\sm\set{0}$. Let $W'$ be any infinite dihedral reflection subgroup of $W$. 
 By  Lemma \ref{x7.2}(c), there is $w\in W$ such that  the closure in $\ray(V) $ of 
the $W'$-orbit of 
$\real_{\geq 0}w\al$ contains a point of $R_{W',0}$. By Corollary \ref{x5.5}(d), the 
closure  of the $W$-orbit $E_{\al}$ of 
$\real_{\geq 0}\al$ contains
$R_{W',0}$ for all infinite dihedral subgroups $W'$ of $W$.  
That is, $R'_{0}\seq\overline{E}_{\al}$ and hence 
$R_{0}=\ol{R'_{0}}\seq \ol{E}_{\al}$ by Theorem \ref{x5.6}(b).
This proves the  assertion of (a).

For (b),  let $\al\in\ol{ \mc{Z}}\sm\set{0}$.  By Lemma \ref{x7.1}, we have 
 $\real(W\al)=\real\Pi$.   One may  therefore choose 
 $\G\seq W\al\seq \ol{\mc{Z}}$  which is a basis for $\real \Pi$.
  Then $\real_{>0}\G\seq \ri(\real_{\geq 0}\G)\seq \ri(\ol{\mc{Z}})=\ri(\mc{Z})$. Let $\al'\in \real_{>0}\G$.
  By (a) and Theorem \ref{x5.3}(a),   it follows that 
 $ \ol{\mc{Z}}\sreq \conv(\cup \ol{E}_{\al'}) \sreq \conv(\cup R_{0})=\ol{\mc{Z}}$. By  Lemma \ref{xA.12}(b) and (d) this gives
 $\ol{\mc{Z}}= \conv\bigl(\cl(\bigcup_{ w\in W }\real_{\geq 0}w(\al'))\bigr)=
     \cl\bigl(\conv(\bigcup_{ w\in W }\real_{\geq 0}w(\al'))\bigr)$. 
  Since   $\al'\in \conv(\cup_{w\in W}\real_{\geq 0}w\al)$, this implies that  \begin{equation*}
  \ol{\mc{Z}}\sreq \cl\bigl(\conv(\bigcup_{w\in W} \real_{\geq 0}w(\al))\bigr)\sreq
  \cl\bigl(\conv(\bigcup_{w\in W} \real_{\geq 0}w(\al'))\bigr)=\ol{\mc{Z}}. 
  \end{equation*} 
  This establishes that  $\ol{\mc{Z}}=
\cl\bigl(\conv(\bigcup E_{\al}\bigr))$. By Lemma \ref{xA.12}, 
$\cl\bigl(\conv(\bigcup E_{\al}\bigr))=\ol{\mc{Z}}=\conv\bigl(\cl(\bigcup 
E_{\al}\bigr))$, completing the proof of (b). 
\end{proof}

\subsection{}  \label{x7.6} For  finite $W$, $\mc{Z}=\set{0}$ is the only non-empty  $W$-invariant   cone  contained in $\real_{\geq 0}\Pi$, 
since any such cone is contained in   $\real_{\geq 0}\Pi\cap w_{S}(\real_{\geq 0}\Pi)=\real_{\geq 0}\Pi\cap-\real_{\geq 0}\Pi=\set{0}$ where $w_{S}$ is the longest element.
For irreducible infinite $W$, one has instead: 
\begin{thm}
Suppose that  $W$ is irreducible and  infinite. Then
$\ol{\mc{Z}}$ is the unique non-zero $W$-invariant closed pointed cone contained in $\real_{\geq 0}\Pi$.
 \end{thm}
\begin{proof} It has already been shown that $\ol{\mc{Z}}$ has the properties listed. 
Let $C$ be any cone with these properties, and fix   $\al\in C$ with $\al\neq 0$. Then $C\seq \cap_{w\in W}w(\real_{\geq 0}\Pi)=\mc{Y}=\ol{\mc{Z}}$ by Theorem \ref{x5.1}.  
Hence $\al\in \ol{\mc{Z}}\sm\set{0}$. Therefore, by Proposition \ref{x7.5},  
$C\sreq \cl\bigl(\conv (\cup_{w\in W}\real_{\geq 0} w\al)\bigr)=\ol{\mc{Z}}$.
\end{proof}

\subsection{}\label{x7.7} If $(W,S)$ is   affine or  dihedral of indefinite type, then $\ol{\mc{Z}}$ has either one or two extreme rays, which form a single $W$-orbit.  This behavior is exceptional amongst the infinite irreducible Coxeter groups. 

\begin{lem} Suppose that $(W,S)$ is  irreducible of indefinite type and  $\vert S\vert \geq 3$. Then \begin{num}
\item The $W$-orbit of any ray contained in $\ol{\mc{Z}}$ is infinite.
\item $\ol{\mc{Z}}$ is not a polyhedral cone.
\end{num}
\end{lem}
\begin{proof} Suppose that $0\neq\al\in \ol{\mc{Z}}$ and there are only 
finitely many distinct  rays in the $W$-orbit of  $\real_{\geq 0}\al$, say
$\real_{\geq 0} \al_{i}$ for $i=1,\ldots, n$.
Let $C:=\real_{\geq 0}\G$, where $\G=\set{\al_{1},\ldots, \al_{n}}$, be the 
polyhedral cone spanned by these rays.  Then $C$ has the properties 
listed in  Theorem \ref{x7.6}, so $C=\ol{\mc{Z}}$.  Let $W'$ be the 
pointwise stabilizer in $W$  of the set $\set{\real_{\geq 0}\al_{i}\mid i=1,\ldots, n}$ of  rays, so $W'$ is a normal subgroup of $W$  of finite 
index. Each $w\in W'$ has each element of $\G$ as an 
eigenvector. Since $\G$ spans $\real\G=\real{\mc{Z}}=\real\Pi$
and $W$ acts faithfully on $\Phi\seq \real \Pi$,  
 it  follows that $W'$ is abelian.  By Remark \ref{x4.5}(2), $(W,S)$ is  affine, a contradiction. Now (b) follows from (a) since if $\ol{\mc{Z}}$ is polyhedral, the $W$-orbit of an  extreme ray of $\ol{\mc{Z}}$ is contained in the (finite) set of extreme rays of $\ol{\mc{Z}}$.
\end{proof}
\subsection{} \label{x7.8}Let $\rext$ (resp., $\rexp$) denote the set of extreme (resp., exposed) rays of $\ol{\mc{Z}}$ and
$Z:=\ray(\mc{Z})$ denote the set of rays of  $\mc{Z}$, so the set of  rays of $\ol{\mc{Z}}$ is $\ray(\ol{\mc{Z}})= \ol{Z}$ (where the latter  closure is  taken in $\ray(V)$). Note that 
\begin{equation}
R_{0}\seq \ol{Z}\cap Q
\end{equation} by  Theorem \ref{x5.4} and Proposition \ref{x5.3}
and 
\begin{equation}
\rexp\seq \rext\seq\ol{\rexp}=\ol{\rext}\seq \ol{Z}
\end{equation} for general reasons (see Lemma \ref{xA.11}(b)).

\begin{prop}  Let $W$ be arbitrary. Then
\begin{num}
\item $\ol{\mc{Z}}$ is the set consisting of ($0$ and) all 
sums
$y_{1}+\ldots +y_{n}$ where  $n>0$ and  all $y_{i}\in \cup\rext$. 
\item $\mc{Q}\cap \ol{\mc{Z}}$ is the set consisting of ($0$ and) all 
sums
$y_{1}+\ldots +y_{n}$ where  $n>0$, all $y_{i}\in \cup\rext$  and 
$\mpair{y_{i},y_{j}}=0$ for $i,j=1,\ldots, n$.
\end{num}\end{prop}
\begin{proof}
Part (a) is a general property of the extreme rays of a closed salient cone  (see Lemma \ref{xA.11}(b))  and (b) follows from (a) and Lemma \ref{x5.9} 
\end{proof}

\subsection{}\label{x7.9}  For $\bt\in \Phi_{+}$, let $R_{\bt,0}:=\ol{R_{\bt,+}}\sm R_{\bt,+}\seq R_{0}$ where $R_{\bt,+}:=\mset{\real_{\geq 0}
\g\mid \g\in \Phi_{+}\cap W\bt}\seq R_{+}$. Since $R_{\bt,0}=\ol{R_{\bt,+}}\cap R_{0}$, $R_{\bt,0}$ is closed in $\ray(V)$.
 \begin{cor} Assume that $W$ is irreducible and indefinite, $\al\in \ol{\mc{Z}}\sm\set{0}$ and $\bt\in \Phi_{+}$.
\begin{num}\item
 $\ri(\mc{Z})=\ri(\conv(\cup E_{\al}))\seq\conv(\cup E_{\al}))\seq\ol{\mc{Z}}$.\item  If $\al\in \ri(\mc{Z})$, then $\conv(\cup E_{\al})\sm\set{0}=\ri(\mc{Z})$ and  $\mpair{\al,\al}<0$.  \item 
$\ol{\rext}\seq \ol{E_{\al}}\seq \ol{Z}$. \item $\ol{\rext}\seq R_{\bt,0}\seq R_{0}\seq \ol{Z}\cap Q$. \end{num}
\end{cor}
\begin{rem} (1) If $W$ is irreducible affine, then, using the explicit description of $\mc{Z}$,  (a)--(d) hold with the following change: in (b), $\mpair{\al,\al}=0$. 

(2) If $W$ is irreducible indefinite of rank at least three, then  $E_{\al}$ 
is infinite by  Lemma \ref{x7.7} and is contained in the compact set 
$\ray(\real_{\geq 0}\Pi)$, so $E_{\al}'\neq \eset$.  Let 
$0\neq \bt\in \ol{\mc{Z}}$ with $\real_{\geq 0}\bt\in E_{\al}'$. 
Since $E_{\al}'$ is closed and $W$-invariant,
one has $\ol{E_{\beta}}\seq E_{\al}'$. By applying (c) to $\bt$ instead of 
$\al$, one gets $\ol\rext\seq \ol{E_{\beta}}\seq E_{\al}'\seq \ol{Z}$, which 
improves (c) under these extra assumptions.

 (3)  From
   \cite[Example 2.16]{HLR}, one sees  that, even for the standard 
   root system of an irreducible finite rank Coxeter system, one may have $\rext\sneq R_{0}$.

(4) It will be shown in  \cite{DHR} that (for any irreducible$W$) $\ol{\rext}=R_{0}$.  This leads to 
improvements in several results of this subsection and additional results not discussed here. It has been asked by Hohlweg and Ripoll whether (for irreducible $W$) $R_{0}=\ol{Z}\cap Q$.
\end{rem}
\begin{proof} In (a), the equality  holds since $\ri(\mc{Z})=\ri(\ol{\mc{Z}})=\ri(\cl(\conv(\cup E_{\al})))=
  \ri(\conv(\cup E_{\al}))$, and the inclusions are trivial from the definitions. For (b), assume $\al\in \ri(\mc{Z})$.
  Then $\ri(\mc{Z})\seq \conv(\cup(E_{\al}))\sm\set{0}$ by (a).
  On the other hand, since $\cup E_{\al}$ is contained in the salient cone $\real_{\geq 0}\Pi$ and  \begin{equation*}
  \cup(E_{\al})\sm\set{0}=\mset{\lambda w(\al)\mid \lambda\in \real_{>0}, w\in W}\seq \ri(\mc{Z}),
  \end{equation*} one has $\conv(\cup(E_{\al}))\sm\set{0}=\conv(\cup(E_{\al})\sm\set{0})\seq \ri(\mc{Z})$.
  If also $W$ is indefinite, then there is some $\al'\in \mc{K}$ of the form $\al'=\sum_{\g\in \Pi}c_{\g}\g$ with all $c_{\g}>0$ and $\mpair{\al',\Pi}\seq \real_{<0}$. Then $\al'\in \ri(\mc{K})\seq \ri(\mc{Z})$ and
  $\mpair{\al',\al'}<0$. Using  Lemma \ref{x5.9}, one has
  $\mpair{\g,\g}<0$ for all $\g\in \conv(\cup(E_{\al'})\sm\set{0})= \ri(\mc{Z})$ and so $\mpair{\al,\al}<0$. This proves (b).

  Since $\conv(\cup\overline{E_{\al}})=\ol{\mc{Z} }$ by Proposition \ref{x7.5}, it follows from Lemma \ref{xA.11}(b) that  $\overline{E_{\al}}\sreq \rext$. Taking closures gives  
  $\overline{Z}\sreq \overline{E_{\al}}\sreq \ol{\rext}$, proving (c).  
    In (d),  $R_{\bt,0}\seq R_{0}\seq \ol{Z}\cap Q$ has already been 
    noted. Observe that since $W\bt$ is infinite (by  Proposition 
    \ref{x1.18}(a)), it follows that
  $R_{\bt,+}$ has a limit point, which is necessarily in  $R_{0}$ and  so 
  not in $R_{\bt,+}$.  Hence there is some ray $\real_{\geq 0}\g
\in R_{\bt,0}$. Taking $\al=\g$ in (c), shows (in terms of the natural 
$W$-action on $\ray(V)$) that
$\ol{\rext}\seq \ol{E_{\al}}\seq \ol{W R_{\bt,0}}=R_{\bt,0}$ and proves 
(d).\end{proof}

\subsection{} \label{x7.10}   We conclude this section 
with some    miscellaneous properties  of $\ol{\mc{Z}}$ related to its 
facial structure, about which much less is known than for $\mc{Z}$.
These facts  have interesting consequences which we do not go into here (but see \S\ref{x9} for some applications to universal  Coxeter groups).

\begin{prop} Let   $\al\in \ol{\mc{Z}}$. Denote the minimal face of $\ol{\mc{Z}}$ containing $\al$ by $\ol{\mc{Z}}_{\al}$. 
\begin{num}
\item  $\ol{\mc{Z}}\cap \al^{\perp}$ is a face of $\ol{\mc{Z}}$.
\item If   $\al$ is isotropic,  then $\ol{\mc{Z}}_{\al} $ is a totally isotropic face of $\ol{\mc{Z}}$ and 
 $ \al\in \ol{\mc{Z}}_{\al}\seq  \ol{\mc{Z}}\cap \al^{\perp}$. 
\item If $(w_{n})_{n\in \Nat}$ is any sequence in $W$, then $F:=\mset{\bt\in \ol{\mc{Z}}\mid \lim_{n\to\infty}w_{n}\bt=0}$ is a totally  isotropic  face of $\ol{\mc{Z}}$. 
\item Let $W'$ be an irreducible  reflection subgroup such that $\mpair{\Pi_{W'},\al}\seq\real_{\geq 0}$ but $\Pi_{W'}\not \seq \al^{\perp}$. Then  $W'$ is finite.
\end{num}\end{prop}
\begin{proof}
In the terminology  of Appendix A, $(\ol{\mc{Z}},-\ol{\mc{Z}})$  is  
semidual pair of stable cones with respect to $\mpair{-,-}$ and 
$\ol{\mc{Z}}\cap \al^{\perp}$ is  an exposed face of $\mc{Z}$ with 
respect to  this semidual pair.  In particular, (a) holds. Now assume that $\al$ is isotropic. It is trivial that $ \al\in \ol{\mc{Z}}_{\al}\seq  \ol{\mc{Z}}\cap \al^{\perp}$.  Recall from the Appendix  that \begin{equation*}
\ol{\mc{Z}}_{\al}=\mset{x\in \ol{\mc{Z}}\mid tx+(1-t)y=\al\text{ \rm for some $y\in \ol{\mc{Z}}$ and $0<t<1$}}.
\end{equation*}  To show $\ol{\mc{Z}}_{\al}$ is totally isotropic, it suffices to show that 
$\mpair{x,x}=0$ for all $x\in \ol{\mc{Z}}_{\al}$. Choose $y\in \ol{\mc{Z}}_{\al}$ and $0<t<1$ with $\al=tx+(1-t)y$.  Since $\al$ is isotropic, Lemma \ref{x5.9} implies that $tx$ is isotropic and  hence  $x$ is isotropic as required for (b).

To prove (c), note first that $F$ is obviously a pointed cone contained in 
$\ol{\mc{Z}}$.  To show $F$ is a face, suppose that $x,y\in \ol{\mc{Z}}$ 
with $x+y=\bt\in F$.  Since 
$\mpair{\rho,w_{n}x}+\mpair{\rho,w_{n}(y)}=\mpair{\rho,w_{n}\beta}\to 0$ 
as $n\to \infty$ with $\mpair{\rho,w_{n}x}, \mpair{\rho,w_{n}(y)}\geq 0$, it 
follows that  $\mpair{\rho,w_{n}x}\to 0$ and  $\mpair{\rho,w_{n}(y)}\to 0$.  
By Lemma \ref{xA.11},
$w_{n}x\to 0$ and so $x\in F$. To see $F$ is totally isotropic, it suffices 
to show it contains no non-isotropic vector. This follows from  Lemma 
\ref{x6.1}(f).

Finally, we prove (d). Assume to the contrary that $W'$ is infinite.  Let $\Delta:=\Pi_{W'}\cap \al^{\perp}\sneq \Pi_{W'}$. 
Lemma \ref{x1.10} gives $\stab_{W'}(\al)=W_{\Delta}$, since $\al\in \mc{C}_{W'}$. Hence $\vert W'\al\vert =\vert W'/W_{\Delta}\vert$ is infinite by Lemma \ref{x1.20}. By \eqref{x4.4.1}, $\mset{\mpair{\al-w\al,\rho}\mid w\in W'}$   is not bounded above. But for any $w\in W'$,
 $w\al\in \ol{\mc{Z}}\seq\real_{\geq 0}\Pi$. Hence $\mpair{\al-w\al,\rho}=\mpair{\al,\rho}-\mpair{w\al,\rho}\leq \mpair{\al,\rho}$, a contradiction which proves (d). 
\end{proof} 

\subsection{} \label{x7.11} This subsection formulates a  result on limit points of 
$W$-orbits of  ordered tuples of rays spanned by roots or in the 
imaginary cone. Other related  results    can be obtained by similar  arguments. 

Assume  $\Phi$ is irreducible of indefinite type, $m\in \Nat_{\geq 1}$
and  that  $\al_{1},\ldots, \al_{m}\in \Phi\cup({\mc{Z}}\sm\set{0})$ are 
all non-isotropic (e.g. $\al_{i}\in \Phi\cup\ri({\mc{Z}})$ for all $i$).
Consider the tuple $\al:=(\real_{\geq 0}\al_{1},
  \ldots, \real_{\geq 0}\al_{m})$ of rays as a point of $(\ray(V))^{m}=\ray(V)\times\ldots\times \ray(V)$, which we equip with the product topology (which makes it compact) and  diagonal $W$-action.
    The $W$-orbit of $\al$ in $(\ray(V))^{m}$ is infinite since the $W$-orbit of each ray $\real_{\geq 0}\al_{i}$ is infinite
  (by Proposition \ref{x1.18}, Lemma \ref{x7.7} and an easy direct check if $W$ is dihedral).
 Let 
$(\bt_{1},\ldots, \bt_{m})\in V^{m}$ be such that
 $\bt:=(\real_{\geq 0}\bt_{1},\ldots, \real_{\geq 0}\bt_{m})$ is a limit point
  of the $W$-orbit $W\al$ in $(\ray(V))^{m}$. Replacing some $\al_{i}\in \Phi$ by their 
  negatives  and $\bt_{i}$ by suitable  scalar multiples if necessary, we suppose without loss of generality by 
  \ref{x5.3}--\ref{x5.4}   that  $\bt_{i}\in \Phi_{+}\cup(\ol{\mc{Z}}\sm\set{0})$ for all $i$. Note that $\bt_{i}\in \ol{\mc{Z}}$ if $\al_{i}\in \mc{Z}$, and the rays  $\real_{\geq 0}\bt_{i}$ need not be pairwise  distinct even if the $\real_{\geq 0}\al_{i}$ are.
  
  \begin{prop} Let assumptions be as above. For some $i$ with $1\leq i\leq m$,  $\bt_{i}$ is isotropic.  
   For any $i$ with $\bt_{i}$ isotropic,  $\bt_{i}\in \ol{\mc{Z}}\cap \mc{Q}$    and  $H:=\real\Pi\cap \bt_{i}^{\perp}$  is a supporting   hyperplane of $\ol{\mc{Z}}$ in $\real{\mc{Z}}$ such that $H$ contains all of    $\set{\bt_{1},\ldots, \bt_{m}}$.  In particular, there is a proper face of $\ol{\mc{Z}}$ containing all the $\beta_{j}$ which are isotropic. 
   \end{prop}
   \begin{rem} A more symmetric choice of supporting hyperplane with these properties  would be $\real\Pi\cap \g^{\perp}$ where $\g\in \ol{\mc{Z}}\cap\mc{Q}$ is the sum of all  $\beta_{i}$ which are isotropic.\end{rem}
\begin{proof} There is  some sequence $(w_{n})_{n\in \Nat}$  in $W$ such that
the tuples $w_{n}\al$ for $n\in \Nat$ are  pairwise distinct and 
$w_{n}\al\to \bt$ as $n\to\infty$. Passing to a subsequence of 
$(w_{n})$ if necessary, we may suppose firstly that 
$w_{n}\al_{i}\in \real_{\geq 0}\Pi$ for all $i$ and secondly, by  \eqref{x4.4.1}  
and Lemma \ref{x4.3}(d), that for some $i$ with $1\leq i\leq m$,
  $\mpair{\rho,w_{n}\al_{i}}\to\infty$ as $n\to\infty$. Now there is $\epsilon>0$ such that
   $\mpair{\rho, w_{n}\al_{j}}>\epsilon $  for all $j=1,\ldots, m$ and all $n$ (for instance, using Lemma \ref{x1.21}(b) for $j$ with $\al_{j}\in \Phi$ and  Lemma  \ref{x6.1}(f) for other $j$).
   Note that for any $j$, $k$,  $\mpair{\rho,w_{n}\al_{j}}^{-1}w_{n}\al_{j}\to \mpair{\rho, \bt_{j}}^{-1}\bt_{j}$ in $V$  as $n\to\infty$ and 
   hence, since  $\mpair{w_{n}\al_{j},w_{n}\al_{k}}=\mpair{\al_{j},\al_{k}}$, 
   \begin{equation*}
   \mpair{\rho,w_{n}\al_{j}}^{-1}\mpair{\rho,w_{n}\al_{k}}^{-1}
   \mpair{\al_{j},\al_{k}}\to \mpair{\rho, \bt_{j}}^{-1}\mpair{\rho, \bt_{k}}^{-1} 
   \mpair{\bt_{j},\bt_{k}}
   \end{equation*} in $\real$ as $n\to\infty$.
   Since the $\al_{j}$ are assumed non-isotropic, 
   taking $k=j$ shows that $\bt_{j}$ is isotropic if and only if $\mpair{\rho,w_{n}\al_{j}}\to\infty$ as $n\to \infty$ (in particular, there exists an isotropic  $\bt_{i}$ from above).  Also, for any isotropic $\bt_{j}$,     
$\mpair{\bt_{j},\bt_{k}}=0$ for all $k$.
Now fix $i$ such that $\bt_{i}$ is isotropic. Obviously 
$\bt_{i}\in \ol{\mc{Z}}\cap\mc{Q}$ since $\bt_{i}\not\in \Phi_{+}$. 
By Lemma \ref{x5.9}, $\mpair{\bt_{i},\mc{Z}}\seq \real_{\leq 0}$ and by Lemma \ref{x7.1}, $\real\mc{Z}=\real\Pi\not \seq\bt_{i}^{\perp}$ (recall  that $\Pi$ is irreducible  of indefinite type). Hence $H:=\real\Pi\cap\bt_{i}^{\perp}$ is a supporting hyperplane of $\ol{\mc{Z}}$ at  $\bt_{i}$ as required.
\end{proof}
\subsection{} \label{x7.12} For certain special  irreducible $\Phi$ (for instance,  $\Phi$  affine or $\Phi$ weakly hyperbolic as defined in \S\ref{x9}),  the non-zero isotropic faces of $\ol{\mc{Z}}$ are one-dimensional (they are rays). That this is not the case in general is shown by the following example. 
\begin{exmp}  Consider the based root system $(\Phi, \Pi)$ in $(V,\mpair{-,-})$  in  \cite[Example 5.8]{HLR}. It has linearly independent simple 
roots  $\Pi=\set{\al,\bt,\g,\delta,\epsilon}$, where   $\mpair{\al,\bt}=\mpair{\delta,\epsilon}=-1$, 
$\mpair{\bt,\g}=\mpair{\g,\delta}=-1/2$ and all  inner products of  other distinct simple roots are
$0$.  Assume $V=\real \Pi$.  The  form $\mpair{-,-}$ has signature $(3,1,1)$ and its radical  is spanned by $\al+\bt-\g-\delta$. Note  $z:=\al+\bt+\delta+\epsilon\in\mc{K}\seq \mc{Z}$ is isotropic. Let $J=\set{s_{\alpha},s_{\bt}}$, $K:=\set{s_{\delta},s_{\epsilon}}$ and $I:=J\cup K$ so $I,J,K\seq S$ and the components of $I$ are $J$, $K$, which are irreducible affine. Then     since $z^{\perp}=\real \Pi_{I}$ and $I,J,K$ are facial,    Theorem \ref{x5.1}  gives
\begin{equation*} \ol{\mc{Z}}\cap  z^{\perp}=\ol{\mc{Z}}\cap \real\Pi_{I}=\ol{\mc{Z}}_{I}=\ol{\mc{Z}}_{J}+\ol{\mc{Z}}_{K}=\real_{\geq 0}(\al+\bt)+\real_{\geq 0} (\delta+\epsilon).
\end{equation*}   So $F:=\ol{\mc{Z}}\cap  z^{\perp}=\real_{\geq 0}\set{\al+\bt, \delta+\epsilon}\seq\mc{K}\seq  \mc{Z}$ is a two-dimensional 
totally isotropic face of $\ol{\mc{Z}}$. Clearly, $\real F$ is a maximal totally isotropic subspace of $(\real \Pi,\mpair{-,-})$. We claim that every ray  in $F$  is a limit ray of  rays spanned by positive roots. 

To see this, it  suffices to show that for any fixed $t\in \real$ with $0\leq t\leq 1$, 
the ray $\real_{\geq 0}(t(\al+\bt)+(1-t)(\delta+\epsilon))$ is a limit ray of   rays spanned by positive roots. For  any $k,l\in \Nat$, set $\al_{2k+1}:=(s_{\al}s_{\bt})^{k}(\al):=(2k+1)\al +2k\bt\in \Phi_{+}$ and
$\epsilon_{2l+1}:=(s_{\epsilon}s_{\delta})^{l}(\epsilon)=(2l+1)\epsilon+2l \delta\in \Phi_{+}$.
Since $\mpair{\al_{2k+1},\epsilon_{2l+1}}=0$, one has 
 \begin{equation*}\begin{split}
  s_{\al_{2k+1}}s_{\epsilon_{2l+1}}(\g)&=\g-\mpair{\g,\ck\al_{2k+1}}\al_{2k+1}-\mpair{\g,\ck\epsilon_{2l+1}}\epsilon_{2l+1}=\g+2k\al_{2k+1}+2l\epsilon_{2l+1}\\&=
  \g+2k(2k+1)\al+4k^{2}\bt+4l^{2}\delta+2l(2l+1)\epsilon\in \Phi_{+}.
  \end{split}
 \end{equation*}
  One may choose sequences $(k_{n}),(l_{n})$ in $\Nat$ with $k_{n},l_{n}\to \infty$  and $\frac{k_{n}^{2}}{k_{n}^{2}+l_{n}^{2}}\to t$ as $n\to \infty$.
 It is easy to see from above that  as $n\to \infty$,
 \begin{equation*}
 \frac{1}{4(k_{n}^{2}+l_{n}^{2})}s_{\al_{2k_{n}+1}}s_{\epsilon_{2l_{n}+1}}(\g)\to t(\al+\bt)+(1-t)(\delta+\epsilon)
 \end{equation*}  from which  the claim follows readily.  
 
Now   define the quotient  space
$(V',\mpair{-,-}')$ of $(V,\mpair{-,-})$  by its radical, 
where $V':=V/\real(\al+\bt-\g-\delta)$. Let $L\colon V\to V'$ be the natural  
map.  Then  $(\Phi,\Pi)$ on $(V,\mpair{-,-})$  is a canonical lift of 
a based root system $(\Phi',\Pi'):=(L(\Phi),L(\Pi))$  for 
$(V',\mpair{-,-}')$, with $L$ as the associated canonical map
(see \ref{x1.4}  but note that the roles of $\Phi$ and $\Phi'$ are swapped  here in relation to the notation there).  From Proposition \ref{x5.10}, $L$ induces a 
surjective map  $R_{W,0}\to R_{W',0}$ on corresponding sets of limit 
rays of positive roots.  This induced map sends all limit rays $\real_{\geq 0}(t(\al+\bt)+(1-t)(\g+\delta)) $ with $0\leq t\leq 1$  to the same limit ray $\real_{\geq 0}(L(\alpha+\bt))$ and in particular the induced map is not  bijective. Hence the set of limit rays of positive roots  can depend
not only on the  Coxeter group $W$ but also on the chosen root system for it, even for irreducible $W$.
 \end{exmp}
 
\section{Supports}
\label{x8}
\subsection{}  
We say \label{x8.1} $\Delta$ is a support of $v\in \real_{\geq 0}\Pi$ (with respect to $\Pi$) if 
$\Delta\subseteq \Pi$  and there is an expression
$v=\sum_{\alpha\in \Delta}c_{\alpha}\alpha$ where all $c_{\alpha}>0$.
The only support of $0$ is $\emptyset$, but if $\Pi$ is linearly dependent, an arbitrary  element $v\in \real_{\geq 0}\Pi$ may have many supports.  Note that for any $\D\seq \Pi$, $\D$ is a set of representatives for the extreme rays of the polyhedral cone $\real_{\geq 0}\D$, which has as its relative interior the set of elements 
of $\real_{\geq 0}\Pi$ with $\D$ as a support.

Define $\mcKc$ to be the subset of $\mc{K}$ of all non-zero elements with some connected support,
and $\mcZc:=\cup_{w\in W}w\mcKc$. 
\begin{lem} Let $\alpha,\gamma \in \real_{\geq 0}\Pi$.
 \begin{num} 
 \item The set of supports of $\alpha$ is closed under finite unions.
 \item If $\alpha$, $\gamma$ have connected supports $\Delta$, $\Delta'$  and $\Delta$, $\Delta'$ are not separated (e.g.  $\mpair{\alpha,\gamma}\neq 0$), then  for any $c,d\in \real_{>0}$, $c\alpha+d\gamma$ has a connected support $\Delta \cup \Delta'$.
\item If $\alpha$ has  supports  $\Delta$ and $\Delta'$, and $\Delta$ and $\Delta'$ are separated, then $\alpha$ is isotropic.
\item If $\Delta$ and $\Delta'$  are connected supports of $v$ and $v$ is non-isotropic, then
$\Delta\cup\Delta'$ is a connected support of $v$.
\item If $\alpha$ is non-isotropic with some connected support, then the set of its connected supports  has an inclusion-maximal element.
\item  Let $\Delta\subseteq \Pi$ and $\Delta':=\real\Delta\cap \Pi$. Then each $\beta\in \Delta'\setminus \Delta$ is joined to some vertex of $\Delta$.
The number of connected components of $\Delta'$ is no more than the number of connected components of $\Delta$;
in particular, if $\Delta$ is connected, then $\Delta'$ is connected.
\item If $\Delta$ is a support of $\alpha$, and  $\Delta\seq \Delta'\seq \real\Delta\cap \Pi$, then $\Delta'$  is a support of $\alpha$.
\item Any $\beta\in \Phi_{+}$ has a connected support.
\item Any element of $\mcZc$ has a connected support. More precisely, if $\Delta$ is a connected support of $k\in \mcKc$, then for any $w\in W$, $wk$ has a connected support $\Delta'$ containing $\Delta$.
\item Let $W'$ be a finitely generated  reflection subgroup of $W$, and  $\beta\in \real_{\geq 0}\Pi_{W'}$
have a connected support with respect to $\Pi_{W'}$. Then $\beta\in \real_{\geq 0}\Pi$  and $\beta$ has  a connected support with respect
to $\Pi$.
\item Any element of $\mc{K}$ (resp., $\mc{Z}$) can be written as a sum of finitely many pairwise orthogonal elements of $\mcKc$ (resp., $\mcZc$).
 \end{num}\end{lem}
 \begin{rem} (1) Except for (e) and (g), which  require just  finiteness of $\Pi$, the above do not require the assumptions \ref{x4.1}(i)--(iii). 
 
 (2)  One might study along similar lines the properties of  the support polytopes
 $\mset{(c_{\alpha})_{\alpha\in \Pi}\mid \text{\rm all } c_{\alpha}\in \real_{\geq 0} , v=\sum_{\alpha}c_{\alpha}\alpha}$ of elements $v\in \real_{\geq 0}\Pi$. \end{rem} 
 \begin{proof} We omit the simple proofs of  (a)--(f). For (g), note that $\alpha$ is in the relative interior of  $\real_{\geq 0}\Delta$ and hence in that of $\real_{\geq 0}\Delta'$.
  We prove (h) by induction on $l(s_{\beta})$. If $l(\beta)=1$, then $\set{\beta}$ is a connected support of $\beta$. Otherwise, there is some $\gamma\in \Pi$ with $l(s_{\gamma}s_{\beta}s_{\gamma})=l(s_{\beta})-2$. Set $\delta:=s_{\gamma}(\beta)\in \Phi_{+}$, which has connected support  by induction and satisfies  $c:=-\mpair{\delta,\ck\gamma}>0$ by \cite[Lemma 3.4]{DySd}. By (c),
$\beta=s_{\gamma}(\delta)=\delta+c\gamma$ has a connected support. 

For (i), let $k\in \mcKc$ with connected support $\Delta$. We show $z\in Wk$ has a connected support $\Delta'$ containing $\Delta$ by induction on the minimum value  $n$ of $l(w)$ such that $z=wk$ with $w\in W$. If $n=0$, the result is trivial. Otherwise,
    write $w=s_{\alpha_{n}}\ldots s_{\alpha_{1}}$ with $\alpha \in \Pi$ and $n\geq 1$ minimal.
     By induction, $z':=s_{\alpha_{n-1}}\ldots s_{\alpha_{1}}k$ has a connected support
     $\Delta'\supseteq \Delta$.
We claim that  \[c:=\mpair{\ck\alpha_{n},z'}=\mpair{s_{\alpha_{1}}\cdots s_{\alpha _{n-1}}(\ck\alpha_{n}),k}<0.\]  For if 
$c>0$, then $s_{\alpha_{1}}\cdots s_{\alpha _{n-1}}(\alpha_{n})\in \Phi_{-}$ since $k\in -\mc{C}$, and 
$s_{\alpha_{n}}\ldots s_{\alpha_{1}}$ couldn't be reduced by Lemma \ref{x1.13}(b); if $c=0$, then
$wk=s_{\alpha_{n-1}}\cdots s_{\alpha_{1}}k$ contrary to minimality of $n$.
Hence $wk=z'-c\alpha_{n}$ has a connected support $\Delta'\cup\set{\alpha_{n}}$ by (c) since $-c>0$. 
This proves (i).

Next, we prove (j). Write $\beta=\sum_{\gamma\in \Delta}c_{\gamma}\gamma$
where $\Delta\subseteq \Pi_{W'}$ is connected and each $c_{\gamma}> 0$. Also write $\Delta=\set{\delta_{1},\ldots, \delta_{n}}$ where for each $i$ with $1\leq i\leq n$, there is some
$j<i$ with $\mpair{\delta_{j},\delta_{i}}\neq 0$.  By (h), each $\delta_{i}$ has some connected support $\Delta_{i}$. By  repeated application of (c), it follows that $\Delta_{i}':=\Delta_{1}\cup \cdots \cup \Delta_{i}$ is connected  for $i=1,2,\ldots, n$. But $\Delta_{n}'$ is clearly a support 
for $\beta$, and (j) follows.

To prove (k), write $\al\in \mc{K}$ as $\al=\sum_{\bt\in \Delta} c_{\bt}\bt$ for some $\Delta\seq \Pi$ where all $c_{\bt}>0$. Let $\Delta_{1},\ldots, \Delta_{m}$ be the components of  $\Delta$. Let
$\al_{i}:=\sum_{\bt\in \Delta_{i}}c_{\bt}\bt\in \mc{K}$, so $\al=\al_{1}+\ldots +\al_{m}$ and $\mpair{\al_{i},\al_{j}}=0$ if $i\neq j$. Then $\mpair{\al_{i},\bt}=\mpair{\al,\bt}\leq 0$ for all 
$\bt\in \Delta_{i}$ and $\mpair{\al_{i},\Pi\sm\Delta_{i}}\seq \real_{\leq 0}$ so $\al_{i}\in \mcKc$.  Now if $\g\in \mc{Z}$, choose $w\in W$ so  
$\al:=w^{-1}\g\in \mc{K}$. Write $\al=\al_{1}+\ldots+\al_{m}$ as above.
Then $\g=\g_{1}+\ldots+\g_{m}$ where $\g_{i}:=w\al_{i}\in \mcZc$ are pairwise orthogonal as required.
 \end{proof}
 \subsection{}\label{x8.2} By Lemma \ref{x8.1}(a), any $\alpha\in \real_{\geq 0}\Pi$ has a maximum (under inclusion) support $\Delta$. It is the set of representatives (in $\Pi$) of the extreme rays of the inclusion minimal element of the set of faces of the polyhedral cone
 $\real_{\geq 0 }\Pi$ containing  $\alpha$. We call $\Delta$ the facial support of $\alpha$.

  The following simple   examples  illustrate features of the above notion of supports if
$\Pi$ is not linearly independent.

\begin{exmp} Suppose that $\Pi=\set{\alpha,\beta,\gamma,\delta}$  has irreducible affine components $\set{\alpha,\beta}$ and $\set{\gamma,\delta}$ and that the space of linear relations amongst
elements of $\Pi$ is spanned over $\real$ by the single relation $\alpha+\beta-\gamma-\delta=0$. 

The non-isotropic element  (root) $v=s_{\alpha}(\beta)=2\alpha+\beta$ has a 
connected support $\set{\alpha,\beta}$, a disconnected support
$\set{\alpha,\gamma,\delta}$ (since  $v=\alpha+\gamma+\delta$) and disconnected  facial support $\Pi$ (since $v=\frac{3}{2}\alpha+\frac{1}{2}\beta+\frac{1}{2}\gamma+\frac{1}{2}\delta$). 

The isotropic element  $v'=\alpha+\beta=\gamma+\delta$ spans the 
isotropic ray of both $\set{\alpha,\beta}$ and $\set{\gamma,\delta}$, 
which are its  connected supports (and are separated from one 
another), and it has disconnected  facial  support $\Pi$. \end{exmp}

\subsection{} \label{x8.3} In this subsection, we do not use or assume the conditions \ref{x4.1}(i)--(iii).
More generally than in the above example,  let $\Phi$ be any root system
 with positive roots $\Phi_{+}$ and simple roots $\Pi$
in $V$ as in \ref{x1.3}.
Suppose that $U$ is a subspace of the radical of $\mpair{-,-}$ such that no non-zero
non-negative combination of simple roots is in $U$. Then there is an induced bilinear form 
on $V/U$ and the image of $\Phi$ in the quotient  $ V/U$ is a root system with simple roots given by    the image $\Pi'$ of $\Pi$ in $V/U$, and with Coxeter system $(W',S')$ naturally identified with that associated with $\Pi$.  The canonical epimorphism $p\colon V\rightarrow V/U$
induces a bijection of  corresponding root systems, positive root systems and systems of simple roots (cf. \cite[6.1]{K}).
We remark that any root system (regarded as a subset of its linear 
span, with the restriction of the ambient bilinear form) arises in this way 
by this process of  dividing out a suitable  subspace of the radical of the 
ambient vector space of  some  root system    with linearly independent 
simple roots (see \ref{x1.4}). We also remark that this 
construction provides a ``canonical support'' of any positive root 
$\alpha$ in $V/U$, namely the image under $p$ of the  support  of the 
unique positive  root $\tilde \alpha\in V$ with $p(\tilde \alpha)=\alpha$. 
The canonical support need not coincide with the facial support; for 
example, in Example \ref{x8.2},  $v$ has canonical support $\set{\alpha,\beta}$.

In any case, one may construct similar examples to Example \ref{x8.2}  involving $W$ with infinite type components by taking suitable $\Pi$ and subspaces of the radical. 
\subsection{} \label{x8.4} Return to the standing assumptions \ref{x4.1}(i)--(iii). 
The next result concerns   supports of elements of $\mc{K}$.
 \begin{lem}  Let $0\neq v\in \mc{K}$.
Consider any expression $v=\sum_{\alpha\in \Delta}c_{\alpha}\alpha$ with  $\Delta\subseteq \Pi$ and with all 
$c_{\alpha}>0$ (e.g. $\Delta$ could be  the facial support of $v$). Let $\Delta_{1},\ldots, \Delta_{n}$ denote the irreducible components of 
$\Delta$ and set $v_{i}:=\sum_{\alpha\in \Delta_{i}}
c_{\alpha}\alpha$.
\begin{num}\item We have $v=v_{1}+\ldots + v_{n}$ where $\mpair{v_{i}, v_{j}}=0$ if $i\neq j$ and each $v_{i}\in \mcKc$.
\item Either $\Delta_{i}$ is of indefinite type (in which  case $\mpair{v,v_{i}}=\mpair{v_{i},v_{i}}<0$ and there is some $\alpha\in \Delta_{i}$ with $\mpair{v_{i},\alpha}<0$) or of affine type (in which case $\mpair{v,v_{i}}=\mpair{v_{i},v_{i}}=0$ and $v_{i}$ spans the isotropic ray of $\Delta_{i}$).
\item  $v$ is isotropic if and only if  all $\Delta_{i}$ are of affine type.
\item (cf.  \cite[Exercise 5.9]{Kac}) For each $i$ for which $\Delta_{i}$  is of indefinite type, the connected components of $\mset{\alpha\in \Delta_{i}\mid \mpair{v_{i},\alpha}=0}$ are all of finite type.
 \item If $v$ is isotropic, then $v$ has a connected support if and only if   $v$ spans the isotropic ray  of some  irreducible (necessarily affine)
component of the facial support of $v$. In that case, any connected support of $v$ is a connected component of the facial support of $v$.
\item If $v$ is non-isotropic then it  has a connected support if and only if   there is a unique component $\Gamma$ of the facial support  of $v$ which is of indefinite type. Then any connected support of $v$ is contained in $\Gamma$.\end{num}\end{lem}
\begin{proof}   Part (a) is from the proof of \ref{x8.1}(a). We  prove (b). We have $\mpair{v,\alpha}=\mpair{v_{i},\alpha}\leq 0$
for all  $\alpha\in \Delta_{i}$. Since $\Delta_{i}$ is connected, it is of either   affine or indefinite type; it is of affine  type if and only if   $\mpair{v_{i},\alpha}= 0$ for all $\alpha\in \Delta$, in which case $v_{i}$ is a representative of the isotropic ray of $\Delta_{i}$.
We have \[\mpair{v_{i},v_{i}}=\sum_{\alpha\in \Delta_{i}}c_{\alpha }\mpair{v_{i},\alpha}\leq 0\]
with equality if and only if  $\Delta_{i}$ is of affine type. This proves (b).

By (a) and (b),  $ \mpair{v,v}=\sum_{i}\mpair{v_{i},v_{i}}\leq 0$, with equality if and only if  all $\Delta_{i}$ are affine, proving (c). 

Part (d) may be deduced from   Proposition \ref{x7.10}(d). We provide instead the following alternative argument. Let $\Gamma$ be a connected component of  $\mset{\alpha\in \Delta_{i}\mid \mpair{v_{i},\alpha}=0}$.   By (b), we have $\Gamma\subsetneq \Delta_{i}$.
Let $v_{\Gamma}:= \sum_{\gamma\in \Gamma} c_{\gamma}\gamma$. Then   for $\alpha\in \Gamma$, $\mpair{v_{i},\alpha}=0$ so
\[\mpair{v_{\Gamma},\alpha}=\sum_{\gamma\in \Gamma} c_{\gamma}\mpair{ \gamma, \alpha}=-\sum_{\gamma\in \Delta_{i}\setminus \Gamma }c_{\gamma} \mpair{\gamma, \alpha}\geq 0.\]
Moreover, since $\Delta_{i}$ is connected, there is some $\alpha\in \Gamma$ which is joined to some $\gamma\in 
 \Delta_{i}\setminus \Gamma $, and then $\mpair{v_{\Gamma},\alpha}>0$. By the classification
 \cite{Kac} of matrices (see \ref{x4.5}), it follows that $\Gamma$ is of finite type.

For the proofs of (e)--(f),  choose the expression for $v$ so $\Delta$ is the facial support of $v$. We may apply (a)--(c)  above to both the facial support  $\Delta$ of $v$, and also some
arbitrary (possibly connected) support $\Delta'$ of $v$.  Necessarily, we have $\Delta'\subseteq \Delta$. Each irreducible component   $\Delta'_{j}$ of $\Delta'$ , for $j=1,\ldots, m$, is contained in some irreducible component $\Delta_{i_{j}}$
of $\Delta$.  Write $v=\sum_{j}v_{j}'$ where $v_{j}'$ has support  $ \Delta_{j}'$.

For (e),  suppose $v$ is isotropic.
  Then each  $\Delta_{j}'=\Delta_{i_{j}}$, since
 $\Delta_{j}'\seq \Delta_{i_{j}}$ are both irreducible affine. In particular,   if $\Delta' $ is connected, then $\Delta'=\Delta'_{1}=\Delta_{i_{1}}$ and by (b)--(c), $v$ spans the isotropic ray of $\Delta'=\Delta_{i_{1}}$. Conversely, if $v$ spans the isotropic ray of $\Delta_{i}$, then $\Delta_{i}$ is a connected support of $v$. This proves (e).

Now for (f),  suppose that $v$ is non-isotropic. If $v_{i}$ is non-isotropic, then
 $i_{j}=i$ for some $i$, since otherwise $\mpair{v_{i},v_{i}}=\mpair{v,v_{i}}=\sum_{j}\mpair{v_{j}',v_{i}}=0$. It follows that the number of  components $\Delta'_{j}$ of $\Delta'$ is at least equal  to the number of non-affine components
$\Delta_{i}$ of $\Delta$. If we take $\Delta'=\Delta_{1}'$ to be any connected support of $v$, it cannot be of affine type by (b).  It follows that there is exactly one indefinite component $\Delta_{i_{1}}$ of $\Delta$, and $\Delta'\subseteq \Delta_{i_{1}}$, proving (f).
\end{proof}

\subsection{} \label{x8.5} Here we supplement the information on facial subsets of $\Pi$ given by Lemma \ref{x2.10} by additional facts under  the assumptions \ref{x4.1}(i)--(iii).
\begin{lem}\begin{num}
\item (\cite[Theorem 4, 1)]{V}, \cite[2.2.4]{K}).
Any  subset $I$ of $S$ such that all irreducible components of $\Pi_{I}$ are of finite type is facial.
\item (\cite[Theorem 4, 2)]{V}).
 Let $I$ be a subset of $S$ such that $I^{\perp}=\emptyset$ and every component of $\Pi_{I}$ is of affine type. Then $I$ is facial.
\item Let $I$ be a facial subset of $S$. Let $M$ be  a subset of $I$ such that $\Pi_{M}$ is a union of irreducible components of $\Pi_{I}$ and $\Pi_{M} $ contains all the  irreducible affine components of $\Pi_{I}$. Then $M$ is a facial subset of $S$.
\item Let $I\seq S$. Then the facial closure of $W_{I}$ coincides with the standard facial closure of $W_{I}$.
\end{num}\end{lem}
\begin{proof} Note that $\emptyset =\Pi\cap \rho^{\perp}$ is facial in $\Pi$, so $\emptyset$ is facial in $S$ and  (a) follows from \ref{x2.10}(b).   In (b), let $I_{1},\ldots, I_{p}$ be the irreducible components
of $I$ (all of affine type), and  write $\mc{K}_{I_{i}}=\real_{\geq 0}\delta_{i}$.
Then $z:=-(\delta_{1}+\cdots +\delta_{p})\in -\mc{K}_{I}\seq -\mc{K}\seq \mc{C}$ and
$z^{\perp}\cap \Pi=\Pi_{I}$ since $z\perp \Pi_{I}$, $\Pi_{I}$ is a support of $-z$ and $I^{\perp}=\eset$.
This proves  (b).

Now we prove (c) (compare Proposition  \ref{x2.13}).
Let $v_{J}=\sum_{\alpha\in \Phi_{J,+}}\alpha$ where $J\subseteq I$ is such that
 $\Pi_{J}$ is the union of all  finite type  irreducible components of $\Pi_{I}$ which are not in 
 $\Pi_{M}$.  It is well known  that
$\mpair{v_{J},\alpha}=1>0$ for all $\alpha\in \Pi_{J}$ (this is  easily seen since 
$s_{\alpha}(v_{J})=v_{J}-2\alpha$, using that $s_{\alpha}$ permutes 
$\Phi_{J,+}\setminus \set{\alpha}$).
Let $L\subseteq I$ be such that $\Pi_{L}$ is the union of all  indefinite type components of 
$\Pi_{I}$ which are not in $\Pi_{M}$.
Since each component of $\Pi_{L}$ is indefinite, there is  by \ref{x4.5} some element 
$v_{L}\in V$ expressible as  a strictly negative real   linear  combination  of elements 
of $\Pi_{L}$ such that $\mpair{v_{L},\alpha}>0$ for all $\alpha\in \Pi_{L}$. We have
\[\mpair{v_{L},\Pi_{L}}\subseteq \real_{>0},\quad \mpair{v_{L},\Pi_{I\setminus L}}\seq\set{0},
\quad \mpair{v_{J},\Pi_{J}}\subseteq \real_{>0},\quad \mpair{v_{J},\Pi_{I\setminus J}}\seq\set{0}.\]
Hence  $v_{J}+v_{L}\in \mc{C}_{I}$
and $\Pi_{I}\cap (v_{J}+v_{L})^{\perp}=\Pi_{M}$.
By definition,  $M$ is a facial subset of $I$ (i.e.  it is facial for $(W_{I},I)$ with simple  roots 
 $\Pi_{I}$) and by \ref{x4.2}(f), $M$ is facial in $S$ as required.

Finally, we prove (d). For any $I\seq S$, let $I_{\infty}$ (resp., $I_{0}$) denote the union of all infinite type (resp., 
finite type) components of $I$,  so $I=I_{\infty}\cup I_{0}$ with $I_{0}$ and $I_{\infty}$ separated. Assume first that $I$ is special i.e. $I=I_{\infty}$. Write the facial closure of $W_{I}$ as 
$dW_{K}d^{-1}$ where $K\seq S$ is facial and, without loss of generality,  $d\in W^{K}$. 
Then $\Pi_{I}\seq d\Pi_{K}$ and hence $d^{-1}\Pi_{I}\seq \Pi_{K}$. By Remark \ref{x1.15}, it 
follows that $d^{-1}\in W_{I^{\perp}}$ and $\Pi_{I}=d^{-1}\Pi_{I}\seq \Pi_{K}$.
Hence $W_{K}$ is a facial subgroup of minimal rank containing $W_{I}$
 (since $W_{K}\sreq W_{I}$ has the same rank as  the facial closure  $dW_{K}d^{-1}$ of 
 $W_{I}$) and it is therefore  the facial closure of $W_{I}$.  Finally,  $W_{K}$ is the standard 
 facial closure of $W_{I}$ since $W_{K}$ is standard facial and the
  standard facial closure of any 
 subset   of $W$contains the facial closure of that subset.

Now consider the case of general $I\seq S$.  We proceed by downward induction on $\vert I\vert$, 
the assertion to be proved being trivial if $I=S$. Suppose now that $I\sneq S$ and denote its facial 
closure as $W'$.  Let $J\seq S$,  where $J\sreq I_{\infty}$, be such that
$W_{J}$ is the    standard facial closure of $W_{I_{\infty}}$. Note that $J$ is special, for $I_{\infty}\seq J_{\infty}$, where $J_{\infty}$ is a  facial subset of $S$ by (c), and so $J=J_{\infty}$. From the special case $I=I_{\infty}$ already treated, $W_{J}$ is the facial closure of $W_{I_{\infty}}$, so
$W'\sreq W_{J}$ and  $W'$ is the facial closure of $W_{J\cup I_{0}}$. 
 Also, the standard facial 
closure of $W_{I}$ coincides with that of $W_{J\cup I_{0}}$
If $I\sneq J\cup I_{0}$, then by induction, the standard  facial closure $W_{K}$ of 
$W_{J\cup I_{0}}$ (and $W_{I}$) is equal to $W'$ by induction. If $I=J\cup I_{0}$, then $I_{\infty}\seq J\seq I_{\infty}$, $I_{\infty}=J$ is facial
 and hence $I=J\cup I_{0}$ is facial 
by Lemma \ref{x2.9}(c). In either case, we have the desired conclusion. 
 \end{proof}

\subsection{} \label{x8.6}  We next describe the relation of  the  imaginary cone  to the  facial subsets of $\Pi$ in the most general affine case.

In this subsection only, let $\Pi_{i}$ for $i\in I$ be the irreducible components of $\Pi$, and assume that all $\Pi_{i}$ are  of affine type.  Each set $\Pi_{i}$ is linearly independent. Let $W_{i}$ be the  affine Weyl group corresponding to $\Pi_{i}$. Write $\mc{K}_{i}=\real_{\geq 0}\delta_{i}$ where $\delta_{i}\in \real_{\geq 0}\Pi_{i}
\subseteq \real_{\geq 0}\Pi$ spans the isotropic ray of $\Pi_{i}$. We have $\mc{Z}_{i}:=W_{i}\mc{K}_{i}=\mc{K}_{i}$
and by \ref{x3.2}(d) that $\mc{Z}=\sum_{i} \mc{Z}_{i}=\sum_{i}\mc{K}_{i}=\mc{K}$ which is a polyhedral cone. Note that some of the $\delta_{i}$ 
may   lie in the relative interior of  faces of $\mc{K}$ of dimension greater than $1$ or may span the same extreme ray of $\mc{K}$.

\begin{lem}\begin{num}\item Let $P$ be a face of the cone $ \mc{Z}$, 
and $J=\mset{i\in I\mid \delta_{i}\in P}$.
For $i\in I$, let  $\Delta_{i}\subseteq \Pi_{i}$ be
  such that  $\Delta_{i}:=\Pi_{i}$ if $i\in J$, and $\Delta_{i}$ is a proper subset of 
 $\Pi_{i}$ otherwise. Then $\Delta:=\cup_{i\in I}\Delta_{i}$ is a facial subset of $\Pi$.
 \item Each facial subset  $\Delta$  of $\Pi$ arises as in $\text{\rm (a)}$ from a unique choice of face $P$ of 
 $\mc{Z}$ and sets $\Delta_{i}\subseteq \Pi_{i}$ satisfying the conditions of $\text{\rm (a)}$. 
  \end{num}\end{lem}
  
\begin{rem} (1) It is easy to  see that (up to linear isomorphism), 
any pair $(C,U)$ of a polyhedral cone $C$ in a vector space
$U$ with $C=\real U$ is isomorphic to  a pair $(\mc{Z},\real\mc{Z})$
arising as above from some affine root system $\Phi$.  Moreover, given any multiset $R$ of rays in $C$ including all the extreme rays of $C$ with multiplicity at least one, one could choose $\Phi$ so in addition, $R$ corresponds under the isomorphism to  the multiset of  isotropic rays of  the irreducible components of $\Phi$

(2) One obtains from Lemmas \ref{x8.6} and  \ref{x2.13}   a complete description of  all the  facial subsets
of $\Pi$ when $\Pi$ has only finite type and affine type components.
In this case, one easily sees  that the faces of $\mc{Z}$ are in natural bijection with the special facial subgroups $W'$ of $W$, via $W'\mapsto\mc{Z}_{W'}$. We shall show this holds in general  in Section \ref{x10}.
  \end{rem}

  \begin{proof}  We first prove (b).
  Choose $\phi\in \mc{C}$ and 
  consider the corresponding facial subset $\Delta:=\Pi\cap \phi^{\perp}$ of $\Pi$.
  We have $\mpair{\phi, \mc{Z}}\subseteq \real_{\geq 0}$. Let $P:=\mc{Z}\cap \phi^{\perp}$, which is a face of $\mc{Z}$. Let $J:=\mset{i\in I\mid \delta_{i}\in P}$ and
  $\Delta_{i}:=\Delta\cap \Pi_{i}$.  Note that $\delta_{i}$ is a strictly positive linear  combination of $\Pi_{i}$ and $\mpair{\phi, \Pi_{i}}\subseteq \real_{\geq 0}$.
   If $i\in J$,  then $\mpair{\phi, \delta_{i}}=0$, so
  it follows   $\mpair{\phi,\Pi_{i}}=0$, $\Pi_{i}\subseteq \Delta$
 and therefore  $\Delta_{i}=\Pi_{i}$ in this case. On the other hand, if $i\not\in J$, then 
 $\mpair{\phi,\delta_{i}
}>0$ so $\Delta_{i}\subsetneq \Pi_{i}$. Hence $\Delta$ arises as in (a).
Clearly, $\Delta_{i}=\Pi_{i}\cap \Delta$, $J=\mset{i\in I\mid \Delta_{i}=\Pi_{i}}$, and  $P=\sum_{j\in J}\real_{\geq 0}\delta_{j}$
  are uniquely determined by $\Delta$.

Conversely, let $P$, $J$, $\Delta_{i}$ be as in (a). 
Choose $\Gamma_{i}\subseteq \Pi_{i}$ as follows. If $i\in J$, let $\Gamma_{i}=\Pi'_{i}$ where $\Pi'_{i}\subseteq \Pi_{i}$ is the set of simple  roots of the finite Weyl group corresponding to $\Phi_{i}$ (in fact, we need only  $\Pi_{i}'\subseteq \Pi_{i}$ and $\vert \Pi_{i}\setminus \Pi_{i}'\vert=1$) .
If $i\in I\setminus J$, let $\Gamma_{i}$ be any subset of $\Pi_{i}$ such that $\vert \Pi_{i}\setminus \Gamma_{i}\vert =1$ and $\Delta_{i}\subseteq  \Gamma_{i}$.
Note that $\Gamma:=\cup_{i}\Gamma_{i}$  has all components of finite type; hence it is  facial 
 and  linearly independent  and the restriction of $\mpair{-,-}$ to $\real\Gamma$ is positive definite by Lemmas  \ref{x8.5}(a)  and \ref{x2.9}(c). It follows that $\real\Pi=\real\Gamma\oplus \real\mc{Z}$, an orthogonal direct sum with $\real\mc{Z}$ as radical, where $\real \Pi$ is positive semidefinite.
Choose, by non-degeneracy of $\mpair{-,-}$,  a subspace $U'\supseteq \real \Pi$ of $V$ such that the restriction of $\mpair{-,-}$ to $U'$ is non-singular and   such that the codimension of $\real\Pi$ in $U'$ is equal to the  dimension $N$ of $\real\mc{Z}$. Write $(\real\Gamma)^{\perp}\cap U'=\real{\mc Z}+U$ where $U$ is a totally isotropic subspace of $U'$ of dimension $N$. Then the restriction of $\mpair{-,-}$ to $\real\mc{Z}+U$
is non-singular of  signature $(N,N,0)$ (it is a direct sum of $N$ hyperbolic planes) and the induced bilinear form $\real\mc{Z}\times U\rightarrow \real$ is a perfect pairing.
Choose $\psi\in U$ such that  $\mpair{\psi,\mc{Z}}\subseteq \real_{\geq 0}$ and 
  $P=\mc{Z}\cap \psi^{\perp}$, so  $J=\mset{j\in I\mid \delta_{j}\in \psi^{\perp}}$. 
  Note that $\mpair{\psi, \Pi_{i}}=0$ for $i\in J$ and $\mpair{\psi, \Pi_{i}\setminus \Gamma_{i}}\subseteq \real_{>0}$ for $i\in I\setminus J$ since $\delta_{i}$ is a strictly positive linear combination of all the elements of $\Pi_{i}$, $\vert \Pi_{i}\setminus \Gamma_{i}\vert\leq 1$ and $\mpair{\psi,\Gamma_{i}}=0$.
   For each $i\in I\setminus J$, choose $\gamma_{i}\in \real\Gamma_{i}$ such that 
  $\mpair{\gamma_{i},\Delta_{i}}=0$ and $\mpair{\gamma_{i},\Gamma_{i}\setminus \Delta_{i}}\subseteq \real_{>0}$ (which is possible since $\Gamma_{i}$ is linearly independent and $\mpair{-,-}$ is positive definite on $\real\Gamma_{i}$). Let $\phi:=\psi+\epsilon \sum_{i\in I\setminus J} \gamma_{i}$ where $\epsilon >0$. It is easy to see that for sufficiently small $\epsilon>0$,  $\mpair{\phi,\Pi}\subseteq \real_{\geq 0}$ and $\Pi\cap \phi^{\perp}=\cup_{i\in I}\Delta_{i}=\Delta$. Hence the set $\Delta$ as in (a) is facial.
   \end{proof}

\subsection{} \label{x8.7} From Proposition \ref{x3.2}(a), we have $\mpair{z,z'}\leq 0$ for all $z,z'\in \mc{Z}$. 
Since each element $z'$ of $\mcZc$ is $W$-conjugate to an element $k$ of $\mcKc$, 
the result below determines  the pairs  $z,z'\in \mcZc$ with $\mpair{z,z'}=0$ up to $W$-conjugacy. 
\begin{lem} Let $k\in \mcKc$ and $z\in \mcZc$ both be non-zero.
Let $\Gamma$, $\Delta$ be  connected    supports  of $k$ and $z$ respectively,  chosen so $\Delta$ contains some connected support of the unique element of $Wz\cap \mc{K}$  (which is possible by Lemma $\text{\rm \ref{x8.4}(i)}$).
Then $\mpair{z,k}=0$ if and only if  one or both  of the following two conditions holds:
 \begin{conds}\item  $\Delta$ and $\Gamma$ are separated.
\item $\Gamma$ is of affine type and $z$, $k$ both span the  isotropic ray $\real_{\geq 0}\delta$ of $\Gamma$.\end{conds}
\end{lem}
\begin{proof} The ``if'' direction is trivial.  For the converse, suppose $\mpair{z,k}=0$.
Note that neither $\Gamma$ nor  $\Delta$ is of finite type,    by  \ref{x8.5}(b).
Write $z=\sum_{\alpha\in \Delta}c_{\alpha}\alpha$ where  all $c_{\alpha}> 0$. 
 We have
\[0= \mpair{k,z}=\sum_{\alpha\in \Delta}c_{\alpha} \mpair{k,\alpha}.\]
Since all $c_{\alpha}>0$ and $\mpair{k,\alpha}\leq 0$, it follows that 
$\mpair{k,\alpha}=0$ for all $\alpha\in \Delta$.
This readily implies that  either $\Delta$ is separated from $\Gamma$ or $\Delta$ is contained in  $\Gamma$.
Suppose that (i) doesn't hold, so  $\Delta\subseteq \Gamma$. If $\Gamma$ is of indefinite type,
\ref{x8.4}(d) implies that every connected component of $\mset{\beta\in \Gamma\mid \mpair{\beta, k}=0}$ is of finite type. Since $\Delta$ is contained in one of these components, we would have $\Delta$ of finite type contrary to above. Therefore 
$\Gamma$ must be of affine type.
Since $\Delta$ is not of finite type and $\Delta\subseteq \Gamma$, we have $\Gamma=\Delta$ of affine type.
But $\mpair{-,-}$ restricted to $\real \Delta$  is positive semi-definite with radical $\real\delta$. Since $\mpair{k,k}=\mpair{z,z}\leq 0$ and 
$k,z\in \real_{\geq 0}\Gamma$, we conclude that  (ii) holds.
\end{proof}

\subsection{}  \label{x8.8}The following description of isotropic rays in $\mc{Z}$ with connected support
is an immediate corollary (cf. \cite[Prop 5.7]{K}).

\begin{cor}  Any non-zero  element $\delta$ of $\mcZc$ is isotropic if and only if  it is  $W$-conjugate  
to a representative of the isotropic ray of some irreducible affine standard  parabolic subsystem of $\Phi$.\end{cor}

\section{Hyperbolic and Universal Coxeter groups} \label{x9} In this section, we discuss the imaginary cone and its closure in the case of ``hyperbolic''  and ``universal'' root systems. These two classes of root system play an important role in the study of the  imaginary cone in general. We also mention some open questions. 
 \subsection{}\label{x9.1} Consider the following  conditions on $(V,\mpair{-,-})$ and $(\Phi,\Pi)$. 

 \begin{conds}
 \item The restriction of $\mpair{-,-}$ to $\real\Pi$ has signature $(n,1,0)$ for some $n\in \Nat_{\geq 1}$
 (i.e. $\real \Pi$ is the orthogonal direct sum of a positive definite subspace of dimension $n\geq 1$ and a negative definite subspace of dimension $1$).
 \item No non-empty proper facial subset of $\Pi$ has all its irreducible components of indefinite type.
 \item Every proper facial subset   of $\Pi$ has each of  its connected components  of finite or affine type.
 \item   Every proper facial  subset   of $\Pi$ has each of  its connected components  of finite  type.\end{conds} 
 
 Clearly (iv) implies (iii) which in turn implies (ii).
  We shall say that the root system  $\Phi$ (or  Coxeter group $W$) is \emph{weakly hyperbolic} if (i) holds
 (in case $V=\real \Pi$, this was called hyperbolic in \cite[4.5--4.6]{K} and \cite{HRT}). We say that $(W,S)$ is \emph{hyperbolic} if 
 (i) and  (iii) hold, and that $(W,S)$ is \emph{compact hyperbolic} if (i) and (iv)  hold.

\subsection{} \label{x9.2} If $W$ is weakly hyperbolic, 
 let $\Pi_{1},\ldots, \Pi_{m}$ be the irreducible components of $\Pi$.
 We have $\real\Pi=\sum_{i}\real \Pi_{i}$ where $\real \Pi_{i}\perp \real \Pi_{j}$ if $i\neq j$.
 Moreover, $\real \Pi_{i}\cap \sum_{j\neq i}\real\Pi_{j}$ is in the radical $\set{0}$ of $\mpair{-,-}$ restricted to $\real \Pi$, so 
 in fact $\real\Pi=\oplus_{i}\real \Pi_{i}$, an orthogonal direct sum. By considering signatures of the restrictions of $\mpair{-,-}$ to $\real \Pi_{i}$, it follows that all but one of the $\Pi_{i}$ are of finite type and one of them, say $\Pi_{1}$, is of indefinite type. By \ref{x8.5}(c), $\Pi_{1}$ is a facial subset of $\Pi$.  If (ii) holds as well, then $\Pi_{1}=\Pi$ i.e. $\Pi$ is irreducible. 
 In particular, if $W$ is hyperbolic, it is irreducible of indefinite type.
 
 Note that if $W$ is  weakly hyperbolic, the Witt index (dimension of a maximal totally isotropic subspace) of the restriction of $\mpair{-,-}$ to $\real \Pi$ 
 is $1$.  This implies  that  any non-zero  totally isotropic face (cf. \ref{x7.10}) of $\ol{\mc{Z}}$ is  one-dimensional (i.e. a ray), which makes the structure of $\ol{\mc{Z}}$, $R_{0}$ etc much simpler  than in general.
 \begin{rem} If $V=\real \Pi$, the notions of hyperbolic and compact hyperbolic above essentially coincide with those in \cite{Bour},  \cite{Hum} and \cite{Kac} for those  (special) root systems  which are both of the type here  and  of the  type  considered in these sources (which all assume in particular that simple roots are  linearly independent).  \end{rem}

\subsection{}\label{x9.3} The following  corollary of Proposition \ref{x5.8} is required below. 
\begin{cor} Suppose that $(W,S)$ is irreducible and satisfies the condition
 $\text{\rm \ref{x9.1}(ii)} $
(as well as $\text{\rm \ref{x4.1}(i)--(iii)}$).  Then $\mc{K}=-\mc{C}\cap \real \Pi$ and  $\mc{Z}=-\mc{X}\cap \real\Pi$.
\end{cor}
\begin{proof} 
Under the assumptions here, $F_{\mathrm{m.ind}}  =\emptyset$. Proposition \ref{x5.8}(a) therefore gives $\mc{K}=-\mc{C}\cap \real \Pi$.  Since $W\real\Pi=\real \Pi$, this implies \begin{equation*}
\mc{Z}=\bigcup_{w\in W}\,w (-\mc{C}\cap \real{\Pi})=\real \Pi\cap \bigcup_{w\in W}w (-\mc{C})=\real\Pi\cap -\mc{X}.
\qedhere\end{equation*} 
\end{proof}

 \subsection{} \label{x9.4} If $W$ is   weakly hyperbolic on $V$
 there is a basis $x_{0},\ldots, x_{n}$ of $V$ such that $\mpair{x_{i},x_{j}}=0$ for $i\neq j$,
 $\mpair{x_{i},x_{i}}=1$ for $i=1,\ldots n$ and $\mpair{x_{0},x_{0}}=-1$.

 Let $\mc{L}=\mset{\sum_{i}\lambda_{i}x_{i}\mid \lambda_{0}>\sqrt{\sum_{i=1}^{n}\lambda_{i}^{2}}}$;
 this is an open convex cone in $V$ with closure 
 $\overline{\mc{L}}=\mset{\sum_{i}\lambda_{i}x_{i}\mid \lambda_{0}\geq \sqrt{\sum_{i=1}^{n}\lambda_{i}^{2}}}$ satisfying $\overline{\mc{L}}^{*}=-\overline{\mc{L}}$. (These well-known assertions follow readily from the Cauchy-Schwarz inequality). The cones $\mc{L}$ and $-\mc{L}$ are the two connected components of $\mset{v\in V\mid \mpair{v,v}< 0}$.  We have $0\neq \mc{Z}\subseteq \overline{\mc{L}}\cup -\overline{\mc{L}}$. Replacing $x_{n}$ by $-x_{n}$ if necessary, assume $\mc{Z}\cap \overline{\mc{L}}\neq \set{0}$.
  \begin{prop}  Let $W$ be irreducible and weakly hyperbolic, and let $\mc{L}$ be as defined above. 
 Assume $V=\real \Pi$. Then   $\mc{L}$ is $W$-invariant and \begin{num}
  \item  $\text{\rm(\cite{Max},\cite[(5.10.2)]{Kac}, \cite[Proposition 3.7(i)]{HRT})}$  $ \mc{Z}\subseteq \overline{\mc{L}}$.
  \item $\text{\rm(\cite{Max},  \cite[Proposition 3.7(ii)]{HRT}}$ $-\mc{L}\subseteq \mc{X}$. 
  \item $\text{\rm (compare \cite[Ex 5.15]{Kac})}$ If the condition $\text{\rm  \ref{x9.1}(ii)}$ holds (e.g. $W$ is hyperbolic), then $\mc{L}\subseteq \mc{Z}=-\mc{X}\subseteq  \overline{\mc{L}}$ and $\overline{\mc{Z}}=-\overline{\mc{X}}=\mc{X}^{*}=\overline{\mc{L}}\seq \real_{\geq 0}\Pi$. Further,
	  $\mc{Z}\setminus \mc{L}$ is the set of all   isotropic elements of $\mc{Z}$.
  \item If $W$ is  compact hyperbolic, then $\mc{Z}=-\mc{X}=\mc{L}\cup\set{0}$. \end{num}
 \end{prop}
 \begin{proof} Let $\mc{L}':=\overline{\mc{L}}\sm\set{0}$, so the connected components of $\mc{L}'':=\mset{v\in V\sm\set{0}\mid \mpair{v,v}\leq 0}$ are $\pm \mc{L}'$. By Proposition \ref{x3.2}, $\mc{Z}\sm\set{0}$ is a connected subset of
$ \mc{L}''$ which intersects $\mc{L}'$ nontrivially, so  
 $\mc{Z}\sm\set{0}\seq\mc{L}'$ and (a) is proved. 
 The $W$-action  preserves $\mc{L}''$ but cannot interchange the two components $\pm \mc{L}'$ since it preserves $\mc{Z}\sm\set{0}\seq \mc{L}'$.  Therefore  $\mc{L}' $
  (and hence its interior $\mc{L}$) is $W$-invariant.  Taking duals in (a)   gives $-\overline{\mc{L}}\seq\mc{Z}^{*}= \overline{\mc{X}}$ by 
 Theorem \ref{x5.1}. Hence the interior  $-\mc{L}$ of 
 $-\overline{\mc{L}}$ is contained in the interior of $\overline{\mc{X}}$, which is in turn contained in $\mc{X}$ since $\mc{X}$ is a cone. This proves (b).

  Assume from now  that \ref{x9.1}(ii) holds. The first part of  (c) follows from (a)--(b), 
 Lemma \ref{x9.3} and Theorem \ref{x5.1} (together with the definition of $\mc{Y}$). The  second part of (c) is trivial since  $\overline{\mc{L}}\setminus \mc{L}$
 is the set of all isotropic elements of $\overline{\mc{L}}$. Now the  isotropic elements of $\mc{Z}$
 are  $W$-conjugate to the isotropic elements of $\mc{K}$, which are completely described 
 in Lemma \ref{x8.4}. In particular, there are no non-zero isotropic elements of $\mc{Z}$ unless
 there is a (necessarily proper  since $W$ is of indefinite type)  facial subset $I$ of $S$ such that $\Pi_{I}$ has only affine components. Hence (d) follows from the definitions and (c).
   \end{proof}
   \subsection{} \label{x9.5}
    The imaginary cone $\mc{Z}_{W'}$ of a reflection subgroup $W'$ is not in general the intersection of $\mc{Z}$  with $\real\Pi_{W'}$, as simple examples show.
    However, the next Lemma describes some important special situations in which this is true
    (in addition to those in Corollary \ref{x3.5}).
   
   \begin{cor} Let $W'$ be a  finitely generated reflection subgroup of $W$
   which is either finite,   irreducible affine  or   hyperbolic  (e.g. any dihedral reflection subgroup).  Then     $\mc{Z}_{W'}=\mc{Z}\cap \real\Pi_{W'}$ and  $\overline{\mc{Z}_{W'}}=\overline{\mc{Z}}\cap \real\Pi_{W'}$.
      \end{cor}
    \begin{proof}  Let ${L} '_{W'}:=\mset{z\in \real
\Pi_{W'}\mid \mpair{z,z}\leq 0}$ and ${L} _{W'}:={L} '_{W'}\cap \real_{\geq 0}\Pi_{W'}$. We first prove that 
\begin{equation}\label{x9.5.1}
{L} '_{W'}={L} _{W'}\cup -  {L} _{W'},\quad \overline{\mc{Z}_{W'}}={L} _{W'}\end{equation}
     by considering cases according to the type of $W'$.
      If $W'$ is finite, then $\overline{\mc{Z}_{W'}}=0$ by \ref{x4.5} and \ref{x3.2}(d),
while  ${L} _{W'}={L} _{W'}'=0$ by Lemma \ref{x2.9}.   If instead $W'$ is affine, then by 4.4, 
 $\overline{\mc{Z}_{W'}}=\mc{Z}_{W'}=\real_{\geq 0} \delta$ (where $\delta$ spans the isotropic ray of $W'$),
 ${L} _{W'}=\real_{\geq 0}\delta$ and ${L} '_{W'}=\real\delta$.  Finally, if  $W'$ is of hyperbolic type, then   \eqref{x9.5.1} follows  from Proposition \ref{x9.4}(c).  Next, we claim that
  \begin{equation}\label{x9.5.2} 
  L'_{W'}\cap \real_{\geq 0}\Pi=L_{W'},\qquad \mc{K}_{W'}=-\mc{C}_{W'}\cap L_{W'}.\end{equation}
   The first equation above follows directly  from the first equation in \eqref{x9.5.1}, and the second   follows from the definitions and Proposition \ref{x3.2}(c).

 Now by Lemma \ref{x3.2}(c), $\overline{\mc{Z}}\cap \real\Pi_{W'}\subseteq {L} _{W'}'$. We get from Theorem \ref{x6.3} that
 \[\overline{ \mc{Z}_{W'}}\subseteq \overline{\mc{Z}}\cap \real\Pi_{{W'}}\subseteq {L} '_{W'}\cap \real_{\geq 0}\Pi\subseteq {L} _{W'}=\overline{\mc{Z}_{W'}}\] by \eqref{x9.5.1}--\eqref{x9.5.2}.
 Next, note that   $\mc{Z}_{W'}\subseteq \mc{Z}\cap \real\Pi_{W'}$
 by \ref{x6.3} again. To prove the reverse inclusion, let $z\in  \mc{Z}\cap \real\Pi_{W'}$. Choose $w\in W$ with $wz\in \mc{K}$, and let $W'':=wW'w^{-1}$ which is a reflection subgroup of $W$ of the same type (finite, irreducible affine or hyperbolic) as $W'$.
 Using Proposition \ref{x3.2}(c) for $W$, and  \eqref{x9.5.2} for $W''$, we have \[wz\in \mc{K}\cap \real\Pi_{W''}= -\mc{C}\cap \real_{\geq 0}\Pi
\cap     {L} '_{W''}\subseteq -\mc{C}_{{W''}}\cap     {L} _{W''} = \mc{K}_{W''}\subseteq \mc{Z}_{W''} \]     and so $z\in w^{-1}\mc{Z}_{W''}=\mc{Z}_{W'}$ as required.     
           \end{proof}
 
\subsection{Some questions} \label{x9.6}
We next collect some (mostly open) questions which have arisen in the 
course of this work or are suggested by study of examples.  One 
preliminary remark is that some  properties of the closed imaginary 
cone  which hold for irreducible root systems fail trivially 
for reducible ones. Correspondingly, even for irreducible $W$,  points 
of this cones may fail  to have some property of interest if  the point is 
conjugate to a point in the cone attached to a proper facial parabolic 
subgroup (since the latter may be reducible).   Therefore define the set of \emph{generic points} of  $\ol{\mc{Z}}$ to be 
 \begin{equation*}
\Zg:=\mset{\al\in \ol{\mc{Z}}\mid w\al\not \in \ol{\mc{Z}_{I}} \text{ \rm for all $w\in W$ and facial $I\sneq S$}}.
\end{equation*}
(We  don't make a corresponding definition for $\mc{Z}$ since it just yields the relative interior of $\mc{Z}$, as follows from results in  \S\ref{x10}.)

 The following  definitions are more naturally formulated  in terms of a compact convex  base of a closed salient  cone   as in Lemma 7.10, but for uniformity, we  continue in terms of cones.   
 Let $C$ be  any closed salient  cone in $V$.   For $p\in V$, 
 set $C(p):=\ccl(C-p)$. Fix non-zero $p\in  \rb(C)$.
 Call $p$  a \emph{smooth boundary point} if
  there is a unique absolutely supporting linear hyperplane for $C$ in $\real C$ which contains $p$.
 Say that  $p$  is a
 \emph{round boundary point} if 
there is some (necessarily unique)  linear hyperplane $H$  in $\real C$  containing 
the 
line  $\real p$ such that   $C(p)$
is the union of $\real p$ with  one of the two  open halfspaces in 
$\real(C-p)$ determined by $H$.
  Finally, 
say
that   $p$ is a \emph{flat boundary  point}
  if it is in the relative interior of an exposed face of $C$ of 
  codimension $1$.  Note that round or flat boundary points are smooth, but not conversely.

\begin{quest}\label{q9.6}  Let $(W,S)$ be an irreducible finite rank Coxeter system of indefinite type.
\begin{num}
\item  Is $\Zg\seq\mc{Z}\cup\mc{Q}$? More strongly, is every  proper face $F$ of $\ol{\mc{Z}}$ such that $\ri(F)\cap \Zg\neq\eset$  totally isotropic?
\item (See \ref{x7.9})  Is $R_{0}=\ol{Z}\cap {Q}$ (raised by Hohlweg and Ripoll)? If not, is it at least true that $\Zg\cap \mc{Q}\seq \cup R_{0}$?
\item Is $\overline{\mc{Z}}$  equal to the  topological closure of the convex hull of the union of  the set of  its  round  boundary points and  $\set{0}$? 
\item Is every extreme ray of $\ol{\mc{Z}}$ exposed (cf. Lemma \ref{xA.11}(b)). More generally, is every face of $\ol{\mc{Z}}$ exposed? Is it even true that $(\ol{\mc{Z}},\ol{\mc{X}})$ is a dual pair of stable cones (cf. Theorem \ref{x10.3})?
\item Do any two distinct,  maximal, totally isotropic faces of $\ol{\mc{Z}}$,  such that  each of them   contains  some point of $\Zg$, intersect in $\set{0}$?
\end{num}   
\end{quest}

\begin{rem} (1) Results in \cite{DHR} are relevant to the study of several of these questions. It should not be difficult to settle some of them for hyperbolic groups (for example, for hyperbolic $\Phi$, one 
has $\ol{\mc{Z}}\seq \mc{Z}\cup\mc{Q}$, more strongly than (a), by 
Proposition \ref{x9.4}).  See \ref{e9.18}  for an example showing 
the necessity of  restrictions in (a), (b), (e) for non-hyperbolic $\Phi$. 
 
 (2)    There is   considerable 
 diversity even  in the imaginary cones of  rank three universal Coxeter 
 groups. Depending on the 
 root system, $\rb(\mc{Z})\sm\set{0}$ may either  consist entirely of 
 round boundary points and be  a differentiable 
 submanifold of $V$ or  contain  no round boundary point  which  has a 
 neighborhood in $\rb(\mc{Z})\sm\set{0}$  which is a differentiable submanifold 
 of $V$.  Similarly in this case, $R_{0}$ may be topologically a circle 
 or a Cantor set. 
  \end{rem}

\subsection{} \label{x9.7} It would be  natural to study the dynamics of the action of 
special elements of $W$ and subgroups of $W$ on the closed 
imaginary cone, the disposition of eigenvectors and eigenspaces  in relation to the faces 
etc. The following question in this vein  is suggested by Lemma 
\ref{x6.1}(f).\begin{quest} If 
$\Phi$ is irreducible of indefinite type  and $\al\in\Zg\cap \mc{Q}$,  
is $0\in \ol{W\al}$. \end{quest}
The condition that $\al$ be generic  is included to exclude possibilities like $\al=\al'+\al''$
where $\al'\in \mc{Z}$ and $\al''\in \ol{\mc{Z}}\cap\mc{Q}$ have separated facial supports and $\al'$ spans the isotropic ray of an affine standard parabolic subsystem; in such a case, $0\in \ol{W\al}$ is not possible.

\subsection{}\label{x9.8} Other natural questions involve ubiquity of  
   reflection subgroups which are universal Coxeter groups, and their properties.
   Note that the based root system $(\Phi,\Pi)$  is the root system of a universal Coxeter group if and only if $\mpair{\al,\bt}\leq -1$ for all distinct $\al,\bt\in \Pi$.   In this case, we also say that $(\Phi,\Pi)$ is \emph{universal}. 
 Using \cite{DyRef} (see \cite[Remark 3.12]{DyTh}), it is equivalent  to require $\vert\mpair{\al,\bt}\vert\geq 1$ for all $\al,\beta $ in $\Phi$.
 We shall say that $(\Phi,\Pi)$ is a \emph{generic} universal based root system  if  $\mpair{\al,\bt}< -1$ for all distinct $\al,\bt\in \Pi$. 
 (Using Corollary \ref{x8.8}, this is easily seen to be equivalent to
 $\vert \mpair{\al,\bt}\vert > 1$ for all $\al,\bt\in \Phi$ with $\bt\neq \pm \al$, but we shall not need this fact).

    Assume   that $(W,S)$ is  an 
  irreducible Coxeter system of  indefinite type and of finite rank at least 
  two. 
We define a metric on the   set of  non-zero closed pointed   cones which are contained   in
 $\real_{\geq 0}\Pi$ by using a Hausdorff metric (see e.g. \cite{Web}) on their intersections 
 with  $\mset{v\in \real_{\geq 0}\Pi\mid \mpair{v,\rho}=1}$ (which is a compact convex base of $\real_{\geq 0}\Pi$). 

\begin{quest} \label{q9.8} Does there exist a finite rank   reflection subgroup $W'$ of $W$ such that $(\Phi_{W'},\Pi_{W'})$ is a generic universal based root system and 
 $\real \Pi_{W'}=\real
\Pi$.   Is there   a sequence of such reflection
subgroups $W'_{n}$, $n\in \Nat$, such that  both sequences $\real_{\geq 0}\Pi_{W'_{n}}$ and 
$\overline{\mc{Z}_{W'_{n}}}$, for $n\in \Nat$, converge  to  $\overline{\mc{Z}}$ in the above 
metric?\end{quest}
This was established for hyperbolic groups in \cite{Edgar}. It will be shown in    \cite{DHR} that the question has an affirmative answer; in fact, we will prove more general results about approximation of non-zero faces of $\ol{\mc{Z}}$.

 \subsection{Generic universal root systems}\label{x9.9} 
 Because of the general importance of  generic universal root systems in view of the affirmative answer  in \cite{DHR} to the previous question, we discuss some of their properties in the remainder of this section. 
 
  \begin{ass}  In \ref{x9.10}--\ref{x9.17},  it is assumed unless otherwise stated (in addition to the standing assumptions  \ref{x4.1}(i)--(iii)) that $(\Phi,\Pi)$ is generic universal (i.e. $\mpair{\al,\bt}<-1$ for all  $a\neq \bt\in \Pi$) and, to avoid trivialities, that  $\vert \Pi\vert \geq 2$. \end{ass}
  
  Give the finite set $\Pi$ the discrete topology and  $\prod_{n\in \Nat}\Pi$ the product topology. Let $C$ the the closed (topological) subspace of 
 $\prod_{n\in \Nat}\Pi$ consisting of all sequences 
  $\bt=(\bt_{n})_{n\in \Nat}$ in   $\Pi$ with $\bt_{i}\neq \bt_{i+1}$ for all
   $i$.  In general, $C$ is a  compact, totally disconnected,  metrizable  
   space.  If $\vert \Pi\vert >2$, then $C$ is perfect  (i.e. has no isolated 
   points) and it is well known that these conditions then imply that  $C$ 
   is a Cantor space i.e. it is homeomorphic to the (standard, ternary) 
   Cantor set.   If $\vert \Pi\vert =2$, then $C$ is a discrete two-point 
   space. 
   For any  $\bt=(\bt_{n})_{n\in \Nat}\in C$, set 
  \begin{equation*}
  F_{\beta}:=\mset{z\in \ol{\mc{Z}}\mid \lim_{n\to\infty}s_{\bt_{n-1}}\cdots
    s_{\beta_{0}}z=0}.
  \end{equation*}
      The  theorem below is our main result on generic universal root systems. It shows in particular that  $\ol{\mc{Z}}\sm\mc{Z}\seq \mc{Q}$ and that the  space of connected components of $\ol{\mc{Z}}\sm\mc{Z}$ (its quotient space identifying the components to points) is homeomorphic to $C$.
\begin{thm}\begin{num}
\item $\mc{Z}\cap \mc{Q}=\set{0}$ and $\ol{\mc{Z}}\sm\mc{Z}\seq \mc{Q}$. 
\item For $\bt=(\bt_{n})_{n\in \Nat}\in C$, $\ol{\mc{Z}}\sm F_{\bt}=\mset{z\in \ol{\mc{Z}}\mid \lim_{n\to \infty}\mpair{\rho, s_{\bt_{n-1}}\cdots
    s_{\beta_{0}}z}=+\infty}$. 
\item  There is    a well-defined function
 $b\colon \ol{\mc{Z}}\sm\mc{Z}\to C$ defined  by setting   $b(z):=\bt$ if $\bt\in C$ and $z\in F_{\bt}\sm\set{0}$. \item The fibers of $b$, namely the sets $b^{-1}(\bt)=F_{\bt}\sm\set{0}$, for $\bt\in C$, are the connected components of $\ol{\mc{Z}}\sm\mc{Z}$.

\item $\set{F_{\bt}\mid \bt\in C}$  is the set of all faces of $\ol{\mc{Z}}$ which are maximal in the set of faces $F$ of $\ol{\mc{Z}}$ satisfying $F\cap \mc{Z}=\set{0}$. It is also  the set of all maximal totally isotropic faces of $\ol{\mc{Z}}$.
\item  The map  $b$ is a quotient map. It  factors as a composite 
$b=qp$ where $p\colon\ol{\mc{Z}}\sm  \mc{Z}\to \ray(\mc{Q}\cap \ol{\mc{Z}})$ is an open quotient map $\al \mapsto \real_{\geq 0}\al$ and where $q\colon 
\ray(\mc{Q}\cap \ol{\mc{Z}})\to C$ is a closed quotient map given by 
$\real_{\geq 0}\al\mapsto  b(\al)$. 

\end{num}
\end{thm}  
 Subsections \ref{x9.10}-\ref{x9.17} 
give the  proof of the theorem, and related facts.

\subsection{} \label{x9.10}
  Some of the following  results, especially the early ones,  can be generalized, with more technical statements,  to hold for all universal $(\Phi,\Pi)$.   We shall find it convenient to use again  the notation $H_{\phi}^{\prec}$ introduced in \ref{x2.3}.  Also, for any non-zero $v\in \real_{\geq 0}\Pi$, write $\wh{v}:=\frac{1}{\mpair{\rho,v}}v$ where we have fixed $\rho\in V$ with $\mpair{\rho, \Pi}\seq \real_{\geq 0}$. Set $\epsilon_{\rho}:=\min(\mset{\mpair{\rho,\bt}\mid \bt\in \Pi})>0$.
  Throughout the proof, whenever multiple root systems are under consideration and  it is necessary to indicate notationally the dependence  of some object we have defined on the root system $(\Phi,\Pi)$, we  shall do so by attaching a subscript $\Pi$ (or $W$). 
  
 For any distinct   $\al$ and $\bt$ in $\Pi$, write 
 $\mpair{\al,\bt}=-\cosh \lambda$ where 
 $\lambda=\lambda_{\al,\bt}> 0$. Let  
 $u_{\Pi}(\al,\bt)=u(\al,\bt):=e^{\lambda}\al+\bt\in \mc{Q}$  and
  $u'_{\Pi}(\al,\bt)=u'(\al,\bt):=(\cosh\lambda) \al+\bt\in  \al^{\perp}$.
  It is easily verified from our general results (or by 
  direct computation using  for instance  the discussion of dihedral root systems in 
 \cite{HLR}),  that the imaginary cone of 
 $W'=W_{\al,\bt}=\mpair{s_{\al},s_{\bt}}$ is 
${\mc{Z}}_{W'}=\real_{>0}\set{u(\al,\bt),u(\bt,\al)}\cup\set{0}$, with 
closure  $\ol{\mc{Z}}_{W'}=\real_{\geq 0}\set{u(\al,\bt),u(\bt,\al)}$ 
and the 
fundamental domain for $W'$ on $\mc{Z}_{W'}$ is $ \mc{K}_{W'}=
\real_{\geq 0}\set{u'(\al,\bt),u'(\bt,\al)}$.
 Also define  \begin{equation*}
 \mc{K}^{+}:=\real_{\geq 0}\mset{u(\al,\bt)\mid \al,\bt\in \Pi,\al\neq \bt}=\real_{\geq 0}\Bigl(\ \bigcup_{\al\neq \bt \in \Pi}\ol{\mc{Z}}_{W_{\set{\al,\bt}}}\Bigr)\seq{\ol{\mc{Z}}}.
 \end{equation*} 

It is important for our purposes to understand the relative position of the rays through the points 
\begin{align*}&x_{1}:=\alpha &&x_{2}:=u(\al,\bt)=e^{\lambda}\alpha+\beta\\ &x_{3}:=u'(\al,\bt)=(\cosh \lambda)\alpha +\beta
 &&x_{4}:=u'(\bt,\al)=\alpha+(\cosh \lambda)\beta\\ &x_{5}:=u(\bt,\al)=\alpha+e^{\lambda}\beta &&x_{6}:=\beta
\end{align*} where $\al\neq \beta\in \Pi$. 
Define the \emph{slope} of $ u:=a\alpha+b\beta$, where $a,b\in \real_{\geq 0}$  are not both zero, as $m(u)=\frac{b}{a}\in \real_{\geq 0}\cup\set{\infty}$.
 Then \begin{multline*}m(x_{1})=0<m(x_{2})=e^{-\lambda}<m(x_{3})=\frac{1}{\cosh \lambda}\\<m(x_{4})=\cosh(\lambda)<m(x_{5})=e^{\lambda}<m(x_{6})=\infty.\end{multline*}  This implies that   \begin{equation}\label{x9.10.1}
 \real_{\geq 0 }x_{j}\seq \real_{\geq 0} x_{i}+\real_{\geq 0}x_{k} \iff 
i\leq j\leq k, \quad \text{ if $1\leq i\leq k\leq 6$, $1\leq j\leq 6$.} \end{equation}
Similar results  are formulated in an affine setting in \cite{HLR}; the condition in  \eqref{x9.10.1} is equivalent to $\wh x_{j}\in \conv(\set{\wh x_{i},\wh x_{j}})$ and the result is represented by the diagram
\begin{equation*}\xymatrix{
\ar@{-}[r]&{\wh x_{1}}\ar@{-}[r]&{\wh x_{2}}\ar@{-}[r]&{\wh x_{3}}\ar@{-}[r]&{\wh x_{4}}\ar@{-}[r]&{\wh x_{5}}\ar@{-}[r]&{\wh x_{6}}\ar@{-}[r]&}
\end{equation*} of points on the affine line spanned by the $\wh x_{i}$.

\subsection{}  \label{x9.11}The lemma below describes basic relations between $\mc{K}$, $\mc{K}^{+}$, $\ol{\mc{Z}}$ and $\mc{Q}$.

\begin{lem}\begin{num}\item 
$\mc{K}=
\real_{\geq 0}\mset{u'(\al,\bt)\mid \al\neq \bt\in \Pi}=\real_{\geq 0}\bigl(\bigcup_{\al\neq \bt\in \Pi}\mc{K}_{W_{\set{\al,\bt}}}\bigr)$.
\item $\real_{\geq 0}\Pi=\mc{K}\dot\cup \bigl(\dot\bigcup_{\al\in \Pi}(\real_{\geq 0}\Pi\cap H^{>}_{\al})\bigr)$.
\item ${\mc{Z}}\cap \mc{Q}=\set{0}$.
 \item $\mc{K}\seq \mc{K}^{+}\seq\ol{\mc{Z}}$.
\item  $\ol{\mc{Z}}=\ol{\cup_{w\in W}\,w\mc{K}^{+}}$.

  \end{num} 
\end{lem}
\begin{proof} We first prove (a)-(b) assuming that $\Pi$ is linearly independent. By definition, 
  $\mc{K}=  \real_{\geq 0}\Pi\cap \bigcap_{\al\in \Pi}H_{\al}^{\leq }$.  One
   easily  checks 
   \begin{equation*}
    \real_{\geq 0}\Pi\cap H_{\al}^{=}=\real_{\geq 0}\mset{u'_{\al,\bt}\mid \bt\in \Pi\sm\set{\al}}.
    \end{equation*}
    Also,  for any $\al\in \Pi$, 
   $H_{\al}^{<}$  does not contain $\al$, and it 
    contains $u'(\bt,\g)$, where $\bt\neq \g$ in $\Pi$, if and only if 
    $\bt \neq \al$. It follows that $\mset{u'(\al,\bt)\mid \al,\bt\in \Pi, \al\neq \bt}$ is a 
    set of representatives for the extreme rays of the polyhedral cone $\mc{K}$, proving the first equality in (a).     
    The second equality in (a) is clear from the formulae in \ref{x9.10}.  
    Recall that the facets of $\mc{K}$ are its  codimension one faces. These are the   cones
    \begin{equation*}
    \mc{K}\cap \real_{\geq 0}(\Pi\sm\set{\al})=\real_{\geq 0}\mset{u'_{\bt,\g}\mid \bt,\g\in \Pi\sm\set{\al}, \bt\neq \g}
    \end{equation*} for $\al\in \Pi$  (provided $\vert \Pi\vert \geq 3$) and  $
    \mc{K}\cap \al^{\perp}= \real_{\geq 0}\Pi\cap H_{\al}^{=}$ for $\al\in \Pi$.  Moreover,
    \begin{equation*}
    \real_{\geq 0}\Pi\cap H_{\al}^{\geq 0}=\real_{\geq 0}(\set{\al}\cup\mset{u'_{\al,\bt}\mid \bt\in \Pi\sm\set{\al}})
    \end{equation*} for $\al\in \Pi$. It is easy to see from these equations that (b) also holds. Next, to prove (a)--(b)  in general (if $\Pi$ is is possibly  linearly dependent) 
  choose a canonical lift $(\Phi',\Pi')$ in $(V',\mpair{-,-}')$ of $(\Phi,\Pi)$ with canonical map $L\colon V'\to V$ as in \ref{x1.4}. 
 For any  distinct $\al',\bt'\in \Pi'$,  one has $L(u'_{\Pi'}(\al',\bt'))=u'_{\Pi}(L(\al'),L(\bt'))$ since $\mpair{L(\al'),L(\bt')}=\mpair{\al',\bt'}'$. 
 Part (a) follows directly by applying $L$ to the formula (a) for 
 $(\Phi',\Pi')$ (known since $\Pi'$ is linearly independent) using Proposition \ref{x3.6}. Part (b) follows similarly using 
 $\real_{\geq 0}\Pi'\cap L^{-1}(\real_{\geq 0}\Pi\cap H_{L(\al')}^{>})=\real_{\geq 0}\Pi'\cap H_{\al'}^{>} $ in $V'$  and $L(\real_{\geq 0}\Pi'\cap  H_{\al'}^{>})=\real_{\geq 0}\Pi\cap H_{L(\al')}^{>} $ in $V$,  for $\al'\in \Pi'$.

    For (c), note  $\mcKc=\mc{K}\sm\set{0}$ and
    $\mcZc=\mc{Z}\sm\set{0}$ since any non-zero element  of $\real_{\geq 0}\Pi$ has a connected support (in fact, any non-empty subset of $\Pi$ is connected). Since there are no irreducible  affine standard parabolic subsystems of $\Phi$, part (c) holds by Corollary \ref{x8.8}.
  In (d)  the first (resp., second)  inclusion holds  since for all $\al\neq \bt$ in $\Pi$,
  $ \mc{K}_{W_{\al,\bt}}\subset\ol{\mc{Z}}_{W_{\al,\bt}}$ (resp., 
 $ \ol{\mc{Z}}_{W_{\al,\bt}}\seq \ol{\mc{Z}}$).   From (d) and the definition  $\mc{Z}=\cup_{w\in W}w(\mc{K})$, we have  \begin{equation*}
 \ol{\mc{Z}}=\ol{\cup_{w\in W}\,w\ol{\mc{Z}}}\,\sreq\, 
 \ol{\cup_{w\in W}\,w\mc{K}^{+}}\,\sreq \, \ol{\cup_{w\in W}\,w\mc{K}}=\ol{\mc{Z}}
 \end{equation*}
which proves (e). 
     \end{proof}
 \subsection{} \label{x9.12}
The   next lemma establishes a  crucial special fact about generic universal root systems:   the positive root cone is the union of the imaginary cone and certain  cones $\mc{D}_{\al}\sm\set{0}$, indexed by 
simple roots $\al$, the union of which contains all positive roots and also the intersection of the positive root cone and isotropic cone. This fact plays a significant role in the study of root systems of general Coxeter groups in \cite{DHR}. For any  $\al\in \Pi$, define
     \begin{equation*}
     \mc{D}_{\alpha}=\mc{D}_{\Pi,\al}:=\real_{\geq 0}(\set{\al}\cup\set {u(\al,\bt)\mid \bt\in \Pi\sm\set{\al}})\seq \real_{\geq 0}\Pi
     \end{equation*} \begin{lem}\begin{num}
\item For $\al\in \Pi$, $\mc{D}_{\al}\sm\set{0}\seq H_{\al}^{>}\cap \bigcap_{\bt\in \Pi\sm\set{\al}}H_{\bt}^{<}$. In particular, $\mc{D}_{\al}\cap \mc{D}_{\bt}=\set{0}$ if $\al\neq \bt$ in $\Pi$.
\item There is a constant $k_{\rho}>0$ (depending on $\rho$) such that if $\g\in \bigcup_{\al\in \Pi}\mc{D}_{\al}$ and $\bt\in \Pi$ then $\vert \mpair{\g,\ck \bt}\vert\geq k_{\rho}\, \vert \mpair{\rho,\g}\vert$. \item  There is a constant $k_{\rho}'>0$  such that if $\g\in \mc{D}_{\al}$ and $\g'\in \mc{D}_{\bt}$ where $\al,\bt\in \Pi$ are distinct, then $\mpair{\g,\g'}\leq  -k_{\rho}'\mpair{\g,\rho}\mpair{\g',\rho}$.
\item $\real_{\geq 0}\Pi=\mc{Z}\cup \bigl(\bigcup_{\al\in \Pi}\mc{D}_{\al}\bigr)$.
 \item $\Phi_{+}\cup((\mc{Q}\cap \real_{\geq 0}\Pi)\sm\set{0})\seq \real_{\geq 0}\Pi\sm{\mc{Z}}\seq \bigcup_{\al\in \Pi}\mc{D}_{\alpha}$. 

\item Let $\g\in\real_{\geq 0}\Pi\sm{\mc{Z}}$. Then for $\al\in \Pi$,  $\mpair{\g,\al}\geq 0$ if and only if    $\g\in \mc{D}_{\al}$ if and only if  $\mpair{\g,\al}> 0$. 
These conditions are satisfied by a unique $\al\in \Pi$.

  \end{num} 
\end{lem}
\begin{proof}  
Part (a) holds since for $\al\neq \beta \in \Pi$ , one has $\mpair{\al,\al}>0$  and $\mpair{\al,u(\al,\bt)}>0$
while for  $\g\in \Pi\sm\set{\al}$, one has $\mpair{\g,\al}<0$  and $\mpair{\g,u(\al,\bt)}<0$. For (b), note that  for any $\al,\bt\in \Pi$ 
and $\g\in \mc{D}_{\al}\sm\set{0}$, one has  $\vert \mpair{\g,\ck \bt}\vert\geq  k_{\al,\bt} \mpair{\g,\rho}$ by (a), where \begin{equation*}k_{\al,\bt}:=
\min\left(\mset{ \frac{\vert\mpair{\ck\bt,u}\vert}{\vert\mpair{\rho,u}\vert} \mid u\in \set{\al}\cup\mset{u(\al,\g)\mid \g\in \Pi\sm\set{\al}}}\right). 
\end{equation*}
Part (b) holds setting $k_{\rho}:=\min(\mset{k_{\al,\bt}\mid \al,\bt\in \Pi})$. The proof of (c) is similar to that of (b) and we omit it. 
 
 In the  proof of (d), assume first    that $\Pi$ is linearly independent. 
 Using this assumption, choose  for each $\al\in \Pi$ an element $\phi_{\al}\in V$ such 
 that $\mpair{\phi_{\al},u(\al,\bt)}=0$  for all $\bt\in \Pi\sm\set{\al}$ and $\mpair{\phi_{\al},\al}> 0$.
 Let $\al\neq \bt$ in $\Pi$.  It follows from \eqref{x9.10.1}  that  $u({\al,\bt})\in \real_{>0}\set{\al,\g}$ for each $\g\in \set{u'({\al,\bt}), u'(\bt,\al), u(\bt,\al),\bt} $. It  follows that $\mpair{\phi_{\al},\g}<0$ for such $\g$.  That is, the hyperplane 
$H_{\phi_{a}}^{=}\cap \real \Pi$ in $\real \Pi$ strictly separates $\al$ and $\g$. 
 It follows that it also strictly  separates $\al$ from $\Pi\sm \set{\al}$ and $\al$ from $u_{\g,\delta}$ for any $\gamma\neq \delta$ in $\Pi\sm\set{\al}$.
  One
   easily  checks that for any $\al\in \Pi$, 
   $\real_{\geq 0}\Pi\cap H_{\phi_\al}^{=}=\real_{\geq 0}
   \mset{u(\al,\bt)\mid \bt\in \Pi\sm\set{\al}}$. An argument very similar to that in the proof of Lemma \ref{x9.11}(a)
  shows that $\mc{K}^{+} =\real_{\geq 0}\Pi\cap\bigcap_{\al\in \Pi} H_{\phi_{\al}}^{\leq}$.    The facets of $\mc{K}^{+}$ are the cones
    $\mc{K}^{+}\cap \real_{\geq 0}(\Pi\sm\set{\al})=\real_{\geq 0}\mset{u_{\bt,\g}\mid \bt\neq \g\in \Pi\sm\set{\al}}=\mc{K}^{+}_{\Pi\sm\set{\al}}$ for $\al\in \Pi$  (provided $\vert \Pi\vert \geq 3$) and  $\mc{K}^{+}\cap H_{\phi_{\al}}^{=}=\real_{\geq 0}\mset{u_{\al,\bt}\mid \bt\in \Pi\sm\set{\al}}$ for $\al\in \Pi$. 
 Also, $\mc{D}_{\al}=\real_{\geq 0}\Pi\cap H_{\phi_{\al}}^{\geq}$.  This makes it clear that $\real_{\geq 0}\Pi=\mc{K}^{+}\cup \bigl(\bigcup_{\al\in \Pi}\mc{D}_{\al}\bigr)$.
  
Now clearly  $\real_{\geq 0}\Pi\,\sreq\,\mc{Z}\cup \bigl(\bigcup_{\al\in \Pi}\mc{D}_{\al}\bigr)$.
We prove the reverse inclusion  by induction on $n:=\vert \Pi\vert$.  If $n=2$, say $\Pi=\set{\al,\bt}$, this follows easily from \eqref{x9.10.1} (in fact, the right hand side is a partition of the left hand side in this case). Now assume  $n\geq 3$.  From above, it suffices  to show $\mc{K}^{+}\seq\mc{Z}\cup \bigl(\bigcup_{\al\in \Pi}\mc{D}_{\al}\bigr)$.
Let $z\in \mc{K}^{+}$. Now clearly $\real \mc{K}^{+}=\real \Pi=\real{\mc{Z}}$. Hence $\ri(\mc{K}^{+})\seq \ri(\ol{\mc{Z}})=\ri(\mc{Z})\seq \mc{Z}$.
Therefore if $z\in \ri(\mc{K}^{+})$, then $z\in\mc{Z}$ as required. 
In the contrary case, $z$ is in some facet of $\mc{K}^{+}$.
Now for $\al \in \Phi$, the facet  $\mc{K}^{+}\cap H_{\phi_{\al}}^{=}=\real_{\geq 0}\cap H_{\phi_{\al}}^{=}$ is also a facet of $\mc{D}_{\al}$.
Hence we may assume that $z$ is in a facet $\mc{K}^{+}_{\Delta}$ of $\mc{K}^{+}$  where $\Delta:=\Pi\sm\set{\al}$ for some $\al\in \Pi$.
By induction, $z\in \mc{Z}_{\Delta}\cup(\bigcup_{\bt\in \Delta}
D_{\Delta,\bt})$. By Remark \ref{x3.5}(2),  $\mc{Z}_{\Delta}=\mc{Z}\cap \real_{\geq 0}\Delta$. Also,  
 $\mc{D}_{\Delta,\bt}=\real_{\geq 0}(\set{\bt}\cup\set{u_{\bt,\g}\mid \g\in \Delta\sm \set{\bt}})=\mc{D}_{\bt}\cap \real_{\geq 0}\Delta$. Therefore
 $z\in \mc{Z}\cup(\bigcup_{\bt\in \Pi}\mc{D}_{\bt})$ as required to complete the proof of (d) in the case of linearly independent $\Pi$.
 
 If $\Pi$ is linearly dependent,  choose a canonical lift $(\Phi',\Pi')$
  in $(V',\mpair{-,-}')$ of $(\Phi,\Pi)$ with canonical map 
  $L\colon V'\to V$ as in \ref{x1.4}. 
 For any  distinct $\al',\bt'\in \Pi'$,  one has 
 $L(u_{\Pi'}(\al',\bt'))=u_{\Pi}(L(\al'),L(\bt'))$ since $\mpair{L(\al'),L(\bt')}=
 \mpair{\al',\bt'}'$. The definitions therefore  give $L(\mc{D}_{\Pi',\al'})=
 \mc{D}_{\Pi,L(\al')}$ for $\al\in \Pi'$. By the  linear independence of 
 $\Pi'$  and the case previously  treated,  we therefore get
 \begin{equation*}\begin{split}
 \real_{\geq 0}\Pi&=L(\real_{\geq 0}\Pi')=
 L(\mc{Z}_{\Pi'}\cup(\bigcup_{\al'\in \Pi'}\mc{D}_{\Pi',\al'}))=
 L(\mc{Z}_{\Pi'})\cup(\bigcup_{\al'\in \Pi'}L(\mc{D}_{\Pi',\al'}))
 \\&= \mc{Z}_{\Pi}\cup(\bigcup_{\al\in L(\Pi')}\mc{D}_{\Pi,\al})=
 \mc{Z}_{\Pi}\cup(\bigcup_{\al\in \Pi}\mc{D}_{\Pi,\al})\end{split}
 \end{equation*} using Proposition \ref{x3.6},
 as required to complete the proof of (d) in all cases.
 
The second inclusion in (e) follows from (d). For the first, note that its 
left hand side is contained in $\real_{\geq 0}\Pi\sm\set{0}$. 
Hence  it suffices  to show that $\Phi_{+}\cap \mc{Z}=\eset$ and 
$\mc{Z}\cap \mc{Q}=\set{0}$.  The former is clear since
 $\mpair{\al,\al}=1$ for $\al\in \Phi$ while $\mpair{\al,\al}\leq 0$ for 
 $\al\in \mc{Z}$ by 
Proposition \ref{x3.2}(c), while the latter holds by 
  Lemma \ref{x9.11}(b).  Part  (f) follows from (a) and (e).

\end{proof}

\subsection{} \label{x9.13} We shall define below a subset $\mc{Z}_{\infty}$ of 
$\real_{\geq 0}\Pi$ and  a function $b\colon  \mc{Z}_{\infty}\to C$
(which will turn out to coincide with $b$ in Theorem \ref{x9.9}(a)). To do this, fix
$\g\in\real_{\geq 0}\Pi $ and   attempt to associate to $\g$  a sequence
$b(\g)=(\bt_{n})_{n\in \Nat}\in C$, using Lemma \ref{x9.11}(b) as follows. Set $\g_{0}:=\g$. Then $\bt_{0}$ is defined if and only if $\g_{0}\in \real_{\geq 0}\Pi\sm\mc{K}$, in which case
  $\bt_{0}$  is defined to be the unique element of $\Pi$ such that 
$\mpair{\bt_{0},\g}>0$.  In general,  for $m\in \Nat$,
 $\bt_{m+1}$ is defined if and only 
if $\bt_{0},\ldots,\bt_{m}$ are all defined  and, further,  
$\g_{m+1}:=s_{\bt_{m}}\cdots s_{\bt_{0}}(\g)\in \real_{\geq 0}\Pi\sm{\mc{K}}$;
then  we define $\bt_{m+1}$ to be the unique element of $\Pi$ with 
$\mpair{\bt_{m+1},\g_{m+1}}>0$.  Note $\bt_{m+1}\neq \bt_{m}$ when they are both defined. Define $\mc{Z}_{\infty}$ to be the set of all $\g\in\real_{\geq 0}\Pi$ such that 
$\bt_{n}$ is defined for all $n\in \Nat$, and define 
$b\colon  \mc{Z}_{\infty}\to C$ by $b(\g):=(\bt_{n})_{n\in \Nat}$ for any $\g\in \mc{Z}_{\infty}$.  Note that $\mc{Z}_{\infty}$ is a  possibly non-convex, blunt cone  and $b$ is constant on $\real_{> 0}$-orbits on $\mc{Z}_{\infty}$.

\begin{lem}\begin{num}
\item $b\colon \mc{Z}_{\infty}\to C$ is continuous. 
\item For $\bt=(\bt_{n})_{n}\in C$, $b^{-1}(\bt)=
\mc{Z}_{\infty}\cap\, \bigcap_{i\in \Nat} H_{\bt_{i}'}^{>}$ where 
$\bt_{i}':=s_{\bt_{0}}\cdots s_{\bt_{i-1}}(\bt_{i})$.
\end{num}
\end{lem}\begin{proof}
To prove (a), it will suffice to show that  for $n\in \Nat$ and  
$\beta_{0},\ldots, \beta_{n}\in\Pi$ with $\beta_{i-1}\neq \beta_{i}$ for 
$i=1,\ldots,n$, the set 
$U:=\mset{\g\in \mc{Z}_{\infty}\mid (b(\g))_{i}=\bt_{i} 
\text{ \rm  for $i=1,\ldots, n$ }}$ is open in $\mc{Z}_{\infty}$.
But from the definition of $b$, for any  $\g\in \mc{Z}_{\infty}$, one has $\g\in U$  if and only if \begin{equation*}
\mpair{\g,\beta_{0}}>0,
\mpair{s_{\beta_{0}}\g,\beta_{1}}>0,\ldots,
\mpair{s_{\beta_{n-1}}\cdots s_{\beta_{0}}(\g),\beta_{n}}>0.
\end{equation*}
This implies  that 
$U=\mc{Z}_{\infty}\cap \bigcap_{i=0}^{n} H_{\bt_{i}'}^{>}$ where 
$\bt_{i}'=s_{\bt_{0}}\cdots s_{\bt_{i-1}}(\bt_{i})$ for $i=0,\ldots, n$. This 
description shows $U$ is open in ${\mc{Z}}_{\infty}$, proving (a). A 
similar argument shows that for $\g\in \mc{Z}_{\infty}$ and
 $\bt=(\bt_{n})_{n}\in C$, one has  $b(\g)=\bt$ if and only if
$\mpair{\g,\bt_{i}'}>0$ for all $i\in \Nat$, where $\bt'_{i}$ is as in (b).
This completes the proof of (b).
\end{proof}
\subsection{} \label{x9.14}
In order to describe  and prove some further properties of  the function $b$, it is convenient  to partially order $V$ by 
$v_{1}\leq v_{2}$ if  $v_{2}-v_{1}\in \real_{\geq 0}\Pi$, 
for $v_{1},v_{2}\in V$.  For $ \g\in  \real_{\geq 0}\Pi\sm\mc{K}$ and $\bt\in \Pi$, we have $s_{\beta}(\g)<\g$ if  $\mpair{\g,\bt}
>0$ and $s_{\beta}(\g)>\g$ if $\mpair{\g,\bt}<0$, these being the only 
possibilities (any coarser partial order on $V$ with these properties could be used in place of $\leq $ below).

\begin{lem} Let $\g'\in \mc{Z}_{\infty}$, $b(\g') =\bt'=(\bt'_{n})_{n}\in C$ and
$\al=(\al_{n})_{n}\in C$  with $\al\neq \bt'$. 
Let $m:=\max(\mset{n\in \Nat\mid  \al_{j}=\bt'_{j}\text{ for $j=0,\ldots, n-1$}})\in \Nat$. For any $n\in \Nat$, write
$\g'_{n}:=s_{\bt'_{n-1}}\cdots s_{\bt'_{0}}(\g')\in \real_{\geq 0}\Pi$  and  set
$\delta_{n}:=s_{\al_{n-1}}\cdots s_{\al_{0}}(\g')$.

\begin{num}
\item $\g'_{0}>\g'_{1}>\g'_{2}\ldots$.
\item $\delta_{0}>\delta_{1}>\ldots >\delta_{m}$ and  
$\delta_{m}<\delta_{m+1}<\delta_{m+2}<\ldots$.
\item $\g'_{n}=\delta_{n}$ for $n=0,1,\ldots, m$ and $\delta_{n}>\g'_{n}$ 
for all $n\geq m+1$. 
\item $\mc{Z}_{\infty}=\ol{\mc{Z}}\sm\mc{Z}$. 
\item $\lim_{n\to \infty }\g'_{n}=0$ and $\g'\in \mc{Q}$.
\item $\mpair{\rho,\delta_{n}}\to\infty$ as $n\to \infty$.
\item If $z\in \mc{Z}\sm\set{0}$ and $(w_{n})_{n\in \Nat}$ is any sequence of pairwise distinct elements of $W$, then $\mpair{\rho,w_{n}z}\to\infty$ as $n\to \infty$.\end{num}
\end{lem}

\begin{proof}
For $\tau\in \real_{\geq 0}\Pi\sm\mc{K}$, there is a unique
$\epsilon\in \Pi$ such that $\mpair{\tau,\epsilon}>0$, by   Lemma \ref{x9.11}, and for any $\epsilon'\in \Pi\sm\set{\epsilon}$,
one has $s_{\epsilon'}\tau>\tau$ and  $s_{\epsilon'}\tau\in \real_{\geq 0}\Pi\sm\mc{K}$. The claims in parts (a)--(c) on $\g_{n}'$, $\delta_{n}$ follow easily from these observations by induction on $n$. 

Next we prove (d).  For any  
$\g\in \ol{\mc{Z}}\sm\mc{Z}\seq \real_{\geq 0}\Pi\sm \mc{K}$, there is a unique   $\bt\in \Pi$ with 
$\mpair{\bt,\g}>0$ by Lemma \ref{x9.11}(b);  further,  one has  $s_{\bt}(\g)\in \ol{\mc{Z}}\sm \mc{Z}$
since  $\ol{\mc{Z}}\sm \mc{Z}$ is $W$-invariant. It follows by induction on $n$ that for all $n\in \Nat$, the simple root  $\bt_{n}$ in the 
construction of $b(\g)$ above is  defined, so  $b(\g)$ is defined  and 
$\ol{\mc{Z}}\sm \mc{Z}\seq \mc{Z}_{\infty}$. For the reverse 
inclusion,  it will suffice   by  Theorem \ref{x5.1}(b) and definition of $\mc{Y}$  to show that   $\g'\not\in \mc{Z}$ and 
$w(\gamma')\in \real_{\geq 0}\Pi$ for all $w\in W$
(since  $\g'\in \mc{Z}_{\infty}$ is arbitrary). 
Since $\bt'_{i}\neq \bt'_{i+1}$ for all $i$,  one has 
$s_{\bt'_{0}}\cdots s_{\bt'_{k}}(\bt'_{k+1})\in \Phi_{+}$ for all $k\in \Nat$, and these roots are pairwise distinct. Since  $\mpair{\bt'_{k+1},\g'_{k+1}}>0$, we have 
$\mpair{s_{\bt'_{0}}\cdots s_{\bt'_{k}}(\bt'_{k+1}),\g'}>0$ for all  $k$. Lemma \ref{x1.10}(a) implies that $\g'\not\in -\mc{X}$ and then
 Proposition \ref{x3.2}(a) implies that $\g'\not\in \mc{Z}$.
 Now we  let $w\in W$ and show $w\g'\in \real_{\geq 0}\Pi$.  By choosing $\al$ appropriately, we may assume without loss of generality that $w=s_{\al_{p-1}}\cdots s_{\al_{0}}$ where $p\in \Nat$. Then $w(\g')=\delta_{p}\geq \g'_{p}\in \real_{\geq 0}\Pi$ by (a)--(c).
By definition of $\leq$, we also have $w(\g')-\g'_{p}\in \real_{\geq 0}\Pi$ so $w(\g')\in \real_{\geq 0}\Pi$ as required. This completes the proof of (d).

For (e),  let $\lambda:=\inf(\mset{\mpair{\rho,w\g'}\mid w\in W})\in \real_{\geq 0}$.
 By (d) and Lemma \ref{x6.1},  there 
exists $v\in \ol{W\g'}\cap \mc{K}$ with $\rho(v)=\lambda$ and a sequence 
$(w_{n})_{n\in \Nat}$ in $W$ such that $w_{n}\g'\to v$. 
 For each $n\in \Nat$, write $w_{n}=x_{n}y_{n}$, where 
$l(w_{n})=l(x_{n})+l(y_{n})$, as follows.  Take a reduced expression
$w_{n}=s_{\epsilon_{q}}\ldots s_{\epsilon_{0}}$ (depending on $n$) with $\epsilon_{i}\in \Pi$, and
set $p=\max(\mset{n\in \Nat\mid   n\leq q+1,  \epsilon_{i}=\gamma'_{i} \text{ for $i=0,\ldots, n-1$}})$. Define $y_{n}=s_{\epsilon_{p-{1}}}\cdots s_{\epsilon_{0}}$ and $x_{n}=s_{\epsilon_{q}}\cdots s_{\epsilon_{p}}$.
Also write $m_{n}:=l(y_{n})$.
By choosing  $\al$ appropriately depending on $w_{n}$, (a)--(c) imply that
$\g'\geq y_{n} \g'$ and $w_{n}\g'\geq y_{n}\g'$ for all $n$.
Hence $\lambda \leq \mpair{\rho,y_{n}\g'}\leq \mpair{\rho, w_{n}\g'}$ for all $n$.
As $n\to \infty$, $\mpair{\rho, w_{n}\g'}\to\lambda$ and so 
$\mpair{\rho,y_{n}\g'}\to \lambda$ also. But $w_{n}\g'-y_{n}\g'\in \real_{\geq 0}\Pi$ and $\mpair{w_{n}\g'-y_{n}\g',\rho}\to 0$ as $n\to \infty$ imply that $w_{n}\g'-y_{n}\g'\to 0$ also. Since $w_{n}\g'\to v$, we get $y_{n}\g'\to v$ as well.
But  $y_{n}\g'=\g'_{m_{n}}$.
Since the sequence $(\mpair{\rho, \g'_{n}})_{n}$ in $\real_{\geq 0}$   is strictly monotonic decreasing and  $\mpair{\rho,\g'_{m_{n}}}\to \lambda$ as $n\to \infty$, we must have, by definition of $\lambda$,    $m_{n}\to\infty$  and 
$\mpair{\rho, \g'_{n}}\to \lambda$ as $n\to \infty$. 
Now for any $k$ with  $m_{k}>n$, we have $\g'_{n}-\g'_{m_{k}}\in \real_{\geq 0}\Pi$.
Letting $k\to \infty$ gives $\g'_{n}-v\in \real_{\geq 0}\Pi$.
Since $\mpair{\g'_{n}-v,\rho}\to \lambda-\lambda=0$ as $n\to \infty$, we get $\g'_{n}\to v$ as $n\to \infty$. Hence $\lim_{n\to \infty}\g'_{n}=v\in \mc{K}$. 

We show next that $v=0$. Suppose to the contrary that $v\neq 0$.
Then $\mpair{\rho, \g'_{n}}\geq \lambda>0$ for all $n\in \Nat$.
Since  $\g'$ is in the $W$-invariant set $\ol{\mc{Z}}\sm \mc{Z}$, and 
  $\g'_{n}\in W\gamma$, we get  $\g'_{n}\in \ol{\mc{Z}}\sm\mc{Z}\seq
\cup_{\al\in \Pi}\mc{D}_{\al}$. Let $\g'_{n}\in \mc{D}_{\tau_{n}}$, where $\tau_{n}\in \Pi$.  By Lemma \ref{x9.12}(a)-(b),
$\mpair{\g'_{n},\ck\tau_{n}}\geq k_{\rho}\mpair{\rho,\g'_{n}}\geq \lambda k_{\rho}>0$ and similarly $\mpair{\g'_{n},\ck\sigma}\leq - \lambda k_{\rho}<0$ for $\sigma\in \Pi\sm\set{\tau_{n}}$.  
Since   $\g'_{n}\to v$ as $n\to\infty$, there is 
some  $\tau\in \Pi$ such that  $\tau_{n}=\tau$ for all sufficiently large $n$, and    $v\in\mc{D}_{\tau}$. But then
$\mpair{v,\ck\tau}\geq k_{\rho}\mpair{\rho, v}>0$ contrary to $v\in \mc{K}$. This shows that $v=0$. 
 Now $w_{n}\g'\to v$ and $\mpair{\g',\g'}=\mpair{w_{n}\g',w_{n}\g'}\to \mpair{v,v}=0$ as $n\to \infty$ implies that $\g'$ is isotropic.  This completes the proof of (e).

Now we prove  (f). For any $l\in \Nat$, $\delta_{l+1}=s_{\al_{l}}(\delta_{l})=\delta_{l}-\mpair{\delta_{l},\ck\al_{l}}\al_{l}$. Assume $l\geq m$. Then 
$\mpair{\delta_{l},\ck\al_{l}}<0$ since $\delta_{l+l}>\delta_{l}$, so  
$-\mpair{\delta_{l},\ck\al_{l}}\geq k_{\rho}\mpair{\rho,\delta_{l}}\geq k_{\rho}\mpair{\rho,\delta_{m}}$ where $k_{\rho}$ is as in Lemma \ref{x9.12}, using $\delta_{l}-\delta_{m}\in \real_{\geq 0}\Pi$ for the last inequality. Hence 
\begin{equation*} 
\mpair{\rho,\delta_{l+1}-\delta_{l}}=-\mpair{\delta_{l},\ck\alpha_{l}}\mpair{\rho,\alpha_{l}}\geq k_{\rho}\mpair{\rho,\delta_{m}}\mpair{\rho,\al_{l}}\geq\epsilon_{\rho} k_{\rho}\mpair{\rho,\delta_{m}}.
\end{equation*} Setting $K:=\epsilon_{\rho} k_{\rho}\mpair{\rho,\delta_{m}}>0$,
we get $\mpair{\rho,\delta_{N+m}}\geq \mpair{\rho,\delta_{m}}+NK$ for all $N\in \Nat$ which implies (f) (the  above estimates  suffice for purposes here but are obviously far from sharp). 

For the proof  of (g), we may assume without loss of generality that $z\in \mc{K}\sm \set{0}$. Hence $v:=-z\in \mc{C}$.  Fix $N\in \real_{\geq 0}$. By \ref{x4.3}(a),  there are only finitely many $x\in Wz$ with $\mpair{v-x,\rho}\leq N$. 
By  Lemma \ref{x1.10} and Lemma \ref{x9.11}(e), $\Stab_{W}(v)$ is a standard parabolic subgroup of $W$ of rank $0$ or $1$, so it is finite. Hence for each $x\in Wz$ with $\mpair{v-x,\rho}\leq N$, there are only finitely many $w\in Wv$ with $wv=x$. This implies that  there are only finitely many $w\in W$ with $\mpair{v-wv,\rho}\leq N$ i.e. $\mpair{wz,\rho}\leq N-\mpair{v,\rho}$. Since $N\geq 0$ is arbitrary, this establishes (g).
\end{proof}

\subsection{} \label{x9.15} 
This section will establish  that the fibers of $b$ are all non-empty, 
which is an important point in the proof of Theorem \ref{x9.9}. 
Though not strictly necessary for purposes here, we establish a stronger fact  involving 
the \emph{dominance order}  $\preceq$ of $\Phi$, which  is the partial 
order such that $\al\preceq \beta$ if for all
 $w\in W$ with $w(\bt)\in \Phi_{-}$, one has $w(\al)\in \Phi_{-}$
 (see  \cite{BH1}, \cite{BjBr} or \cite{HRT}). If $\al,\bt\in \Phi_{+}$, one 
 has $\al\preceq \bt$ if and only if $\mpair{\al,\bt}\geq 1$ and 
 $l(s_{\al})\leq l(s_{\beta})$. The above  two characterizations of dominance order apply to arbitrary based root systems. For universal root systems
 (generic or not), the second criterion immediately implies the following fact:  for
$\al,\bt\in \Phi_{+}$, one has $\al\preceq \bt$ if and only if, writing (in the 
unique possible way) $\beta=s_{\bt_{1}}\ldots s_{\bt_{n-1}}(\bt_{n})$ 
where all $\bt_{i}\in \Pi$ and  $l(s_{\bt})=2n-1$, one has 
$\al=s_{\beta_{1}}\ldots s_{\beta_{m-1}}(\bt_{m})$ for some
 $1\leq m\leq n$.

  Recall also  the definitions from 
\S\ref{x5} of the set $R_{+}$ of rays spanned by positive roots and the 
set $R_{0}=\Acc(R_{+})$ of its limit rays in $\ray(V)$. Note 
that, since limit rays of the set of  root rays are isotropic,  
$(\bigcup R_{0})\setminus\set{0}\seq \ol{\mc{Z}}\sm \mc{Z}=\mc{Z}_{\infty}
$ by Lemmas \ref{x9.11} and \ref{x9.14}. Since 
$b\colon \mc{Z}_{\infty}\to C$  
 has  blunt, possibly non-convex cones as its fibers, it determines a 
 partition  $R_{0}=
 \dot\cup _{\bt\in C}(R_{0}\cap\ray({b^{- 1}(\bt)\cup\set{0}}))$ (except it  is not  yet clear that all the sets on the right are non-empty). 
The following proposition  gives a  condition 
for  a sequence of positive root rays to  have all its limit rays  in one set  in this partition, and makes it clear that these sets are all non-empty.

 \begin{prop} Fix $\bt=(\bt_{n})_{n\in \Nat}\in C$. 
Set $\bt_{n}':=s_{\bt_{0}}\cdots s_{\bt_{n-1}}(\bt_{n})\in \Phi_{+}$. 

 \begin{num}
  \item Let  $\al_{n}\in \Phi_{+}$ for all $n\in \Nat$.    Then   $\Acc((\real_{\geq 0}\al_{n})_{n})\seq R_{0}\cap\ray({b^{-1}(\bt)\cup\set{0}})$ if and only if for all $m\in \Nat$, one has $\bt'_{m}\preceq \al_{n}$ for almost all $n\in \Nat$.
 \item    $\Acc((\real_{\geq 0} \bt'_{n})_{ n\in \Nat})\seq \ray(b^{-1}(\bt)\cup\set{0})$. Hence  $b$ is surjective.
 \end{num}\end{prop}
 
\begin{proof}   It is sufficient to prove (a)  under the extra assumption 
that $(\real_{\geq 0}\al_{n})_{n}$ has a unique limit ray  
$\real_{\geq 0}\al$ i.e.  $\set{\real_{\geq 0}\al}=
\Acc((\real_{\geq 0}\al_{n})_{n})$ or equivalently,
 $\wh \al=\lim_{n\to\infty}\wh \al_{n}$. First consider the case that $\al$
  is non-isotropic.  Without loss of generality, assume $\al\in \Phi_{+}$. 
   Then  $\al_{n}=\al$ for almost all $n$.  Moreover 
   $\real_{\geq 0}\al\not \in R_{0}$ and, since each positive root dominates only 
   finitely many positive roots and the $\bt_{i}$ are pairwise distinct, 
   there is some $m\in \Nat$ such that  $\bt_{m}\leq \al_{n}$ for no
    $n\in \Nat$.  Hence (a) holds if $\al$ is 
 non-isotropic. Now assume that $\al$ is isotropic.  Then, since $\real_{\geq 0}\al$ is a limit ray of the set of rays spanned by positive roots,  one has
 $\al\in\mc{Z}_{\infty}$, as previously noted, so $b(\al)$ is defined.
 In this case, we have  $\real_{\geq 0}\al \in R_{0}\cap \ray(b^{-1}(\bt)\cup\set{0})$ if and only if  $b(\al)=\bt$.
In the notation of  Lemma \ref{x9.13}(b), we have     $b(\al)=(\bt_{n})_{n}$ 
if and only if for all $m\in \Nat$, $\mpair{\al,\bt'_{m}}>0$. 

We claim that 
this holds if 
and only if  for all $m$, $\mpair{\al_{n},\bt'_{m}}>0$ for almost all $n$.
Fix $m$. Since $\wh \al_{n}\to\wh \al$, it is clear that $\mpair{\al,\bt'_{m}}>0$ implies 
$\mpair{\al_{n},\bt'_{m}}>0$ for almost all $n$. Conversely,
suppose that $\mpair{\al_{n},\bt'_{m}}>0$ for almost all $n$.
Since $\wh\al_{n}\to \wh \al$, it is clear that 
$\mpair{\wh \al,\bt'_{m}}\geq 0$ and we have to show equality cannot 
occur. Write $w:=s_{\bt_{m-1}}\ldots s_{\bt_{0}}$. We have to show that $\mpair{\wh\al,w^{-1}\bt_{m}}\neq 0$ i.e. $\mpair{w\wh\al,\bt_{m}}\neq 0$. But $0\neq \wh\al\in \ol{\mc{Z}}\cap\mc{Q}$ and hence  
$0\neq w\wh\al\in \ol{\mc{Z}}\cap\mc{Q}\seq \real_{\geq 0}\Pi\cap\mc{Q}$.
By Lemma  \ref{x9.12}, 
$\vert \mpair{w\wh\al,\ck\bt_{m}}\vert \geq 
k_{\rho}\mpair{\rho, w\wh\al}>0$. This   proves the claim.

 Thus, $b(\al)=\bt$ if and only if for all $m$, $\mpair{\al_{n},\bt'_{m}}>0$ 
for almost all $n$.
Equivalently since $\Phi$ is universal, $b(\al)=\bt$ if and only if  for all $m$,  $\mpair{\al_{n},\bt'_{m}}\geq 1$ 
for almost all $n$.
Since $l(s_{\al_{n}})\to \infty$ as $n\to \infty$,  this  is equivalent to
the condition that for all $m$, $\bt_{m}'\preceq \al_{n}$ for almost all
 $n$. This proves (a). Part (b) follows from (a) since 
 $\bt_{0}'\prec \bt_{1}'\prec\bt_{2}'\prec\ldots $ (and any sequence of 
 distinct  root rays has at least one limit ray). \end{proof}

 \subsection{Proof of Theorem \ref{x9.9}} \label{x9.16}  
 The first part of \ref{x9.9}(a) is already established as Lemma 
 \ref{x9.11}(c). The second part of  (a) follows from Lemma 
 \ref{x9.14}(d)--(e), since $\g'\in\mc{Z}_{\infty}$ there is arbitrary.

 Let  $\bt=(\bt_{n})_{n}\in C$ and $w_{n}:=s_{\bt_{n-1}}\ldots s_{\bt_{0}}\in W$.   Then   
  \begin{equation}\label{x9.16.1}
  b^{-1}(\bt)\cup\set{0}\seq\mset{z\in \ol{\mc{Z}}\mid \lim_{n\to \infty}w_{n}z=0}
  \end{equation} by  Lemma \ref{x9.14}(e).
  Also, \begin{equation}\label{x9.16.2}
  \ol{\mc{Z}}\sm ({b^{-1}(\bt)}\cup\set{0})\seq
 \mset{z\in \ol{\mc{Z}}\mid \lim_{n\to \infty}\mpair{\rho,w_{n}z}=\infty}
  \end{equation} 
  To see this, note that $\mc{Z}\sm\set{0}$ is contained in the right hand 
  side by Lemma \ref{x9.14}(g),  and any point of
  $\ol{\mc{Z}}\sm(\mc{Z}\cup b^{-1}(\bt))$ is in $b^{-1}(\bt')$ for some 
  $\bt'\neq \bt$ in $C$ and so is in the right hand side by using Lemma 
  \ref{x9.14}(f). Now the left hand sides of \eqref{x9.16.1}--\eqref{x9.16.2}
  form a partition of $\ol{\mc{Z}}$ and the right hand sides are disjoint 
  subsets of $\ol{\mc{Z}}$. Hence the inclusions in 
  \eqref{x9.16.1}--\eqref{x9.16.2} are actually equalities. By definition of $F_{\bt}$, this  proves that  $b^{-1}(\bt)\cup\set{0}=F_{\bt}$ and  that 
   Theorem \ref{x9.9}(b) holds.
  It also implies that  the sets $F_{\bt}\sm\set{0}$ for $\bt\in C$ form a 
  partition of $\ol{\mc{Z}}\sm\mc{Z}$, which proves  Theorem
   \ref{x9.9}(c) and that the function $b$ so defined  coincides 
   with the function $b$ defined in \ref{x9.13}.

  Note that for $\bt\in C$,  $F_{\bt}$ is a totally
 isotropic face of $\ol{\mc{Z}}$ by  Lemma \ref{x7.10}(c)
 and  $F_{\bt}\cap \mc{Z}=\set{0}$ by (a).
Hence   $b^{-1}(\bt)=F_{\bt}\sm\set{0}$ is a 
 blunt cone, which is non-empty by Proposition \ref{x9.15}(b).
 So $b^{-1}(\bt)$ is non-empty and connected. Since $C$ is totally disconnected, its connected components are its singleton subsets and 
therefore  $\bt^{-1}(\bt)$ is a connected component of the domain $\ol{\mc{Z}}\sm\mc{Z}$ of $b$. This proves Theorem
   \ref{x9.9}(d).
   
   Next we prove Theorem \ref{x9.9}(e). Suppose $\bt\in C$.  Let $F$ be any face of $\ol{\mc{Z}}$  such that $F\cap \mc{Z}=\set{0}$ and  $F\sreq F_{\bt}$.
   Then $F\sm\set{0}$ is a connected subset of $\ol{\mc{Z}}\sm\mc{Z}$ and it  contains the component $F_{\bt}\sm\set{0}$, so by (b),
   $F\sm\set{0}=F_{\bt}\sm\set{0}$ and $F=F_\bt$.
   On the other hand, let   $G$ be any face of $\ol{\mc{Z}}$  with the property  $G\cap\mc{Z}=\set{0}$.  If $G=\set{0}$, then $G\seq F_{\bt}$ for all $\bt\in C$.  Otherwise, $G\sm\set{0}$ is a non-empty connected subset of $\ol{\mc{Z}}\sm\mc{Z} $ and is therefore contained in $F_{\bt}\sm\set{0}$ for a unique $\bt\in C$. This completes the proof of (e).
   
Before the proof of Theorem \ref{x9.9}(f),  we recall relevant generalities about quotient maps. A map $q\colon X\to Y$ of topological spaces is said to be a 
quotient map if it is surjective and a set $U$ is open in $Y$ if and only if 
$q^{-1}(U)$ is open in $X$. In that case, any continuous map $f\colon X
\to Z$ which is constant on each fiber $q^{-1}(y)$, for $y\in Y$,  
factors uniquely as a composite $f=f'q$ where $f'\colon Y\to Z$ is 
continuous.  Any continuous,  open (resp., closed), surjective map 
$q\colon X\to Y$ is a quotient map and then given any subset $Y'$ of $Y$, 
the restriction of $q$ to a map $q'\colon X'\to Y$ is also an open 
(resp., closed) quotient map where $X':=q^{-1}(Y')$.

Consider the  space $\ray(V)$. For the  part of the argument involving $p$, it is  convenient  to identify $\ray(V)$ 
with the unit sphere $\mathbb{S}$ in $V$ 
(with respect to some  norm $\Vert \ \Vert$ on $V$ induced by a  
positive definite inner product on $V$) by the homeomorphism 
$\real_{\geq 0}\al\mapsto \frac{1}{\Vert\al\Vert}\al$ for $0\neq \al\in V$. The natural continuous map 
$\al\mapsto \real_{\geq 0}\al\colon V\sm\set{0}\to \ray (V)$ then identifies with the map $r\colon V\sm\set{0}\to \mathbb{S}$ given by $\al\mapsto \frac{1}{\Vert\al\Vert}\al$. It  is an open 
quotient map
(in fact, it maps open balls in $V\sm\set{0}$ to open balls in $\mathbb{S}$, where the balls  in $\mathbb{S}$ are defined with  respect to the metric obtained by restriction of the metric on $V$ induced by $\Vert\ \Vert$). It follows that for any subset $X$ of $V$ which is a union of rays, the natural map $ X\sm\set{0}\to \ray(X\cup\set{0})$ given by restriction of $r$ is an open quotient map, where $X\sm\set{0}$ (resp., $\ray(X)$) is topologized as a subspace of $V$ (resp., $\ray(V)$).

This applies in particular to the set $X=\ol{\mc{Z}}\cap \mc{Q}=\mc{Z}_{\infty}\cup\set{0}=(\ol{\mc{Z}}\sm\mc{Z})\cup\set{0}$.
It follows that  $b$ factors  
 as $b=qp$ where  $p$, $q$  are as in \ref{x9.9}(f) except that $q$ is so far only known to be 
continuous.  To show $q$ is a closed quotient map,  recall the identification (used in the proof of Proposition \ref{x5.3}) of $\ray(\real_{\geq 0}\Pi)$  with the polytope $P=\conv(\widehat {\Pi})$, where $\wh \Pi:=\mset{\wh\al\mid \al\in \Pi}$,   in the affine hyperplane $H=\mset{v\in V\mid \mpair{\rho,v}=1}$. The identification is given by $\real_{\geq 0}\al\mapsto \wh\al$ for $0\neq \al\in\real_{\geq 0}\Pi$.  Under this identification, 
$\ray(\mc{Q}\cap \ol{\mc{Z}})$ identifies with the closed, hence  compact, subset $\mset{\wh\al\mid \al\in \mc{Z}_{\infty}}=\mc{Q}\cap \ol{\mc{Z}}\cap H$ of $P$. Since $q$ is a continuous  map between  compact Hausdorff spaces, it is closed. Since $b$ is surjective, so is $q$ and therefore $q$ is a quotient map as claimed. As it is  a composite of quotient maps, $b$ is a quotient map too.

\subsection{}\label{x9.17}  Retain the identification $\ray(\mc{Q}\cap \ol{\mc{Z}})\seq \ray(\real_{\geq 0} \Pi)= P$ from the end of the preceding proof. 
(One can naturally  identify $\ray(V)$ itself with the boundary of  any convex polytope $P'$  with $P$ as a face and $0$ as interior point, if desired.) 
The identification induces an identification of  $\ray(\ol{\mc{Z}})$ with a  convex subset $Z$ of $P$  and $\ray(\ol{\mc{Z}})$ with the compact convex set $\ol{Z}$.   Let $Q:=\ray(\mc{Q})$.
 The fibers of $q$ are precisely  the maximal isotropic faces 
(isotropic in the sense of being contained in $Q$) of $\ol{Z}$   and they are also  the  connected components of   $\ol{Z}\cap Q$. Further, $C$ is a quotient space of $\ol{Z}\cap Q$, identifying each connected component   to a point.
For  arbitrary generic 
universal based root systems, we leave the determination of the dimensions of these connected components, description of their  faces and face lattices etc as an open problem.   If however $\Phi$ is in addition assumed to be  weakly hyperbolic, the connected components of $\ol{Z}\cap Q$ reduce to points se questions can be answered easily (see  \cite{DHR} for some   properties of weakly hyperbolic root systems in general). 
\begin{cor} Suppose that $\Phi$ is weakly hyperbolic and generic universal.
\begin{num}
\item  Let $\bt=(\bt_{n})_{n\in \Nat}\in C$. 
Set $\bt_{n}':=s_{\bt_{0}}\cdots s_{\bt_{n-1}}(\bt_{n})\in \Phi_{+}$. Then the limit  
$\lim_{n\to \infty}\real_{\geq 0}\bt_{n}'=\real_{\geq 0} z_{\bt}$ exists in $\ray(V)$  and 
$\ray(b^{-1}(\bt)\cup\set{0})=\set{\real_{\geq 0}z_{\bt}}$. 
\item  $R_{0}=\ray(\mc{Q}\cap \ol{\mc{Z}})$.
 \item The   map $q\colon R_{0}\to C$ is a homeomorphism. \end{num}\end{cor}
 \begin{rem} It is easily checked that a generic universal root system is weakly hyperbolic if it has  rank three, or if it has rank at least four and all inner products of distinct simple roots are (sufficiently close to) equal. 
   \end{rem}
\begin{proof}
The  isotropic face $F_{\bt}$ is one-dimensional (i.e. it is a ray) by the discussion in \ref{x9.2}. Hence $\ray(F_{\bt})=\set{\real_{\geq 0}z_{\beta}}$ for any  non-zero  $z_{\bt}\in F_{\bt}$. The sequence $\real_{\geq 0}\bt_{n}$ has at least one limit point,  and any such limit point must be in $\ray(F_{\bt})$ by Proposition \ref{x9.16}. Part (a) follows.  One has $R_{0}\seq \ray(\mc{Q}\cap \ol{\mc{Z}})$ in general and (a) implies the reverse inclusion in (b). Finally,  (a) implies that the fibers of the quotient map $q$ are singletons, so (c) follows from (b).
\end{proof}

\subsection{} \label{x9.18} The proofs of the  above results about generic universal root systems use a number of ingredients which may be shown to have at least partial analogues for general root systems, but the situation in general is far more complex. 
  The following example illustrates some of the subtleties which occur. 
 \begin{exmp}   \label{e9.18}
   Suppose $(\Phi,\Pi)$ is such that   $\Pi=\mset{\al,\bt,\g,\delta,\epsilon}$ is  linearly independent 
and  that $\mpair{\al,\bt}$ and $\mpair{\delta,\epsilon}$ are less than $-1$ while 
 $\mpair{\bt,\g}$, $\mpair{\al,\g}$,  $\mpair{\g,\delta}$ and $\mpair{\g,\epsilon}$  are 
 negative and  the  pairings of all other  
  pairs of distinct roots are equal to zero. Write $\set{\al,\bt,\delta,\epsilon}=\Pi_{I}$ where $I\seq S$. Since 
  $\Pi_{I}$ has two components $\Pi_{J}:=\set{\al,\bt}$ and $\Pi_{K}:=\set{\g,\delta}$,  
  our general results  imply that
  $\mc{Z}\cap \real\Pi_{I}=\mc{Z}_{I}=\mc{Z}_{J}+\mc{Z}_{K}$ and similarly $\ol{\mc{Z}}\cap \real\Pi_{I}=\ol{\mc{Z}}_{I}=\ol{\mc{Z}}_{J}+\ol{\mc{Z}}_{K}$. 
  
  Note  $(\Phi_{I,}\Pi_{J})$, $(\Phi_{K},\Pi_{K})$ are (dihedral)  generic 
  universal; use for them the notation introduced previously for such 
  systems.  Then $\mc{Z}_{J}=\set{0}\cup \real_{>0}\set{u(\al,\bt), u(\bt,\al)} $ and $\ol{\mc{Z}}_{J}=\real_{\geq 0} 
  \set{u(\al,\bt), u(\bt,\al)}$. So $\ol{\mc{Z}}_{I}=\real_{\geq 0} \set{u(\al,\bt), u(\bt,\al), u(\delta,\epsilon), u(\epsilon,\delta)}$, which is  a  four-dimensional simplicial cone,  while  
  \begin{multline*} {\mc{Z}}_{I}=\real_{>0} \set{u(\al,\bt), u(\bt,\al), u(\delta,\epsilon), u(\epsilon,\delta)}\\
  \dot\cup\, \real_{>0} \set{u(\al,\bt), u(\bt,\al) } \,\dot\cup\,
  \real_{>0} \set{u(\delta,\epsilon), u(\epsilon,\delta)}\,\dot\cup\,\set{0}\end{multline*}
  is the union of the relative interior of $\ol{\mc{Z}}_{I}$,
   the relative  interiors of two  opposite two-dimensional faces  of 
   $\ol{\mc{Z}}_{I}$  and $\set{0}$.  Note that $\ol{\mc{Z}}_{I}\cap \mc{Q}$ is the union of the four two-dimensional faces of $\ol{\mc{Z}}_{I}$ whose relative interiors are not contained in $\mc{Z}_{I}$.
       Hence the relative interiors of the facets of 
   $\ol{\mc{Z}}_{I}$  consist of non-isotropic points; these  relative interiors 
   are contained  in $\ol{\mc{Z}}$ but not in $\mc{Z} $, so 
   $\ol{\mc{Z}}\not \seq \mc{Z}\cup\mc{Q}$ even though $W$ is 
   irreducible. The facets of $\ol{\mc{Z}}_{I}$ are faces of $\ol{\mc{Z}}$;  
   they  are neither  closures of faces of ${\mc{Z}}$, nor  totally    isotropic.
   
   The set of limit rays of positive root rays  from  $\Phi_{I}$ is   \begin{equation*}
   R_{W_{I},0}=
  \set{\real_{\geq 0}u(\al,\bt), \real_{\geq 0}u(\bt,\al), \real_{\geq 0}
  u(\delta,\epsilon), \real_{\geq 0}u(\epsilon,\delta)}.
   \end{equation*}
  This is properly contained in  $\ol{\mc{Z}}_{I}\cap \mc{Q}$.
  On the other hand, computations like those in Example \ref{x7.12}
 show that any ray in  $\ol{\mc{Z}}_{I}\cap \mc{Q}$ is  an element of  $R_{0}$.
 
   For generic choices of the inner products as above, the form $\mpair{-,-}$ on $\real\Pi$ has Witt index $2$ and 
  the four  above-mentioned  two-dimensional faces of $\ol{\mc{Z}}_{I}$  are therefore maximal totally isotropic faces of $\ol{\mc{Z}}$. Any one of the four extreme rays of $\ol{\mc{Z}}_{I}$ is equal to an  intersection of two of these faces, so not all pairs of distinct faces from amongst  these four maximal isotropic faces of $\ol{\mc{Z}}$  intersect in $\set{0}$. 
   
   In particular, the example shows  that the restriction to 
 generic points in  Question \ref{q9.6}(a),(e) is necessary in general (it is not necessary in the generic universal case)  and that  the restriction  to irreducible $\Phi$ is necessary in Question \ref{q9.6}(b).
       \end{exmp}
 \section{Facial structure of the imaginary cone}
\label{x10}
\subsection{} Recall from \ref{x2.8}   the notion of facial closure of a subset of $W$. 
The following lemma collects properties of facial closures under the assumptions \ref{x4.1}(i)--(iii).
\label{x10.1}
\begin{lem}
 \begin{num}
 \item A standard facial subgroup of a standard facial subgroup of $W$ is a  standard facial subgroup of $W$. 
  \item A facial subgroup of a facial subgroup of $W$ is a facial subgroup of $W$. 
 \item Any intersection of facial subgroups of $W$ is facial.
\item The facial closure of any  $X\subseteq W$  exists. It is 
 equal to  the intersection  of all  facial subgroups  containing $X$. It is also the unique  facial subgroup of $W$ containing $X$ and  of minimal rank amongst such facial subgroups.
\item Let $W_{i}$ for $i\in I$ be reflection subgroups of $W$, all with no finite type component.
Then the facial closure $W'$ of $\cup_{i}W_{i}$ has no finite type components.
\end{num}\end{lem}
\begin{proof}  Part (a) follows from \ref{x4.2}(f). For (b), let $W_{2}$ be facial in $W_{1}$ where $W_{1}$ is facial in $W$.
Say $W_{1}=yW_{1}'y^{-1}$ where $y\in W$ and  $W_{1}'$ is standard facial in  $W$.
Then $y^{-1}W_{2}y$ is facial in $y^{-1}W_{1}y=W_{1}'$, say $y^{-1}W_{2}y=xW_{2}'x^{-1}$ where $W_{2}'$ is standard facial in $W_{1}'$. Then $W_{2}=yxW_{2}'(yx)^{-1}$ where $yx\in W$ and $W_{2}'$ is standard facial in $W$ by (a). This proves (b).  Parts (c)--(d) follow from the discussion in \ref{x2.8}, since $W$ is assumed here to be of finite rank.

For (e), note first that $W'$ is finitely generated since it is parabolic in $W$.   Let $W''\subseteq W'$ be the reflection subgroup of $W'$ such that $\Pi_{W''}$ is equal to the union of the infinite type components of $\Pi_{W'}$. By \ref{x8.5}(c), $W'' $ is facial in
$W$. Now any infinite type component of $\Phi_{W_{i}}$ is contained in some infinite type 
 component of $\Phi_{W''}$ and hence it is contained  in one of the irreducible components of $\Phi_{W''}$.
It follows that $\Phi_{W_{i}}\subseteq \Phi_{W''}$  for all $i$ and hence that 
$W_{i}\subseteq W''$.
Since $W''$ is facial, we get $W'\subseteq W''$ by definition of facial closure, hence $W''=W'$.
   \end{proof}  
   \subsection{} \label{x10.2} The above implies that the set of facial subgroups of $W$, 
   ordered by inclusion, is a complete lattice. We now define a certain subposet $\mc{W}$ of 
   this lattice,
   which plays an important role in this section. The subposet 
   is itself a complete lattice with join given by restriction of the join in the family of facial 
   subgroups. 

 Let $\mc{W}=\mset{wW_{I}w^{-1}\mid w\in W, I\in F_{\mathrm{inf}} }$  denote the set of  all 
 special facial subgroups of $W$ i.e facial subgroups  with no finite irreducible components.  
 We always consider $\mc{W}$ with partial order given by inclusion. For any parabolic 
 subgroups  $W'\subsetneq W''$ of $W$, the rank of $W'$ is strictly less than that
of $W''$. It follows than any non-empty subset of $\mc{W}$ has at least one  minimal element 
and at least one  maximal element, and any chain $W_{0}\subsetneq \ldots \subsetneq
  W_{n}$ in $\mc{W}$ has length $n$ bounded by the rank of $W$.
\begin{prop}\begin{num}\item  The family $\mc{W}$ of subgroups of $W$ forms a complete 
lattice.
\item The join of a subset of  $\mc{W}$ is the facial closure of its union.
\item The meet of  a family $W_{i}$ for a $I$ in any index set $I$ is $W'$, where 
$W'\in \mc{W}$ is such that  $\Pi_{W'}$
is the union of all infinite type components of $\Pi_{\cap_{i}W_{i}}$.
\end{num}\end{prop}
\begin{proof} Let $W_{i}$ for $i\in I$ be a family of elements of $\mc{W}$.
By  Lemma \ref{x10.1}(e),  the facial closure $W'$ of $\cup_{i}W_{i}$  is a facial subgroup  which has no infinite type components; clearly, $W'$ is the minimum element of $\mc{W}$ containing all $W_{i}$. This proves that $\mc{W}$ is a complete  lattice with join as in (b).
Now we consider the meet of the family $\set{W_{i}}_{i\in I}$. Note that $W'$ as defined in (c) is in $\mc{W}$ by Lemma \ref{x10.1}(c) and \ref{x8.5}(c). Clearly, $W'\subseteq W_{i}$ for all $i$.
 If $W''\in \mc{W}$ with $W''\subseteq W_{i}$ for all $i$, then  $W''\subseteq \cap_{i}W_{I}$. Here, $\cap_{i}W_{i}$ is a facial subgroup. Hence each (necessarily infinite type) component of $\Phi_{W''} $ is contained in some (infinite type) component of $\Phi_{ \cap_{i}W_{I}}$, and hence it is contained in $\Phi_{W'}$.
This implies that $\Phi_{W''}\subseteq \Phi_{W'}$ and $W''\subseteq W'$. That is, $W'$ is the meet of $\set{W_{i}}_{i\in I}$ as required.
\end{proof}
\begin{rem} Given two elements of $\mc{W}$ each with a given expression as a conjugate
of a standard facial subgroup, one can compute a similar expression for their meet and join 
in $\mc{W}$.  In fact, as in the proof of 
\ref{x2.17}(c), one can reduce easily to the case in which the two subgroups are $W_{J}$ 
and $dW_{K}d^{-1}$ where $J$, $K$ are special facial subsets of $J$ and  $d$ is of minimal 
length  in $W_{J\cup J^{\perp}} dW_{K\cup K^{\perp}}$. Then by Kilmoyer's formula 
(Proposition \ref{x2.12}) their meet   is  $W_{L}$ where $L:=(J\cap dKd^{-1})_{\infty}$ is the 
union of all infinite type components of $J\cap dKd^{-1}$. By  Lemma \ref{x2.17}  and 
Lemma \ref{x8.5}, their join is $W_{M}$ where $W_{M}$ is the standard facial closure of  
$J\cup \set{d}\cup K$ i.e. $M$ is the smallest standard facial subset of $S$ containing $J\cup K\cup\set{s_{1},\ldots, s_{n}}$ for some reduced  expression $d=s_{1}\cdots s_{n}$ of $d$. 
We do not know if  the standard facial and standard parabolic closures of  $J\cup \set{d}\cup K$ coincide under these conditions. \end{rem}

\subsection{}\label{x10.3}
Henceforward, extensive use will be made of the various notions of 
faces of cones as described in Appendix A, and  especially the notions 
 of semidual and dual pairs of pointed cones (\ref{xA.6}--\ref{xA.9}).
 
 By Theorem \ref{x5.1}, $(\ol{\mc{Z}},\ol{\mc{X}})$ is a semidual pair of  stable cones (with respect to the pairing $\mpair{-,-}$). It follows trivially  that $P=({\mc{Z}},{\mc{X}})$ is also a semidual pair of (pointed) cones for the same pairing, satisfying the condition \ref{xA.9}(i).  Following \ref{xA.6}, we define the Galois connection corresponding to $P$ between the  power sets
 $\mc{P}(\mc{Z})$ and  $\mc{P}(\mc{X})$. It is specified by order-reversing maps
 \begin{equation*}
 A\mapsto A^{\dag}:=\mc{X}\cap A^{\perp}\colon \mc{P}(\mc{Z})\to\mc{P}(\mc{X})
 \end{equation*}
 and 
  \begin{equation*}
 B\mapsto B^{\#}:=\mc{Z}\cap B^{\perp}\colon \mc{P}(\mc{X})\to\mc{P}(\mc{Z}).
 \end{equation*}
 The complete lattices $\Ext_{\neq \eset}(\mc{Z})$ and 
 $\Ext_{\neq \eset}(\mc{X})$ of non-empty faces of $\mc{Z}$ and 
 $\mc{X}$ are defined as in \ref{xA.3}. They contain (as subposets) 
 the complete lattices of stable subsets (which we call stable faces)  of $\mc{Z}$ and $\mc{X}$ for the above Galois connection.
 
 For any reflection subgroup $W'$ of $W$, define    $\mc{X}(W'):=\mc{X}\cap \Pi_{W'}^{\perp}=\mset{ x\in \mc{X}\mid x^{\perp}\cap \Phi\supseteq\Phi_{W'}}$.  Observe that for any $x\in \mc{X}$, $x^{\perp}\cap \Phi=\Phi_{W''}$ for some facial reflection subgroup $W''$. Hence $x\in \mc{X}_{W'}$ if and only if   $W''\supseteq W'$. The sets $\mc{X}(W')$ are primarily of interest here for $W'\in \mc{W}$.

 The following theorem summarizes the main results of this section.
 \begin{thm}
 \begin{num}
 \item $P=(\mc{Z},\mc{X})$ is a dual pair of cones in the sense of 
 $\text{ \rm Definition }$ $\text{\rm \ref{xA.9}}$.
 \item The map $W'\mapsto \mc{Z}_{W'}$ induces an isomorphism of complete lattices $\mc{W}\to \Ext_{\neq \eset}(\mc{Z})$.
  \item The map $W'\mapsto \mc{X}(W')$ induces an isomorphism of complete lattices $\mc{W}\to \Ext_{\neq \eset}(\mc{X})\op$ (where $\op$ indicates the opposite poset).
  \item For $W'\in \mc{W}$, $\mc{Z}_{W'}^{\dag}=\mc{X}(W')$ and
  $\mc{X}(W')^{\#}=\mc{Z}_{W'}$.
 \end{num}
 \end{thm}
 The proof   of  the theorem will occupy \ref{x10.4}--\ref{x10.11}. 
 \begin{rem} (1) For $W'\in \mc{W}$,  $\mc{Z}_{W'}$ is both the imaginary  cone  of $W'$ and a face of  
$\mc{Z}$. However, 
the Tits cone   of  $W'$ is $\mc{X}_{W'}\supseteq \mc{X}$, whereas
$\mc{X}({W'})\subseteq \mc{X}$ is a face of $\mc{X}$.

(2)  We discuss here the main known results in the Kac-Moody setting (these are analogous to, but not consequence of, the  results established  about the Tits cone here  in the special case in  which the simple roots are linearly independent).
  In the Kac-Moody setting,  \cite{Loo} studied  exposed faces of the Tits cone attached to special  standard 
  parabolic subgroups, obtaining in particular formulae for their setwise and pointwise stabilizers. In \cite[Kapitel 6]{Sl1} (see also \cite{Sl2}) it was 
  shown that    the exposed  faces  of the Tits cone are the $W$-translates of the special exposed faces defined in \cite{Loo}. Amongst other results,  \cite[4.1]{Mok1} (see also \cite{Mok2}) gave another 
  proof of this   showing as well that any  non-empty   face (extreme subset)  of the Tits cone is  exposed.  An algebraic 
description of meet and join of two  faces of the Tits cone was given in \cite[Theorem 4(b)]{Mok6} (the analogue here is given  as Remark \ref{x10.2}). 
 \end{rem}

\subsection{}\label{x10.4}  As in \ref{x5.7}, $F_{\mathrm{inf}} $ denotes the set of special facial subsets of $S$ (i.e.  with no finite type components).
  For  $I\subseteq S$, 
let $I^{\perp}:=\mset{r\in  S\mid r\not\in I,\  rs=sr\text{ for all }s\in I}$ and let $\tilde I\subseteq I$ be such that
  $\Pi_{\tilde I}$ is the union of  the infinite type  components of $\Pi_{I}$.

\begin{lem}\begin{num}
\item For any  facial subset $I$  of $S$,  $\tilde I$ is a facial subset of $S$,  and $\mc{Z}_{\tilde I}=\mc{Z}_{I}$.

 \item The  sets $w\mc{Z}_{I}$ where $I\in F_{\mathrm{inf}} $ and $w\in W$ are  faces of $\mc{Z}$.
  \item If $I\in F_{\mathrm{inf}} $, there exists $p\in \mc{K}_{I}$ such that the facial support of $p$ is
 $\Pi_{I}$ and such that for any $\alpha\in \Pi$, one has $\mpair{p,\alpha}=0$ if and only if  either
 $\alpha$ is in an affine component of $\Pi_{I}$ or $\alpha\in \Pi_{I^{\perp}}$.
 \item  Let  $I,J\in F_{\mathrm{inf}} $ with $I\seq J$. Then 
  the normalizer of $W_{I}$ is  $N(W_{I})=W_{I\cup I^{\perp}}=W_{I}W_{I^{\perp}}=W_{I^{\perp}}W_{I}$ and $N(W_{I})N(W_{J})=W_{I^{\perp}}W_{J}$. 
  \item If $x,y\in W$ and $I,J\in F_{\mathrm{inf}} $, the   conditions $\text{\rm (i)--(iii)}$ below are equivalent: \begin{subconds}
 \item $x\mc{Z}_{I}\subseteq y\mc{Z}_{J}$.
 \item  $I\subseteq J$ and $x^{-1}y\in N(W_{I})N(W_{J})$.
 \item $xW_{I}x^{-1}\seq yW_{J}y^{-1}$.
 \end{subconds}
  \end{num}
\end{lem}
\begin{proof} Part (a) follows from \ref{x8.5}(c) and \ref{x3.2}(d).
Now  $\mc{Z}=\mc{Z}_{S}$ itself is a stable face of $\mc{Z}$. For any 
proper  facial subset  $I\subseteq S$, we have (with notation as in
 \ref{x5.8}) $\mc{Z}\subseteq \real_{\geq 0}\Pi\subseteq
  H_{\phi_{I}}^{\geq 0}$  and $\mc{Z}\cap H_{\phi_{I}}^{=}=
  \mc{Z}\cap \real\Pi_{I}=\mc{Z}_{I}$  is a face (in fact, an absolutely exposed face, in the terminology of Appendix \ref{xA}) of $\mc{Z}$ by
   \ref{x3.4}(b). Since $W$ acts on $\mc{Z}$, $w\mc{Z}_{I}$ is a stable 
   face of $\mc{Z}$ for any $w\in W$ and facial $I\subseteq S$.
By (a), $w\mc{Z}_{I}=w\mc{Z}_{\tilde I}$ proving (b).

 Choose $p\in \mc{K}_{I}$ to satisfy (c) as follows.
Let $I_{1},\ldots, I_{n}$ be subsets of $I$ such that $\Pi_{I_{j}}$ are the irreducible components of $\Pi_{I}$. For each $i=1,\ldots, n$, choose by \ref{x4.5} some  $0\neq p_{i} \in \mc{K}_{i}$ such that 
 $\Pi_{I_{i}}$ is a support of  $p_{i}$ and, if $\Pi_{I_{i}}$ is indefinite,
$\mpair{p_{i},\alpha}<0$ for all $\alpha\in \Pi_{I_{i}}$.  If $\Pi_{I_{i}}$ is affine, then  $p_{i}$ necessarily is a representative of  its isotropic ray. Set $p=p_{1}+\ldots +p_{n}$.
Clearly, $\Pi_{I}$ is a support of $p$; hence $\Pi_{I}$ is the facial support of $p$ since $I$ is facial. We have $\mpair{p,\alpha}\leq 0 $ for all $\alpha\in \Pi\setminus \Pi_{I}$, with equality if and only if  $s_{\alpha}\in I^{\perp}$.
Summarizing, for any  $\alpha\in \Pi$, we have 
\begin{equation*} \begin{cases}\mpair{p,\alpha} =0,&\text{ if $\alpha\in \Pi_{I^{\perp}}$, or 
                                                                       $\alpha\in \Pi_{I_{i}}$ with $\Pi_{I_{i}}$ affine}\\
                                                                    \mpair{p,\alpha}   <0,&\text{otherwise.}\end{cases}\end{equation*} This proves (c).

It is well known (see \cite{BH2}) that   for any reflection subgroup $W'$,
one has $N(W')=W' N_{W}(\Pi_{W'})$ where  
$N_{W}(\Pi_{W'})=\mset{w\in W\mid w\Pi_{W'}=\Pi_{W'}}$ 
(this also follows from \ref{x1.6}). Since all irreducible components of 
$\Pi_{I}$ are of  infinite type, it is an easy consequence of \ref{x1.15}
 (see \cite{DyRig}) that  if $x\in W$ and  $x(\Pi_{I})=\Pi_{J}$ where 
 $J\seq S$, then $I=J$ and $x\in W_{I^{\perp}}$.  Hence
$N_{W}(\Pi_{I})=W_{I^{\perp}}$, $N(W_{I})=W_{I}N(\Pi_{I})=W_{I}W_{I^{\perp}}=W_{I\cup I^{\perp}}$ and  \begin{equation*}
N(W_{I})N(W_{J})=W_{I^{\perp}}W_{I}W_{J^{\perp}}W_{J}=
W_{I^{\perp}}W_{J^{\perp}}W_{I}W_{J}=W_{I^{\perp}}W_{J}
\end{equation*} if $I\seq J$.  This proves (d).

Next, we prove  (e).  To show (i) implies (ii),
it will suffice to show that if $\mc{Z}_{I}\subseteq y\mc{Z}_{J}$,
then $I\seq J$ and  $y\in N(W_{I})N(W_{J})$.  Choose $p$ as in (c).
 Then  $p\in \mc{Z}_{I}\subseteq 
y\mc{Z}_{J}=yW_{J}\mc{K}_{J}$, so $p=ywk$ for some 
 $w\in W_{J}$ and $k\in \mc{K}_{J}$.
Since $\mc{K}_{I},\mc{K}_{J} \subseteq \mc{K}\subseteq -\mc{C}$ by 
Proposition \ref{x3.4}(a), we get $p,k\in -\mc{C}$. This implies by 
Lemma \ref{x1.10}  that $p=k$  and $yw\in W_{L}$ where 
$\Pi_{L}=\mset{\alpha\in \Pi\mid \mpair{p,\alpha}=0}$.  Clearly  
$W_{L}\seq N(W_{I})$ so $yw\in N(W_{I})$.
Therefore   $y=(yw)w^{-1}\in N(W_{I})W_{J}\subseteq N(W_{I})N(W_{J})
$.
Also, since $p=k\in \mc{K}_{J}$,  $p$ has a support $I$, and $J$ is 
facial, it follows  that $I\subseteq J$.  This completes the proof that (i) implies (ii). It is easy to see that (ii) implies (iii).  If (iii) holds, then $x\mc{Z}_{I}=\mc{Z}_{xW_{I}x^{-1}}\seq\mc{Z}_{yW_{J}y^{-1}}=y\mc{Z}_{J}$, by Proposition \ref{x3.4}(a) and Theorem \ref{x6.3}, so (iii) implies (i). 
This proves (e).
\end{proof}

\subsection{} \label{x10.5} Partially order the power sets  $\mc{P}(\mc{Z})$ and $\mc{P}(\mc{X})$ by inclusion. \begin{lem}  \begin{num}\item The map $W'\mapsto \mc{Z}_{W'}\colon \mc{W}\xrightarrow{\cong}\mset{\mc{Z}_{W'}\mid W'\in \mc{W}}$ is a poset isomorphism.
\item  The map $W'\mapsto\mc{X}(W')\colon \mc{W}\op\xrightarrow{\cong}\mset{\mc{X}(W')\mid W'\in \mc{W}}$  is a poset  isomorphism.
\end{num}
\end{lem}
\begin{proof} Observe that $\mc{W}=\mset{wW_{I}w^{-1}\mid w\in W, I\in F_{\mathrm{all}} }$. Therefore, (a) follows directly from the equivalence of (i) and (iii) in
Lemma \ref{x10.4}(e). For (b), note that for any facial subgroup $W'$ of $W$, there exists $z\in \mc{X}$ such that $\Phi\cap z^{\perp} =\Phi_{W'}$. It follows from this  that for facial subgroups $W',W''$ of $W$,  one has $\mc{X}(W')\subseteq \mc{X}(W'')$ if and only if  $W''\subseteq W'$. Restricting to $W',W''\in \mc{W}$ gives (b).
\end{proof}
  
\subsection{}\label{x10.6}
 In \ref{x10.6}--\ref{x10.7}, we fix $W'\in \mc{W}$ and let facial support of elements of $\real_{\geq 0}\Pi_{W'}$ be taken with respect to $\Pi_{W'}$ unless otherwise indicated.
 \begin{lem} For any $z\in \mc{Z}_{W'} $, there are two mutually exclusive possibilities: either $z$ lies in  $\mc{Z}_{W''}$ for some  $W''\in \mc{W}$ with  $W''\subsetneq W'$, or $wz$ has facial support $\Pi_{W'}$ for all $w\in W'$.
 \end{lem}
\begin{proof}
Using the $W'$ action, assume without loss of generality that 
 $z=k\in \mc{K}_{W'}$.
Observe that the  facial support $\Delta$ of $k$ (with respect to $\Pi_{W'}$)   is of the form
$\Delta=\Pi_{W''}$ for some $W''\in \mc{W}$ with $W''$ standard parabolic (even  standard facial) in  $W'$.
In fact, let $\Delta'$ (resp., $\Delta''$)  be the union of the infinite (resp., finite) type components  of $\Delta$.
Since $\Delta$ is standard facial in $\Pi_{W'}$, and $k\in \mc{K}$, it 
follows using Lemma \ref{x3.4}, Proposition \ref{x3.2}(d) and 
\ref{x4.5}  that \[k\in\real_{\geq 0}\Delta \cap \mc{K}_{W'}=
\mc{K}_{W_{\Delta}}=\mc{K}_{W_{\Delta'}}+\mc{K}_{W_{\Delta''}}=
\mc{K}_{W_{\Delta'}}.\] One has (e.g. by \ref{x2.13} and \ref{x2.9}) that  
$\real \Delta=\real\Delta'\oplus \real \Delta''$ (orthogonal direct sum) 
where $ \Delta''$ is linearly independent. 
 This implies that any support of $k$ is contained in $\Delta'$, so
  $\Delta'=\Delta$ has no finite type components.
 Then $W''=W_{\Delta}\in \mc{W}$ by Lemma \ref{x10.1}(b).
Now if $W''\neq W'$, then $k$ lies in a proper face $\mc{Z}_{W''}$ of 
$\mc{Z}_{W'}$
and has proper facial support $\Delta\subsetneq\Pi_{W'}$ in $\Pi_{W'}$
 On the other hand, if $W''=W'$,
then $wk$ has facial support $\Pi_{W'}$  for all $w\in W'$, since 
$wk-k\in \real_{\geq 0}\Pi_{W'}$
by \ref{x1.10}. It is clear from the proof that the two possibilities 
mentioned in the statement of the Lemma are mutually exclusive, and 
the proof of the Lemma is complete.
\end{proof}

\subsection{}\label{x10.7} The next lemma,   in which  for conciseness, $C^{\circ}:=\ri(C)$ for any cone $C$,  characterizes the relative interior of  $\mc{Z}_{W'}$.  \begin{lem} Let  $W'\in \mc{W}$ and $z\in  \mc{Z}_{W'}$. Write  $z=wk$ with $w\in W'$ and $k\in \mc{K}_{W'}$ Then the following conditions are equivalent:
\begin{conds}\item  $z\in \mc{Z}_{W'}^{\circ }$.
\item   $wz$ has facial support $\Pi_{W'}$ for all $w\in W'$. 
\item $wz$ has facial support $\Pi_{wW'w^{-1}}$ with respect to $\Pi_{wW'w^{-1}}$,
for all $w\in W$.
\item  the facial support of $k$ with resect to $\Pi_{W'}$ is $\Pi_{W'}$.  \end{conds}\end{lem}

\begin{proof}   Making use of the $W'$-action, we may assume that $z=k\in \mc{K}_{W'}$. 
Since no point of a proper face of $\mc{Z}_{W'}$ is in $\mc{Z}_{W'}^{\circ }$, the preceding Lemma shows that (i) implies (ii).  Clearly (ii) implies (iv).
Conditions (ii) and (iii) are easily seen to be equivalent using 
\ref{x1.6}--\ref{x1.7} and Proposition \ref{x3.2}(e).  To complete the proof, we show that (iv) implies (i). Using the $W'$-action, it will suffice to show that $k\in \mc{Z}_{W'}^{\circ }$.
Now by definition of $\mc{W}$, $\Pi_{W'}$ has only infinite type components. Let
$\Omega$ denote the union of all indefinite type components of $\Pi_{W'}$ and 
$\Gamma$ denote the union of all other (affine type) components of $\Pi_{W'}$.
Write $k=\sum_{\alpha\in \Pi_{W'}}c_{\alpha }\alpha$ where all $c_{\alpha}>0$. By
the proof of \ref{x3.2}(d), $k_{\Omega}:=\sum_{\alpha\in \Omega}c_{\alpha}\alpha\in \mc{K}_{W_{\Omega}}$ and $k_{\Gamma}:=\sum_{\alpha\in \Gamma}c_{\alpha}\alpha\in \mc{K}_{W_{\Gamma}}$. Further, from Lemma \ref{x8.4}, it follows that 
for each component $\Gamma'$ of $\Gamma$, $k_{\Gamma'}:=\sum_{\alpha\in \Gamma'}c_{\alpha}\alpha$ is a representative of the isotropic ray $\mc{K}_{W_{\Gamma'}}$  of $W_{\Gamma'}$.
From \ref{x8.6}, $\mc{Z}_{\Gamma}=\mc{K}_{\Gamma}=\sum_{{\Gamma'}}\mc{K}_{W_{\Gamma'}}$ and so clearly  $k_{\Gamma}\in \mc{Z}_{\Gamma}^{\circ }$.
It will be enough to show that $k_{\Omega}\in \mc{Z}_{\Omega}^{\circ }$.
For then  $k=k_{\Omega}+k_{\Gamma}\in   \mc{Z}_{\Omega}^{\circ }+\mc{Z}_{\Gamma}^{\circ }\subseteq \mc{Z}_{W'}^{\circ }$ (using \ref{x3.2}(d) and \eqref{xA.3.14}).

To prove that  $k_{\Omega}\in \mc{Z}_{\Omega}^{\circ }$, we may and do assume without loss of generality that $\Pi=\Pi_{W'}=\Omega$, so $k_{\Omega}=k$ and $\Pi_{W'}$  has only indefinite components. By \ref{x4.5},  $\real \mc{Z}_{W'}=\real \Pi_{W'}$ and by Lemma \ref{x8.4}(d), $\Delta:=k^{\perp}\cap \Pi_{W'}$ has only finite type  components. Let $W'':=W_{\Delta}$. We choose an open neighborhood $U$ of $k$ in $V$
which does not intersect the hyperplane  $\alpha^{\perp}$ for any $\alpha\in \Pi_{W'}\setminus \Delta$, and such that $U':=U\cap \real\Pi_{W'}\subseteq \real_{\geq 0}\Pi_{W'}$ (which is possible since $k\in(\real_{\geq 0}\Pi_{W'})^{\circ }$). Since $k$ is $W''$-stable, we may assume without loss of generality that $U$ (and hence $U'$) is $W''$-stable also. We have $\mpair{\Pi_{W'}\setminus \Delta,U}\subseteq \real_{<0}$ since 
$\mpair{\Pi_{W'}\setminus \Delta,k}\subseteq \real_{<0}$.
 We claim that $U'$ is an open neighborhood of $k$ in $\real\mc{Z}_{W'}$ with $U'\subseteq \mc{Z}_{W'}$. By construction, $U'$ is an open neighborhood of $k$ in $\real\mc{Z}_{W'}$. To see that  $U'\subseteq \mc{Z}_{W'}$, recall that since $W''$ is finite, we have $\mc{X}_{W''}=\cup_{w\in W''}w\mc{C}=V$ by \ref{x1.10}(h).  So if $u\in U'$, 
 there is $u'\in -\mc{C}_{W''}$ and $w\in W''$ with $u=wu'$.
 By $W''$-invariance of $U'$, we have $u'\in U'$ also, so $\mpair{\Pi_{W'}\setminus
 \Delta ,u'}\subseteq \real_{<0}$. We have  $\mpair{u',\Delta}\subseteq \real_{\leq 0}$ since $u'\in  -\mc{C}_{W''}$. It follows that $u'\in \real_{\geq 0}\Pi_{W'}\cap -\mc{C}_{W'}=\mc{K}_{W'}$
 and $u=wu'\in W'\mc{K}_{W'}=\mc{Z}_{W'}$ since $w\in W''\subseteq W'$. This 
 shows that $U'\subseteq \real{Z}_{W'}$ as claimed, and completes the  proof of the Lemma.
\end{proof}

\subsection{}\label{x10.8}  The statement below uses the notation $\mc{C}(I)$ defined in \ref{x2.3}.
\begin{lem}  Let  $W'\in \mc{W}$, and  write $W'=wW_{I}w^{-1}$ for some 
$w\in W$ and $I\in F_{\mathrm{inf}} $. \begin{num}\item 
Let  $z\in \mc{Z}_{W'}^{\circ }$.   Then $  \set{z}^{\dag}=(\mc{Z}_{W'})^{\dag}=\mc{X}\cap (\real \Pi_{W'})^{\perp}=\mc{X}({W'})$. 
\item  Choose $v\in \mc{C}(I)$. 
Then $ \mc{Z}_{W'}= (\mc{X}(W'))^{\#}= (w\mc{C}(I))^{\#}= \set{wv}^{\#}$.
\item $\mc{Z}_{W'}$  is a stable face of $\mc{Z}$
dual (with respect to the Galois connection from $P$) to the stable face  $\mc{X}(W')$ of $\mc{X}$. \end{num} 
\end{lem}
\begin{proof}
In (a),  it is obvious from the definitions  that 
\[ \set{z}^{\dag}\supseteq (\mc{Z}_{W'})^{\dag}\supseteq \mc{X}\cap (\real \Pi_{W'})^{\perp}= \mc{X}({W'}).\]
Using the $W$-action, we may assume for the proof of the reverse inclusions that
$W'$ is standard facial, say $W'=W_{I}$ where $I\in F_{\mathrm{inf}} $.
Let $x\in \set{z}^{\dag}=\mc{X}\cap z^{\perp}$. Choose $w\in W$ so $wx\in \mc{C}$.
Thus, $wx\in (wz)^{\perp}$. By Lemma \ref{x10.7},   $wz=\sum_{\alpha\in \Pi_{wW'w^{-1}}}c_{\alpha}\alpha$ with all $c_{\alpha}>0$. It follows that $wx\in (\Pi_{wW'w^{-1}})^{\perp}$ and so
$x\in( \Pi_{W'})^{\perp}\cap \mc{X}$. This proves (a).
In (b), we have  $v\in \mc{C}_{I}\seq \mc{C}\cap\Pi_{I}^{\perp}\seq \mc{X}\cap \Pi_{I}^{\perp}=\mc{X}(W_{I})$ and thus
\begin{multline*} \mc{Z}_{W_{I}}\subseteq (\mc{X}(W_{I}))^{\#}\subseteq (\mc{C}(I))^{\#}\subseteq \set{v}^{\#}=\mc{Z}\cap v^{\perp}=\\ \mc{Z}\cap \real_{\geq 0}\Pi\cap v^{\perp}=\mc{Z}\cap \real_{\geq 0}\Pi_{I}=\mc{Z}_{W_{I}}\end{multline*} where  the first inclusion comes from (a)
and the last equality from Lemma \ref{x3.4}. Hence equality holds throughout. Acting on the resulting equation  by $w$ gives (b).
Part (c) is immediate from (a)--(b) and the definitions.
\end{proof} 

\subsection{} \label{x10.9}
For $W'\in \mc{W}$, let $\mc{X}'(W')\subseteq \mc{X}(W')$ denote  the  set of all   $x\in \mc{X}$ such that  $x^{\perp}\cap \Phi =\Phi_{W''}$ where $W''$ is the facial reflection subgroup  such that  $\Pi_{W'}$ is the union of all infinite type  components of $\Pi_{W''}$. It is easy to see that  for $W'\in \mc{W}$, we have 
  \begin{equation*}\mc{X}({W'})=\dot \bigcup_{\substack{W''\in \mc{W}\\ W''\supseteq W'}}\mc{X}'({W''}).\end{equation*}

\begin{lem} If  $W'\in \mc{W}$,  then $\mc{X}(W')^{\circ }=\mc{X}'(W')$.  \end{lem} 
\begin{proof}
 Since $\mc{X}'(W'')$ is in the proper face $\mc{X}(W'')$  (and therefore in the relative boundary)  of $\mc{X}(W')$ for  $W''\in \mc{W}$ with $W''\supsetneq W'$, we have $\mc{X}(W')^{\circ }\subseteq\mc{X}'(W')$. Hence it remains to prove that if $v\in \mc{X}'({W'})$, there is an open neighborhood $U$ of $v$ in $\mc{X}(W')$ such that $U\subseteq \mc{X}'(W')$.
Choose $w\in W$ with $x:=wv\in \mc{C}$. Let $W_{I}$ be the stabilizer of  $x$ in $W$
and $\tilde I\subseteq I$ be as in \ref{x10.4}. Then $wW'w^{-1}=W_{\tilde I}$, $x=wv\in w\mc{X}'({W'})=\mc{X}'({W_{\tilde I}})$
and $ w\mc{X}({W'})=\mc{X}({W_{\tilde I}})$.   

Set $J=I\setminus \tilde I$, so $W_{J}$ is finite and $J$ is the  union of the finite type connected components of ${I}$. We have $\Pi_{I}=x^{\perp} \cap \Pi$. Hence we  may 
choose an open neighborhood $U'$ of $x$ in $V$ such that $\mpair{U',\Pi\setminus \Pi_{I}}\subseteq \real_{>0}$. Since $W_{J}$ stabilizes $x$, we may assume without loss of generality, 
that $U'$ is $W_{J}$-invariant. 

We claim that $U'':=U'\cap \mc{X}(W_{\tilde I})$ is an open
neighborhood of $x$ contained in $\mc{X}'(W_{\tilde I})$. Clearly, $U''$ is open in 
$\mc{X}(W_{\tilde I})$. Let $p\in {U}''$. Since $W_{J}$ is finite, we have $\mc{X}_{W_{J}}=V\ni p$,
so there is some $y\in W_{J}$ such that $yp\in \mc{C}_{W_{J}}$.
This means that $\mpair{yp, \Pi_{J}}\subseteq \real_{\geq 0}$.
Also, $\mpair{yp,\Pi\setminus \Pi_{I}}\subseteq \real_{> 0}$ since  $yp\in U'$ by $W_{J}$-invariance of $U'$. We have $p^{\perp}\cap \Phi=\Phi_{W'''}$  for some facial reflection subgroup
$W'''$. Since $p\in \mc{X}(W_{\tilde I})$,  we have $W'''\supseteq W_{\tilde I}$.
Hence $(yp)^{\perp}\cap \Phi=y\Phi_{W'''}\supseteq y \Phi_{\tilde I}=\Phi_{\tilde I}$ since $y\in W_{J}$. This shows that $\mpair{yp,\Pi_{\tilde I}}=0$. 
Since $\Pi=\Pi_{J}\cup( \Pi\setminus \Pi_{I})\cup \Pi_{\tilde I}$, 
we have now seen that $yp\in \mc{C}\subseteq \mc{X}$ with $\Pi_{\tilde I}\subseteq (yp)^{\perp}
\cap \Pi \subseteq \Pi_{I}$. Hence $yp\in \mc{X}'(W_{\tilde I})$ and $p\in\mc{X}'(y^{-1}W_{\tilde I}y) =
\mc{X}'(W_{\tilde I})$. We conclude that $U''\subseteq \mc{X}'(W_{\tilde I})$ as claimed.
Finally, $U:=w^{-1}U''$ is an open neighborhood of $w^{-1}x=v$ in $w^{-1}(\mc{X}(W_{\tilde I}))=\mc{X}(W')$ such that $U\subseteq w^{-1}(\mc{X}'(W_{\tilde I}))=\mc{X}'(W')$.
This completes the proof.
\end{proof}

\subsection{} \label{x10.10} The following is the last  lemma required for the proof of Theorem \ref{x10.3}.
\begin{lem} \begin{num} \item 
The sets $\mc{X}(W')^{\circ }$ for $W'\in \mc{W}$ are pairwise disjoint and have union $\mc{X}$. 
Further, for $W'\in \mc{W}$, we have 
  \begin{equation*}\mc{X}({W'})=\dot \bigcup_{\substack{W''\in \mc{W}\\ W''\supseteq W'}}\mc{X}({W''})^{\circ }.\end{equation*}
\item  The sets  $\mc{Z}_{W'}^{\circ }$ for $W'\in \mc{W}$ are pairwise disjoint, and have union $\mc{Z}$. Further,   for $W'\in \mc{W}$,
\[ \mc{Z}_{W'}=\dot \bigcup_{\substack{W''\in \mc{W}\\ W''\subseteq W'}}\mc{Z}_{W''}^{\circ }.\]
 \end{num} 
\end{lem}
\begin{proof} Part (a) follows directly from  \ref{x10.9} and the definitions.

For (b), suppose first that $z\in \mc{Z}_{W'}^{\circ }\cap \mc{Z}_{W''}^{\circ }$ where $W',W''\in \mc{Z}_{W'}$.
Without loss of generality, $W'\not\subseteq W''$. By Corollary \ref{x6.4},  $z\in \mc{Z}_{W'}\cap \mc{Z}_{W''}=\mc{Z}_{W'''}$ where $W'''\in \mc{W}$ has $\Pi_{W'''}$ equal to the union of the infinite type components of $\Pi_{W'\cap W''}$. But then, since $W'''\subsetneq W'$, it follows  by Lemma  \ref{x10.7}  that $z\not\in \mc{Z}_{W'}^{\circ }$, a contradiction. 
Hence the sets $\mc{Z}_{W'}^{\circ }$ for $w'\in \mc{W}$ are pairwise disjoint.

Now let $z\in \mc{Z}_{W'}$ where $W'\in \mc{W}$. Choose a minimal element $W''\in \mc{W}$
with $W''\subseteq W'$ and  $z\in \mc{Z}_{W''}$. By  Lemma  \ref{x10.7} again,
we have $z\in \mc{Z}_{W''}^{\circ }$.  Hence 
\[ \mc{Z}_{W'}\subseteq\dot \bigcup_{\substack{W''\in \mc{W}\\ W''\subseteq W'}}\mc{Z}_{W''}^{\circ }.\]
The reverse inclusion holds since for $W''\in \mc{W}$ with $W''\seq W'$, $\mc{Z}_{W''}^{\circ }\subseteq \mc{Z}_{W''}\subseteq \mc{Z}_{{W'}}$ by Theorem \ref{x6.3}.

\end{proof}
\subsection{Proof of Theorem \ref{x10.3}}\label{x10.11}
 Lemma \ref{x10.8}(c) states that for $W'\in \mc{W}$, $\mc{Z}_{W'}$  is a stable face of $\mc{Z}$
dual  to the stable face 
 $\mc{X}(W')$ of $\mc{X}$. Now by \eqref{xA.3.6}, $\mc{Z}$ is the disjoint union of the
  relative interiors of its non-empty faces; comparing  with Lemma \ref{x10.10}(b) shows that 
  $\Ext_{\neq\eset}(\mc{Z})=\mset{\mc{Z}_{W'}\mid W'\in \mc{W}}$ and that every non-empty face of $\mc{Z}$ is a stable face. 
   Similarly,  Lemma \ref{x10.10}(b) implies  that 
   $\Ext_{\neq\eset}(\mc{Z})=\mset{\mc{X}_{W'}\mid W'\in \mc{W}}$ and that 
   every non-empty face of $\mc{X}$ is a stable face. Together, these 
   verify the condition \ref{xA.9}(ii), and validity of \ref{xA.9}(i) has already been noted.  This proves \ref{x10.3}(a). Now \ref{x10.3}(b)--(c) follow from Lemma \ref{x10.5}(a)--(b) and \ref{x10.3}(d) follows from Lemma  \ref{x10.8}. This completes the proof of Theorem \ref{x10.3}.
 
 \subsection{} \label{x10.12} We conclude  this section with the following observation.
 \begin{cor} Let $W'$ be a finitely generated reflection subgroup of $W$. Then  the map defined by
 $F\mapsto F\cap \mc{Z}_{W'}:\Ext_{\neq \eset}(\mc{Z})\to
 \Ext_{\neq \eset}(\mc{Z}_{W'})$,   for $F$ a face of $\mc{Z}_{W}$,
 preserves  meets of arbitrary subsets of $\Ext_{\neq \eset}(\mc{Z})$.  \end{cor}
 \begin{proof} This follows  since  for any non-empty   face $F$  of 
 $\mc{Z}$, $F\cap \mc{Z}_{W'}$ is a non-empty face of $\mc{Z}_{W'}$ by  Corollary  \ref{x6.4} and the results of this section, and
 the meet of faces in either lattice is given by their intersection.
 \end{proof}

\section{Facial closures of reflection subgroups}
\label{x11}
This subsection  describes results which lead to certain algorithms for computing  facial closures of reflection subgroups. Since  specification of  general real numbers in finite terms and   and algorithmic computation with them is  impossible, we shall assume that oracles are available for performing arithmetic  computations with arbitrary real numbers and determining their signs. By an $\real$-algorithm, we shall mean an algorithm  with any necessary real arithmetic and sign or equality  tests  handled by such  oracles. 
We shall not be more precise, as we shall use this terminology only for informal comments.

 Under mild additional conditions (loosely, that $(V,\mpair{-,-})$ and $(\Phi,\Pi)$ are defined, in  an obvious  natural sense,  over a    subfield $K$ of $\real$  with algorithmically computable  arithmetic operations and sign and equality tests), these $\real$-algorithms readily  adapt to give   bona-fide (finite, terminating)  algorithms. 
The main point is to replace statements involving the various subsets of $V$ (especially $V$ and  cones $\real_{\geq 0}\Pi$, $\mc{C}$, $\mc{X}$, $\mc{ K}$, $\mc{Z}$) appearing in the theory by  statements involving the subsets of their $K$-points. We note that any finite extension field of $\rat$ in $\real$ is  a ``computable'' field in this sense, and each  finite rank Coxeter group is associated to  some 
 $(V,\mpair{-,-})$ and $(\Phi,\Pi)$ defined over a suitable computable field  $K$). 

   \subsection{}\label{x11.1} 
The first lemma gives rise to an $\real$-algorithm which determines, for $v\in \mc{C}$
(resp., $v\in \mc{Z}$) some $w\in W$ with $wv\in \mc{C}$ (resp., $wv\in \mc{K}$). 
 \begin{lem} For $v\in \mc{X}$, let   $n_{v}:=\vert\mset{\alpha\in \Phi_{+}\mid \mpair{v,\alpha}<0}\vert\in \Nat$.
 \begin{num}\item $n_{v}=0$ if and only if    $v\in \mc{C}$.
 \item If $n_{v}>0$, then there is $\alpha\in \Pi$ with
   $\mpair{\alpha,v}<0$, and then $s_{\alpha}v\in \mc{X}$ with $n_{s_{\alpha}v}=n_{v}-1$.
   If $w'\in W$ with $w'(s_{\alpha}v)\in \mc{C}$, then $wv\in \mc{C}$ where $w=w's_{\alpha}$.
   \item Let $v\in\mc{Z}$. Then $-v\in \mc{X}$. Further,  for $w\in W$, $wv\in\mc{K}$ if and only if  $w(-v)\in \mc{C}$.  \end{num}
\end{lem}
\begin{proof} Parts (a)--(b)  are  assertions involved in a standard proof of \ref{x1.10}(a)
(see \cite{Kac});
the main point for (b) is that $s_{\alpha}$ fixes the set $\Phi_{+}\setminus \set{\alpha}$ of positive roots. Part (c) follows readily from the definitions since $wv\in \real_{\geq 0}\Pi$ for all $w\in W$.
\end{proof}
     
  \subsection{} \label{x11.2}There are standard $\real$-algorithms of convex geometry which determine the facial subsets of $\Pi$ and the facial support of elements $v\in \real_{\geq 0}\Pi$
  (for any  $v$ specified as a non-negative $\real$-linear combination of elements of $\Pi$).
  Given a finitely generated reflection subgroup $W'$, specified by a finite set of reflections generating it, there are algorithms to compute $\chi(W')$ in \cite{Ref} and hence there are $\real$-algorithms for computing $\Pi_{W'}$.   We will also use that for $w\in W$, algorithms based on 
  \ref{x1.6}-\ref{x1.7}  enable one to compute the minimal length element of $W'w$ for $w\in W$, and in particular, compute  $\chi(w^{-1}W'w)$ from $\chi(W')$. 
  
  Given $\Pi_{W'}$, standard  $\real$-algorithms of convex geometry enable one  to compute the polyhedral cone $\mc{K}_{W'}$ (specified either  by the inequalities defining its minimal set of closed supporting half-spaces, or by a set of representatives of its extreme rays).
  In turn from this, one may give an $\real$-algorithm to determine a point of $\mc{K}_{W'}^{\circ }$ (as a sum of  representatives of the all the extreme rays of $\mc{K}_{W'}$ is in $\mc{K}_{W'}^{\circ }$).  Since   $\mc{K}_{W'}^{\circ }\subseteq \mc{Z}_{W'}^{\circ }$, it follows that  
there is  a $\real$-algorithm to determine a point of $\mc{Z}_{W'}^{\circ }$.

Generally, when we say that there is an $\real$-algorithm to compute a finitely-generated reflection subgroup $W''$ from certain data, we mean that there is a $\real$-algorithm to compute $\chi(W'')$ from that data.  
  We will also use in \ref{x11.5} the fact that there are $\real$-algorithms to determine, from $\chi(W')$  for any any finitely-generated reflection subgroup $W'$, a subset  $J\subseteq S$ and element  $x\in W$ such that $xW_{J}x^{-1}$ is  the parabolic closure of $W'$ (\cite{DyParClos}). 
  This latter algorithm is  sophisticated (compared to others discussed here), as it depends on the solvability of the conjugacy problem for Coxeter groups, which was for some time an open problem (see \cite{K}).
  A simpler algorithm is known for computing parabolic closures of finite (reflection) subgroups (see \cite{K}), and this will suffice for the results  through \ref{x11.4}.

\subsection{}   The next Lemma gives rise to an $\real$-algorithm to  determine, for $v\in \mc{Z}$,  the element $W'\in \mc{W}$  with $v\in\mc{Z}_{W'}^{\circ }$.\label{x11.3}   It in fact determines an expression for $W'$ as a $W$-conjugate of a standard facial subgroup $W_{I}$.
 Note from Lemma  \ref{xA.3} that $\mc{Z}_{W'}$ is the minimum face of $\mc{Z}$ containing $v$.

 \begin{lem}
   Let $v\in \mc{Z}$. Choose $w\in W$ with $wv\in \mc{K}$, and 
  let the facial support of $wv$ be $\Pi_{I}$ where $I\subseteq S$ is facial. 
   Then $W':=w^{-1}W_{I}w\in\mc{W}$ and $v\in \mc{Z}_{W'}^{\circ }$.\end{lem}
\begin{proof}  By Lemma \ref{x8.4}(b), $I$ is special, so $W_{I},W'\in \mc{W}$.
We have $wv\in\real_{\geq 0} \Pi_{I}\cap -\mc{C}\subseteq \real_{\geq 0} \Pi_{I}\cap -\mc{C}_{I}=\mc{K}_{I}$. Since $\Pi_{I}$ is the facial support of $wv$, we get $wv\in \mc{Z}_{I}^{\circ }$ by Lemma \ref{x10.7}
and so  $v\in w^{-1}\mc{Z}_{I}^{\circ }= \mc{Z}_{W'}^{\circ }$ as required.
    \end{proof}  
   \subsection{}  \label{x11.4} The next result describes the minimal face of $\mc{Z}$ containing  $\mc{Z}_{W'}$ for a reflection subgroup $W'$, and also the facial closure of $W'$ in terms
   of this minimal face and extra data determined from the root systems. 
   
   \begin{prop} Let $W'$ be a finitely generated reflection subgroup of $W$.
   Let $ W''$  be the reflection subgroup of $W'$ such that
   $\Pi_{ W''}$  is the union of all infinite type
   components of $\Pi_{W'}$.   Let    $U'''$ be the reflection subgroup of $W''$  with $\Pi_{U'''}=\Pi_{W'}\setminus \Pi_{W''}$ i.e. the components of $\Pi_{U'''}$ are the finite type components of $\Pi_{W'}$.
   \begin{num}
    \item Let $p\in \mc{Z}_{W'}^{\circ } $. Choose
  $U'\in \mc{W}$, 
   so that
  $p\in \mc{Z}_{U'}^{\circ }$. 
 Then  $U'$  is the facial closure of $ W''$,
 and  $\mc{Z}_{U'}$ is the minimum face of $\mc{Z}$  containing $\mc{Z}_{W'}$. 
  \item  $\mc{Z}_{W'}^{\circ }\subseteq \mc{Z}_{U'}^{\circ }$.
  \item Every component of $\Pi_{U'''}$ is of finite type, and is either contained
  in $\real\Pi_{U'}$ or  is orthogonal to $\real\Pi_{U'}$.
    \item Let $\Delta:=\Pi_{U'''}\cap \Pi_{U'}^{\perp}$. Then
   $\Delta=\Pi_{W'}\cap \Pi_{U'}^{\perp}$ and $\Delta$   is the union of  the (finite type)   components of
   $\Pi_{W'}$ which are not contained in $\real\Pi_{U'}$.         \item Let $U''$ denote the  facial (equivalently, parabolic) closure of $W_{\Delta}$. Then  $\Phi_{{U''}}=\Phi\cap \real\Delta\subseteq \Pi_{U'}^{\perp}$ and the facial closure  $U$  of
   $W'$ satisfies $\Pi_{U}=\Pi_{U'}\dot\cup \, \Pi_{U''}$. In particular, the facial closures of $W'$ and $W''$ both have the same infinite type components.
  \item  Write  $U'=wW_{I}w^{-1}$ where $I\subseteq S$ is facial, $w\in W$
  and, without  loss of generality, $N(w^{-1})\cap W_{I\cup I^{\perp}}=\emptyset$.  
  Then $w^{-1}\Delta\subseteq \Phi_{I^{\perp}}$.
  \item 
  We may choose $x\in W_{I}^{\perp}$ and  $J\subseteq I^{\perp}$, so that the facial closure of $w^{-1}W_{\Delta}w$
  is  $xW_{J}x^{-1}$  where   without loss of generality, $N(x^{-1})\cap J=\emptyset$.
   Then $I$, $J$ and $I\cup J$ are facial,  and  \[ U''= wxW_{J}(wx)^{-1},\qquad
   U'= (wx)W_{I}(wx)^{-1}, \qquad  
  U=(wx)W_{I\cup J}(wx)^{-1}.\]
 \end{num}
   \end{prop}
   \begin{rem} The above leads to an $\real$-algorithm for computing $U$, $U'$ and $U''$
  from $\chi(W')$  as follows. First, compute $p\in \mc{Z}_{W'}^{\circ } $ by \ref{x11.2}. 
    Compute $U'$ as in (a)  and express $U'=w'W_{I}w^{\prime-1}$ where $I\subseteq S$ is facial and $w\in W$, using   \ref{x11.3}. Compute  the minimal length  element $w^{-1}$ of $W_{I\cup I^{\perp}}w^{\prime-1}$; then 
    $U'=wW_{I}w^{-1}$ as in (f). Compute $\Delta=\Pi_{W'}\cap \Pi_{U'}^{\perp}$ as in (d).
    Since  $w^{-1}W_{\Delta}w$ is a finite  reflection subgroup of $W_{I^{\perp}}$, one  may  by \ref{x11.2} compute    $x'\in W_{I^{\perp}}$ and $J\subseteq I^{\perp}$ with  the parabolic closure of $w^{-1}W_{\Delta}w$ equal to $x^{\prime}W_{J}x^{\prime-1}$.  Let $x^{-1}\in W_{I^{\perp}}$ be the minimal length  element of 
    $W_{J}x^{\prime -1}$. Then $x^{\prime}W_{J}x^{\prime-1}=xW_{J}x^{-1}$ is the facial closure of 
    $w^{-1}W_{\Delta}w$ as in (g), since parabolic and facial closures of finite (reflection) subgroups coincide by \ref{x8.5}(a). Let $z^{-1}$ be the minimal length element of
     $W_{I\cup J}(wx)^{-1}$. 
    Then     $\chi(U'')=zI z^{-1}$, $\chi(U')=zJz^{-1}$ and $\chi(U)=z(I\cup J) z^{-1}$ using that $I$, $J$ are separated. 
      \end{rem}
     \begin{proof} Let $\mc{Z}_{  X}$, with 
     $  X\in \mc{W}$, denote the inclusion-minimal face of $\mc{Z}$ containing 
   $\mc{Z}_{W'}=\mc{Z}_{ W''}$. By Lemma \ref{xA.3}, (b) holds.   It follows that $U'= X$, where $U'$ is as defined in (a), and  (b) holds.
   
To complete the proof of (a), it remains to show that $U'$ is equal to the facial closure $W'''$ of $W''$.   Certainly   $\mc{Z}_{W'}=\mc{Z}_{W''}\subseteq \mc{Z}_{W'''}$.  We have $W'''\in \mc{W}$ by \ref{x10.1}(e) and  we conclude that $U'= X\subseteq W'''$ by definition of $X$. Conversely, let us prove that $W'''\subseteq U'$. Since $U'$ is facial, it will suffice to show that $W''\subseteq U'$.
   Write $U'=wW_{I}w^{-1}$ as in (f).  Let $p\in \mc{Z}_{W'}^{\circ }=\mc{Z}_{W''}^{\circ }$.  Since $W''$ has no finite type component,  Lemma \ref{x10.7} implies that   $w^{-1}p=\sum_{\alpha\in \Pi_{w^{-1}W''w}}c_{\alpha}\alpha$ where all
   $c_{\alpha}>0$. But  by (b),
    $w^{-1}p\in \mc{Z}_{w^{-1 }U'w}
  = \mc{Z}^{\circ }_{W_{I}}\subseteq \real_{\geq 0}\Pi_{I}$.
  Since $\Pi_{I}$ is facial, we have  $\Pi_{w^{-1}W''w}\subseteq \real\Pi_{I}\cap \Phi=\Phi_{I}$. Hence  $ w^{-1}W'' w\subseteq W_{I}$
 and  $W''\subseteq wW_{i}w^{-1}=U'$.
 This  completes the proof of (a).

Next we prove (c). Let $w$, $I$ be as in (f).
      Since $U'''$ centralizes $W''$, it  normalizes the facial closure $U'$ of $W''$.
 Therefore $w^{-1}U'''w$ normalizes $w^{-1}U'w=W_{I}$.
  Note that since all components of $\Pi_{I}$ are of infinite type (since those of $\Pi_{U'}$ are),
 \ref{x1.15}  implies that the normalizer of $W_{I}$ is  $W_{I}W_{I^\perp}$ (cf. Remark \ref{x1.15}).
 Hence $ w^{-1}U'''w\subseteq  W_{I}W_{I^\perp}$.
 Every root of, and hence every irreducible component of, $\Pi_{w^{-1}U'''w}$ is contained therefore in either   $\Phi_{I} $ or $ \Phi_{I^{\perp}}$, since  $I$ and $I^{\perp}$ are separated.  
 Since $U'=wW_{I}w^{-1}$,  every irreducible component of $\Pi_{U'''}$ is contained  in either   $\real\Pi_{U'} $ or $\Pi_{U'}^{\perp}$, proving (c). Then (d) is clear  since every component of $\Pi_{W'}$ is either of infinite type (hence contained in $\real \Pi_{W''}\subseteq \real\Pi_{U'}$) or is a component of $\Pi_{U'''}$. Part (f) is also clear since  from the above since $w^{-1}\Delta\subseteq \Phi_{I\cup I^{\perp}}=\Phi_{I}\cup \Phi_{I^{\perp}}$,  and  $w^{-1}\Delta\subseteq w^{-1}\Pi_{U'}^{\perp}\subseteq \Pi_{I}^{\perp}$ implies $w^{-1}\Delta\cap \Phi_{I}=\emptyset$.
 
 Since $w^{-1}W_{\Delta}w$ is finite, its facial closure coincides    (by \ref{x8.5}(a)) with its parabolic closure, which
 is clearly contained in $W_{I^{\perp}}$ and so may be written as in (g). We have already seen $I$ is facial. Then $J$ is facial  and $I\cup J$ is facial by \ref{x8.5}, since  $W_{J}$ is finite.
 The formulae for $U'$ in (g) holds since $x$ centralizes $W_{I}$. The formula for  the facial closure  $U''$  of $W_{\Delta}$ in (g) follows  since $w^{-1}U''w$ is the facial closure of $w^{-1}W_{\Delta }w$.  Let us now prove the formula  $U=(wx)W_{I\cup J}(wx)^{-1}$ from (g). 
 We have 
 \[ \Pi_{W'}=\Pi_{W''}\cup \Pi_{U'''}\subseteq\Phi_{U'}\cup \Pi_{U'''}=\Phi_{U'}\cup \Delta.\]
 Hence \[ x^{-1}w^{-1}\Pi_{W'}\subseteq x^{-1} w^{-1}\Phi_{U'}\cup x^{-1}w^{-1}\Delta
\subseteq x^{-1}\Phi_{I}\cup \Phi_{J}\subseteq \Phi_{I\cup J}.\]  In turn this implies that $W'\subseteq wxW_{I\cup J}(wx)^{-1}$, so $U\subseteq wxW_{I\cup J}(wx)^{-1}$. On the other hand,
$\Phi_{U}\supseteq \Phi_{W'}\supseteq \Phi_{W''}$ and $U$ facial implies $\Phi_{U}\supseteq \Phi_{U'}$. Further, $\Phi_{U}\supseteq \Phi_{W'}\supseteq \Phi_{U'''}\supseteq \Delta$ implies
that $\Phi_{U}\supseteq (\real \Delta \cap \Phi)$ by Corollary \ref{x2.5} since $U$ is facial.
We have $\real w^{-1}\Delta\cap \Phi=\real x\Pi_{J}\cap \Phi$ from Lemma \ref{x2.9}(c) and  the  definition of $J$.
Hence \[(wx)^{-1 } \Phi_{U}\supseteq (wx)^{-1}(\Phi_{U'}\cup  (\real \Delta \cap \Phi))
\supseteq  x^{-1 }\Phi_{I}\cup (\real\Pi_{J}\cap \Phi)\supseteq \Phi_{I}\cup \Phi_{J}=\Phi_{I\cup J}\] since $J\subseteq I^{\perp}$ and 
so $U\supseteq wxW_{I\cup J}(wx)^{-1}$. This proves the formula for $U$ in (g).
Finally, (e) follows from the formula for $U''$, $U'$ and $U$ in (g) noting $I$ and $J$ are separated.
 
   \end{proof}

   \subsection{} \label{x11.5}  The above subsection provides a  geometric algorithm for calculating the facial closure of a finitely generated reflection subgroup. This subsection  gives an alternative, more algebraic algorithm for 
   determining the facial closure of an arbitrary finitely generated subgroup $W'$ of $W$.

First, let $X$ be a finite set of generators for the subgroup $W'$. Let $wW_{J}w^{-1}$ denote the parabolic closure of $X$; suitable $w\in W$ and $J\subseteq S$  may be computed by an $\real$-algorithm described in \cite{DyParClos}.  Using \ref{x1.15}, one may compute all  finitely many  standard parabolic subgroups, say $W_{K_{i}}$, for $i=1,\ldots, n$,  which are conjugate to $W_{J}$ and determine $x_{i}\in W$ with $K_{i}=x_{i}Jx_{i}^{-1}$.  
Amongst  all the   standard facial   subgroups of $W$ containing any of the sets $K_{i}$, 
choose one of minimal rank; say $L\subseteq S$ is facial ${L}\supseteq{K_{j}}$,  and $\vert L\vert \leq\vert L'\vert  $ for any facial $L'\subseteq S$ with $L'\supseteq K_{i}$ for some $i$.  Then we claim that  the facial closure $W''$ of $W'$ is $W''=U$ where $U:=wx_{j}^{-1}W_{L}x_{j}w^{-1}$.

To see this, note first that  $U$  is a facial subgroup of $W$ containing $W'$. 
Write $W''=yW_{M}y^{-1}$ for some facial $M\subseteq S$ and some 
$y\in W$. By Lemma \ref{x10.1}(d), it will suffice to show that  the ranks
 $\vert L\vert $ of $U$ and $\vert M\vert $ of $W''$ satisfy 
 $\vert L\vert\leq \vert M\vert$.  Now  since $W''$ is parabolic and contains 
 $W'$, we have $W''\supseteq wW_{J}w^{-1}$. Then
  $W_{M}\supseteq  y^{-1}wW_{J}w^{-1}y$. This implies that the right hand 
  side is a parabolic subgroup of $W_{M}$, so it is conjugate in $W_{M}$ to 
  some standard parabolic 
subgroup of $W_{M}$; this standard parabolic is one of 
$W_{K_{1}},\ldots, W_{K_{n}}$,  say  $W_{K_{i}}$, since it is $W$-conjugate 
to 
$W_{J}$. Since $M\supseteq K_{i}$, we have $\vert M\vert\geq \vert L\vert$ 
as required,  by choice of $L$.
We conclude that 
$W''=U$ as claimed.

\begin{rem} The facial closure of a subgroup $W'=\mpair{X}$ of $W$, where 
$X\seq W$ is not necessarily finite, is  the parabolic subgroup of maximal   rank (necessarily bounded by $\vert S\vert$)  amongst the  facial closures of 
 subgroups $\mpair{X'}$ of $W'$ such that $X'\seq X$ is finite.
In case $W'$ is a reflection subgroup, this can be made more precise as 
follows.
Let $W''$ be  a finitely-generated  reflection subgroup of $W'$ such that
$\real \Pi_{W''}=\real\Pi_{W'}$ (which is possible since $V$ is finite dimensional).
Then the facial closure $U$ of $W''$ is equal to the facial closure $U'$ of $W'$.
For obviously $U'\supseteq U$. On the other hand,  we have by \ref{x2.5} since $U$ is facial that 
\[\Phi_{U}=\real\Phi_{U}\cap \Phi\supseteq \real\Phi_{W''}\cap \Phi=\real\Phi_{W'}\cap \Phi\supseteq \Phi_{W'}\] so $U\supseteq W'$ and hence $U\supseteq U'$. \end{rem}
  \section{The imaginary cone in general }
  \label{x12} 
     \subsection{} In   \label{x12.1}    this  section, $(W,S)$ is a general Coxeter system realized as in \ref{x1.3}.
     
     We first make  some general comments   about the relation of this general situation to the more special case in which the additional assumptions  \ref{x4.1}(i)--(iii) hold.  
    
An extension or restriction of quadratic space of a based root system, in the sense of Remark \ref{x1.3},   does not change the abstract group $W$, or the roots $\Phi$,  simple roots $\Pi$,  or cones $\mc{Y}$, $\mc{Z}$ or $\mc{K}$,  (as subsets of $\real \Pi$) though in general it may   change $\mc{X}$, $\mc{X}^{*}$,  the set   facial subsets of $S$ etc.
 Recall also that any based root system $(\Phi,\Pi)$ on $(V,\mpair{-,-})$ has ample extensions in the sense of Remark \ref{x1.3}(2); these may even be chosen so the extended quadratic space is non-singular (and finite-dimensional if $V$ is finite-dimensional).
On the other hand, if the form $\mpair{-,-}$ on $V$ is non-singular, then any finite dimensional subspace of $V$ is contained in a finite dimensional subspace $U'$ such that the restriction of the form $\mpair{-,-}$ to $U$ is non-singular. Then $V=U'\oplus U^{\prime\perp}$.

These remarks imply  that for the study of $\mc{Z}$, there is no loss of generality  in assuming that the quadratic space is non-singular and ample.  Further, if  $\Pi$ is finite, there is no loss of generality for studying $\mc{Z}$ in assuming    \ref{x4.1}(i)--(iii).

  \subsection{} \label{x12.2} With the above remarks in hand,  we show Theorem \ref{x6.3}
  still holds under the more relaxed assumptions here. 
     
     \begin{thm} \begin{num}\item $\mc{Z}:=\bigcup_{W'}\mc{Z}_{W'}$
  where in the union, $W'$ ranges over  any cofinal  subfamily  of the
inclusion-ordered family of   all finitely generated reflection subgroups of $W$.
  \item  For any reflection subgroup $W'$ of $W$, we have $\mc{Z}_{W'}\subseteq \mc{Z}$.
  \end{num}
     \end{thm}
     \begin{proof} Observe that  $\mc{Z}_{W'}\subseteq \mc{Z}_{W''}$ for any finitely generated reflection subgroups $W'\subseteq W''$ of $W$, by reducing  to the case in which  \ref{x4.1}(i)---(iii) hold  using the above remarks and using Theorem \ref{x6.3}.

     Write the family of reflection subgroups in (a) as $W'_{j}$ for $j$ in an index set $J$, 
     and let $G$ be the family of finite subsets of $S$.
     For each $j\in J$, there is some finite $I_{j}\in G$ with $ W'_{j}\subseteq W_{I_{j}}$, since $W'_{j}$ is finitely generated. Also, since $\set{W'_{j}}_{j\in J}$ is cofinal in the family of finitely generated reflection subgroups, there is for each  $I\in G$  there is some $j_{I}\in J$ with
     $W_{I}\subseteq W'_{j_{I}}$. By the above remarks, we have
     $\mc{Z}_{j}\subseteq \mc{Z}_{I_{j}}$ for $j\in J$ and $\mc{Z}_{I}\subseteq \mc{Z}_{W_{j_{I}}} $
     for $I\in {G}$. Hence 
     \[\cup_{j\in J}\mc{Z}_{W'_{j}}=\cup_{I\in {G}}\mc{Z}_{I}=\mc{Z}\] by Lemma
     \ref{x3.3}, proving (a). 
     
     To prove (b), let $H$ be the family of all finitely generated reflection subgroups of $W$.
     By (a) applied successively  to $W'$ and $W$, we have that
     \[\mc{Z}_{W'}=\bigcup_{\substack{W''\in H\\ W''\subseteq W'}}\mc{Z}_{W''}\subseteq 
     \bigcup_{W''\in H}\mc{Z}_{W''}=\mc{Z}\] as required.
       \end{proof}
       
       \subsection{} \label{x12.3} We have not investigated the facial structure of $\mc{Z}$ in general. 
       However, as a corollary of the results of this section,    the relative interior of $\mc{Z}$ may be described under more general assumptions than those of Sections \ref{x4}--\ref{x11}.       \begin{cor}  Let $W'$ be a subgroup of $W$ such that $\real\mc{Z}_{W'}$ is finite dimensional.
       Let $G$ be the family of all finitely generated reflection subgroups $W''$ of $W$ such that
       $\real\mc{Z}_{W''}=\real\mc{Z}_{W'}$.  Let $Y:=\cup_{W''\in G}\mc{Z}_{W''}^{\circ }$. Then
       \begin{num}\item $G$ is cofinal in the family of all finite rank reflection subgroups of $W'$.
  \item   $\mc{Z}_{W'}^{\circ }=Y$.\end{num}
       \end{cor}
       \begin{proof}  By Theorem \ref{x12.2}, we have $\mc{Z}_{W'}=\cup_{W''\subseteq W'}
       \mc{Z}_{W''}$ where the union is over all finite rank reflection subgroups $W''$ of $W'$.
      From this and  Theorem \ref{x12.2},  one gets (a).      For $W''\in G$, $\mc{Z}_{W''}^{\circ }$ is an open subset of 
        $\real\mc{Z}_{W''}=\real\mc{Z}_{W'}$ and is contained in $\mc{Z}_{W'}$.
        Hence  $Y$ is an open subset of $\real \mc{Z}_{W'}$
        contained in $\mc{Z}_{W'}$ and  it follows that
         $Y\subseteq \mc{Z}_{W'}^{\circ }$. For the reverse inclusion, let  $\alpha\in \mc{Z}_{W'}^{\circ }$.
         Then there is a (finite) basis $\Gamma$ of  $\real \mc{Z}_{W'}$ such that
         $\Gamma\subseteq \mc{Z}_{W'}$ and 
         $\alpha\in \real_{>0}\Gamma$. Using (a) and Theorem \ref{x12.2}, there is some $W''\in G$ such that
        $\Gamma\subseteq \mc{Z}_{W''}$. Then $\alpha\in \real_{>0}\Gamma\subseteq \mc{Z}_{W''}^{\circ }\subseteq Y$, completing the proof.
        \end{proof}
        \subsection{} \label{x12.4} Finally, we return to the assumptions of Sections \ref{x4}--\ref{x11} and  describe the minimal face of $\mc{Z}_{W}$ containing $\mc{Z}_{W'} $ for an arbitrary (not necessarily finite rank) reflection subgroup $W'$ of $W$.
        \begin{cor} Assume that the conditions $\text{\rm \ref{x4.1}(i)--(iii)}$ hold.
        Let $W'$ be an arbitrary reflection subgroup of $W$ and let $U$ denote the facial closure of
        $W'$.
        \begin{num}\item $\mc{Z}_{U}$ is the minimum face of $\mc{Z}$ which contains
        $\mc{Z}_{W'}$.
        \item $\mc{Z}_{W'}^{\circ }\subseteq \mc{Z}_{U}^{\circ }$.
            \end{num}
        \end{cor} 
        \begin{proof} Suppose first that  $W'$ is finitely generated. Let $W''$, $U'$,
        $U''$ be as in \ref{x11.4}. By \ref{x11.4}, $\mc{Z}_{U'}=\mc{Z}_{U}$ is the minimal face of
        $\mc{Z}$ containing 
        $\mc{Z}_{W''}=\mc{Z}_{W'}$, and $\mc{Z}_{W'}^{\circ }  \subseteq \mc{Z}_{U}^{\circ }$.

        In general, note that the family $G'$ of finite rank  reflection subgroups $W'''$ of $W'$ with $\real\Pi_{W'''}=\real
\Pi_{W'}$ and $\real\mc{Z}_{W'''}=\real\mc{Z}_{W'}$ is cofinal in the family of all finite rank reflections subgroups of  $W$.  Any $W'''\in G'$ has $U$ as facial closure by Remark \ref{x11.5}, and $\mc{Z}_{W'''}\subseteq \mc{Z}_{U}$. Both (a)--(b) now  follow easily from the special case in the previous paragraph, using that 
$\mc{Z}_{W'}=\cup_{W'''\in G'}\mc{Z}_{W'''}$ by Theorem \ref{x12.2} and 
 $\mc{Z}_{W'}^{\circ }=\cup_{W'''\in G'}\mc{Z}_{W'''}^{\circ }$
by Corollary \ref{x12.3}.

        \end{proof}

\appendix \section{Facial structure of cones} 
\label{xA} This section   discusses  (largely without proof)    some mostly standard 
facts concerning purely algebraic aspects of facial structure of cones, which are  used  in this paper or may be  helpful in understanding  the main results.  These  facts can be found  scattered
 in   various sources between which there is not always agreement  on  terminology and conventions,  and for the readers  convenience, we    record them here in a uniform way.
   Though, strictly,  these results  are only applied in this paper  in finite dimensional spaces,  
so   the reader may consider such spaces if desired, 
we discuss many of them  them more generally to clarify the role played by finite dimensionality and   since the general versions  may be useful in  possible extensions of this work. As  references for this material,  see 
\cite{BourEsp} and   \cite{Bar}  for   general real vector spaces, and, for finite 
dimensional spaces,   \cite{Rock},   \cite{Web}, \cite{StW},
 \cite{FlVan}, \cite{Br} and  \cite{Gl}.    For    background   on partially 
  ordered sets, lattices, directed sets, Galois connections etc which  
  is also used here and elsewhere in the paper, see  for instance \cite{DavPr}.  In this paper, 
  the only Galois connections which occur are those  between 
  power sets $\mc{P}(X)$ and $\mc{P}(Y)$, associated to a relation 
  $R\seq X\times Y$ as in \cite[7.22]{DavPr}.

\subsection{} \label{xA.1}  In this section,   all vector spaces are endowed  with a standard topology,  
which we call the \emph{finest locally convex topology} (\flc topology, for short)  defined below.   Though some  of the properties stated in the following subsections hold  for  all locally convex (Hausdorff) vector  topologies (see \cite{BourEsp} for details), many do not.

    A  subset $C$ of a real affine space $V$ is said to be 
\emph{algebraically open} if for all affine lines $l$ of $V$, the 
 intersection $l\cap C$ is an open subset of $l$ (in the topology obtained by transferring the standard topology on $\real $ to $l$ by an affine isomorphism $\real\to l$). 
The   \flc topology on $V$ is  the (Hausdorff)  topology on $V$  so the  open 
subsets    are the  unions of algebraically open \emph{convex} subsets of 
$V$. The \flc topology   on $V$ is clearly  invariant under all affine   automorphisms of $V$.  It is easily seen  (\cite[(3.3), Problem 4]{Bar})  that if $V$ is a vector space, the \flc topology is  the finest topology making $V$ a locally convex topological vector 
space.  (By \cite[Ch II, \S4, no. 2, 2]{BourEsp}, the \flc topology  coincides with the topology determined by the   family of all seminorms on $V$.) Other easily checked  properties of  \flc topology are as follows.  If $V$ is finite dimensional, the \flc topology is the standard topology on $V$ (defined for example by the norm associated to a positive definite inner product on $V$). In general, any affine subset $U$ of $V$ is closed
(hence each affine   function $V\to\real$ is continuous)   and (if $U\neq \eset$)   the  subspace and \flc  topology on $U$ coincide.  
The \flc topology on a  product  $V_{1}\times V_{2}$ of real affine  spaces $V_{1},V_{2}$ is the product  of the \flc topologies on $V_{1}$ and $V_{2}$. 
 See \cite{Bar}, \cite[Ch II]{BourEsp} for  more details  and further properties. 

\subsection{}\label{xA.2}  In the remainder of this section, let $V$ be a real vector space, in its \flc topology.
   Fix a  convex set $\mc{Y}\seq V$.    A point $v$ of $\mc{Y}$ is called an \emph{algebraically 
interior point} of $\mc{Y}$ in $V$ if for each affine line $l$ of $V$ with $v\in l$,
the point $v$ is in the interior (with respect to $l$) of $l\cap \mc{Y}$ (see \cite[Ch II, (1.5)]{Bar}). It is easily shown that  $\ri(\mc{Y})$ is the set of all algebraically interior 
points of $\mc{Y}$ in  $\aff(\mc{Y})$.
Using \cite[Ch II, \S 2, no. 6, Proposition 16 and Corollaire 1] {BourEsp} together with  $\cl(\mc{Y})\seq \aff(\mc{Y})$ (which holds since we use \flc topology) shows  
\begin{equation}\label{xA.2.1}
\text{\rm if $x\in \ri(\mc{Y})$ and $y\in \overline{\mc{Y}}$, then  
$tx+(1-t)y\in \ri(\mc{Y})$ for $0<t<1$,}
\end{equation}
\begin{equation}\label{xA.2.2}
\text{\rm if $\ri(\mc{Y})\neq \eset$, then
 $\ol{\ri(\mc{Y})}=\ol{\mc{Y}}$ and $\ri(\ol{\mc{Y}})=\ri({\mc{Y}})$.}
\end{equation} (For examples of non-empty cones without any relative interior points in the \flc topology, see  \cite[Ch II, (1.4)]{Bar} and \cite[Proposition 11.1]{Gl}).
    
\subsection{} \label{xA.3} A subset $\mc{C}$ of $\mc{Y}$ is said to be a \emph{face} or an 
 \emph{extreme subset}  of $\mc{Y}$ if it is convex  and for all 
   $c\in \mc{C}$, $y_{1},y_{2}\in \mc{Y}$ and $t\in (0,1)$
   with $c=ty_{1}+(1-t)y_{2}$, one has $y_{1},y_{2}\in \mc{C}$. 
    See  \cite{Bar}  for general background on faces. A face $\mc{C}$ of $\mc{Y}$ is \emph{proper} if $\mc{C}\neq \mc{Y}$.   Any intersection or directed 
   union 
   of faces of $\mc{Y}$ is a face of 
   $\mc{Y}$,  and a face of a face of $\mc{Y}$ is an
     face of $\mc{Y}$. In particular, the set of faces     of $\mc{Y}$, ordered by inclusion, forms a complete lattice
     $\Ext(\mc{Y})$,  called here  the \emph{face lattice} of $\mc{Y}$.
   The proof  of \cite[Theorem 18.1]{Rock} shows 
  \begin{equation}\label{xA.3.1}\text{\rm if  } \mc{C}\in \Ext(\mc{Y}), \mc{Z}\seq\mc{Y}, \mc{Z} \text{ \rm is convex and } \mc{C}\cap \ri(\mc{Z})\neq \eset, \text{ \rm then }\mc{Z}\seq \mc{C}.
   \end{equation} 
      Combining the  proofs of
  \cite[(2.4.4)]{StW} 
 and
   \cite[Theorem 2.6.2]{Web} shows that 
   if $\mc{C}\seq \mc{Y}$, 
\begin{equation} \label{xA.3.2}\mc{C}\in \Ext(\mc{Y})\iff 
\text{\rm  ( $\mc{Y}\sm\mc{C}$ is convex and   
$\mc{C}=\aff(\mc{C})\cap \mc{Y}$).} 
\end{equation}      Since $\aff(\mc{C})$ is closed in $V$, it follows that any face of 
   $\mc{Y}$ is closed in $\mc{Y}$.
    
  For any  $\mc{K}\seq\mc{Y}$, there is an inclusion minimal face $\mc{Y}_{\mc{K}}$ of $\mc{Y}$ containing 
  $\mc{K}$, namely the intersection of all faces of
   $\mc{Y}$ which contain $\mc{K}$.  If also 
   $\mc{K}'\seq \mc{Y}$, one 
  clearly has $\mc{Y}_{\mc{K}\cup\mc{K}'}=\mc{Y}_{\mc{K}}\vee\mc{Y}_{\mc{K}'}$ where $\vee$ denotes join in $\Ext(\mc{Y})$.
 For  $y\in \mc{Y}$ let  $\mc{Y}_{y}:=\mc{Y}_{\set{y}}$ and
  $U_{y}=\aff(\mc{Y}_{y})$.
  From  \eqref{xA.3.2}, 
   $\mc{Y}_{y}=\mc{Y}\cap U_{y}$ and by
    \cite[Ch II, \S7, Ex3]{BourEsp},  $U_{y}$ is the inclusion-largest   affine  subset $U$ of $V$ containing $\set{y}$ such that
   $y$ is an algebraically interior point of $U\cap \mc{Y}$ in $U$.
  In particular,   $y\in \ri(\mc{Y}_{y})$. It follows  from this 
  and \eqref{xA.3.1} that any convex set $\mc{Z}\seq \mc{Y}$
  with a relative interior point $y$ is contained in $\mc{Y}_{y}$. 
  Since any face $\mc{Z}$ of $\mc{Y}$ which contains $y$ is convex and  contains $\mc{Y}_{y}$, one has 
   \begin{equation}\label{xA.3.3}
  \set{\mc{Y}_{y}\mid y\in \mc{Y}}=\mset{ \mc{Z}\in \Ext(\mc{Y})\mid\ri( \mc{Z})\neq \eset}
  \end{equation} 
 Similarly,  for $y,y'\in \mc{Y}$, one has 
  \begin{equation}\label{xA.3.4}
  \mc{Y}_{y'}=\mc{Y}_{y}\iff y'\in \ri(\mc{Y}_{y})\iff  \ri(\mc{Y}_{y'})=\ri(\mc{Y}_{y})
  \end{equation} 
      and   \begin{equation}\label{xA.3.5}
 \ri(\mc{Y}_{y'})\cap \rb(\mc{Y}_{y})\neq \eset\iff \mc{Y}_{y'}\seq \rb(\mc{Y}_{y})\iff \mc{Y}_{y'}\sneq \mc{Y}_{y}.
 \end{equation}
Hence  
      \begin{equation}\label{xA.3.6}
   \text{$\set{\ri(\mc{Y}_{y})\mid y\in \mc{Y}}$ is a partition of $\mc{Y}$ 
       into relatively open  convex subsets.}
   \end{equation}  
   and 
   \begin{equation}\label{xA.3.7}
  \mc{Y}\cap ( \rb\mc{Y})=\bigcup_{\substack{\mc{C}\in \Ext(\mc{Y})\\ \mc{C}\neq \mc{Y}}}\mc{C}=\bigcup_{\substack{\mc{C}\in \Ext(\mc{Y})\\ \mc{C}\neq \mc{Y}}}\ri(\mc{C}).
   \end{equation}  
  Note that  \begin{equation}\label{xA.3.8}
  \text{   $ \mc{C}\in \Ext(\mc{Y})\implies \mc{C}=\bigcup_{y\in \mc{C}}\mc{Y}_{y}$ (directed union)}
    \end{equation}
  since   for $y\in \mc{C}$, $\mc{Y}_{y}\seq \mc{C}$ and for all $y, y'\in \mc{C}$, one has  $\mc{Y}_{y}\vee \mc{Y}_{y'}=\mc{Y}_{\set{y,y'}}=\mc{Y}_{y''}$ where $y'':=\frac{1}{2}y+\frac{1}{2}y'\in \mc{C}$.
  Recall that  $\mc{F}\in \Ext(\mc{Y})$ is said to be \emph{lattice-compact}  if  whenever
   $\mc{F}\seq \bigvee_{\mc{G}\in I} \mc{G}$ where $I\seq  \Ext(\mc{Y})$, one has  $\mc{F}\seq \bigvee_{\mc{G}\in I'} \mc{G}$ for 
   some finite $I'\seq I$. 
   Using \eqref{xA.3.1} and \eqref{xA.3.8}, one easily shows that
   \begin{equation} \label{xA.3.9}\text{the set of lattice-compact elements of $\Ext(\mc{Y})$ is $\set{\mc{Y}_{y}\mid y\in\mc{Y}}\cup\set{\eset}$}.    \end{equation}    By definition (see \cite{DavPr}), \eqref{xA.3.8}--\eqref{xA.3.9} show  that $\Ext(\mc{Y})$ is an \emph{algebraic  lattice}.  
   
    Let $U$ be any finite    dimensional   affine subset   of $V$.  By \eqref{xA.3.2}, 
     \begin{equation}\label{xA.3.10}
   \text{ if   $C\seq  \Ext(\mc{Y})\cap \mc{P}(U)$  is totally ordered, then $\vert C\vert \leq \dim(U)+2$}   \end{equation}
    since $\dim(U)+2$ is the maximum cardinality  of  a  flag 
    of  affine subsets   of  $U$.     Let 
      $\mc{C}\in  \Ext(\mc{Y})\cap \mc{P}(U)$. In \eqref{xA.3.8},  each $\mc{Y}_{y}$ is contained in $\Ext(\mc{Y})\cap \mc{P}(U)$. It  follows that if $\mc{C}\neq \eset$, then 
      $\mc{C}=\mc{Y}_{y}$ for some $y\in \mc{C}$ and so
       $\ri(\mc{C})\neq \eset$. Taking $U=V$,  one  recovers  the well-known fact
       (see \cite[Theorem 6.2]{Rock}) that  
   \begin{equation}\label{xA.3.11}\text{\rm if $\dim(V)$ is finite,    any convex set $\mc{Y}\neq \eset$ has a relative interior point.}
    \end{equation}   Then  \eqref{xA.3.3}  implies that
   \begin{equation}\label{xA.3.12}
    \text{ if $\dim(V)$ is finite,  $\Ext(\mc{Y})=\set{\eset}\cup
       \set{\mc{Y}_{y}\mid y\in \mc{Y}}$. }  
   \end{equation} 
   
   Two other facts we shall use (see \cite[Theorem 6.5, Corollary 6.6.2]{Rock}) are the following. Let $C_{1},C_{2}$ be convex subsets of $V$ where $V$ is finite dimensional. Then
   
  \begin{equation}\label{xA.3.13}
   \ri(C_{1}\cap C_{2})=\ri(C_{1})\cap\ri(C_{2}) \text{\ \rm if $\ri(C_{1})\cap\ri(C_{2})\neq 0$}
   \end{equation}
   \begin{equation}\label{xA.3.14}
   \ri(C_{1}+C_{2})=\ri(C_{1})+\ri(C_{2}) 
   \end{equation}

   The proof of the following lemma is left to the reader. 
  
\begin{lem} Suppose that $\dim(V)$ is finite and $\mc{Y}\seq V$ is a  convex set. Let  $C$ be  a non-empty convex subset of $\mc{Y}$. 
 \begin{num}\item The following conditions $\text{\rm (i)--(iii)}$ on $F\in \Ext(\mc{Y})$  are equivalent: 
\begin{subconds}\item $C\seq F$  \item $\ri(C)\seq F$.
\item $\ri(C)\cap F\neq \eset$.
\end{subconds}
 \item The following conditions $\text{\rm (i)--(iv)}$ on $F\in \Ext(\mc{Y})$  are equivalent: \begin{subconds}
\item  $F$ is minimal with $C\seq F$.
\item $\ri(F)\sreq \ri(C)$
\item $\ri(F)\cap \ri(C)\neq \eset$.
\item $F$ is maximal with  $\ri(F)\cap C\neq \eset$. \end{subconds}Moreover, there is a unique $F$ satisfying $\text{\rm (i)--(iv)}$.
\end{num}\end{lem}

 \subsection{}\label{xA.4}  Assume now that $\mc{Y}$ is a cone, as will be the case in all 
   applications in this paper.  Then a subset $\mc{C}$ of $\mc{Y}$ 
   is a face  if and only if it is a cone and   for all 
   $c\in \mc{C}$ and  $y_{1},y_{2}\in \mc{Y}$    with 
   $c=y_{1}+y_{2}$, one has $y_{1},y_{2}\in \mc{C}$ 
   (see \cite[Lemma 10.2(a),(c)]{Gl}).
   For $\mc{C}\in \Ext(\mc{Y})$, $\mc{C}$ and $\mc{Y}\sm\mc{C}$ 
   are cones  and if $\mc{C}\neq \eset$, then  
   $\aff(\mc{C})=\lin(\mc{C})$.    
For $\mc{K}\seq \mc{Y}$, one has
\begin{equation}\label{xA.4.1}
\text{$\mc{Y}_{\mc{K}}=\mc{Y}_{\ccl(\mc{K})}$ and, if $\mc{K}$ is a 
 cone,
 $\mc{Y}_{\mc{K}}= (\mc{K}-\mc{Y})\cap \mc{Y}$}
\end{equation}  as is easily 
 checked (see \cite[Lemma 10.2(e)]{Gl} in case 
 $\mc{K}=\real_{>0}x$).   If $0\in \mc{Y}$, then every non-empty 
 face
 $\mc{C}$ of $\mc{Y}$ contains $0$,
 and $\Ext_{\neq\eset}(\mc{Y}):=\Ext(\mc{Y})\sm\set{\eset}$ is 
 itself a complete lattice in which  meets and joins of its
 (possibly infinite)  subsets are the same  as in
 $\Ext(\mc{C})$.    A ray  which is a face of $\mc{Y}$ is called an \emph{extreme ray} of $\mc{Y}$.
 
 The next result is easily proved from \eqref{xA.2.1}--\eqref{xA.2.2}.
 \begin{lem} Suppose $\mc{Y}\seq V$ is a cone and $\rho\in \ri(\mc{Y})$. Then 
 \begin{num}
 \item
 $\cl(\mc{Y})=\mset{z\in V\mid z+t\rho\in \mc{Y}\text{ \rm for all $t\in \real_{> 0}$}}$.
 \item $\ri(\mc{Y})=\mset{z+t\rho\mid z\in \ri(\mc{Y}), t\in \real_{> 0}}=\mset{z+t\rho\mid z\in \mc{Y}, t\in \real_{> 0}}=
 \mset{z+t\rho\mid z\in \cl(\mc{Y}), t\in \real_{> 0}}$.  \end{num}
 \end{lem}
 \subsection{} For the rest of this section, fix a   bilinear map 
 $\mpair{-,-}\colon V\times U\to \real$  where $V$, $U$ are real vector spaces. 

In general,   if $A\seq V$, $B\seq U$,  write $\mpair{A,B}$ for $\mset{\mpair{a,b}\mid a\in A,b\in B}$.
 Partially order the power sets $\mc{P}(V),\mc{P}(U)$  by inclusion.  
Define  maps $C\mapsto C^{*}$ and $C\mapsto C^{\perp}$  from
 $\mc{P}(V)\to \mc{P}(U)$ by   \begin{equation}
 C^{*}:=\mset{u\in U\mid \mpair{C,u}\subseteq \real_{\geq 0}},\qquad
 C^{\perp}:=\mset{u\in U\mid \mpair{C,u}\subseteq \set{0}}
 \end{equation} for each subset $C$ of $V$. By symmetry, define maps
  $D\mapsto D^{*}$ and $D\mapsto D^{\perp}$  from 
  $\mc{P}(U)\to \mc{P}(V)$ (we rely on context rather than a more elaborate notation  to 
  distinguish the  meaning of $C^{\perp}$, $C^{*}$, and also $\wcl(C)$  as defined below, if $C\subseteq U\cap V$).
  
   Note $C^{*}$ is a  pointed cone  in $U$; in fact, $C^{*}$ contains the right radical $V^{\perp}$ of $\mpair{-,-}$. If $C$ is a cone, $C^{*}$ is called the \emph{dual cone} of $C$. 
      It is easy to see that the pair of  order-reversing maps 
 $C\mapsto C^{*}\colon \mc{P}(V)\to \mc{P}(U)$  and 
 $D\mapsto D^{*}\colon \mc{P}(U)\to \mc{P}(V)$ define a Galois 
 connection between $\mc{P}(U)$ and $\mc{P}(V)$ (see \cite{DavPr}). General properties of Galois connections imply the following. For $C\seq V$, one has $C\seq C^{**}$; further, 
 $C=C^{**}$ if and only if $C=D^{*}$ for some 
 $D\seq U$. If $C=C^{**}$,  $C$ is said to be a \emph{stable cone} in 
 $V$. The map $C\mapsto C^{*}$ defines an inclusion-reversing  
 bijection from the set of  stable  cones of $V$ to the set of stable cones 
 of $U$. 
 \begin{rem} (1) The weak topology on $V$ is defined to be   the weakest topology on $V$ so all linear maps 
  $\mpair{-,u}\colon V\to \real$ for $u\in U$ are continuous (it is Hausdorff if  and only 
  if the left radical $U^{\perp}$ of $\mpair{-,-}$ is zero). Denote the weak closure operator on $V$ as $\wcl$. It follows from the \emph{bipolar theorem}  (see \cite[Ch II,\ S6, no. 3, Th\'eor\`eme 1 and  Proposition 4(ii)]{BourEsp} or \cite[(5.3)]{Bar}) 
  that for $C\seq V$, 
  \begin{equation}
 C^{**}=\wcl(\real_{\geq 0}C) 
 \end{equation}  and that the stable cones in $V$ are the non-empty weakly closed  cones in $V$.
 
   (2) Similarly, the maps $\perp$ between $\mc{P}(V)$ and $\mc{P}(V)$ determine a Galois connection  with the weakly closed subspaces of $U$ and  $V$ as the stable sets.
 \end{rem}
\subsection{}\label{xA.6}  Let $ \mc{Y}, \mc{Z}$ be  cones in $V$ and 
$U$ respectively  such that 
$\mpair{ \mc{Y}, \mc{Z}}\subseteq \real_{\geq 0}$ i.e. 
$ \mc{Z}\subseteq  \mc{Y}^{*}$ or, equivalently, 
$ \mc{Y}\subseteq  \mc{Z}^{*}$. Such a pair of cones will be called 
 a \emph{semidual pair of cones}. There is then   a  \emph{transpose semidual pair}
 $P^{\text{\rm tr}}:=(\mc{Z},\mc{Y})$ with respect to the bilinear map $\mpair{-,-}\circ \iota\colon U\times V\to \real$ where $\iota\colon U\times V\to V\times U$ is the isomorphism $(u,v)\mapsto (v,u)$. 
  If $(\mc{Y},\mc{Z})$ is a 
 semidual pair with respect to $\mpair{-,-}$, so is $(\mc{Y},\mc{Z}^{**})$.
  If $ (\mc{Y}, \mc{Z})$ is a semidual pair such that $ \mc{Z}= \mc{Y}^{*}$ and 
 $ \mc{Y}= \mc{Z}^{*}$, then  $( \mc{Y}, \mc{Z})$ is said to 
 be a \emph{semidual  pair of stable cones}.  For any  $\mc{Y}\seq V$, $(\mc{Y}^{**},\mc{Y}^{*})$  is a semidual pair of stable cones.
 
Assume below that $P=(\mc{Y}, \mc{Z})$ is a semidual pair of cones.
Let $z\in \mc{Z}$. Then  $z^{*}\sreq\mc{Z}^{*} =\mc{Y}^{**} \sreq \mc{Y}$.  
  If $z^{\perp }\neq V$ and $z^{\perp}\cap \mc{Y}\neq \eset$, then $z^{*}$ is called  a \emph{supporting (homogeneous)
  half-space} of $\mc{Y}$  and $z^{\perp}=\rb(z^{*})$ is called a
   \emph{supporting (linear) hyperplane} of $\mc{Y}$. 
 Any subset  of $\mc{Y}$ which is either equal to $\mc{Y}$ or  of the form $\mc{Y}\cap  z^{\perp}$ for  an arbitrary 
  $z\in \mc{Z}$  is a face of $\mc{Y}$. Such a face  will be  called   an
  \emph{exposed face} of  $\mc{Y}$.  Thus
  the intersections of $\mc{Y}$ with its supporting hyperplanes are the proper, non-empty exposed faces of $\mc{Y}$, and possibly $\mc{Y}$ itself.  Note  that 
  \begin{equation}\label{xA.6.1}
  \text{ if $z,z'\in \mc{Z}$,  
  $(\mc{Y}\cap z^{\perp})\cap (\mc{Y}\cap z^{\prime \perp})=
  \mc{Y}\cap (z+z')^{\perp}$ where  $z+z'\in \mc{Z}$.}
  \end{equation} Hence the set of   exposed subsets of $\mc{Y}$ forms a meet semilattice $\Exp_{P}(\mc{Y})$  with    maximum element $\mc{Y}$. It is not necessarily a  complete meet semilattice.  
   
    An arbitrary (possibly empty) intersection of exposed faces of $\mc{Y}$ is called a 
  \emph{semi-exposed} face of $\mc{Y}$. 
  A ray $\real_{\geq 0}\alpha$ which is a face (resp., an exposed face, semiexposed face ) of a (pointed) cone $\mc{Y}$ is called an extreme ray (resp., exposed ray, semiexposed ray) of $\mc{Y}$.   The set of all 
  semi-exposed subsets of $\mc{Y}$ forms a complete
   lattice $\SExp_{P}(\mc{Y})$ with maximum element $\mc{Y}$. Define the order-reversing  maps $Y\mapsto Y^{\dag}\colon \mc{P}(\mc{ Y})\rightarrow \mc{P}(\mc{Z})$ 
given by $Y^{\dag}:= \mc{Z}\cap Y^{\perp}$ and   
 $Z\rightarrow Z^{\#}\colon \mc{P}(\mc{Z})\rightarrow \mc{P}(\mc{Y})$  given by $Z^{\#}:= \mc{Y}\cap Z^{\perp}$. Since 
$Y\subseteq Z^{\#}$ if and only if $\mpair{Y,Z}\seq \set{0}$ if and only if   $Z\subseteq Y^{\dag}$, these 
maps also  define a Galois connection. The \emph{stable subsets} of   $ \mc{Y}$
 for  this Galois connection are  the subsets $Y$ of
  $\mc{Y}$ with $Y=Y^{\dag\#}$ or, equivalently, with $Y=Z^{\#}$
  for some $Z\seq \mc{Z}$.  Note that for all $z\in\mc{Z}$,  
   $z^{\#}=\mc{Y}\cap z^{\perp}\in \Exp_{P}(\mc{Y})$ is an exposed 
   face of $\mc{Y}$ and that for $Z\seq \mc{Z}$, 
    $Z^{\#}=\cap_{z\in Z}z^{\#}$. Hence    the stable subsets of 
    $\mc{Y}$ are precisely its semi-exposed subsets. Similar results hold  for $\mc{Z}$ by symmetry. Standard properties of Galois connections imply that  the maps $Z\mapsto Z^{\#}$ and $Y\mapsto Y^{\dag}$ restrict to mutually inverse, inclusion-reversing bijections between the stable subsets of $ \mc{Y}$ and $ \mc{Z}$ i.e. between $\SExp_{P}(\mc{Y})$ and $\SExp_{P}(\mc{Z})$.

Any semi-exposed face   is  a (weakly closed and hence  closed in \flc topology) face of  $\mc{Y}$,  and   it contains  $0$ if $0\in \mc{Y}$
(since this is readily checked  for  exposed subsets  
  $z^{\#}=\mc{Y}\cap z^{\perp}$, where $z\in \mc{Z}$).   
  For any $Z\seq \mc{Z}$, $Z^{\#}=
  Z^{\#\dag\#}=\bigcap_{z\in Z^{\#\dag}} z^{\#}$
  is (by  \eqref{xA.6.1} and the fact that $\mc{Z}^{\#\dag}$ is a cone)  the intersection 
   of the directed downwards (by inclusion) family of exposed subsets $z^{\#}$ of $\mc{Y}$, for $z\in\mc{Z}^{\#\dag}$. 
  Using  \eqref{xA.3.10} now shows that if $V$ is 
   finite dimensional, then any semi-exposed face of $\mc{Y}$ is 
   exposed. Hence 
   \begin{equation}\label{xA.6.2}
   \text{if $\dim(V)$ is finite, $\Exp_{P}(\mc{Y})=\SExp_{P}(\mc{Y})$ is a complete lattice.}
   \end{equation}
   
   For $z\in\mc{Z}$, one has $z\subseteq z^{\#\dag}$ where the right hand side is an exposed subset, and hence in particular $z^{\#\dag}$ is a weakly closed face of $\mc{Z}$.
   It follows that
   \begin{equation}\label{xA.6.3}
   z\in\ri(\mc{Z}_{z})\subseteq \mc{Z}_{z}\subseteq \cl(\mc{Z}_{z})\subseteq \wcl(\mc{Z}_{z})\subseteq  z^{\#\dag}. \end{equation} 
  Applying the inclusion reversing map  $\#$ to this and using $z^{\#}=z^{\#\dag\#}$  shows that  
    \begin{equation}\label{xA.6.4}
   z^{\#}=(\ri(\mc{Z}_{z}))^{\#}= (\mc{Z}_{z})^{\#}=( \cl(\mc{Z}_{z}))^{\#}=(\wcl (\mc{Z}_{z}))^{\#}. \end{equation} 
   
\begin{rem}   Using \eqref{xA.6.4},  the  above Galois connection between subsets of $\mc{Y}$ and $\mc{Z}$  can be described in terms of a Galois connection between subsets of  the  ``quotient'' sets   $\mc{Y}':=\mset{\ri(\mc{Y}_{y})\mid y\in \mc{Y}}$ and 
$\mc{Z}':=\mset{\ri(\mc{Z}_{z})\mid z\in \mc{Z}}$ of $\mc{Y}$, $\mc{Z}$ (of relative interiors of the non-empty lattice-compact faces of $\mc{Y}$, $\mc{Z}$)  as follows.
There are order-reversing maps $f\colon \mc{P}(\mc{Y}')\to\mc{P}( \mc{Z}')$  defined by 
  \begin{equation*}
  f(Y):=\mset{z\in \mc{Z}'\mid \mpair{z,y}\seq\set{0}\text{ \rm for all 
  $y\in Y$}} 
  \end{equation*} for $Y\seq \mc{Y}'$  and        
  $g\colon \mc{P}(\mc{Z}')\to\mc{P}( \mc{Y}')$  defined analogously 
  by symmetry.   It is easy to see that $f$ and $g$ define a Galois 
  connection between $\mc{P}(\mc{Y}')$ and $\mc{P}(\mc{Z}')$.
   Define the   natural surjection (``quotient map'')  $\pi\colon \mc{Y}\to \mc{Y}'$ 
    by $y\mapsto \ri(\mc{Y}_{y})$, and define
    $\rho\colon \mc{Z}\to \mc{Z}'$ similarly by
     $z\mapsto \ri(\mc{Z}_{z})$. 
   Then for $Y\seq \mc{Y}$, one has $Y^{\dag}=\rho^{-1}(f(\pi(Y)))$ and  for $Z\seq \mc{Z}$, one has $Z^{\#}=\pi^{-1}(g(\rho(Z))$.\end{rem}

 \subsection{} \label{xA.7} 
   In the specially important  case that   $U=\Hom_{\real}(V,\real)$ is the dual space of  
   $V$, $\mpair{-,-}\colon \mc{Y}\times \mc{Z}\to \real$ is
    the canonical evaluation pairing, and 
   $\mc{Z}:=\mc{Y}^{*}$ is the dual cone of $\mc{Y}$,
    supporting half-spaces and hyperplanes of $\mc{Y}$,
   and   exposed and semi-exposed faces of $\mc{Y}$, 
   will be distinguished 
   by the adjective ``absolute.''  These absolute notions   are 
   completely determined by  $\mc{Y}$ as a  subset of the vector 
   space $V$ alone.  Denote the set of  absolutely semi-exposed 
   (resp., absolutely exposed) faces of $\mc{Y}$ by $\ASExp(\mc{Y})$ (resp.,
   $\AExp(\mc{Y})$). Note that for a semidual pair $(\mc{Y},\mc{Z})$ 
    associated to an  arbitrary bilinear map $\mpair{-,-}$, one has 
    $\AExp(\mc{Y})\seq\Exp(\mc{Y}) $ and therefore $\SExp(\mc{Y})\seq\ASExp(\mc{Y})$,
    but equality need not hold   in general
     if the map $u\mapsto \mpair{-,u}\colon U\mapsto\Hom_{\real}(V,\real) $ is not surjective (e.g.  $\mpair{-,-}$ is zero)
      or if $\mc{Z}\sneq \mc{Y}^{*}$ (e.g. $\mc{Z}=\set{0}\neq \mc{Y}^{*}$). 
      
      Special properties of the absolute notions are as follows (cf.  
      \cite[(3.6.5)]{StW} for the key fact (a) in the finite dimensional setting).    \begin{lem}\begin{num}
  \item  If $\ri(\mc{Y})\neq \eset$, then for any 
      $y\in \mc{Y}\cap \rb(\mc{Y})$, there is a proper absolutely exposed face 
      $\mc{C}\sneq \mc{Y}$ of $\mc{Y}$ with $\mc{Y}_{y}\seq \mc{C}$.
      \item If $\ri(\mc{Y})\neq \eset$, then $\mc{Y}\cap \rb(\mc{Y})$ is the union of the proper, absolutely exposed faces of $\mc{Y}$.
      \item If $\dim(V)$ is finite and $\mc{C}$ is a non-empty proper
      face of $\mc{Y}$, there is a sequence $\mc{C}=\mc{C}_{0},\ldots, \mc{C}_{n}=\mc{Y}$ of faces of $\mc{Y}$ such that for $i=1,\ldots, n$, $\mc{C}_{i-1}$ is a proper, absolutely exposed face of $\mc{C}_{i}$.  
  \end{num}
   \end{lem}  
   \begin{proof}
   The proof of (a) easily reduces to the case in which $\aff(\mc{Y})=V$. Then  $\mc{Y}$ has an algebraically interior point $u$.
   By \eqref{xA.3.4} and \cite[Corollary 1.7]{Bar}, there is an affine hyperplane which separates $\ri(\mc{Y}_{u})$ and $\ri(\mc{Y}_{y})$ i.e some non-zero $f\in \Hom_{\real}(V,\real)$ and $a\in \real$ such that
   $f(\ri(\mc{Y}_{y}))\seq \real_{\leq a}$ and  $f(\ri(\mc{Y}_{u}))\seq \real_{\geq a}$. Since $\ri(\mc{Y}_{u})$ is algebraically open, we  have $f(\ri(\mc{Y}_{u}))\seq \real_{> a}$  Also,  $f(\cl(\mc{Y}_{y}))\seq \real_{\leq a}$ and  $f(\cl(\mc{Y}_{u}))\seq \real_{\geq a}$ since $f$ is continuous.
   Since $0\in \cl(\mc{Y}_{y}))\seq \cl(\mc{Y}_{u}))$, this gives $a=0$, $f(\mc{Y}_{y})=\set{ 0}$ and  $\set{0}\neq f(\mc{Y}_{u})\seq\real_{\geq 0}$. Hence $\mc{C}:=\mc{Y}\cap \ker f$ is as required for (a). 
   Part (b) follows from (a) and \eqref{xA.3.7}. Part (c) follows from (a) using \eqref{xA.3.10}--\eqref{xA.3.12}.
     \end{proof}

\subsection{} \label{xA.8} Suppose henceforward in this section that $U,V$ are finite dimensional and  that
$\mpair{-,-}$ is non-singular.  The flc and weak topologies on 
$V$, $U$ are therefore their standard topologies as finite dimensional real vector spaces.

Consider  a semidual pair 
$P=(\mc{Y},\mc{Z})$ of  pointed  cones with respect to $\mpair{-,-}$. Then
$\Ext_{\neq\eset}(\mc{Y})\sreq \ASExp(\mc{Y})\sreq \SExp_{P}
(\mc{Y})$. Here,   $\ASExp(\mc{Y})=\AExp(\mc{Y})$ and $\SExp_{P}
(\mc{Y})=\Exp_{P}(\mc{Y})$ by \eqref{xA.6.2}. If $\mc{Z}=\mc{Y}^{*}$, then (using the  
the natural isomorphism   $ U\to \Hom_{\real}(V,\real)$ given by  
$u\mapsto \mpair{-,u}$  to identify $U$ with the dual space of
 $V$),   one sees  that $\ASExp(\mc{Y})=\SExp_{P}(\mc{Y})$, but trivial examples show one may have $\ASExp(\mc{Y})\neq \SExp_{P}(\mc{Y})$ if $\mc{Z}\sneq\mc{Y}^{*}$.

  The following example, which  arises by homogenizing a standard example (see \cite[Fig 2.11]{Web}) of a convex  set with
an extreme point which is not exposed,
shows that even if $\mc{Z}=\mc{Y}^{*}$ and 
 $\mc{Y}=\mc{Z}^{*}$, one may have  $\Ext_{\neq\eset}(\mc{Y})\neq\ASExp(\mc{Y})$. 
  \begin{exmp}
  Let $U=V=\real^{3}$ with $\mpair{-,-}$ given by the standard  (dot) inner product. Let $Y:=\mset{(x,y,1)\in \real^{3}\mid -1\leq y\leq 1,  -1-\sqrt{1-y^{2}}\leq x\leq 1+ \sqrt{1-y^{2}}}$, $\mc{Y}:=\real_{\geq 0}Y$, 
  $Z:=\mset{(x,y,1)\in \real^{3}\mid   -(1-y^{2})\leq 2x\leq {1-y^{2}}}$ and $\mc{Z}:=\real_{\geq 0}Z$.
  Then one can check that $(\mc{Y},\mc{Z})$ is a semidual pair of stable salient cones (so $\mc{Y}=\mc{Z}^{*}$ and $\mc{Z}=\mc{Y}^{*}$), that $\Ext_{\neq\eset}(\mc{Z})= \SExp_{P^{\text{\rm tr}}}(\mc{Z})$ and that $\Ext_{\neq\eset}(\mc{Y})\neq \SExp_{P}(\mc{Y})$.

   \end{exmp}

 \subsection{}\label{xA.9}  Condition (ii) in the  following  definition isolates   a rather subtle     special  property  of a semidual pair of cones which plays an important role in Section \ref{x11}.  
\begin{defn}
A  semidual pair  $P=(\mc{Y},\mc{Z})$ of pointed cones  (with respect to $\mpair{-,-}$)  is a \emph{dual pair} of cones  if the following conditions hold:\begin{conds}
\item  $\mc{Y}^{*}=\mc{Z}^{**}$ (equivalently, $\mc{Z}^{*}=\mc{Y}^{**}$).
\item $\Ext_{\neq\eset}(\mc{Y})= \SExp_{P}(\mc{Y})$ and 
$\Ext_{\neq\eset}(\mc{Z})= \SExp_{P^{\text{\rm tr}}}(\mc{Z})$.
\end{conds}
  \end{defn} 
  The discussion in  \ref{xA.8} shows that  if $P$ is a dual pair, then
 $\Ext_{\neq\eset}(\mc{Y})= \ASExp(\mc{Y})$ and 
 $\Ext_{\neq\eset}(\mc{Z})= \ASExp(\mc{Z})$.  
  Example \ref{xA.8} shows that even  a semidual pair 
  $(\mc{Y},\mc{Z})$ of stable (i.e closed) cones need not be a dual pair in the  above sense. In our principal applications, $\mc{Y}$ and $\mc{Z}$ are not necessarily closed.

     \subsection{} \label{xA.10} A cone $\mc{Y}$ in $V$ is said to be polyhedral if $\mc{Y}=\real_{\geq 0}\Gamma$ for some finite subset $\Gamma\seq V$. 
Standard properties of polyhedral cones show they provide examples of dual pairs of stable cones.
\begin{lem}  Let $\mc{Y}$ be a polyhedral cone in $V$. 
Then \begin{num}
\item $\mc{Z}:=\mc{Y}^{*}$ is a  polyhedral cone with $\mc{Z}^{*}=\mc{Y}$.
\item $\mc{Y}$ and $\mc{Z}$ are closed and  have finitely many faces.
\item Every face of $\mc{Y}$ or $\mc{Z}$ is a polyhedral cone.
\item Every face of $\mc{Y}$ is an  exposed face and every face of a face of $\mc{Y}$ is a face of $\mc{Y}$. Similarly for $\mc{Z}$.
\item  If $\mc{Y}$, $\mc{Z}$ are pointed, then $(\mc{Y},\mc{Z})$ is a dual pair of stable cones.
 \end{num}
   \end{lem} 
   \begin{proof} Parts (a)--(d) are standard, and they  imply (e) by the definition of a dual pair of cones.
     \end{proof}
     
\subsection{}\label{xA.11} Part (a) of the Lemma below  recalls a standard correspondence between   compact convex sets  and closed salient cones, while part (b)  states   analogues  for cones of  Minkowski's and Straszewicz' Theorems  on  compact convex sets.
\begin{lem} Let $\mc{Y}$ be a non-empty cone in $V$.
\begin{num}
\item
  The following conditions are equivalent:
\begin{subconds}
\item $\mc{Y}$ is closed and salient.
\item $\mc{Y}$ is closed and $\mc{Y}^{*}$ is generating..
\item $\mc{Y}$ is closed  and $\mc{Y}^{*}$ has an interior point in $U$.
\item $\mc{Y}=\mc{Z}^{*}$ for some $\mc{Z}\seq U$ which has an interior point in $U$.
\item There is an affine subset $H$ of $V$ with $0\not\in H$
such that $H\cap \mc{Y}$ is a compact (necessarily convex) base of $\mc{Y}$.
\item  There exists some   compact convex base $B$ of $\mc{Y}$.
\end{subconds}

\item  Let $\mc{Y}$ be a closed salient cone in $V$. Then $\mc{Y}=
\real_{\geq 0}\Gamma=\cl(\real_{\geq 0}\Gamma')$
where $\Gamma$ (resp., $\Gamma'$) is the union of the extreme (resp., exposed) rays of $\mc{Y}$.  Further, $\G$ is the minimum (under inclusion)  union of rays of $\mc{Y}$  with $\mc{Y}=\real_{\geq 0}\G$.
\end{num}
\end{lem}
\begin{rem}
If $V\neq 0$ and $\mc{Z}$ is as in $\text{\rm (a)(iii)}$, then
 $\text{\rm (iv)}$ is satisfied by taking  $H$ to be  the affine 
 hyperplane 
$H=\mset{v\in V\mid\mpair{v,u}=1}$, where $u$ is  any interior 
point  of $\mc{Z}$. In that case, the sets $U_{\epsilon}:=\mset{v\in \mc{Y}\mid \mpair{v,u}<\epsilon}$ form a basis of neighborhoods of $0$ in $\mc{Y}$, with compact closures $\ol{U_{\epsilon}}=\mset{v\in \mc{Y}\mid \mpair{v,u}\leq\epsilon}$.
\end{rem}
\begin{proof}
For (a),  \cite[Exercise 2.13]{FlVan}, \cite[8.6]{Bar}, and  \cite[Ch II, \S 7, no. 3, esp. Exemples 1]{BourEsp}.
For (b), take a base $B=H\cap \mc{Y}$ as in (a)(v). Then the extreme rays of $\mc{Y}$ are the rays spanned by the extreme points of $B$
(see \cite[8.4]{Bar}), and the assertions involving them follows from
Minkowski's theorem (i.e. the sharp form of the finite dimensional Krein-Milman theorem; see \cite[Thorem (3.3)]{Bar}). 
The other part of (b) follows similarly from Straszewicz' Theorem (see \cite[Theorem 2.6.21]{Web}  and \cite[Exercise 3.14]{FlVan}).
\end{proof}
\subsection{} \label{xA.12} Let $\mc{Y}$ be a closed salient cone in $V$  and let $B$ be a compact convex base of $\mc{Y}$.
The map  $\mc{U}\to K:=\mc{U}\cap B$ induces a bijection between the subsets of $\mc{Y}$ which are pointed, possibly non-convex cones, and the subsets $K$ of $B$. The inverse bijection is given by $K\mapsto \mc{U}=\real_{\geq 0}K$.  Useful properties of this correspondence are listed below.
\begin{lem}    Let $\mc{U}\seq\mc{Y}$ be a pointed, possibly non-convex cone, and set ${K}:=\mc{U}\cap {B}$.
\begin{num}
\item  $\cl(\mc{U})$ is a pointed, possibly non-convex cone satisfying  $\cl(\mc{U})\cap {B}=\cl({K})$.
 \item   $\conv(\mc{U})=\real_{\geq 0}\mc{U}$ is a  pointed cone satisfying    $\real_{\geq 0}\mc{U}\cap {B}=\conv({K})$.
 \item $K':=\cl(\conv(K))=\conv(\cl(K))$.
\item $\mc{U}':=\cl(\real_{\geq 0}\mc{U})=\real_{\geq 0}\cl(\mc{U})$. 
\item  $\mc{U}'\cap K=K'$. \end{num} 
\begin{proof}
The straightforward proofs of (a)--(b) are omitted.
Compactness of $\cl(K)\seq B$ implies, by 
a  standard consequence (\cite[Ch I, Corollary (2.4)]{Bar}) of Carath\'eodory's theorem, that $\conv(\cl(K))$ is compact.  
The second equality   in (c) therefore holds since both its sides are the inclusion-smallest closed convex set containing $K$.
By (a)--(b),  $\cl(\real_{\geq 0}\mc{U})$ and $\real_{\geq 0}\cl(\mc{U})$ are pointed cones in $\mc{Y}$. By  (a)--(c),  these two cones  have the same intersection $K'$ with $K$. The rest of the lemma follows. 
\end{proof}

\end{lem}


\begin{thebibliography}{10}

\bibitem{Bar}
Alexander Barvinok.
\newblock {\em A course in convexity}, volume~54 of {\em Graduate Studies in
  Mathematics}.
\newblock American Mathematical Society, Providence, RI, 2002.

\bibitem{BjBr}
Anders Bj{\"o}rner and Francesco Brenti.
\newblock {\em Combinatorics of {C}oxeter groups}, volume 231 of {\em Graduate
  Texts in Mathematics}.
\newblock Springer, New York, 2005.

\bibitem{DySd}
C{\'e}dric Bonnaf{\'e} and Matthew~J. Dyer.
\newblock Semidirect product decomposition of {C}oxeter groups.
\newblock {\em Comm. Algebra}, 38(4):1549--1574, 2010.

\bibitem{Bour}
N.~Bourbaki.
\newblock {\em \'{E}l\'ements de math\'ematique. {F}asc. {XXXIV}. {G}roupes et
  alg\`ebres de {L}ie. {C}hapitre {IV}: {G}roupes de {C}oxeter et syst\`emes de
  {T}its. {C}hapitre {V}: {G}roupes engendr\'es par des r\'eflexions.
  {C}hapitre {VI}: syst\`emes de racines}.
\newblock Actualit\'es Scientifiques et Industrielles, No. 1337. Hermann,
  Paris, 1968.

\bibitem{BourEsp}
Nicolas Bourbaki.
\newblock {\em Espaces vectoriels topologiques. {C}hapitres 1 \`a 5}.
\newblock Masson, Paris, new edition, 1981.
\newblock {\'E}l{\'e}ments de math{\'e}matique. [Elements of mathematics].

\bibitem{BH1}
Brigitte Brink and Robert~B. Howlett.
\newblock A finiteness property and an automatic structure for {C}oxeter
  groups.
\newblock {\em Math. Ann.}, 296(1):179--190, 1993.

\bibitem{BH2}
Brigitte Brink and Robert~B. Howlett.
\newblock Normalizers of parabolic subgroups in {C}oxeter groups.
\newblock {\em Invent. Math.}, 136(2):323--351, 1999.

\bibitem{Br}
Arne Br{\o}ndsted.
\newblock {\em An introduction to convex polytopes}, volume~90 of {\em Graduate
  Texts in Mathematics}.
\newblock Springer-Verlag, New York, 1983.

\bibitem{Ca}
Roger~W. Carter.
\newblock {\em Finite groups of {L}ie type}.
\newblock Wiley Classics Library. John Wiley \& Sons Ltd., Chichester, 1993.

\bibitem{DavPr}
B.~A. Davey and H.~A. Priestley.
\newblock {\em Introduction to lattices and order}.
\newblock Cambridge University Press, New York, second edition, 2002.

\bibitem{dlH}
Pierre de~la Harpe.
\newblock Groupes de {C}oxeter infinis non affines.
\newblock {\em Exposition. Math.}, 5(1):91--96, 1987.

\bibitem{D}
Vinay~V. Deodhar.
\newblock On the root system of a {C}oxeter group.
\newblock {\em Comm. Algebra}, 10(6):611--630, 1982.

\bibitem{DyTh}
M.~J. Dyer.
\newblock {\em {Hecke} Algebras and Reflections in {Coxeter} Groups}.
\newblock PhD thesis, Univ. of Sydney, 1987.

\bibitem{DyQuo}
M.~J. Dyer.
\newblock Quotients of twisted {B}ruhat orders.
\newblock {\em J. Algebra}, 163(3):861--879, 1994.

\bibitem{RankTwo}
M.~J. Dyer.
\newblock Rank two detection of singularities of {S}chubert varieties.
\newblock {\em {\rm Unpublished manuscript}, {\tt
  http://www.nd.edu/$\sim$dyer/papers/index.html}}, 2001.

\bibitem{DyRig}
M.~J. Dyer.
\newblock On rigidity of abstract root systems of {C}oxeter groups.
\newblock {\em {\tt arXiv:1011.2270 [math.GR]}}, 2010.

\bibitem{Ref}
Matthew Dyer.
\newblock Reflection subgroups of {C}oxeter systems.
\newblock {\em J. Algebra}, 135(1):57--73, 1990.

\bibitem{DyRef}
Matthew Dyer.
\newblock Reflection subgroups of {C}oxeter systems.
\newblock {\em J. Algebra}, 135(1):57--73, 1990.

\bibitem{BrGr}
Matthew Dyer.
\newblock On the ``{B}ruhat graph'' of a {C}oxeter system.
\newblock {\em Compositio Math.}, 78(2):185--191, 1991.

\bibitem{DyParClos}
Matthew Dyer.
\newblock On parabolic closures in {C}oxeter groups.
\newblock {\em J. Group Theory}, 13(3):441--446, 2010.

\bibitem{DHR}
Matthew Dyer, Christophe Hohlweg, and Vivien Ripoll.
\newblock Asymptotical behaviour of roots of infinite {C}oxeter groups ii.
\newblock {\em In preparation}, 2012.

\bibitem{Edgar}
Tom Edgar.
\newblock {\em Dominance and regularity in {C}oxeter groups}.
\newblock ProQuest LLC, Ann Arbor, MI, 2009.
\newblock Thesis (Ph.D.)--University of Notre Dame.

\bibitem{FlVan}
Monique Florenzano and Cuong Le~Van.
\newblock {\em Finite dimensional convexity and optimization}, volume~13 of
  {\em Studies in Economic Theory}.
\newblock Springer-Verlag, Berlin, 2001.
\newblock In cooperation with Pascal Gourdel.

\bibitem{FuDomHeir}
Xiang Fu.
\newblock The dominance hierarchy in root systems of {C}oxeter groups.
\newblock {\em {\tt arXiv:1108.2940 [math.RT]}}, 2010.

\bibitem{FuCone}
Xiang Fu.
\newblock Coxeter groups, imaginary cones and dominance.
\newblock {\em {\tt arXiv:1108.5232 [math.RT]}}, 2011.

\bibitem{FuNonortho}
Xiang Fu.
\newblock Non-orthogonal geometric realizations of {C}oxeter groups.
\newblock {\em {\tt arXiv:1112.3429 [math.RT]}}, Preprint, 2011.

\bibitem{Gl}
Helge Gl{\"o}ckner.
\newblock Positive definite functions on infinite-dimensional convex cones.
\newblock {\em Mem. Amer. Math. Soc.}, 166(789):xiv+128, 2003.

\bibitem{HeeIm}
J.-Y. H\'ee.
\newblock Le c\^{o}ne imaginaire d'une base de racines sur $\mathbb{R}$.
\newblock {\em Unpublished manuscript}.

\bibitem{HeeTh}
J.-Y. H\'ee.
\newblock {\em Sur la torsion de {S}teinberg-{R}ee des groupes de {C}hevalley
  et des groupes de {K}ac-{M}oody}.
\newblock PhD thesis, Universit\'e de Paris-Sud, Orsay, 1993.

\bibitem{HLR}
Christophe Hohlweg, Jean-Phillipe Labb\'e, and Vivien Ripoll.
\newblock Asymptotical behaviour of roots of infinite {C}oxeter groups {I}.
\newblock {\em \tt arXiv:1112.5415v2 [math.GR]}, 2012.

\bibitem{HRT}
R.~B. Howlett, P.~J. Rowley, and D.~E. Taylor.
\newblock On outer automorphism groups of {C}oxeter groups.
\newblock {\em Manuscripta Math.}, 93(4):499--513, 1997.

\bibitem{Hum}
James~E. Humphreys.
\newblock {\em Reflection groups and {C}oxeter groups}, volume~29 of {\em
  Cambridge Studies in Advanced Mathematics}.
\newblock Cambridge University Press, Cambridge, 1990.

\bibitem{Kac}
Victor~G. Kac.
\newblock {\em Infinite-dimensional {L}ie algebras}.
\newblock Cambridge University Press, Cambridge, 1990.

\bibitem{K}
Daan Krammer.
\newblock The conjugacy problem for {C}oxeter groups.
\newblock {\em Groups Geom. Dyn.}, 3(1):71--171, 2009.

\bibitem{Loo}
Eduard Looijenga.
\newblock Invariant theory for generalized root systems.
\newblock {\em Invent. Math.}, 61(1):1--32, 1980.

\bibitem{MV}
G.~A. Margulis and {\`E}.~B. Vinberg.
\newblock Some linear groups virtually having a free quotient.
\newblock {\em J. Lie Theory}, 10(1):171--180, 2000.

\bibitem{Max}
George Maxwell.
\newblock The normal subgroups of finite and affine {C}oxeter groups.
\newblock {\em Proc. London Math. Soc. (3)}, 76(2):359--382, 1998.

\bibitem{Mok1}
C.~Mokler.
\newblock {\em Die Monoidvervollst\"andigung einer {K}ac-{M}oody-{G}ruppe,
  dissertation}.
\newblock PhD thesis, Hamburg, 1996.

\bibitem{Mok7}
C.~Mokler.
\newblock An algebraic geometric model of an action of the face monoid
  associated to a {K}ac-{M}oody group on its building.
\newblock {\em \tt arXiv:0906.5059 [math.RT]}, 2009.

\bibitem{Mok2}
Claus Mokler.
\newblock An analogue of a reductive algebraic monoid whose unit group is a
  {K}ac-{M}oody group.
\newblock {\em Mem. Amer. Math. Soc.}, 174(823):vi+90, 2005.

\bibitem{Mok6}
Claus Mokler.
\newblock Actions of the face monoid associated to a {K}ac-{M}oody group on its
  building.
\newblock {\em J. Algebra}, 321(9):2384--2421, 2009.

\bibitem{NuidaLocPar}
Koji Nuida.
\newblock Parabolic subgroups in {C}oxeter groups of arbitrary ranks.
\newblock {\em {\tt arXiv:1106.4709 [math.GR]}}, 2011.

\bibitem{RVa}
Werner~C. Rheinboldt and James~S. Vandergraft.
\newblock A simple approach to the {P}erron-{F}robenius theory for positive
  operators on general partially-ordered finite-dimensional linear spaces.
\newblock {\em Math. Comp.}, 27:139--145, 1973.

\bibitem{Rock}
R.~Tyrrell Rockafellar.
\newblock {\em Convex analysis}.
\newblock Princeton Landmarks in Mathematics. Princeton University Press,
  Princeton, NJ, 1997.
\newblock Reprint of the 1970 original, Princeton Paperbacks.

\bibitem{Sl1}
P.~Slodowy.
\newblock {\em Singularit\"aten{K}ac-{M}oody-Liealgberen und assoziierte
  {G}ruppen, Habilitationsschrift}.
\newblock PhD thesis, Bonn, 1984.

\bibitem{Sl2}
Peter Slodowy.
\newblock An adjoint quotient for certain groups attached to {K}ac-{M}oody
  algebras.
\newblock In {\em Infinite-dimensional groups with applications ({B}erkeley,
  {C}alif., 1984)}, volume~4 of {\em Math. Sci. Res. Inst. Publ.}, pages
  307--333. Springer, New York, 1985.

\bibitem{StW}
Josef Stoer and Christoph Witzgall.
\newblock {\em Convexity and optimization in finite dimensions. {I}}.
\newblock Die Grundlehren der mathematischen Wissenschaften, Band 163.
  Springer-Verlag, New York, 1970.

\bibitem{Va}
James~S. Vandergraft.
\newblock Spectral properties of matrices which have invariant cones.
\newblock {\em SIAM J. Appl. Math.}, 16:1208--1222, 1968.

\bibitem{V}
{\`E}.~B. Vinberg.
\newblock Discrete linear groups that are generated by reflections.
\newblock {\em Izv. Akad. Nauk SSSR Ser. Mat.}, 35:1072--1112, 1971.

\bibitem{Wat}
Mark~E. Watkins.
\newblock Infinite paths that contain only shortest paths.
\newblock {\em J. Combin. Theory Ser. B}, 41(3):341--355, 1986.

\bibitem{Web}
Roger Webster.
\newblock {\em Convexity}.
\newblock Oxford Science Publications. The Clarendon Press Oxford University
  Press, New York, 1994.

\bibitem{Xi}
Nanhua Xi.
\newblock Lusztig's {$A$}-function for {C}oxeter groups with complete graphs.
\newblock {\em Bull. Inst. Math. Acad. Sin. (N.S.)}, 7(1):71--90, 2012.

\end{thebibliography}
  \end{document}